\documentclass[a4paper,10pt]{amsart}
\usepackage{amsthm,amssymb,amsmath,amsfonts,mathrsfs,amscd}
\usepackage[left=2.5cm,right=2.5cm,top=2.5cm,bottom=2.5cm]{geometry}

\usepackage{hyperref}
\usepackage{bm}
\usepackage{color}  
\usepackage[all]{xy}
\hypersetup{
    colorlinks=true, 
    linktoc=all,     
    linkcolor=blue,  
}

\newcommand{\Q}{\mathbb{Q}}
\newcommand{\C}{\mathbb{C}}
\newcommand{\R}{\mathbb{R}}
\newcommand{\Z}{\mathbb{Z}}

\newcommand{\A}{\mathbb{A}}

\newcommand{\SL}{\operatorname{SL}}
\newcommand{\GL}{\operatorname{GL}}
\newcommand{\PGL}{\operatorname{PGL}}
\newcommand{\GSp}{\operatorname{GSp}}
\newcommand{\PGSp}{\operatorname{PGSp}}
\newcommand{\Sp}{\operatorname{Sp}}
\newcommand{\SO}{\operatorname{SO}}

\newcommand{\GSO}{\operatorname{GSO}}

\newcommand{\Tr}{\operatorname{Tr}}
\newcommand{\M}{\operatorname{M}}

\newcommand{\Sym}{\operatorname{Sym}}
\newcommand{\SK}{\operatorname{SK}}
\newcommand{\vol}{\operatorname{vol}}

\newtheorem{theorem}{Theorem}[section]
\newtheorem{proposition}[theorem]{Proposition}
\newtheorem{rem}[theorem]{Remark}
\newtheorem{lemma}[theorem]{Lemma}
\newtheorem{corollary}[theorem]{Corollary}

\begin{document}
\title{Pullbacks of Saito--Kurokawa lifts and a central value formula for degree $6$ $L$-series}
\author{Aprameyo Pal, Carlos de Vera-Piquero}

\begin{abstract}
 We prove an explicit central value formula for a family of complex $L$-series of degree $6$ for $\mathrm{GL}_2 \times \mathrm{GL}_3$ which arise as factors of certain Garret--Rankin triple product $L$-series associated with modular forms. Our result generalizes a previous formula of Ichino involving Saito--Kurokawa lifts, and as an application we prove Deligne's conjecture stating the algebraicity of the central values of the considered $L$-series up to the relevant periods. 
\end{abstract}




\address{A. Pal: Fakult\"at f\"ur Mathematik, Universit\"at Duisburg-Essen, Essen, Germany.}
\email{aprameyo.pal@uni-due.de}

\address{C. de Vera-Piquero: Facultat de Matem\`atiques i Inform\`atica, Universitat de Barcelona, Barcelona, Spain.}
\email{cdeverapiquero@gmail.com}

\maketitle

\section{Introduction}

Explicit central value formulas for $L$-series associated with modular forms have always been of interest in number theory. In this paper we prove a central value formula for certain $L$-series of degree $6$, generalizing a result of Ichino \cite{Ichino-pullbacks}, which involves pullbacks of Saito--Kurokawa lifts. This can be seen as yet another evidence of the key role that pullbacks of Siegel Eisenstein series or cusp forms play in the proof of the algebraicity of critical values of certain automorphic $L$-functions. Previous instances of this phenomenon are found for example in the works of Garrett \cite{GarretPullbacks, GarretTriple}, B\"ocherer \cite{Bocherer}, or B\"ocherer--Furusawa--Schulze-Pillot \cite{BFS}, or also in a similar flavour in Ichino--Ikeda \cite{IchinoIkeda}, where special values of certain triple product $L$-series are related with pullbacks of hermitian Maass lifts. All these results, as well as the result in this paper, fit within the range of the `refined global Gross--Prasad conjecture' (cf. \cite{IchinoIkeda-periods}), reflecting the intimate relation between certain periods of automorphic forms on special orthogonal groups and $L$-values.

In order to describe the setting considered in this article, let $k\geq 1$ be an odd integer, and let $f \in S_{2k}^{new}(\Gamma_0(N_f))$ and $g \in S_{k+1}^{new}(\Gamma_0(N_g),\chi)$ be two normalized newforms of weights $2k$ and $k+1$, and levels $N_f$ and $N_g$, respectively. We assume $f$ has trivial nebentypus, whereas $g$ has nebentypus $\chi$ (hence $\chi$ is a Dirichlet character modulo $N_g$). Write $V_{\ell}(f)$ (resp. $V_{\ell}(g)$) for the compatible system of $\ell$-adic Galois representations attached to $f$ (resp. $g$), and denote by $\mathrm{Ad}(V_{\ell}(g))$ the so-called {\em adjoint representation} of $V_{\ell}(g)$. In this paper we are concerned with the complex $L$-series $L(f \otimes \mathrm{Ad}(g),s)$ of degree $6$ for $\GL_2 \times \GL_3$ associated with the tensor product $V_{\ell}(f) \otimes \mathrm{Ad}(V_{\ell}(g))$. This $L$-series can be defined by an Euler product for $\mathrm{Re}(s) \gg 0$, whose local factors at primes $p\nmid N_fN_g$ are given as in \cite[p. 559]{Ichino-pullbacks}. The completed $L$-series 
\[
 \Lambda(f\otimes \mathrm{Ad}(g),s) := \Gamma_{\C}(s)\Gamma_{\C}(s+k)\Gamma_{\C}(s-k+1)L(f \otimes \mathrm{Ad}(g),s),
\]
where $\Gamma_{\C}(s) = 2(2\pi)^{-s}\Gamma(s)$ is the usual complex Gamma function, admits analytic continuation to the whole complex plane and satisfies a functional equation relating its values at $s$ and $2k-s$, with sign 
$\varepsilon(f \otimes\mathrm{Ad}(g)) \in \{\pm 1\}$. Under certain hypotheses, which in particular guarantee that the sign $\varepsilon(f \otimes\mathrm{Ad}(g))$ is $+1$, the main result of this paper is an {\em explicit central value formula} for $\Lambda(f\otimes \mathrm{Ad}(g),k)$. As an immediate corollary, we deduce the algebraicity of such value up to a suitable period, as expected by Deligne's conjecture.

As the eager reader might have already suspected, the $L$-series $L(f \otimes \mathrm{Ad}(g),s)$ is closely related to a suitable triple product Garret--Rankin $L$-series. Indeed, let $f' := f \otimes \chi^{-1}$ be the twist of $f$ by the inverse of the character $\chi$. By construction, the motive associated to the triple tensor product $V_{\ell}(f') \otimes V_{\ell}(g)\otimes V_{\ell}(g)$ is self-dual, and hence the Garret--Rankin $L$-series $L(f'\otimes g\otimes g,s)$ attached to it (or rather, its completed $L$-series) satisfies a functional equation relating its values at $s$ and $4k-s$, with sign $\varepsilon(f',g,g) \in \{\pm 1\}$. In view of the isomorphisms 
\[
 V_{\ell}(g) \otimes V_{\ell}(g) \simeq \det(V_{\ell}(g)) \oplus \mathrm{Sym}^2(V_{\ell}(g)) \simeq \det(V_{\ell}(g)) \otimes \left( \mathbf 1 \oplus \mathrm{Ad}(V_{\ell}(g))\right),
\]
where $\mathrm{Sym}^2(V_{\ell}(g))$ stands for the {\em symmetric square} representation of $V_{\ell}(g)$ and we use that $\mathrm{Sym}^2(V_{\ell}(g)) \simeq \mathrm{Ad}(V_{\ell}(g)) \otimes \det(V_{\ell}(g))$, Artin formalism provides a factorization of complex $L$-series 
\begin{equation}\label{complexfactorization}
 L(f' \otimes g \otimes g, s)  = L(f,s-k)L(f\otimes\mathrm{Ad}(g),s-k).
\end{equation}

It is well-known that the completed $L$-series $\Lambda(f,s) := \Gamma_{\C}(s)L(f,s)$ satisfies a functional equation relating its values at $s$ and $2k-s$, with sign $\varepsilon(f) \in \{ \pm 1\}$, thus the central critical point for the shifted $L$-series $\Lambda(f,s-k)$ is at $s = 2k$. Concerning the central values, suppose that $\Lambda(f \otimes \mathrm{Ad}(g),k)$ is non-zero. If $\Lambda(f,k)$ is non-zero as well, then one can use  \eqref{complexfactorization} straightforward to express the central value $\Lambda(f\otimes \mathrm{Ad}(g),k)$ as a ratio between the central values $\Lambda(f'\otimes g \otimes g, 2k)$ and $\Lambda(f,k)$. When $\Lambda(f,k)$ vanishes, however, the identity in \eqref{complexfactorization} is not directly giving a way to obtain an expression for $\Lambda(f\otimes \mathrm{Ad}(g),k)$.

Despite the above relation to triple product central $L$-values, the approach in this article to obtain an {\em explicit} central value formula for $\Lambda(f \otimes \mathrm{Ad}(g),k)$ does not require determining triple product central $L$-values. Instead, as pointed out at the beginning of the introduction, we generalize a result of Ichino involving Saito--Kurokawa lifts. As a motivation towards our result, suppose in the previous discussion that $N_f = N_g = 1$, hence $\chi$ is trivial as well and $f$ and $g$ are normalized newforms for the full modular group $\Gamma_0(1) = \SL_2(\Z)$. In this case, $f'= f$, and \eqref{complexfactorization} reads 
\[
 L(f\otimes g\otimes g,s)  = L(f,s-k)L(f\otimes \mathrm{Ad}(g),s).
\]
Our choice of weights makes that $\varepsilon(f,g,g) = \varepsilon(f) = -1$, and therefore the sign in the functional equation for $\Lambda(f\otimes \mathrm{Ad}(g),s)$ is $+1$. In this particular setting, Ichino proved in \cite{Ichino-pullbacks} an explicit formula for $\Lambda(f \otimes \mathrm{Ad}(g),k)$, involving a half-integral weight modular form $h \in S_{k+1/2}(\Gamma_0(4))$ associated with $f$ by the Shimura correspondence and its Saito--Kurokawa lift $F \in S_{k+1}(\mathrm{Sp}_2(\Z))$, which is a Siegel modular form of degree $2$. In terms of these lifts, Ichino's formula reads 
\begin{equation}\label{IchinoFormula}
 \Lambda(f\otimes \mathrm{Ad}(g),k) = 2^{k+1} \frac{\langle f,f\rangle}{\langle h,h\rangle} \frac{|\langle F_{|\mathcal H\times\mathcal H},g\times g \rangle|^2}{\langle g,g\rangle^2},
\end{equation}
where $F_{|\mathcal H\times\mathcal H}$ denotes the restriction (or `pullback') of $F$ to $\mathcal H \times \mathcal H$, embedded `diagonally' in Siegel's upper half space $\mathcal H_2$. Our main result can be seen as a generalization of \eqref{IchinoFormula}, when removing the assumption that $N_f = N_g = 1$. However, instead of extending Ichino's arguments, our strategy relies on a more recent result by Qiu \cite{Qiu}.

Indeed, there is a decomposition formula for the $\mathrm{SO}(4)$-period $\mathcal P$ associated with (the restriction of) a Saito--Kurokawa representation of $\PGSp_2$ and an irreducible cuspidal unitary representation of $\mathrm{GSO}(4)$. Here, $\mathrm{GSO}(4)$ and $\mathrm{SO}(4)$ stand for the group of similitudes and the special orthogonal group of a certain $4$-dimensional (split) quadratic space, and $\PGSp_2$ is identified with the special orthogonal group of a suitable $5$-dimensional quadratic space. The proof of this decomposition result by Qiu in fact reduces to a decomposition formula for a global $\SL_2$-period $\mathcal Q$, proved in the same article. This $\SL_2$-period, and the interplay between $\mathcal P$ and $\mathcal Q$, plays a central role in the proof of our main result.

To illustrate our strategy, consider again the general setting in which $f$ and $g$ are of level $N_f$ and $N_g$, respectively, and $\chi$ is not necessarily trivial. Then let $\pi$ (resp. $\tau$) be the automorphic representation of $\PGL_2(\A)$ (resp. $\GL_2(\A)$) associated with $f$ (resp. $g$). The Shimura correspondence, settled and investigated in detail by Waldspurger \cite{Waldspurger80, Waldspurger81} as a theta correspondence for the pair $(\PGL_2,\widetilde{\SL}_2)$, associates to $\pi$ (and a choice of non-trivial additive character of $\A/\Q$) a near equivalence class of automorphic representations $\tilde{\pi}$ of the double metaplectic cover $\widetilde{\SL}_2(\A)$ of $\SL_2(\A)$. Classical Shimura lifts of $f$ give rise to automorphic forms in the representations $\tilde{\pi}$ arising in this theta correspondence. Associated with the representations $\tilde{\pi}$, $\tau$, and a Weil representation $\omega$ depending on the fixed additive character of $\A/\Q$, there is a (global) $\SL_2$-period functional
\[
 \mathcal Q: \tilde{\pi} \otimes \tilde{\pi} \otimes \tau \otimes \tau \otimes \omega \otimes \omega \, \longrightarrow \, \C
\]
(cf. Section \ref{sec:proof} for its precise definition), which by virtue of \cite[Theorem 4.5]{Qiu} decomposes (when it is non-zero) up to certain special $L$-values as a product of local periods 
\[
 \mathcal I_v: \tilde{\pi}_v \otimes \tilde{\pi}_v \otimes \tau_v \otimes \tau_v \otimes \omega_v \otimes \omega_v \, \longrightarrow \, \C
\]
defined by integrating local matrix coefficients. Among the $L$-values showing up in this decomposition formula, one finds $L(1/2,\pi \times \mathrm{ad}\tau)$, which corresponds with the central value $\Lambda(f \otimes \mathrm{Ad}(g),k)$ that we are interested in. Moreover, the non-vanishing of the functional $\mathcal Q$ is essentially controlled by the non-vanishing of the special value $L(1/2,\pi \times \mathrm{ad}\tau)$ (cf. Propositions \ref{prop:nonvanishingQ} and \ref{prop:nonvanishing}). Hence, one can obtain an explicit expression for $\Lambda(f \otimes \mathrm{Ad}(g),k)$ by finding a {\em test vector} on which $\mathcal Q$ does not vanish, and computing the local periods $\mathcal I_v$ when evaluated at such test vector. Besides, as hinted above, the global period $\mathcal Q$ is related to the $\mathrm{SO}(4)$-period $\mathcal P$, when replacing the automorphic representations $\tilde{\pi}$ and $\tau$ with automorphic representations $\Pi$ and $\Upsilon$, of $\GSp_2$ and $\mathrm{GSO}(4)$ respectively, obtained from $\tilde{\pi}$ and $\tau$ via theta correspondence. It is via this relation with $\mathcal P$ that the global period $\mathcal Q$ evaluated at the test vector can be interpreted as a classical Petersson product, therefore leading to the aimed expression for $\Lambda(f \otimes \mathrm{Ad}(g),k)$ in purely classical terms. For example, in Ichino's setting described above, this global automorphic period is the responsible of the factor $|\langle F_{|\mathcal H \times \mathcal H}, g \times g\rangle|^2$ appearing in \eqref{IchinoFormula}. The main novelty of our work is the computation of the above mentioned local $\SL_2$-periods at ramified primes. It is important to remark that these local $\SL_2$-periods have their own interest, and their computation has potential applications in the study of the subconvexity problem for the family of automorphic $L$-functions of the form $L(s,\pi \times \mathrm{ad}\tau)$. This will be explored in a forthcoming work.

Although the strategy that we have just sketched works in a rather general setting, for the sake of clarity and to simplify the (already involved) local computations we will impose some assumptions on $f$ and $g$. Most importantly, we will assume that 
\begin{equation}\label{SF}
 N = N_f = N_g \text{ is {\em odd} and {\em square-free}.} \tag{SF}
\end{equation}
One could easily relax this assumption to require only that $N_f$ and $N_g$ are square-free (but not necessarily equal), at the cost of dealing with more cases when performing the computation of the local periods $\mathcal I_v$ alluded to above. However, we content ourselves with illustrating the method under the assumption \eqref{SF}. 

Besides, let $M$ denote the conductor of the Dirichlet character $\chi$. Thus $M$ is a positive divisor of $N$, and by \eqref{SF} $M$ is square-free as well. If we write $\chi = \prod_{p\mid M} \chi_{(p)}$, where $\chi_{(p)}$ is a Dirichlet character modulo $p$ for each prime $p\mid M$, then we assume that
\begin{equation}\label{Hyp1}
 \chi_{(p)}(-1) = -1 \text{ for all primes } p \mid M. \tag{H1}
\end{equation}
In particular, this implies that $M$ is the product of an even number of primes. We use hypothesis \eqref{Hyp1} to apply a generalized Kohnen formula due to Baruch and Mao \cite{BaruchMao}, recalled in Theorem \ref{thm:BM} below. Finally, it is well-known that the sign $\varepsilon(f) \in \{\pm 1\}$ in the functional equation for $L(f,s)$ might be written as a product of local signs $\varepsilon(f) = \prod_v \varepsilon_v(f)$, where $v$ varies over the rational places, $\varepsilon_v(f) = \varepsilon(1/2,\pi_v) \in \{\pm 1\}$ for all $v$, and $\varepsilon_v(f) = +1$ for all $v\nmid N\infty$.  We will assume that
\begin{equation}\label{Hyp2}
 \varepsilon_p(f) = -1 \text{ for all primes } p \mid M. \tag{H2}
\end{equation}
Under our previous assumptions, hypothesis \eqref{Hyp2} becomes crucial for the non-vanishing of the period functional $\mathcal Q$ (cf. Section \ref{sec:proof}). Observe that if $\chi$ is assumed to be trivial, then $M=1$ and hypotheses \eqref{Hyp1} and \eqref{Hyp2} are empty. In fact, for $\chi$ trivial, the arguments in this paper would be much less technical (for example, it would be enough to use Kohnen's formula instead of its generalization by Baruch--Mao, computations with the Saito--Kurokawa lift would be simpler, and the whole Section 8 would not be needed) and consequently the length of this note would be also considerably reduced.

A comment on signs is now in order. Indeed, the assumption that $k$ is {\em odd} implies $\varepsilon_{\infty}(f) = -1$, and since the triple of weights $(2k,k+1,k+1)$ is `balanced' (i.e. none of the three weights is at least the sum of the other two) one also has $\varepsilon_{\infty}(f'\otimes g \otimes g) = -1$. Besides, the assumption \eqref{SF} together with hypothesis \eqref{Hyp2} imply, by \cite[Section 8]{Prasad}, that $\prod_{p\mid N}\varepsilon_p(f) = \prod_{p\mid N}\varepsilon_p(f'\otimes g \otimes g)$. Therefore, $\varepsilon(f) = \varepsilon(f'\otimes g\otimes g)$, and it follows from \eqref{complexfactorization} that the sign in the functional equation for $\Lambda(f \otimes \mathrm{Ad}(g),s)$ is $+1$. 

Now we can finally state our main result. To do so, let $S_{k+1/2}^{+,new}(4NM,\chi)$ denote Kohnen's subspace of newforms of weight $k+1/2$, level $4NM$ and character $\chi$ (see Section \ref{MF:half-integral} for details). Under assumptions \eqref{SF} and \eqref{Hyp1}, it follows from \cite[Theorem 10.1]{BaruchMao} (cf. Theorem \ref{thm:BM} below) that the subspace of $S_{k+1/2}^{+,new}(4NM,\chi)$ consisting of newforms whose eigenvalues for the Hecke operators at primes $p\nmid 2N$ coincide with those of $f \otimes \chi$ is one-dimensional. Let $h$ be any non-trivial element in this one-dimensional subspace, thus $h$ is a {\em Shimura lift} of $f$ (or of $f\otimes \chi$), and let $F_{\chi} \in S_{k+1}(\Gamma_0^{(2)}(N),\chi)$ be the Saito--Kurokawa lift of $h$, as defined in Section \ref{MF:Siegel}; this is a Siegel cusp form of degree $2$, weight $k+1$, Hecke-type level $\Gamma_0^{(2)}(N)$, and character $\chi$. One may think of $F_{\chi}$ as a Saito--Kurokawa lift of $f \otimes \chi$. 

\begin{theorem}\label{mainthm}
Let $k, N\geq 1$ be odd integers. Let $f \in S_{2k}^{new}(\Gamma_0(N))$ and $g \in S_{k+1}^{new}(\Gamma_0(N),\chi)$ be normalized newforms, and assume \eqref{SF}, \eqref{Hyp1}, and \eqref{Hyp2}. If $h$ and $F_{\chi}$ denote a Shimura lift of $f$ and its Saito--Kurokawa lift as explained above, then 
\begin{equation}\label{mainformula-intro}
 \Lambda(f \otimes \mathrm{Ad}(g), k) = 2^{k+1-\nu(M)}C(N,M,\chi) \frac{\langle f,f \rangle}{\langle h,h\rangle} 
 \frac{|\langle (\mathrm{id}\otimes U_M) F_{\chi |\mathcal H \times \mathcal H}, g \times g\rangle|^2}{\langle g,g\rangle ^2},
\end{equation}
where $U_M = \prod_{p \mid M} U_p$ denotes the product of the usual $p$-th Hecke operators $U_p$, $\nu(M)$ denotes the number of prime divisors of $M$, and
\[
 C(N,M,\chi) = |\chi(2)|^{-2} M^{3-k}N^{-1} \prod_{p\mid N}(p+1)^2\prod_{p\mid M} (p+1). 
\]
\end{theorem}

In particular, under the assumptions of the theorem we have $\Lambda(f \otimes \mathrm{Ad}(g),k) \geq 0$.

\begin{rem}
 Let $f$ and $g$ be as in Theorem \ref{mainthm}, and suppose that $\chi$ is trivial. Then $M = 1$, and hypotheses \eqref{Hyp1} and \eqref{Hyp2} hold trivially, hence the central value formula \eqref{mainformula-intro} reads 
 \[
 \Lambda(f \otimes \mathrm{Ad}(g),k) = 2^{k+1}N^{-1} \prod_{p\mid N}(p+1)^2 \frac{\langle f,f \rangle}{\langle h,h\rangle} 
 \frac{|\langle F_{|\mathcal H \times \mathcal H}, g \times g\rangle|^2}{\langle g,g\rangle ^2}.
 \]
 This formula coincides with the one obtained by S.-Y. Chen in \cite{Chen}, which appeared after a first version of this paper was made available. Instead of using Qiu's decomposition theorems and computing local $\SL_2$-periods, Chen generalizes straightforward the original strategy of Ichino. Besides considering non-trivial nebentype character, as commented above the novelty of our approach is precisely the computation of local $\SL_2$-periods at ramified primes, which have their own interest and applications to other problems. If we further restrict to $N = 1$, observe that we obviously recover Ichino's formula in \eqref{IchinoFormula}.
\end{rem}

\begin{rem}\label{rmk:differentweights}
 If the weight of $g$ is assumed to be $\ell+1$ with $\ell\geq k$ odd, instead of $k+1$, then all the arguments to prove the above central value formula work through by replacing $h$ and $F_{\chi}$ with suitable nearly holomorphic forms obtained from these by applying the relevant derivative operators, and modifying accordingly the local computation for the archimedean period $\mathcal I_{\infty}$ (cf. Section \ref{sec:infty}). 
\end{rem}


Theorem \ref{mainthm} has an immediate application to Deligne's conjecture \cite{Deligne}:

\begin{corollary}\label{coro-alg}
Let $k, N\geq 1$ be odd integers, and $f$, $g$ be as in Theorem \ref{mainthm}. If $\Q(f,g)$ denotes the number field generated by the Fourier coefficients of $f$ and $g$, then 
 \[
  \Lambda(f \otimes \mathrm{Ad}(g),k)^{\mathrm{alg}} := \frac{\Lambda(f \otimes \mathrm{Ad}(g), k)}{\langle g,g\rangle^2 c^+(f)} \in \Q(f,g),
 \]
 where $c^+(f)$ denotes the period associated with the cuspidal form $f$ by Shimura as in \cite{ShimuraPeriods}.
\end{corollary}

\begin{rem}
 When $N=1$, this corollary follows of course from Ichino's formula. And in line with Remark \ref{rmk:differentweights}, when $N=1$ and $f$ and $g$ have weights $2k$ and $\ell+1$, respectively, with $\ell\geq k$ odd integers, the algebraicity of the relevant central value has been recently shown by H. Xue in \cite{HangXue-generalizedIchino}, working with Jacobi forms instead of Saito--Kurokawa lifts.
\end{rem}

Let us close this introduction by pointing out some applications of the present work. One direction that we want to explore aims for a $p$-adic analogue of the factorization of complex $L$-series given in \eqref{complexfactorization}. Even though one could define a two-variable $p$-adic $L$-function associated with $f \otimes \mathrm{Ad}(g)$ by making use of the above hinted relation of $L(f \otimes \mathrm{Ad}(g),s)$ with suitable triple product $L$-series, the explicit central value formula provided by Theorem \ref{mainthm} gives an alternative approach for the construction of such a $p$-adic $L$-function, which is independent of any triple product $p$-adic $L$-function. Namely, one can directly interpolate the explicit expression in Theorem \ref{mainthm} (or rather the algebraic parts as in Corollary \ref{coro-alg}) as $f$ and $g$ vary in Hida families. With this $p$-adic $L$-function at hand, the proof of a factorization of $p$-adic $L$-functions parallel to \eqref{complexfactorization} will require comparing two Euler systems. Namely, the one associated to generalized Gross--Kudla--Schoen diagonal cycles in the product of three Kuga--Sato varieties, and the one arising from Heegner points. This would extend the factorization result of Dasgupta in \cite{Dasgupta} establishing Greenberg's conjecture for the adjoint representation, where instead of a triple product $L$-series one considers a Rankin $L$-series of two cusp forms twisted by a Dirichlet character.

In a completely different direction, the computation of local $\SL_2$-periods leading to the explicit central value formula in Theorem \ref{mainthm}, when relaxing the assumption \eqref{SF} to requiring only that $N_f$ and $N_g$ are square-free, provides an important tool in the study of the subconvexity problem for the family of automorphic $L$-functions of the form $L(s, \pi \times \mathrm{ad} \tau)$. This problem is related to the limiting mass distribution of automorphic forms (`arithmetic quantum unique ergodicity'). See for instance \cite{Nelson1}, \cite{Nelson2}, and references therein.

\subsection{Outline of the paper}

Although the strategy of the proof of Theorem \ref{mainthm} has already been sketched above, let us briefly explain the organization of this paper. Sections 2 and 3 are devoted to recall and set some necessary background material concerning classical modular forms and automorphic forms. Theorem \ref{thm:BM}, due to Baruch--Mao, will play an important role later on in the paper. After this, we review in Section 4 the theory of quadratic spaces, theta functions, and theta lifts, describing also explicit models for quadratic spaces in low dimension that are used in this paper. Section 5 focuses in the three theta correspondences that are involved in the proof of the main result of this paper. In particular, we prove two explicit identities for theta lifts (see Propositions \ref{ThetaIdentity1} and \ref{ThetaIdentity2}) for the theta correspondences between $\GL_2$ and $\mathrm{GSO}(4)$, and between $\widetilde{\SL}_2$ and $\PGSp_2$, respectively, by adapting the ones in \cite[Section 5]{IchinoIkeda} and \cite[Section 7]{Ichino-pullbacks}. In Section 6 we can already prove Theorem \ref{mainthm}, and deduce Corollary \ref{coro-alg}, although the computation of the local periods $\mathcal I_v$ is relegated to Sections 7, 8, and 9, which constitute the most technical part of this article.

\subsection{Notation and measures}

Before closing this introduction, we collect here some notations to be used throughout the paper. We will denote by $\A = \A_{\Q}$ the ring of adeles over $\Q$. We will also write $\hat{\Z} = \prod_p \Z_p$ for the profinite completion of $\Z$, which will be regarded as a subring of $\A$.

We write $\zeta$ for Riemann's zeta function, with the usual Euler product $\zeta = \prod_p \zeta_p$, and $\zeta_{\Q}$ will stand for the completed Riemann zeta function defined by $\zeta_{\Q}(s) := \Gamma_{\R}(s)\zeta(s)$, where $\Gamma_{\R}(s) := \pi^{-s/2}\Gamma(s/2)$ and $\Gamma(s)$ denotes the Gamma function. We also put $\Gamma_{\C}(s) := 2(2\pi)^{-s}\Gamma(s)$.

If $G$ is a connected reductive group over $\Q$, we equip $G(\A)$ with the Tamagawa measure $dg$. The volume of $[G]:=G(\Q) \backslash G(\A)$ with respect to this measure is usually referred to as the Tamagawa number of $G$. For $\GL_2(\A)$, we write $dg = \prod_v dg_v$ as a product of local Haar measures, satisfying $\mathrm{vol}(\GL_2(\Z_p),dg_p) = 1$ for all finite primes $p$. In the case of $\SL_2(\A)$, we choose the local Haar measures to satisfy $\mathrm{vol}(\SL_2(\Z_p)) = \zeta_p(2)^{-1}$ for all finite primes $p$. 

If $V$ is a finite-dimensional quadratic space over $\Q$, with bilinear form $(\,,\,)$, and $\psi$ is an additive character of $\A/\Q$, then we consider the Haar measure on $V(\A)$ which is self-dual with respect to $\psi$, unless otherwise stated. That is to say, the Haar measure such that $\mathcal F(\mathcal F(\phi))(x) = \phi(-x)$, where $\mathcal F(x) = \int_{V(\A)} \phi(y) \psi((x,y))dy$ is the Fourier transform of $\phi$. The orthogonal group $\mathrm{O}(V)$ is {\em not} connected. We choose a measure on $\mathrm{O}(V)(\A)$ as follows: first, we equip $\mathrm{SO}(V)(\A)$ with the Tamagawa measure; secondly, at each place $v$ we extend the local measure on $\mathrm{SO}(V)(\Q_v)$ to the non-identity component of $\mathrm{O}(V)(\Q_v)$; and finally, we consider the measure $dh_v$ on $\mathrm{O}(V)(\Q_v)$ to be half of this extended measure, and define $dh = \prod_v dh_v$. This is the Tamagawa measure on $\mathrm{O}(V)(\A)$, and $[\mathrm{O}(V)] = \mathrm{O}(V)(\Q) \backslash \mathrm{O}(V)(\A)$ has volume $1$ 
with respect to $dh$.

Continue to consider a finite-dimensional quadratic space $V$ as before, and a non-trivial additive character of $\A/\Q$. If $\mathcal S(V(\A))$ denotes the space of Bruhat--Schwartz functions on $V(\A)$, and $\phi_1$, $\phi_2 \in \mathcal S(V(\A))$, we set $\langle \phi_1,\phi_2\rangle = \int_{V(\A)} \phi_1(x)\overline{\phi_2(x)} dx$, where $dx$ is the Haar measure that is self-dual with respect to $\psi$. If $\pi$ is an irreducible cuspidal unitary representation of $G(\A)$, and $f_1, f_2 \in \pi$, we define the pairing $\langle f_1, f_2 \rangle$ to be:
\begin{itemize}
 \item[i)] $\int_{[\SL_2]} f_1(g) \overline{f_2(g)} dg$, if $G = \widetilde{\SL}_2$;
 \item[ii)] $\int_{[\PGL_2]} f_1(g) \overline{f_2(g)} dg$, if $G = \GL_2$;
 \item[iii)] $\int_{[G]} f_1(g) \overline{f_2(g)} dg$, if $G = \mathrm{SO}(V)$ or $\mathrm{O}(V)$.
\end{itemize}

Finally, let $p$ be a prime, and fix a non-trivial additive character $\psi_p: \Q_p \to \C$ of conductor $\Z_p$. If $\mu: \Q_p^{\times} \to \C^{\times}$ is a character such that $\mu(p)=1$ (or equivalently, a character of $\Z_p^{\times}$), and $a \in \Q_p$, we define the Gauss sum 
\[
 \mathfrak G(a,\mu) := \int_{\Z_p^{\times}} \psi(ax)\mu(x) dx,
\]
where $dx$ is the Haar measure which is self-dual with respect to $\psi$, normalized so that $\Z_p$ has volume $1$. It is well-known that $\mathfrak G(a,\mu) = 0$ unless $\mu = \mathbf 1$ or $\mathrm{ord}_p(a) = -\mathrm{cond}(\mu)$, where $c = \mathrm{cond}(\mu)$ is the smallest positive integer such that $\mu$ is trivial on $1 + p^{c}\Z_p$. More precisely, we have 
\[
 \mathfrak G(a,\mu) = \begin{cases}
                       1 - p^{-1} & \text{if } \mu = \mathbf 1, \, \mathrm{ord}_p(a) \geq 0,\\
                       -p^{-1} & \text{if } \mu = \mathbf 1, \, \mathrm{ord}_p(a) = -1,\\
                       0 & \text{if } \mu = \mathbf 1, \, \mathrm{ord}_p(a) < -1,\\
                       |a|^{-1/2}\varepsilon(1/2,\mu^{-1})\mu^{-1}(a)  & \text{if } \mu \neq \mathbf 1, \, \mathrm{ord}_p(a) = -\mathrm{cond}(\mu),\\
                       0 & \text{if } \mu \neq \mathbf 1, \, \mathrm{ord}_p(a) \neq -\mathrm{cond}(\mu).
                      \end{cases}
\]
Here, $\varepsilon(1/2,\mu^{-1})$ is the (local) $\varepsilon$-factor associated with $\mu^{-1}$, which satisfies (among other properties) $\varepsilon(1/2,\mu)\varepsilon(1/2,\mu^{-1}) = \mu(-1)$.

\subsection*{Acknowledgements}

We would like to thank the {\em Essener Seminar f\"ur Algebraische Geometrie und Arithmetik} for providing a pleasant working environment, and particularly to Rodolfo Venerucci for his interest in this work and several fruitful discussions. During the elaboration of this work, both authors were financially supported by the SFB/TR 45 ``Periods, Moduli Spaces, and Arithmetic of Algebraic Varieties''. This project has also received funding from the European Research Council (ERC) under the European Union's Horizon 2020 research and innovation programme (grant agreement No 682152).

\section{Modular forms}\label{sec:MF}

The main purpose of this section is to recall and set up the notation regarding classical modular forms, with a particular focus on some features concerning distinct liftings between spaces of classical modular forms of different nature (namely, modular forms of integral weight, modular forms of half-integral weight, and Siegel modular forms of degree $2$) that will be used in the paper.

\subsection{Integral weight modular forms}\label{MF:integral}

Let $\mathcal H = \{z \in \C: \mathrm{Im}(z) > 0\}$ be the complex upper half plane, on which the group $\GL_2^+(\R)$ of real $2$-by-$2$ matrices with positive determinant acts by
\[
 \left(\begin{array}{cc} a & b \\ c & d\end{array}\right) \cdot z = \frac{az+b}{cz+d}.
\]

Let $\ell\geq 1$ be an integer, and consider the usual action of $\GL_2^+(\R)$ on the space of holomorphic functions $g: \mathcal H \to \C$ defined by 
\[
 g|_{\ell}[\gamma] (z) = j(\gamma,z)^{-\ell}\det(\gamma)^{\ell/2} g(\gamma z),
\]
where $j(\gamma,z) = cz + d$ if $\gamma = \left(\begin{smallmatrix} a & b \\ c & d\end{smallmatrix} \right)$. If $N\geq 1$ is an integer and $\chi: (\Z/N\Z)^{\times} \to \C^{\times}$ is a Dirichlet character modulo $N$, we write $M_{\ell}(N,\chi)$ for the $\C$-vector space of modular forms of weight $\ell$, level $N$ and nebentype character $\chi$. Namely, the space of holomorphic functions $g: \mathcal H \to \C$ satisfying 
\[
 g|_{\ell}[\gamma] = \chi(\gamma) g \quad \text{for all } \gamma \in \Gamma_0(N),
\]
together with the usual holomorphicity condition at the cusps. Here, $\Gamma_0(N)$ denotes the level $N$ Hecke-congruence subgroup of $\Gamma_0(1) = \SL_2(\Z)$, 
\[
 \Gamma_0(N) = \left\lbrace \left(\begin{array}{cc} a & b \\ c & d\end{array}\right) \in \SL_2(\Z): c \equiv 0 \pmod N\right\rbrace,
\]
on which $\chi$ induces a character, which we still denote $\chi$ by a slight abuse of notation, through the rule 
\[
 \left(\begin{array}{cc} a & b \\ c & d\end{array}\right) \, \longmapsto \, \chi(d).
\]
We denote by $S_{\ell}(N,\chi) \subseteq M_{\ell}(N,\chi)$ the subspace of cusp forms, namely the subspace of those modular forms which vanish at all the cusps. If $g_1, g_2 \in M_{\ell}(N,\chi)$, and at least one of them belongs to $S_{\ell}(N,\chi)$, then we consider the Petersson product of $g_1$ and $g_2$ defined as
\[
\langle g_1, g_2 \rangle := \frac{1}{[\SL_2(\Z):\Gamma_0(N)]} \int_{\Gamma_0(N)\backslash \mathcal H} g_1(z) \overline{g_2(z)}y^{\ell-2}dxdy \quad (z = x + iy).
\]

If $d$ is a positive integer, we define operators $V_d$ and $U_d$ by 
\[
 V_d g (z) := d \, g(dz), \quad U_d g (z) := \frac{1}{d} \sum_{j=0}^{d-1} g\left(\frac{z+j}{d}\right).
\]
Then the classical Hecke operators $T_p$, for primes $p\nmid N$, are expressed in terms of $V_p$ and $U_p$: if $g \in M_{\ell}(N,\chi)$, then 
$T_p g = U_p g + \chi(p)p^{\ell-2}V_pg$. A cusp form $g \in S_{\ell}(N,\chi)$ is said to be a Hecke eigenform if $g$ is an eigenvector for all the Hecke operators $T_p$, for $p\nmid N$, and $U_p$, for $p\mid N$.

If $g \in M_{\ell}(N,\chi)$, recall that one disposes of a $q$-expansion (at the infinity cusp) 
\[
 g(q) = \sum_{n\geq 0} a(g,n) q^n.
\]
If $g$ is a cusp form, then $a(g,0)=0$, and we say $g$ is normalized if $a(g,1)=1$. If $g$ is a normalized, new cuspidal Hecke eigenform, then $T_p g = a(g,p) g$ for all primes $p$. And furthermore, $g$ is also an eigenvector for the Atkin--Lehner involutions $W_p$ at each prime $p\mid N$ (cf. \cite{AtkinLehner} for details, especially Theorem 3).

If $\lambda$ is a Dirichlet character modulo $M$ and $g \in M_{\ell}(N,\chi)$, then we write $g \otimes \lambda$ for the unique modular form in $M_{\ell}(NM^2,\chi\lambda^2)$ with $q$ expansion 
\[
 \sum_{n\geq 0} \lambda(n)a(g,n) q^n.
\]
For a careful study of the minimal level of $g \otimes \lambda$, at least for new forms, we refer the reader to \cite{AtkinLi}.

When $\chi = \mathbf 1$ is the trivial character, we will write $S_{\ell}(N) := S_{\ell}(N,\mathbf 1)$. If $f \in S_{\ell}(N)$ has $q$-expansion $\sum_{n>0} a_n q^n$, then we write
\[
 L(f,s) = \sum_{n>0} a_n n^{-s}
\]
for the $L$-series associated with $f$. This Dirichlet series is well-known to converge for $\mathrm{Re}(s)>1+\ell/2$, and further it extends analytically to a holomorphic function on $\C$. The completed $L$-series $\Lambda(f,s) := \Gamma_{\C}(s) L(f,s)$ satisfies a functional equation of the form 
\begin{equation}\label{FnEq:f}
 \Lambda(f,s) = \varepsilon(f) N^{\ell/2-s} \Lambda(f,\ell-s),
\end{equation}
where $\varepsilon(f) \in  \{\pm 1\}$. If $f$ is a new eigenform, then $L(f,s)$ can be expressed as an Euler product, namely 
\[
 L(f,s) = \prod_p (1 - a_p p^{-s} + e_pp^{\ell-1}p^{-2s})^{-1},
\]
where $e_p = 0$ for $p\mid N$, and $e_p=1$ otherwise.

\subsection{Half-integral weight modular forms}\label{MF:half-integral}

We review briefly some aspects of the theory of half-integral weight modular forms, initiated by Shimura \cite{ShimuraHalfIntegral} and further studied by Kohnen \cite{Kohnen80, Kohnen-newforms, Kohnen85} and many others (see, e.g. \cite{KojimaTokuno, BaruchMao}).

For $\gamma = \left( \begin{smallmatrix} a & b \\c & d \end{smallmatrix}\right) \in \Gamma_0(4) \subset \SL_2(\Z)$ and $z \in \C$, define 
\[
 \tilde j(\gamma,z) = \left(\frac{c}{d}\right) \epsilon(d) (cz+d)^{1/2},
\]
where $\left(\frac{\cdot}{\cdot}\right)$ is the Kronecker symbol as defined in \cite[p.442]{ShimuraHalfIntegral}, and $\epsilon(d)$ equals $1$ if $d\equiv 1 \pmod 4$ and $-\sqrt{-1}$ if $d \equiv 3 \pmod 4$. Observe that $\tilde j(\gamma,z)^4 = j(\gamma,z)^2$.

Let $k \geq 1$ be an integer, $N$ be an odd positive integer, and $\chi$ be an even Dirichlet character modulo $N$. Consider the space $S_{k+1/2}(4N,\chi)$ of holomorphic functions $h: \mathcal H \to \C$ satisfying 
\[
 h(\gamma z) = \tilde j(\gamma,z)^{2k+1}\chi(d) h(z) \text{ for all } \gamma = \left( \begin{array}{cc} a & b \\c & d \end{array}\right) \in \Gamma_0(4N) 
\]
and vanishing at the cusps of $\Gamma_0(4N)$. The Petersson product of two cusp forms $h_1$ and $h_2$ in $S_{k+1/2}(4N,\chi)$ is defined, similarly as in the case of integral weight, as 
\[
 \langle h_1, h_2 \rangle := \frac{1}{[\SL_2(\Z):\Gamma_0(4N)]} \int_{\Gamma_0(4N)\backslash\mathcal H} h_1(z) \overline{h_2(z)} y^{k-3/2}dx dy \quad \quad (z = x + i y).
\]

It is well-known that every cusp form $h \in S_{k+1/2}(4N,\chi)$ admits a Fourier expansion of the form 
\[
h(z) = \sum_{n\geq 1} c(n)q^n. 
\]
One defines {\em Kohnen's plus subspace} $S_{k+1/2}^+(4N,\chi) \subset S_{k+1/2}(4N,\chi)$ as the subspace of those cusp forms whose Fourier expansion is of the form 
\[
h(z) = \sum_{\substack{n\geq 1,\\(-1)^kn \equiv 0,1 \,(4)}} c(n)q^n. 
\]
By virtue of \cite[Proposition 1]{Kohnen-newforms}, the space $S_{k+1/2}^+(4N,\chi)$ can be characterized as the eigenspace of a certain hermitean operator satisfying a quadratic equation.

Finally, we recall also that for each prime $p \nmid 2N$ there is a Hecke operator $T_{p^2}$ acting on the space $S_{k+1/2}(4N,\chi)$ (see \cite[Eq. (11)]{Kohnen-newforms}, \cite[Eq. (1-6)]{KojimaTokuno}); on Fourier expansions, it is given by 
\[
 \sum_{n\geq 1} c(n)q^n | T_{p^2} = \sum_{n\geq 1} \left( c(p^2n) + \left(\frac{(-1)^kn}{p}\right)\chi(p)p^{k-1}c(n) + \chi(p)^2p^{2k-1}c(n/p^2)\right)q^n,
\]
where we read $c(n/p^2) = 0$ if $p^2\nmid n$.

Shimura's correspondence establishes lifting maps $\zeta_{k,N,\chi}^D$ from $S_{k+1/2}^+(4N,\chi)$ to the space $S_{2k}(N,\chi^2)$ of cusp forms of weight $2k$, level $N$ and nebentype character $\chi^2$, which depend on the choice of a (fundamental) discriminant $D$. When the character $\chi$ is trivial and $N$ is square-free, there is a well-behaved theory of new forms of half-integral weight, and a linear combination of the lifting maps $\zeta_{k,N}^D$ provides an isomorphism 
\[
 S_{k+1/2}^{+,new}(4N) \, \stackrel{\simeq}{\longrightarrow} S_{2k}^{new}(N)
\]
commuting with the action of Hecke operators. In particular, for each normalized newform $f \in S_{2k}^{new}(N)$ there is a unique half-integral weight cusp form $h \in S_{k+1/2}^{+,new}(4N)$, up to constant multiples, such that $h | T_{p^2} = a(p)h$, where $a(p)$ is the $p$-th Fourier coefficient of $f$. Moreover, If $h \in S_{k+1/2}^{+,new}(4N)$ corresponds to $f \in S_{2k}^{new}(N)$ under this isomorphism, then the $|D|$-th Fourier coefficient $c(|D|)$ of $h$ is related to the special value $L(f,D,k)$ of the complex $L$-series associated to $f$ twisted by the quadratic character $\chi_D$, by Kohnen's formula (cf. \cite[Corollary 1]{Kohnen85}): 
\begin{equation}\label{Kohnenformula-classical}
 \frac{|c(|D|)|^2}{\langle h,h\rangle} = 2^{\nu(N)}\frac{(k-1)!}{\pi^k}|D|^{k-1/2}\frac{L(f,D,k)}{\langle f,f\rangle},
\end{equation}
where $\nu(N)$ is the number of prime divisors of $N$. However, this formula is valid {\em only} for discriminants $D$ such that $(-1)^kD > 0$, and $\left(\frac{D}{\ell}\right) = w_{\ell}$ for all prime divisors $\ell$ of $N$. Here, $w_{\ell}$ denotes the eigenvalue of the $\ell$-th Atkin--Lehner involution acting on $f$.

In general, the lifting maps $\zeta_{k,N,\chi}^D: S_{k+1/2}^+(4N,\chi) \to S_{2k}(N,\chi^2)$ are given in terms of Fourier expansions by the recipe (see \cite[Section 3]{KojimaTokuno}, for instance)
 \begin{equation}\label{Shimura-lifting}
 \sum_{\substack{n\geq 1, \\(-1)^kn \equiv 0,1 \, (4)}} c(n)q^n \, \longmapsto \, \sum_{n\geq 1}\left(\sum_{d\mid n} \left(\frac{D}{d}\right)\chi(d)d^{k-1}c(n^2|D|/d^2)\right)q^n.
\end{equation}

Suppose from now on that $f \in S_{2k}^{new}(N)$ is as in the introduction; in particular, $k, N \geq 1$ are odd, and $N$ is square-free. Let also $\chi$ be an even Dirichlet character modulo $N$, of conductor $M \mid N$, and write $\chi = \prod_{p\mid M} \chi_{(p)}$ through the canonical isomorphism $(\Z/M\Z)^{\times} \simeq \prod_p (\Z/p\Z)^{\times}$, where each $\chi_{(p)}$ is a Dirichlet character modulo $p$. For each prime $p\mid N$, let $w_p \in \{\pm 1\}$ be the eigenvalue of the $p$-th Atkin--Lehner involution acting on $f$, and define a set of fundamental discriminants 
\begin{equation}\label{D(N,M)}
 \mathfrak D(N,M) := \left\lbrace D < 0 \text{ fund. discr.}: \left(\frac{D}{p}\right) = w_p \text{ for all } p \mid N/M, \,\left(\frac{D}{p}\right) = -w_p \text{ for all } p \mid M \right\rbrace.
\end{equation}
Besides, consider the twisted cusp form $f \otimes \chi$, which by \cite[Theorem 3.1]{AtkinLi} belongs to $S_{2k}^{new}(NM,\chi^2)$, and define a subspace of $S_{k+1/2}^+(4NM,\chi)$ by setting
\[
 S_{k+1/2}^{+,new}(4NM,\chi;f \otimes \chi) := \{h \in S_{k+1/2}^+(4NM,\chi): h|T_{p^2} = a_{\chi}(p)h \text{ for all } p\nmid 2N \} \subset S_{k+1/2}^{+,new}(4NM,\chi),
\]
where $a_{\chi}(p)$ denotes the $p$-th Fourier coefficient of $f \otimes \chi$. The lifting maps $\zeta_{k,N,\chi}^D$ map $S_{k+1/2}^{+,new}(4NM,\chi;f \otimes \chi)$ to the one-dimensional subspace of $S_{2k}(NM,\chi^2)$ spanned by the new form $f \otimes \chi$. The theorem below follows from \cite[Theorem 10.1]{BaruchMao} (cf. loc. cit. for a slightly more general statement).

\begin{theorem}[Baruch--Mao]\label{thm:BM}
 Suppose that $\chi_{(p)}(-1) = -1$ for all $p\mid M$. Then the space 
 \[
  S_{k+1/2}^{+,new}(4NM,\chi;f \otimes \chi) \subset S_{k+1/2}^{+,new}(4NM,\chi)
 \]
 is one-dimensional. Moreover, if $h \in S_{k+1/2}^{+,new}(4NM,\chi;f \otimes \chi)$ is a non-zero element, with Fourier expansion $h = \sum c(n)q^n$, and $D \in \mathfrak D(N,M)$, then
  \begin{equation}\label{BMKohnenformula}
   \frac{|c(|D|)|^2}{\langle h, h \rangle} = 2^{\nu(N)}\frac{(k-1)!}{\pi^k}|D|^{k-1/2}\left(\prod_{p\mid M}\frac{p}{p+1}\right)\frac{L(f,D,k)}{\langle f,f\rangle}.
  \end{equation}
\end{theorem}

The identity in \eqref{BMKohnenformula} is a generalization of Kohnen's formula \eqref{Kohnenformula-classical}, relating the special values $L(f,D,k)$ to the $|D|$-th Fourier coefficients of certain cusp forms of half-integral weight, for discriminants $D$ to which \eqref{Kohnenformula-classical} does not apply. At the end of Section \ref{theta:PGL2SL2} we will explain the meaning of the local sign assumptions on $D$ in terms of Waldspurger's theta correspondence between automorphic forms of $\PGL_2$ and $\widetilde{\SL}_2$.

\begin{rem}\label{rmk:BMcompleted}
 In terms of the completed $L$-series $\Lambda(f,D,s) = \Gamma_{\C}(s)L(f,D,s)$, the identity in \eqref{BMKohnenformula} can be rewritten as 
 \[
  \Lambda(f,D,k) = 2^{1-k-\nu(N)} |D|^{1/2-k}|c(|D|)|^2\left(\prod_{p\mid M}\frac{p+1}{p}\right)\frac{\langle f,f\rangle}{\langle h, h \rangle}.
 \]
\end{rem}

Continue to consider a new form $f \in S_{2k}^{new}(N)$ as above, and let $\chi$ be an even Dirichlet character of conductor $M \mid N$ satisfying the hypothesis of Theorem \ref{thm:BM}. In particular, notice that $M$ must be the (square-free) product of an even number of prime divisors of $N$. Let 
\[
 f = \sum_{n\geq 1} a(n)q^n, \quad f \otimes \chi = \sum_{n\geq 1} a_{\chi}(n)q^n
\]
be the Fourier expansions of $f$ and $f\otimes \chi$, respectively, so that $a_{\chi}(n) = \chi(n)a(n)$. We further assume that $f$ is normalized, i.e. $a(1) = 1$, hence $f\otimes \chi$ is normalized as well. By virtue of the above theorem, we can then choose a non-zero cusp form 
\[
 h \in S_{k+1/2}^{+,new}(4NM,\chi;f \otimes \chi).
\]
Now fix a fundamental discriminant $D \in \mathfrak D(N,M)$ such that $L(f,D,k) \neq 0$. By \eqref{BMKohnenformula}, $c(|D|)$ is non-zero, and this implies that $\zeta^D_{k,N,\chi}(h) \neq 0$. Indeed, \eqref{Shimura-lifting} shows that the first Fourier coefficient of $\zeta^D_{k,N,\chi}(h)$ equals $c(|D|)$. It follows that $\zeta^D_{k,N,\chi}(h) = c(|D|) f \otimes \chi$. For later purposes, we determine in the next lemma the Fourier coefficients of $h$ in terms of the Fourier coefficients of $f\otimes \chi$ (hence, in terms of the Fourier coefficients of $f$ as well). 

\begin{lemma}\label{lemma:ShimuraLifts}
 With the above notation and assumptions, for every integer $n\geq 1$ one has 
 \begin{equation}\label{relationFC:ShimuraLifts}
  c(n^2|D|) = c(|D|) \sum_{\substack{d \mid n,\\d > 0}} \mu(d) \left(\frac{D}{d}\right) \chi(d) d^{k-1} a_{\chi}(n/d).
 \end{equation}
\end{lemma}
\begin{proof}
 As observed above, $\zeta_{k,N,\chi}^D$ maps $h$ to $c(|D|) f\otimes \chi$. From \eqref{Shimura-lifting}, we thus have for all $n\geq 1$
  \[
  c(|D|) a_{\chi}(n) = \sum_{\substack{d\mid n,\\d>0}} \left(\frac{D}{d}\right)\chi(d)d^{k-1} c(n^2|D|/d^2).
 \]
 Suppose first that $n\geq 1$ satisfies $(n,ND)=1$. Then we can rewrite this last identity as 
 \[
  c(|D|) \left(\frac{D}{n}\right)\chi(n)^{-1}n^{1-k}a_{\chi}(n) = \sum_{0<t\mid n} \left(\frac{D}{t}\right)\chi(t)^{-1}t^{1-k}c(t^2|D|).
 \]
 By applying the M\"obius inversion formula, one gets  
 \[
  c(n^2|D|) = c(|D|)\sum_{0<d \mid n} \mu(d) \left(\frac{D}{d}\right) \chi(d) d^{k-1}a_{\chi}(n/d).
 \]
 Now suppose that $n\geq 1$, and write $n = n_0 n_N n_D$ where $n_0, n_N, n_D \geq 1$ are integers such that $(n_0,ND) = 1$ and every prime divisor of $n_N$ (resp. $n_D$) divides $N$ (resp. $D$). From \eqref{Shimura-lifting}, we see that $c(|D|)a_{\chi}(n_N) = c(n_N^2|D|)$ and $c(|D|)a_{\chi}(n_D) = c(n_D^2|D|)$. Also, for integers $r,s \geq 1$ with $(r,s)=1$ one has $c(r^2|D|)c(s^2|D|) = c(|D|)c(r^2s^2|D|)$. In particular,  
 $c(n_0^2|D|)c(n_N^2|D|)c(n_D^2|D|) = c(|D|)^2c(n^2|D|)$, which together with the above relations imply that 
 \[
  c(n^2|D|) = \frac{c(n_N^2|D|)}{c(|D|)}\frac{c(n_D^2|D|)}{c(|D|)} c(n_0^2|D|) = a_{\chi}(n_N)a_{\chi}(n_D)c(n_0^2|D|).
 \]
 Now one can apply the previous argument for $c(n_0^2|D|)$, since $(n_0,ND)=1$, to eventually conclude that 
 \[
  c(n^2|D|) = c(|D|) \sum_{0<d \mid n} \mu(d) \left(\frac{D}{d}\right) \chi(d) d^{k-1} a_{\chi}(n/d).
 \]
\end{proof}

\subsection{Siegel modular forms of degree 2 and Saito--Kurokawa lifts}\label{MF:Siegel}

Let  
\[
 \mathcal H_2 = \{Z \in \mathrm{M}_2(\C): Z = {}^tZ, \, \mathrm{Im}(Z) \text{ positive definite}\}
\]
denote Siegel's upper half-space of degree $2$, 
\[
 \mathrm{GSp}_2^+(\R) := \{g \in \mathrm M_{4}(\R): g J_2 {}^tg = \nu(g)J_2, \, \nu(g)>0\}, \quad J_2 = \begin{pmatrix} 0 & \mathrm{Id}_2 \\ -\mathrm{Id}_2 & 0\end{pmatrix},
\]
be the group of symplectic similitudes with positive multiplicators, and let $\mathrm{Sp}_2(\R) = \{g \in \mathrm{GSp}_2^+(\R): \nu(g) = 1\} \subseteq \GL_4(\R)$ be the symplectic group. The group $\mathrm{GSp}_2^+(\R)$ acts on $\mathcal H_2$ by 
\begin{equation}\label{action:GSp2H2}
 gZ = (AZ+B)(CZ+D)^{-1} \quad \text{if } g = \left(\begin{array}{cc} A & B \\ C & D\end{array}\right).
\end{equation}
Put $\Gamma_2 := \mathrm{Sp}_2(\Z) = \mathrm{Sp}_2(\R) \cap \mathrm M_{4}(\Z)$. If $N\geq 1$ is an integer, one defines a Hecke-type congruence subgroup of level $N \geq 1$ of $\Gamma_2$ by 
\begin{equation}\label{Gamma0Siegel}
 \Gamma_0^{(2)}(N) = \left\lbrace g = \begin{pmatrix} A & B \\ C & D \end{pmatrix} \in \Gamma_2: C \equiv 0 \text{ mod } N\right\rbrace.
\end{equation}
Given a Dirichlet character $\chi: (\Z/N\Z)^{\times} \to \C^{\times}$, by a slight abuse of notation we will continue to write $\chi: \Gamma_0^{(2)} \to \C^{\times}$ for the character on $\Gamma_0^{(2)}(N)$ defined by the rule 
\begin{equation}\label{chi:Siegel}
 \left(\begin{array}{cc} A & B \\ C & D\end{array}\right) \, \longmapsto \, \chi(\det(D)).
\end{equation}

Fix now an integer $k\geq 1$, and for a function $F: \mathcal H_2 \to \C$ and an element $g \in \GSp_2^+(\R)$, define 
\[
 (F|_{k+1}[g])(Z) = J(g,Z)^{-k-1}F(gZ),
\]
where $J(g,Z) = \det(CZ+D)$ if $g = \begin{pmatrix} A & B \\ C & D \end{pmatrix}$. The space $M_{k+1}(\Gamma_0^{(2)}(N),\chi)$ of Siegel modular forms of weight $k+1$, level $N$, and character $\chi$ is the space of holomorphic functions $F: \mathcal H_2 \to \C$ such that 
\[
 F|_{k+1}[\gamma] = \chi(\gamma) F \quad \text{ for all } \gamma \in \Gamma_0^{(2)}(N).
\]
Notice that there are no additional conditions of holomorphicity at the cusps (`Koecher's principle'). We will write $S_{k+1}(\Gamma_0^{(2)}(N),\chi) \subseteq M_{k+1}(\Gamma_0^{(2)}(N),\chi)$ for the subspace of Siegel cusp forms. Given $F \in S_{k+1}(\Gamma_0^{(2)}(N),\chi)$, one has a Fourier expansion of the form 
\[
 F(Z) = \sum_{B} A_F(B)e^{2\pi\sqrt{-1}\Tr(BZ)},
\]
where $B$ runs over the set of positive definite, half-integral $2$-by-$2$ symmetric matrices. If 
\[
 B = \left(\begin{array}{cc} n & r/2 \\ r/2 & m\end{array}\right), \quad Z = \left(\begin{array}{cc} \tau & z \\ z & \tau' \end{array}\right),
\]
with $n,r,m \in \Z$ such that $4nm-r^2>0$ and $\tau,\tau'\in \mathcal H$, $z\in \C$ with $\mathrm{Im}(z)^2 < \mathrm{Im}(\tau)\mathrm{Im}(\tau')$, then notice that $e^{2\pi \sqrt{-1}\Tr(BZ)} = e^{2\pi\sqrt{-1}(n\tau + rz + m\tau')}$, so that we can rewrite the Fourier expansion as 
\[
 F(\tau,z,\tau') = \sum_{\substack{n,r,m\in \Z,\\ 4nm-r^2 > 0}} A_F(n,r,m)e^{2\pi\sqrt{-1}(n\tau + rz + m\tau')}.
\]

Given a Siegel modular form $F$ for $\Gamma_0^{(2)}(N)$, one can restrict it to $\mathcal H \times \mathcal H$, diagonally embedded into $\mathcal H_2$. This way, we obtain a modular form on $\mathcal H \times \mathcal H$ for $\Gamma_0(N) \times \Gamma_0(N)$, usually referred to as the ``pullback'' of $F$ to the diagonal. Explicitly, this pullback or restriction is obtained by setting $z = 0$, hence 
\[
 F_{|\mathcal H\times \mathcal H}(\tau,\tau') = \sum_{n,m\in \Z} \left(\sum_{\substack{r \in \Z,\\ r^2 < 4nm}} A_F(n,r,m) \right)e^{2\pi\sqrt{-1}(n\tau + m\tau')}.
\]

Later we will also need a more technical operation on Siegel forms, which we define now. Let $F \in S_{k+1}(\Gamma_0^{(2)}(N),\chi)$ be a Siegel modular form of degree $2$ as before, and let $p$ be a prime dividing $N$. For each $j \in \Z/p\Z$, put 
\[
 u_j := \left(\begin{array}{cccc} 1 & 0 & 0 & 0 \\ 0 & 1 & 0 & p^{-1}j \\ 0 & 0 & 1 & 0 \\ 0 & 0 & 0 & 1\end{array}\right) \in \GSp_2^+(\Q),
\]
and define a function $\mathfrak R_p F$ on $\mathcal H_2$ by setting (notice that $J(u_j,Z)=1$ for all $j$)
\begin{equation}\label{Rp:Siegel}
 \mathfrak R_p F(Z) := \frac{1}{p} \sum_{j \in \Z/p\Z} F(u_j Z) = \frac{1}{p} \sum_{j \in \Z/p\Z} F|_{k+1}[u_j](Z).
\end{equation}

For any integer $m \neq 0$, we write $\Gamma^{\text{param}}(m) \subseteq \Sp_2(\Q)$ for the paramodular subgroup defined by
\[
 \Gamma^{\text{param}}(m) := \Sp_2(\Q) \cap \left(\begin{array}{cccc} \Z & m\Z & \Z & \Z \\ \Z & \Z & \Z & m^{-1}\Z \\ \Z & m\Z & \Z & \Z \\ m\Z & m\Z & m\Z & \Z\end{array}\right).
\]

\begin{lemma}
 Let $F \in S_{k+1}(\Gamma_0^{(2)}(N),\chi)$ be as before, and $p$ be a prime with $p \mid N$. Then 
 \[
  \mathfrak R_p F \in S_{k+1}(\Gamma_0^{(2)}(Np)\cap \Gamma^{\mathrm{param}}(p),\chi).
 \]
 If $p^2 \mid N$, then $\mathfrak R_p F \in S_{k+1}(\Gamma_0^{(2)}(N)\cap \Gamma^{\mathrm{param}}(p),\chi)$.
\end{lemma}
\begin{proof}
 Let $\gamma \in \Gamma_0^{(2)}(Np)\cap \Gamma^{\mathrm{param}}(p)$ be written as a block matrix 
 \[
  \gamma = \left(\begin{array}{cc} A & B \\ C & D\end{array}\right),
 \]
 where $A = \left(\begin{smallmatrix} a_1 & a_2 \\ a_3 & a_4 \end{smallmatrix}\right)$, $B = \left(\begin{smallmatrix} b_1 & b_2 \\ b_3 & b_4 \end{smallmatrix}\right)$, $C = \left(\begin{smallmatrix} c_1 & c_2 \\ c_3 & c_4 \end{smallmatrix}\right)$, $D = \left(\begin{smallmatrix} d_1 & d_2 \\ d_3 & d_4 \end{smallmatrix}\right)$. Notice that all entries are integral, and we have $c_1, c_2, c_3, c_4 \in Np\Z$ and $a_2, d_3 \in p\Z$. Moreover, observe that $a_4, d_4$ are invertible modulo $p$.
 
 Similarly, write $u_j$ as a block matrix, 
 \[
  u_j = \left(\begin{array}{cc} \mathrm{Id}_2 & E_j \\ 0 & \mathrm{Id}_2\end{array}\right).
 \]
 Choosing $i \in \Z/p\Z$ such that $ia_4 \equiv jd_4 \pmod p$, a bit of algebra shows that $u_j \gamma = \gamma' u_i$, where $\gamma' = \left(\begin{smallmatrix} A' & B' \\ C & D' \end{smallmatrix}\right) \in \Gamma_0^{(2)}(N)$ is such that $A' \equiv A$ and $D' \equiv D$ modulo $N$. In particular, $\chi(\gamma') = \chi(\gamma)$ and it follows from the very definition that 
 \[
  \mathfrak R_pF_{|[\gamma]} = \sum F_{|[u_j\gamma]} = \sum F_{|[\gamma' u_i]} = \chi(\gamma') \sum F_{|[u_i]} = \chi(\gamma) \mathfrak R_pF.
 \]
 The second part of the statement follows by checking carefully the omitted algebra.
\end{proof}

\begin{rem}
 We might observe that $u_j \in \GSp_2(\Q)$ is the image of 
\[
 \left(\left(\begin{array}{cc} 1 & 0 \\ 0 & 1\end{array}\right), \left(\begin{array}{cc} 1 & p^{-1}j \\ 0 & 1\end{array}\right)\right) \in \SL_2(\Q)\times\SL_2(\Q),
\]
and hence the pullback $(\mathfrak R_p F)_{|\mathcal H\times \mathcal H}$ coincides with $(\mathrm{id} \otimes V_pU_p) F_{|\mathcal H \times \mathcal H}$.
\end{rem}

Finally, we recall the classical construction of Saito--Kurokawa lifts that we need in this paper. Assume as in the introduction that $k\geq 1$ is odd, $N \geq 1$ is an odd square-free integer, and $f \in S_{2k}^{new}(N)$ is a normalized new cusp form of weight $2k$ and level $N$. Let $\chi$ be an even Dirichlet character of conductor $M$, for some $M \mid N$, satisfying the hypothesis of Theorem \ref{thm:BM}, and let 
\[
 h \in S_{k+1/2}^{+,new}(4NM,\chi; f\otimes \chi) \subset S_{k+1/2}^{+,new}(4NM,\chi)
\]
be a Shimura lift of $f \otimes \chi \in S_{2k}^{new}(NM,\chi^2)$ as in Theorem \ref{thm:BM}. There is an Eichler--Zagier isomorphism 
\[
 \mathbf Z: S_{k+1/2}^{+,new}(4NM,\chi) \, \stackrel{\simeq}{\longrightarrow} \, J_{k+1,1}^{cusp,new}(\Gamma_0(NM)^J,\chi)
\]
between $S_{k+1/2}^{+,new}(4NM,\chi)$ and the space $J_{k+1,1}^{cusp,new}(\Gamma_0(NM)^J,\chi)$ of Jacobi new cusp forms of weight $k+1$, index $1$, level $\Gamma_0(NM)^J$ and character $\chi$ (see \cite{EichlerZagier, Ibukiyama, MR2000}), which together with Maa\ss' lift 
\[
 \mathbf M: J_{k+1,1}^{cusp,new}(\Gamma_0(NM)^J,\chi) \, \hookrightarrow \, S_{k+1}(\Gamma_0^{(2)}(NM),\chi)
\]
gives an injective homomorphism from $S_{k+1/2}^{+,new}(4NM,\chi)$ into $S_{k+1}(\Gamma_0^{(2)}(NM),\chi)$. We will refer to the Siegel modular form $F_{\chi} := \mathbf M(\mathbf Z(h)) \in S_{k+1}(\Gamma_0^{(2)}(NM),\chi)$ associated with $h$ under the composition of $\mathbf Z$ and $\mathbf M$, as a {\em Saito--Kurokawa lift} of $f \otimes \chi$. If we continue to denote by $c(n)$ the Fourier coefficients of $h$, then the Fourier expansion of $F_{\chi}$ reads
\begin{equation}\label{Fchi:Fourier}
 F_{\chi}(Z) =  \sum_{B = \left(\begin{smallmatrix} n & r/2 \\ r/2 & m\end{smallmatrix} \right)} \left( \sum_{\substack{a \mid \gcd(n,r,m),\\ \gcd(a,N)=1}} a^k\chi(a) c\left(\frac{4nm-r^2}{a^2}\right) \right) e^{2\pi \sqrt{-1} \mathrm{Tr}(BZ)}.
\end{equation}

If $p$ is a prime dividing $M$, notice that $p^2 \mid NM$, hence our discussion above implies that $\mathfrak R_p F_{\chi} \in S_{k+1}(\Gamma_0^{(2)}(NM)\cap \Gamma^{\mathrm{param}}(p),\chi)$. By defining the operator $\mathfrak R_M := \prod_{p \mid M} \mathfrak R_p$ as the compositum of the operators $\mathfrak R_p$ for primes $p\mid M$, we thus have 
\[
 \mathfrak R_M F_{\chi} \in S_{k+1}(\Gamma_0^{(2)}(NM)\cap \Gamma^{\mathrm{param}}(M),\chi).
\]

\section{Automorphic forms and representations}\label{sec:AF}

Similarly as in the previous one, the aim of this section is to set up some notation and recall some results concerning the theory of automorphic forms for $\GL_2$, for the metaplectic cover $\widetilde{\SL}_2$ of $\SL_2$, and for the symplectic similitude group $\GSp_2$ of degree $4$. Most of this section can therefore be seen as an automorphic rephrasing of the previous one. We will abbreviate $\A = \A_{\Q}$ for the ring of adeles of $\Q$, and $\A_f$ will denote the subring of finite adeles. Then $\Q$ sits diagonally in $\A$, and for every rational place $v$, the local field $\Q_v$ embeds as a subfield of $\A$ in the $v$-th component. If $\chi$ is a Dirichlet character modulo $N\geq 1$, we will write $\underline{\chi}$ for the adelization of $\chi$. Namely, $\underline{\chi}: \A^{\times}/\Q^{\times} \to \C^{\times}$ is the unique Hecke character such that  $\underline{\chi}(\varpi_q) = \chi(q)$ for every prime $q\nmid N$ and every uniformizer $\varpi_q \in \Q_q^{\times} \hookrightarrow \A^{\times}$ at $q$. For every finite prime $p$, we shall denote by $\underline{\chi}_p$ the restriction of $\underline{\chi}$ to $\Z_p^{\times}$. At primes $p \mid N$, $\underline{\chi}_p$ coincides with the {\em inverse} of the character $\Z_p^{\times} \to \C^{\times}$ inflated from the $p$-th component $\chi_{(p)}$ of $\chi$. 

\subsection{Automorphic forms for $\GL_2$}\label{AF:GL2}

Let us briefly recall how classical modular forms of integral weight give rise to automorphic forms and representations of $\GL_2$. In the following, we identify $\Q^{\times}$ and $\A^{\times}$ with the centers of $\GL_2(\Q)$ and $\GL_2(\A)$, respectively. 

Let $N\geq 1$ be an integer, and consider the compact open subgroup 
\[
 \mathbf K_0(N) = \left \lbrace \left(\begin{array}{cc} a & b \\c & d\end{array}\right) \in \GL_2(\hat{\Z}): c \equiv 0 \pmod N\right\rbrace
\]
of $\GL_2(\A_f)$. By strong approximation, one has $\GL_2(\A) = \GL_2(\Q)\GL_2^+(\R)\mathbf K_0(N)$, where $\GL_2^+(\R)$ stands for the subgroup of 2-by-2 real matrices with positive determinant. If $\chi$ is a Dirichlet character modulo $N$, then $\underline{\chi}$ induces a character of $\mathbf K_0(N)$, which by a slight abuse of notation we still denote $\underline{\chi}$, by 
\[
 \underline{\chi}\left(\left( \begin{array}{cc} a & b \\c & d\end{array}\right)\right) = \underline{\chi}(d).
\]
Observe that this agrees with the Hecke character $\underline{\chi}$ when restricted to $\A^{\times}\cap \mathbf K_0(N)$.

Let $g \in S_{\ell}(N,\chi)$ be a cusp form of weight $\ell$, level $N$ and nebentype character $\chi$. Then it is well-known that $g$ induces an automorphic form $\mathbf g$ for $\GL_2(\A)$, by setting 
\begin{equation}\label{g-to-autform}
 \mathbf g (\gamma \gamma_{\infty} k_0) = g(\gamma_{\infty}(i))(ci + d)^{-\ell}(\det \gamma_{\infty})^{\ell/2}\underline{\chi}(k_0)
\end{equation}
whenever $\gamma \in \GL_2(\Q)$, $\gamma_{\infty} \in \GL_2^+(\R)$ and $k_0 \in \mathbf K_0(N)$. This gives indeed a well-defined function on $\GL_2(\A)$ because $\GL_2(\Q) \cap \GL_2^+(\R)\mathbf K_0(N) = \Gamma_0(N)$, $\GL_2(\A) = \GL_2(\Q)\GL_2^+(\R)\mathbf K_0(N)$, and $g$ satisfies 
\[
 g\left( \left( \begin{array}{cc} a & b \\c & d\end{array}\right)z\right) = \chi(d)(cz+d)^{\ell}g(z) \quad \text{for } \left( \begin{array}{cc} a & b \\c & d\end{array}\right) \in \Gamma_0(N).
\]
Further, it is clear that $\mathbf g$ satisfies $\mathbf g(\gamma z) = \underline{\chi}(z)\mathbf g(\gamma)$ for all $z \in \A^{\times}$.

The function $\mathbf g$ just defined belongs in fact to the space of automorphic forms on $\GL_2(\A)$ with central character $\underline{\chi}$. If $\pi = \pi_g$ denotes the linear span of the right translates of $\mathbf g$ under $\GL_2(\A)$, then $\pi$ is an admissible smooth representation of $\GL_2(\A)$. Since $g$ is an eigenform, it is well-known that $\pi$ is irreducible and decomposes as a restricted tensor product $\otimes_v \pi_v$ of admissible representations of $\GL_2(\Q_v)$, and $\mathbf g = \otimes_v \mathbf g_v$ with $\mathbf g_v \in \pi_v$ for each place $v$ of $\Q$.

When the nebentype character $\chi$ is trivial, $g$ gives rise to an automorphic representation $\pi$ of $\GL_2(\A)$ with trivial central character, so that we can regard $\pi$ as an automorphic representation of $\PGL_2(\A)$. Again, $\pi$ decomposes as a restricted tensor product $\pi = \otimes_v \pi_v$ of admissible representations of $\PGL_2(\Q_v)$. 

We refer the reader to \cite{SchmidtRemarksGL2} for a good account on local types, newforms, and a careful study of local $\varepsilon$-factors at non-archimedean places. If $f \in S_{\ell}^{new}(N)$ is a newform of weight $\ell\geq 1$, level $N$, and trivial nebentype character, then the completed complex $L$-series $\Lambda(f,s)$ associated with $f$ coincides with the $L$-series $L(\pi,s)$ associated with the automorphic representation of $\PGL_2(\A)$ corresponding to $f$. The root number $\varepsilon(f)$ appearing in \eqref{FnEq:f} can be therefore written as the product of local $\varepsilon$-factors $\varepsilon(\pi_v,1/2)$ at places $v \mid N\infty$. At the finite places $p \mid N$, $\varepsilon(\pi_p,1/2)$ coincides with the eigenvalue $w_p \in \{\pm 1\}$ of the $p$-th Atkin--Lehner involution acting on $f$ (cf. \cite[Theorem 3.2.2]{SchmidtRemarksGL2}).

\subsection{Automorphic forms for $\widetilde{\SL}_2$}\label{AF:SL2}

We start by setting down some notation concerning metaplectic groups. If $v$ is a place of $\Q$, we write $\widetilde{\SL}_2(\Q_v)$ for the metaplectic cover (of degree $2$) of $\SL_2(\Q_v)$, and similarly, we denote by $\widetilde{\SL}_2(\A)$ the metaplectic cover (of degree $2$) of $\SL_2(\A)$. We will identify $\widetilde{\SL}_2(\Q_v)$, resp. $\widetilde{\SL}_2(\A)$, with $\SL_2(\Q_v) \times \{\pm 1\}$, resp. $\SL_2(\A) \times \{\pm 1\}$, where the product is given by the rule 
\[
 [g_1, \epsilon_1][g_2, \epsilon_2] = [g_1g_2, \epsilon(g_1,g_2)\epsilon_1\epsilon_2].
\]

At each place $v$, $\epsilon_v(g_1,g_2)$ is defined as follows. First one defines $x: \SL_2(\Q_p) \to \Q_p$ by
\[
 g = \left(\begin{smallmatrix} a & b \\ c & d\end{smallmatrix}\right) \longmapsto x(g) = \begin{cases}
                                                                                          c & \text{if } c \neq 0,\\
                                                                                          d & \text{if } c = 0.
                                                                                         \end{cases}
\]
Then, $\epsilon_v(g_1,g_2) = (x(g_1)x(g_1g_2), x(g_2)x(g_1g_2))_v$, where $(\, , \,)_v$ denotes the Hilbert symbol. When $v$ is a finite place, set also 
\[
 s_v\left(\left(\begin{smallmatrix} a & b \\ c & d\end{smallmatrix}\right)\right) = \begin{cases}
                (c,d)_v & \text{if } cd \neq 0, \, \mathrm{ord}_v(c) \text{ odd}, \\
                1 & \text{otherwise},
               \end{cases}
\]
for $g = \left(\begin{smallmatrix} a & b \\ c & d\end{smallmatrix}\right) \in \SL_2(\Q_p)$, and $s_{\infty}(g) = 1$ for all $g \in \SL_2(\R)$. Then, for each place $v$, $\SL_2(\Q_v)$ embeds as a subgroup of $\widetilde{\SL}_2(\Q_v)$ through $g \mapsto [g, s_v(g)]$. If $v$ is an odd finite prime, then this homomorphism gives a splitting of $\widetilde{\SL}_2(\Q_p)$ over the maximal compact subgroup $\SL_2(\Z_p)$, while for $v = 2$ this is only a splitting over $\Gamma_1(4;\Z_2) \subset \SL_2(\Z_2)$. If $p$ is an odd prime (resp. $p=2$), and $G$ is a subgroup of $\SL_2(\Z_p)$ (resp. $\Gamma_1(4;\Z_2)$), then we will write $\widetilde G \subseteq \widetilde{\SL}_2(\Z_p)$ for the image of $G$ under the previous splitting. We will also regard $\SL_2(\Q)$ as a subgroup of $\widetilde{\SL}_2(\A)$ through the homomorphism $g \mapsto [g, \prod_v s_v(g)]$.

To simplify notation, if $v$ is a place of $\Q$ and $x \in \Q_v$, $\alpha \in \Q_v^{\times}$, we will write $u(x)$, $n(x)$, $t(\alpha)$ for the elements of $\SL_2(\Q_v)$ given by
\[
 u(x) = \left(\begin{array}{cc} 1 & x \\ 0 & 1 \end{array}\right), \, n(x) = \left(\begin{array}{cc} 1 & 0 \\ x & 1 \end{array}\right), \, t(\alpha) = \left(\begin{array}{cc} \alpha & 0 \\ 0 & \alpha^{-1} \end{array}\right).
\]
We will slightly abuse of notation and still write $u(x)$, $n(x)$, $t(\alpha)$ for the elements $[u(x),1]$, $[n(x),1]$, $[t(\alpha),1]$ in $\widetilde{\SL}_2(\Q_v)$ (notice that these coincide with the images of $u(x)$, $n(x)$, $t(\alpha)$, respectively, under the splitting $g \mapsto [g,s_v(g)]$, as $s_v(g) = 1$ in the three cases). We will also write $s = \left( \begin{smallmatrix} 0 & 1 \\ -1 & 0\end{smallmatrix}\right) \in \SL_2(\Q_v)$ (or in $\widetilde{\SL}_2(\Q_v)$). Then observe that 
\[
 u(x) = t(-1) \cdot s \cdot n(-x) \cdot s \quad \text{for all } x \in \Q_v.
\]

Let $k\geq 1$ be an integer, $N$ be a positive integer, and $\chi$ be an even Dirichlet character modulo $4N$ as before. Write $\chi_0 = (\frac{-1}{\cdot})^k\cdot \chi$, and let $\underline{\chi}_0$ denote the associated Hecke character. 

Let $\tilde{\mathcal A}_0$ denote the space of cusp forms on $\SL_2(\Q) \backslash \widetilde{\SL}_2(\A)$, and $\tilde{\rho}$ denote the right regular representation of the Hecke algebra of $\widetilde{\SL}_2(\A)$ on $\tilde{\mathcal A}_0$. Following Waldspurger \cite{Waldspurger81}, we define $\tilde{\mathcal A}_{k+1/2}(4N,\underline{\chi}_0)$ to be the subspace of $\tilde{\mathcal A}_0$ consisting of elements $\varphi$ satisfying the following properties, where $\tilde{\rho}_v$ denotes the restriction of $\tilde{\rho}$ to $\widetilde{\SL}_2(\Q_v)$.
\begin{itemize}
 \item[i)] For each prime $q\nmid 2N$, $\tilde{\rho}_q(\gamma) \varphi = \varphi$ for all $\gamma \in \SL_2(\Z_q)$.
 \item[ii)] For each prime $q\mid N$, $q\neq 2$, $\tilde{\rho}_q(\gamma) \varphi = \underline{\chi}_{0,q}(d)\varphi$ for all $\gamma = \left( \begin{smallmatrix} a & b \\ c & d \end{smallmatrix}\right) \in \Gamma_0(q^{\mathrm{ord}_q(N)}) \subset \SL_2(\Z_q)$.
 \item[iii)] For all $\gamma = \left( \begin{smallmatrix} a & b \\ c & d \end{smallmatrix}\right) \in \Gamma_0(2^{\mathrm{ord}_2(4N)}) \subset \SL_2(\Z_2)$, $\tilde{\rho}_2(\gamma) \varphi = \tilde{\epsilon}_2(\gamma)\underline{\chi}_{0,2}(d)\varphi$.
 \item[iv)] For all $\theta \in \R$, $\tilde{\rho}_{\infty}(\tilde{\kappa}(\theta)) \varphi = e^{i(k+1/2)\theta}\varphi$.
 \item[v)] $\tilde{\rho}_{\infty}(\tilde D) \varphi = \frac{1}{2}(k+1/2)(k-3/2) \varphi$.
\end{itemize}
Here, $\tilde D$ denotes the Casimir element for $\widetilde{\SL}_2(\R)$, and the element $\tilde{\kappa}(\theta)$, for $\theta \in \R$, and the character $\tilde{\epsilon}_2$ of $\widetilde{\Gamma}_0(4) \subseteq \widetilde{\SL}_2(\Q_2)$ are defined as in \cite[p. 382]{Waldspurger81}. 

Let $h \in S_{k+1/2}(4N,\chi)$ be a cusp form of half-integral weight as in Section \ref{MF:half-integral}. Given $z = u+iv \in \mathcal H$, let $b(z) \in \widetilde{\SL}_2(\A)$ be the element which is $1$ at all the finite places and equal to 
\[
 \left(\begin{array}{cc} v^{1/2} & uv^{-1/2} \\ 0 & v^{-1/2}\end{array}\right)
\]
at the real place. Then there exists a unique continuous function $\mathbf h$ on $\SL_2(\Q)\backslash \widetilde{\SL}_2(\A)$ such that
\[
 \mathbf h(b(z)\tilde{\kappa}(\theta)) = v^{k/2+1/4}e^{i(k+1/2)\theta}h(z)
\]
for all $z \in \mathcal H$, $\theta \in \R$. The assignment $h \mapsto \mathbf h$ gives an {\em adelization map} 
\begin{equation}\label{adelization-Waldspurger}
  S_{k+1/2}(4N,\chi) \, \longrightarrow \, \tilde{\mathcal A}_{k+1/2}(4N,\underline{\chi}_0),
\end{equation}
which is an isomorphism (see \cite[Proposition 3, p. 386]{Waldspurger81}). Under this isomorphism, the Fourier coefficients of the classical modular form $h$ are related to the Fourier coefficients of $\mathbf h$ as follows. Let $\psi$ be the standard additive character of $\Q \backslash \A$. That is, $\psi_{\infty}(x) = e^{2\pi i x}$ and, for each finite prime $q$, $\psi_q$ is the unique character of $\Q_q$ with kernel $\Z_q$ and such that $\psi_q(x) = e^{-2\pi i x}$ for $x \in \Z[1/q]$. If $\varphi \in \tilde{\mathcal A}_0$ and $\xi \in \Q$, then the $\xi$-th Fourier coefficient of $\varphi$ is defined to be the function on $\widetilde{\SL}_2(\A)$ given by 
\[
 W_{\varphi,\xi}(\gamma) = \int_{\Q \backslash \A} \varphi(u(x)\gamma) \overline{\psi(\xi x)} dx.
\]
If $h \in S_{k+1/2}(4N,\chi)$ has Fourier expansion 
\[
 h(z) = \sum_{\xi \in \Z, \xi > 0} c(\xi) q^{\xi},
\]
and $h \mapsto \mathbf h$ under the above isomorphism, then one has (cf. \cite[Lemme 3]{Waldspurger81})
\begin{equation}\label{FC:classicalautomorphic}
 c(\xi) = e^{2\pi \xi} W_{\mathbf h,\xi}(1).
\end{equation}

For the rest of this subsection, assume that $k \geq 1$ is odd, $N\geq 1$ is odd and square-free, and $\chi$ is an even Dirichlet character satisfying the hypothesis of Theorem \ref{thm:BM}. Let $f \in S_{2k}^{new}(N)$ and $h\in S_{k+1/2}^{+,new}(4NM,\chi; f\otimes\chi)$ be chosen as in the discussion after Theorem \ref{thm:BM}. By using \eqref{FC:classicalautomorphic}, the relation \eqref{relationFC:ShimuraLifts} proved in Lemma \ref{lemma:ShimuraLifts} can be rephrased in automorphic terms as we will now explain.

Given any $\xi \in \Q_+$, set $\xi = \mathfrak d_{\xi}\mathfrak f_{\xi}^2$, where $\mathfrak d_{\xi} \in \mathbb N$ is such that $-\mathfrak d_{\xi}$ equals the discriminant of $\Q(\sqrt{-\xi})/\Q$. By \eqref{relationFC:ShimuraLifts}, we have
\begin{equation}\label{relationFC:ShimuraLifts-xi}
 c(\xi) = c(\mathfrak d_{\xi}) \sum_{\substack{d\mid \mathfrak f_{\xi},\\d>0}} \mu(d)\chi_{-\xi}(d)\chi(d)d^{k-1}a_{\chi}(\mathfrak f_{\xi}/d).
\end{equation}
For each prime $p \nmid N$, let $\{\alpha_p,\alpha_p^{-1}\}$ be the Satake parameter of $f$ at $p$, so that 
\[
 (1-p^{k-1/2}\alpha_pX)(1-p^{k-1/2}\alpha_p^{-1}X) = 1- a(p)X + p^{2k-1}X^2.
\]
In particular, notice that $a(p) = p^{k-1/2}(\alpha_p + \alpha_p^{-1})$. More generally, for each integer $e\geq 0$ one has 
\[
 a(p^e) = p^{e(k-1/2)} \sum_{i=0}^e \alpha_p^{e-2i}.
\]
In contrast, if $p$ is an (odd) prime dividing $N$, being $N$ square-free we have $a(p) = -p^{k-1}w_p$, where $w_p = w_p(f) \in \{\pm 1\}$ is the eigenvalue of the Atkin--Lehner involution at $p$ acting on $f$. Also, one has $a(p^e) = a(p)^e$ for all integers $e \geq 0$ in this case. For each prime $p\mid N$, define $\alpha_p := p^{1/2-k}a(p) = -p^{-1/2}w_p$. In a similar spirit as for primes not dividing $N$, now $a(p^e) = p^{e(k-1/2)}\alpha_p^e$ for $e\geq 0$.

Given a rational prime $p$, put $e_p := \mathrm{ord}_p(\mathfrak f_{\xi})$ and define a function $\Psi_p(\xi;X) \in \C[X,X^{-1}]$ by 
\[
  \Psi_p(\xi;X) = \begin{cases}
                  \frac{X^{e_p+1}-X^{-e_p-1}}{X-X^{-1}} - p^{-1/2}\chi_{-\xi}(p)\frac{X^{e_p}-X^{-e_p}}{X-X^{-1}}, & \text{if } p \nmid N, e_p\geq 0,\\
                  \chi_{-\xi}(p)(\chi_{-\xi}(p)+w_p)X^{e_p}, & \text{if } p\mid N/M, e_p \geq 0,\\
                  \chi_{-\xi}(p)(\chi_{-\xi}(p)-w_p)X^{e_p}, & \text{if } p\mid M, e_p \geq 0,\\
                  0, & e_p < 0.
                 \end{cases}
\]
Observe first of all that, fixed $\xi$, there are only finitely many primes $p$ with $\mathrm{ord}_p(\xi)\neq 0$. Since $\Psi_p(\xi;X)=1$ whenever $p\nmid N$ and $e_p=0$, we see that $\Psi_p(\xi;X) = 1$ for almost all primes. Secondly, at a prime $p \mid N/M$ (resp. $p\mid M$) we see that $\Psi_p(\xi;X) \neq 0$ if and only if $\chi_{-\xi}(p) = w_p$ (resp. $\chi_{-\xi}(p) = -w_p$). Hence, 
\[
  \prod_p \Psi_p(\xi;X) \neq 0 \iff \xi \in \Z, \, \left(\frac{-\xi}{p}\right) = w_p \text{ for all } p \mid N/M, \, \text{and } \left(\frac{-\xi}{p}\right) = -w_p \text{ for all } p \mid M.
\]

\begin{lemma}\label{lemma:c(xi)}
 If $\xi \in \Z$, and $\nu(N)$ denotes the number of prime factors of $N$, then 
 \begin{equation}\label{cxi-Psip}
  c(\xi) = 2^{-\nu(N)}c(\mathfrak d_{\xi}) \chi(\mathfrak f_{\xi})\mathfrak f_{\xi}^{k-1/2} \prod_{p} \Psi_p(\xi;\alpha_p).
 \end{equation}
\end{lemma}
\begin{proof}
 First notice that \eqref{cxi-Psip} holds if $c(\xi) = 0$ by our above observation, so we may assume that $c(\xi)\neq 0$. Secondly, both sides in \eqref{relationFC:ShimuraLifts-xi} are zero if $\xi$ is not an integer, so we may assume that $\xi \in \Z_+$. Writing $\xi = \mathfrak d_{\xi}\mathfrak f_{\xi}^2$ as before, and setting $\mathfrak f_{\xi} = \mathfrak f_{\xi,N}\mathfrak f_{\xi,0}$, with $\mathfrak f_{\xi,N}, \mathfrak f_{\xi,0}$ integers such that $(\mathfrak f_{\xi,0},N)=1$ and every prime divisor of $\mathfrak f_{\xi,N}$ divides $N$, equation \eqref{relationFC:ShimuraLifts-xi} can be rewritten as
 \[
  c(\xi) =  c(\mathfrak d_{\xi})\chi(\mathfrak f_{\xi})a(\mathfrak f_{\xi,N}) \sum_{\substack{d\mid \mathfrak f_{\xi},\\ d > 0}} \mu(d)\chi_{-\xi}(d)d^{k-1}a(\mathfrak f_{\xi,0}/d).
 \]
 Using the definition of the functions  $\Psi_p(\xi;X)$,  we deduce that 
 \[
    c(\xi) =c(\mathfrak d_{\xi})\chi(\mathfrak f_{\xi})a(\mathfrak f_{\xi,N})  \prod_{p\mid \mathfrak f_{\xi,0}} (a(p^{e_p}) - p^{k-1}\chi_{-\xi}(p)a(p^{e_p-1})) = c(\mathfrak d_{\xi})\chi(\mathfrak f_{\xi})a(\mathfrak f_{\xi,N}) \mathfrak f_{\xi,0}^{k-1/2} \prod_{p\mid \mathfrak f_{\xi,0}} \Psi_p(\xi;\alpha_p).
 \]
 Since $c(\xi) \neq 0$, in particular $\Psi_p(\xi;\alpha_p) \neq 0$ for all $p \mid N$. At each prime $p \mid \mathfrak f_{\xi,N}$, we thus have $\Psi_p(\xi;\alpha_p) = 2 \alpha_p^{e_p}$. We therefore deduce that 
 \[
 a(\mathfrak f_{\xi,N}) =  \prod_{p\mid \mathfrak f_{\xi,N}} a(p^{e_p}) = \mathfrak f_{\xi,N}^{k-1/2} \prod_{p\mid \mathfrak f_{\xi,N}} \alpha_p^{e_p}= 2^{-\nu(\mathfrak f_{\xi,N})}\mathfrak f_{\xi,N}^{k-1/2} \prod_{p\mid \mathfrak f_{\xi,N}} \Psi_p(\xi;\alpha_p).
 \]
 At primes $p\mid N$ with $p\nmid \mathfrak f_{\xi,N}$ (if any), we have $\Psi_p(\xi;\alpha_p) = 2$, hence we can rewrite the above identity as 
 \[
  a(\mathfrak f_{\xi,N}) = 2^{-\nu(N)}\mathfrak f_{\xi,N}^{k-1/2} \prod_{p\mid N} \Psi_p(\xi;\alpha_p).
 \]
 Since $\Psi_p(\xi;\alpha_p)=1$ for all primes $p\nmid N$, we deduce that \eqref{cxi-Psip} holds when $c(\xi) \neq 0$. 
\end{proof}

\subsection{Automorphic forms for $\GSp_2$}\label{AF:GSP2}

We will now set the basic notation and definitions concerning automorphic forms for $\GSp_2(\A)$, which will naturally arise in this paper by adelization of Siegel modular forms as the ones considered in Section \ref{MF:Siegel}. Write 
\[
 \GSp_2 = \{g \in \GL_4: {}^tg J_2 g = \nu(g) J_2: \nu(g) \in \mathbb G_m\}, \quad 
 J_2 = \left(\begin{array}{cc} 0 & \mathrm{Id}_2 \\ -\mathrm{Id}_2 & 0 \end{array} \right),
\]
for the general symplectic group of degree $2$, and recall that $\GSp_2^+(\R)$ acts on Siegel's upper half-space as in \eqref{action:GSp2H2}. Here, $\nu: \GSp_2 \to \mathbb G_m$ is the so-called {\em similitude} (or {\em scale}) {\em morphism}. If $N\geq 1$ is an integer, we set $K_0^{(2)}(N; \hat{\Z}) = \prod_{p} K_0^{(2)}(N;\Z_p) \subseteq \GSp_2(\hat{\Z})$, where for each prime $p$, 
\[
K_0^{(2)}(N;\Z_p) = \left \lbrace g = \begin{pmatrix} A & B \\ C & D \end{pmatrix} \in \GSp_2(\Z_p): C \equiv 0 \text{ mod } N \right \rbrace 
\]
is the local analogue of the congruence subgroup $\Gamma_0^{(2)}(N)$ introduced in \eqref{Gamma0Siegel}. Observe that $K_0^{(2)}(N;\Z_p) = \GSp_2(\Z_p)$ for all primes $p\nmid N$. Compact open subgroups of $\GSp_2(\A)$ of the form $K_0^{(2)}(N)$ will play a special role in the paper, although we will also consider certain subgroups of them. 

Let $\mathbf F: \GSp_2(\A) \to \C$ be an automorphic cusp form, and $\Pi$ be the automorphic representation of $\GSp_2(\A)$ associated with $\mathbf F$ (i.e., the closure of the span of all the right-translates of $\mathbf F$). We suppose that $\Pi$ is irreducible and unitary. By Schur's lemma, $\mathbf F$ has a central character: there exists a Hecke character $\lambda: \Q^{\times}\backslash \A^{\times} \to \C^{\times}$ such that $\mathbf F(zg) = \lambda(z)\mathbf F(g)$ for all $z \in \A^{\times} = Z(\GSp_2(\A))$, $g \in \GSp_2(\A)$. If $\lambda$ is trivial, then $\mathbf F$ is trivial on the center of $\GSp_2(\A)$, and hence one can regard $\mathbf F$ as an automorphic cusp form on $\PGSp_2(\A)$. 

Besides, suppose that we are given an automorphic cusp form $\mathbf F: \PGSp_2(\A) \to \C$, and let $\lambda: \Q^{\times}\backslash \A^{\times} \to \C^{\times}$ be a Hecke character. Then one can define an automorphic cusp form $\mathbf F \otimes \lambda: \GSp_2(\A) \to \C$ by the rule $(\mathbf F \otimes \lambda) (g) := \mathbf F(g) \lambda(\nu(g))$. It is readily seen that the central character of $\mathbf F \otimes \lambda$ is $\lambda^2$.

If $B \in \mathrm{Sym}_2(\Q)$ is a symmetric $2$-by-$2$ matrix, then the $B$-th Fourier coefficient of $\mathbf F$ is by definition the function on $\GSp_2(\A)$ given by
\begin{equation}\label{automFC:GSp2}
 \mathcal W_{\mathbf F, B} (g) = \int_{\mathrm{Sym}_2(\Q)\backslash \mathrm{Sym}_2(\A)} \mathbf F(n(X)g) \overline{\psi(\mathrm{Tr}(BX))} dX, \quad g \in \GSp_2(\A),
\end{equation}
where $n(X) = \left(\begin{smallmatrix} \mathbf 1_2 & X \\ 0 & \mathbf 1_2\end{smallmatrix}\right)$. If $\mathbf F$ is right invariant by some subgroup $K \subseteq \GSp_2(\hat{\Z})$, then notice that the Fourier coefficients $\mathcal W_{\mathbf F, B}$ enjoy also the same invariant property. It is well-known that the collection of all Fourier coefficients $\mathcal W_{\mathbf F, B}$ determine the automorphic form $\mathbf F$.

We will not develop the general theory of automorphic forms and representations of $\GSp_2(\A)$, but rather we will focus on the automorphic forms for $\GSp_2(\A)$ that appear by adelization of Siegel modular forms of degree $2$ as the ones considered in Section \ref{MF:Siegel}. Hence, suppose that $k\geq 1$ is an odd integer, $N\geq 1$ is an integer, and $\chi: (\Z/N\Z)^{\times} \to \C^{\times}$ is a Dirichlet character. As usual, let $\underline{\chi}: \A^{\times} \to \C^{\times}$ be the Hecke character associated with $\chi$ as in previous sections. If $F \in S_{k+1}(\Gamma_0^{(2)}(N),\chi)$ is a Siegel modular form of weight $k+1$, level $\Gamma_0^{(2)}(N)$ and character $\chi$, then $F$ defines an automorphic form $\mathbf F$ for $\GSp_2(\A)$ by setting 
\[
 \mathbf F(g) = \det(g_{\infty})^{(k+1)/2}\det(C\sqrt{-1}+D)^{-k-1}F(g_{\infty}\sqrt{-1}) \underline{\chi}(k),
\]
whenever $g = \gamma g_{\infty} k$ with $\gamma \in \GSp_2(\Q)$, $k \in K_0^{(2)}(N)$, and 
\[
 g_{\infty} = \left(\begin{array}{cc} A & B \\ C & D\end{array}\right) \in \GSp_2^+(\R).
\]
Here, $\underline{\chi}$ induces a character on $K_0^{(2)}(N)$ similarly as in the classical situation explained in \eqref{chi:Siegel}. Since $\GSp_2(\A) = \GSp_2(\Q)\GSp_2^+(\R)K_0^{(2)}(N)$ by strong approximation, this gives indeed a well-defined function on $\GSp_2(\A)$. The fact that $F$ is a Siegel cusp form in $S_{k+1}(\Gamma_0^{(2)}(N),\chi)$ easily implies that $\mathbf F$ is $\GSp_2(\Q)$-invariant on the left, that $\mathbf F(gk) = \underline{\chi}(k)\mathbf F(g)$ for all $g \in \GSp_2(\A)$ and $k \in K_0^{(2)}(N)$, and that the center of $\GSp_2(\A)$ acts through $\underline{\chi} \circ \nu$, where $\nu: \GSp_2(\A) \to \A^{\times}$ denotes the similitude morphism. That is to say, $F(zg) = \underline{\chi}(\nu(z))\mathbf F(g)$ for all $z \in Z(\GSp_2(\A)) = \A^{\times}$, $g \in \GSp_2(\A)$, hence the central character of $\mathbf F$ is $\underline{\chi}^2$. We write $\mathcal S_{k+1}(K_{0}^{(2)}(N),\underline{\chi})$ for the space of automorphic cusp forms on $\GSp_2(\A)$ arising by adelization of Siegel cusp forms in $S_{k+1}(\Gamma_0^{(2)}(N),\chi)$.

The level-raising operator $\mathfrak R_p$ introduced classically in \eqref{Rp:Siegel} can be defined analogously in the automorphic setting. Indeed, let $p$ be a prime dividing $N$, and define 
\[
 \mathfrak R_p \mathbf F := \sum_{j=0}^{p-1} \Pi(u_j)\mathbf F,
\]
where $\Pi$ is the automorphic representation of $\GSp_2(\A)$ associated with $\mathbf F$, acting on $\mathbf F$ by right translation: $\Pi(h)\mathbf F: g \mapsto \mathbf F(gh)$. In line with the classical definition, we now set 
\[
 K^{\text{param}}(p;\Z_p) := \left\lbrace \gamma \in \GSp_2(\Q_p) \cap \left(\begin{array}{cccc} \Z_p & p\Z_p & \Z_p & \Z_p \\ \Z_p & \Z_p & \Z_p & p^{-1}\Z_p \\ \Z_p & p\Z_p & \Z_p & \Z_p \\ p\Z_p & p\Z_p & p\Z_p & \Z_p \end{array}\right): \det(\gamma) \in \Z_p^{\times} \right\rbrace.
\]
Then we write $K(N,p;\Z_p) := K_{0}^{(2)}(N;\Z_p) \cap K^{\text{param}}(p;\Z_p)$, and 
\[
 K(N,p) = K(N,p;\Z_p) \times \prod_{q \neq p} K_{0}^{(2)}(N;\Z_q) \subseteq \GSp_2(\hat{\Z}).
\]

If $F\in S_{k+1}(\Gamma_0^{(2)}(N),\chi)$ is as above, recall that $\mathfrak R_p F$ belongs to $S_{k+1}(\Gamma_0^{(2)}(N')\cap \Gamma^{\text{param}}(p),\chi)$, where $N' = N$ if $p^2 \mid N$, and $N' = Np$ otherwise. We have $K(N',p) \cap \GSp_2(\Q) = \Gamma_0^{(2)}(N')\cap \Gamma^{\text{param}}(p)$, and hence Siegel forms in $S_{k+1}(\Gamma_0^{(2)}(N')\cap \Gamma^{\text{param}}(p),\chi)$ induce by adelization automorphic cusp forms on $\GSp_2(\A)$ on which $K(N',p)$ acts on the right through the character $\underline{\chi}: K(N',p) \to \C^{\times}$. We write $\mathcal S_{k+1}(K(N',p),\underline{\chi})$ for the space of automorphic forms on $\GSp_2(\A)$ obtained by adelization of Siegel forms in $S_{k+1}(\Gamma_0^{(2)}(N')\cap \Gamma^{\text{param}}(p),\chi)$.

\begin{lemma}\label{lemma:RpF}
 Let $F \in S_{k+1}(\Gamma_0^{(2)}(N),\chi)$, and $\mathbf F \in \mathcal S_{k+1}(K_{0}^{(2)}(N),\underline{\chi})$ be its adelization as above. If $p$ is a prime dividing $N$, then $\mathfrak R_p\mathbf F$ is the adelization of $\mathfrak R_p F$. In particular, $\mathfrak R_p \mathbf F \in \mathcal S_{k+1}(K(Np,p),\underline{\chi})$. If $p^2$ divides $N$, then one actually has $ \mathfrak R_p \mathbf F \in \mathcal S_{k+1}(K(N,p),\underline{\chi})$.
\end{lemma}

We now particularize the above discussion to a situation which is of particular interest in this note. Continue to assume that $k\geq 1$ is an odd integer, and let $N\geq 1$ be an odd square-free integer, and $\chi$ be an even Dirichlet character of conductor $M \mid N$. Let $f \in S_{2k}^{new}(\Gamma_0(N))$ be a normalized new cusp form of weight $2k$, level $N$ and trivial nebentype, and let $F_{\chi} \in S_{k+1}(\Gamma_0^{(2)}(NM),\chi)$ be the Saito--Kurokawa lift of $f \otimes \chi$ defined in Section \ref{MF:Siegel}. Write $\mathbf F_{\chi} \in \mathcal S_{k+1}(K_0^{(2)}(NM),\underline{\chi})$ for the adelization of $F_{\chi}$, hence 
\begin{equation}\label{adelizationFchi}
  \mathbf F_{\chi}(g) = \det(g_{\infty})^{(k+1)/2}\det(C\sqrt{-1}+D)^{-k-1}F_{\chi}(g_{\infty}\sqrt{-1}) \underline{\chi}(k),
\end{equation}
for all $g = \gamma g_{\infty} k$ with $\gamma \in \GSp_2(\Q)$, $k \in K_0^{(2)}(N)$, and 
\[
 g_{\infty} = \left(\begin{array}{cc} A & B \\ C & D\end{array}\right) \in \GSp_2^+(\R).
\]

Let $B \in \mathrm{Sym}_2(\Q)$ be a symmetric $2$-by-$2$ matrix, and $\mathcal W_{\mathbf F_{\chi}, B}$ denote the $B$-th Fourier coefficient of $\mathbf F_{\chi}$ as defined above. By strong approximation, together with \eqref{adelizationFchi}, $\mathcal W_{\mathbf F_{\chi}, B}: \GSp_2(\A) \to \C$ is determined by the values $\mathcal W_{\mathbf F_{\chi}, B}(g_{\infty})$ for $g_{\infty} \in \GSp_2^+(\R)$. Every element $g_{\infty} \in \GSp_2^+(\R)$ can be written as
\[
 g_{\infty} = \left(\begin{array}{cc} z\mathbf 1_2 & 0 \\ 0 & z\mathbf 1_2\end{array}\right)
 \left(\begin{array}{cc} \mathbf 1_2 & X \\ 0 & \mathbf 1_2\end{array}\right)
 \left(\begin{array}{cc} A & 0 \\ 0 & {}^tA^{-1} \end{array}\right)
 \left(\begin{array}{cc} \alpha & \beta \\ -\beta & \alpha\end{array}\right),
\]
where $z \in \R^{\times}_+$, $X\in \mathrm{Sym}_2(\R)$, $A \in \GL_2^+(\R)$, and $\left(\begin{smallmatrix} \alpha & \beta \\ -\beta & \alpha\end{smallmatrix}\right) \in \Sp_2(\R)$ with $\mathbf k = \alpha+\sqrt{-1}\beta \in U(2)$. Since 
\[
 \mathbf F_{\chi}\left(\left(\begin{array}{cc} z\mathbf 1_2 & 0 \\ 0 & z\mathbf 1_2\end{array}\right)g \left(\begin{array}{cc} \alpha & \beta \\ -\beta & \alpha\end{array}\right)\right) = \det(\mathbf k)^{k+1} \mathbf F_{\chi}(g)
\]
for $z \in \R^{\times}_+$ and $\left(\begin{smallmatrix} \alpha & \beta \\ -\beta & \alpha\end{smallmatrix}\right) \in \Sp_2(\R)$ as before, we see that $\mathcal W_{\mathbf F_{\chi}, B}$ is actually determined by the values $\mathcal W_{\mathbf F_{\chi}, B}(g_{\infty})$ for elements $g_{\infty} \in \GSp_2^+(\R)$ of the form 
\[
 g_{\infty} = n(X) m(A,1) = \left(\begin{array}{cc} \mathbf 1_2 & X \\ 0 & \mathbf 1_2\end{array}\right)\left(\begin{array}{cc} A & 0 \\ 0 & {}^tA^{-1} \end{array}\right),
\]
with $X\in \mathrm{Sym}_2(\R)$, and $A \in \GL_2^+(\R)$. And for $g_{\infty} = n(X) m(A,1)$, one checks from the definitions that $\mathcal W_{\mathbf F_{\chi}, B}(g_{\infty}) = 0$ unless $B$ is positive definite and half-integral, in which case one has 
\begin{equation}\label{WFchiB:ginfty}
 \mathcal W_{\mathbf F_{\chi}, B}(g_{\infty}) = A_{\chi}(B) \det(Y)^{(k+1)/2}e^{2\pi\sqrt{-1}\mathrm{Tr}(BZ)},
\end{equation}
where $Y = A {}^t A$, $Z = X + \sqrt{-1}Y \in \mathcal H_2$, and $A_{\chi}(B)$ is the $B$-th Fourier coefficient of $F_{\chi}$ (cf. \eqref{Fchi:Fourier}).

Finally, let $p$ be a prime dividing $M$. In particular, $p^2 \mid NM$, and the adelization of $\mathfrak R_p F_{\chi} \in S_{k+1}(\Gamma_0^{(2)}(NM)\cap \Gamma^{\text{param}}(p),\chi)$ is precisely (cf. Lemma \ref{lemma:RpF})
\[
 \mathfrak R_p \mathbf F_{\chi} \in \mathcal S_{k+1}(K(NM,p), \underline{\chi}).
\]

It is not hard to see from the definitions that the Fourier coefficients of $\mathfrak R_p \mathbf F_{\chi}$ are closely related to those of $\mathbf F_{\chi}$. More precisely, one can prove the following lemma, whose proof is left for the reader.

\begin{lemma}
 With notation as above, let $B = \left(\begin{smallmatrix} b_1 & b_2/2 \\ b_2/2 & b_3\end{smallmatrix}\right)$ be a positive definite, half-integral symmetric matrix, and let $g_{\infty} = n(X)m(A,1) \in \GSp_2(\R)$ be as before. Then $\mathcal W_{\mathfrak R_p\mathbf F_{\chi}, B}(g_{\infty}) = 0$ unless $b_3 \in p\Z$, and if this holds then 
 \[
 \mathcal W_{\mathfrak R_p\mathbf F_{\chi}, B}(g_{\infty}) = \mathcal W_{\mathbf F_{\chi}, B}(g_{\infty}).
 \]
\end{lemma}

If we repeat the above for all primes $p$ dividing $M$, or in other words, if we apply the operator $\mathfrak R_M := \prod_{p\mid M} \mathfrak R_p$ to $\mathbf F_{\chi}$, we obtain an automorphic cusp form 
\[
 \mathfrak R_M \mathbf F_{\chi} \in \mathcal S_{k+1}(K(NM,M), \underline{\chi}),
\]
and directly from the previous lemma we deduce the following:

\begin{corollary}\label{FourierCoeffRMFchi}
Let $B$ be a positive definite, half-integral symmetric matrix as in the previous lemma, and let $g_{\infty} = n(X)m(A,1) \in \GSp_2(\R)$ be as above. Then $\mathcal W_{\mathfrak R_M\mathbf F_{\chi}, B}(g_{\infty}) = 0$ unless $b_3 \in M\Z$, and if this holds then 
 \[
 \mathcal W_{\mathfrak R_M\mathbf F_{\chi}, B}(g_{\infty}) = \mathcal W_{\mathbf F_{\chi}, B}(g_{\infty}).
 \]
\end{corollary}

\section{Quadratic spaces and theta lifts}\label{sec:QS-Theta}

This section is devoted to briefly recall the essentials on quadratic spaces and theta lifts. We focus especially in the three theta correspondences that will be considered in this paper, which explain the different lifts of classical modular forms described in Section \ref{sec:MF} in the language of automorphic forms given in Section \ref{sec:AF}.

\subsection{Quadratic spaces}

Let $F$ be a field with $\mathrm{char}(F)\neq 2$, and $V$ be a quadratic space over $F$. That is to say, $V$ is a finite dimensional vector space over $F$ equipped with a non-degenerate symmetric bilinear form $(\, , \,)$. We denote by $Q$ the associated quadratic form on $V$, given by 
\[
 Q(x) = \frac{1}{2}(x,x), \quad x \in V.
\]
If $m = \mathrm{dim}(V)$, fixing a basis $\{v_1,\dots v_m\}$ of $V$ and identifying $V$ with the space of column vectors $F^m$, the bilinear form $(\, , \,)$ determines a matrix $Q \in \GL_m(F)$ by setting $Q = ((v_i,v_j))_{i,j}$. Then we have 
\[
 (x,y) = {}^t x Q y \quad \text{for } x,y \in V.
\]
We define $\det(V)$ to be the image of $\det(Q)$ in $F^{\times}/F^{\times, 2}$. The orthogonal similitude group of $V$ is 
\[
 \mathrm{GO}(V) = \{h \in \GL_m: {}^t hQh = \nu(h)Q, \, \nu(h) \in \mathbb G_m\},
\]
where $\nu: \mathrm{GO}(V) \to \mathbb G_m$ is the so-called similitude morphism (or scale map). From the very definition, observe that $\det(h)^2 = \nu(h)^m$ for every $h \in \mathrm{GO}(V)$. When $m$ is even, set 
\[
 \mathrm{GSO}(V) = \{h \in \mathrm{GO}(V): \det(h) = \nu(h)^{m/2}\}.
\]
Finally, we let $\mathrm O(V) = \ker(\nu)$ denote the orthogonal group of $V$, and write $\mathrm{SO}(V) = \mathrm O(V) \cap \SL_m$ for the special orthogonal group.

\subsection{Explicit realizations in low rank}\label{spaces:lowrank}

In this paper, we are particularly interested in orthogonal groups for vector spaces of dimension $3$, $4$ and $5$. For this reason, we fix here certain explicit realizations that will be used later on to describe automorphic representations for $\mathrm{SO}(V)(\A)$ and $\mathrm{GSO}(V)(\A)$. Our choices follow quite closely the ones in \cite{Ichino-pullbacks, Qiu}.

When $\mathrm{dim}(V) = 3$, one can show that there is a unique quaternion algebra $B$ over $F$ and an element $a \in F^{\times}$ such that $(V,q) \simeq (V_B, aq_B)$, where $V_B = \{x \in B: \mathrm{Tr}_B(x) = 0\}$ is the subspace of elements in $B$ with zero trace (sometimes called `pure quaternions'), and $q_B(x) = -\mathrm{Nm}_B(x)$. The group of invertible elements $B^{\times}$ acts on $V_B$ by conjugation, $b \cdot x = b x b^{-1}$, and this action gives rise to an isomorphism 
\[
 PB^{\times} \, \stackrel{\simeq}{\longrightarrow} \, \mathrm{SO}(V_B,q_B) \simeq \mathrm{SO}(V,q).
\]
When $B = \mathrm{M}_2$ is the split algebra of $2$-by-$2$ matrices, then notice that $PB^{\times} = \PGL_2$, thus the above identifies $\PGL_2$ with the special orthogonal group of a three-dimensional quadratic space.

In dimension $4$, consider the vector space $V_4 := \mathrm{M}_2(F)$ of $2$-by-$2$ matrices, equipped with the quadratic form $q(x) = \det(x)$. The associated non-degenerate bilinear form is $(x,y) = \mathrm{Tr}(xy^*)$, where 
\[
 x^* = \left(\begin{array}{cc} x_4 & -x_2 \\ -x_3 & x_1\end{array}\right) \quad \text{for } x = \left(\begin{array}{cc} x_1 & x_2 \\ x_3 & x_4\end{array}\right) \in \M_2(F).
\]
There is an exact sequence 
\begin{equation}\label{exseq:GSO4}
 1 \, \longrightarrow \, \mathbb G_m \, \stackrel{\iota}{\longrightarrow} \, \GL_2 \times \GL_2 \, \stackrel{\rho}{\longrightarrow} \GSO(V_4) \, \longrightarrow 1,
\end{equation}
where $\iota(a) = (a\mathbf 1_2, a^{-1}\mathbf 1_2)$ and $\rho(h_1,h_2)x = h_1xh_2^*$ for $a \in \mathbb G_m$ and $h_1,h_2 \in \GL_2$. One has $\nu(\rho(h_1,h_2)) = \det(h_1h_2) = \det(h_1)\det(h_2)$. In particular, when $F$ is a number field, automorphic representations of $\GSO(V_4)$ can be seen as automorphic representations of $\GL_2 \times \GL_2$ through the homomorphism $\rho$ in the above short exact sequence. Here we might warn the reader that our choice for the homomorphism $\rho$ in \eqref{exseq:GSO4} agrees with the one on \cite{Qiu} and \cite{GanTakeda}, but differs from the one considered in \cite{Ichino-pullbacks} (or \cite{IchinoIkeda}), which leads to a slightly different model for $\GSO(V_4)$.

Finally, in dimension $5$ we will describe a realization of $\mathrm{SO}(3,2)$, the special orthogonal group of a $5$-dimensional quadratic space $(V,q)$ of Witt index $2$. Although the isomorphism class of such a quadratic space depends on $\det(V)$, the group $\mathrm{SO}(V,q)$ does not. We describe a model $V_5$ of such a quadratic space, with determinant $1$ (modulo $F^{\times,2}$). Namely, start considering the $4$-dimensional space $F^4$ of column vectors, on which $\GSp_2 \subset \GL_4$ acts on the left. Let 
\[
 e_1 = {}^t(1,0,0,0),\, \dots \, , e_4 = {}^t(0,0,0,1)
\]
be the standard basis on $F^4$, and equip $\tilde V := \wedge^2 F^4$ with the non-degenerate symmetric bilinear form $(\, , \,)$ defined by 
\[
 x \wedge y = (x,y) \cdot (e_1 \wedge e_2 \wedge e_3 \wedge e_4), \quad \text{for all } x, y \in \tilde V.
\]
Set $x_0 := e_1 \wedge e_3 + e_2 \wedge e_4$, and define the $5$-dimensional subspace $V_5 \subset \tilde V$ to be the orthogonal complement of the span of $x_0$, i.e.
\[
 V_5 := \{x \in \tilde V: (x,x_0) = 0\}.
\]
Then the homomorphism $\tilde{\rho}:\GSp_2 \to \SO(\tilde V)$ given by $\tilde{\rho}(h) = \nu(h)^{-1} \wedge^2 (h)$ satisfies $\tilde{\rho}(h)x_0 = x_0$, and therefore induces an exact sequence 
\begin{equation}\label{exseq:SO5}
  1 \, \longrightarrow \, \mathbb G_m \, \stackrel{\iota}{\longrightarrow} \, \GSp_2 \, \stackrel{\rho}{\longrightarrow} \SO(V_5) \, \longrightarrow 1,
\end{equation}
where $\iota(a) = a\mathbf 1_4$ for $a \in \mathbb G_m$. This short exact sequence induces an identification 
\[
 \PGSp_2 \simeq \SO(V_5).
\]

We fix an identification of $V_5$ with the $5$-dimensional space $F^5$ of column vectors by 
\[
 \sum_{i=1}^5 x_i v_i \, \longmapsto \, {}^t(x_1,x_2,x_3,x_4,x_5),
\]
where $v_1 = e_2\wedge e_1, \,\, v_2 = e_1 \wedge e_4, \,\, v_3 = e_1\wedge e_3 - e_2\wedge e_4, \,\, v_4 = e_2\wedge e_3, \,\, v_5 = e_3\wedge e_4$. Upon this identification, we consider the non-degenerate bilinear symmetric form $(\,,\,)$ on $V$ defined by $(x,y) = {}^txQy$ for $x,y\in F^5$, where 
\[
 Q = \left(\begin{array}{ccc} & & -1 \\ & Q_1 \\ -1 \end{array}\right), \quad Q_1 = \left(\begin{array}{ccc} 0 & 0 & 1 \\ 0 & 2 & 0 \\ 1 & 0 & 0 \end{array}\right).
\]

We shall distinguish the $3$-dimensional subspace $V_3 \subset V_5$ spanned by $v_2$, $v_3$, $v_4$, equipped with the bilinear form $(x,y) = {}^txQ_1y$, for $x,y \in F^3$, under the identification $V_3 = F^3$ induced by restricting the above one for $V = F^5$. Notice that $V_5 = \langle v_1\rangle \oplus V_3 \oplus \langle -v_5\rangle$, where $v_1$ and $-v_5$ are isotropic vectors with $(v_1,-v_5) = 1$, and $V_3$ is the orthogonal complement of $\langle v_1,-v_5\rangle = \langle v_1,v_5\rangle$.

Also, we distinguish a $4$-dimensional subspace of $V_5$. Indeed, the subspace $\{x \in V: (x,v_3) = 0\} = \langle v_3 \rangle^{\perp} \subset V_5$ is a quadratic $4$-dimensional subspace of $V_5$, and it can be identified with the space $V_4$ defined above by means of the linear map
 \[
  \langle v_3 \rangle^{\perp} \, \longrightarrow \, V_4, \quad x_1 v_2 + x_2 v_1 + x_3 v_5 + x_4 v_4 \, \longmapsto \, \left(\begin{array}{cc} x_1 & x_2\\ x_3 & x_4\end{array}\right).
 \]

By restricting the homomorphism $\rho$ from the exact sequence in \eqref{exseq:GSO4} to 
\[
 G(\SL_2\times\SL_2)^- := \{(h_1,h_2)\in \GL_2\times\GL_2 : \det(h_1)\det(h_2) = 1\} \subseteq \GL_2 \times \GL_2,
\]
one gets an exact sequence 
\begin{equation}\label{exseq:SO4}
  1 \, \longrightarrow \, \mathbb G_m \, \stackrel{\iota}{\longrightarrow} \, G(\SL_2 \times \SL_2)^- \, \stackrel{\rho}{\longrightarrow} \SO(V_4) \, \longrightarrow 1.
\end{equation}
Now notice that $G(\SL_2 \times \SL_2)^-$ is isomorphic to 
\[
  G(\SL_2\times\SL_2) := \{(h_1,h_2)\in \GL_2\times\GL_2 : \det(h_1)\det(h_2)^{-1} = 1\} \subseteq \GL_2 \times \GL_2
\]
through the morphism $(h_1,h_2) \mapsto (h_1,\det(h_2)^{-1}h_2)$. Composing this isomorphism with the natural embedding $G(\SL_2 \times \SL_2) \hookrightarrow \GSp_2$ given by 
\[
  \left(\left(\begin{array}{cc} a_1 & b_1 \\ c_1 & d_1 \end{array}\right), \left(\begin{array}{cc} a_2 & b_2 \\ c_2 & d_2 \end{array}\right)\right) \, \longmapsto \, \left(\begin{array}{cccc} a_1 & 0 & b_1 & 0 \\ 0 & a_2 & 0 & b_2 \\ c_1 & 0 & d_1 & 0 \\ 0 & c_2 & 0 & d_2\end{array}\right),
 \]
one gets a commutative diagram
\[
 \xymatrix{
 1 \ar[r] & \mathbb G_m \ar[r]^{\iota \qquad} \ar@{=}[d] & G(\SL_2 \times \SL_2)^- \ar@{^{(}->}[d] \ar[r]^{\quad \rho} & \mathrm{SO}(V_4) \ar[r] \ar@{^{(}->}[d] & 1 \\
 1 \ar[r] & \mathbb G_m \ar[r]^{\iota} & \GSp_2 \ar[r]^{\rho} & \mathrm{SO}(V_5) \ar[r] & 1
 }
\]
and hence an embedding $\SO(V_4) \subset \SO(V_5)$. This embedding will be of crucial interest later on.

\subsection{Weil representations}

Let now $F$ be a local field with $\mathrm{char}(F)\neq 2$ (for our purposes, one can think of $F$ being $\Q_v$ for a rational place $v$), and $(V,Q)$ be a quadratic space over $F$ of dimension $m$ as above. Let $\mathcal S(V)$ denote the space of locally constant and compactly supported complex-valued functions on $V$. This is usually referred to as the space of Bruhat--Schwartz functions on $V$. If $F$ is archimedean, we rather consider $\mathcal S(V)$ to be the Fock model (which is a smaller subspace, see \cite[Section 2.1.2]{YZZ}). 

We fix a non-trivial additive character $\psi$ of $F$. The Weil representation $\omega_{\psi,V}$ of $\widetilde{\SL}_2(F) \times \mathrm{O}(V)$ on $\mathcal S(V)$, which depends on the choice of the character $\psi$, is given by the following formulae. If $a \in F^{\times}$, $b \in F$, $h \in \mathrm{O}(V)$, and $\phi \in \mathcal S(V)$, then 
\begin{align*}
 \omega_{\psi,V}(h)\phi(x) & = \phi(h^{-1}x), \\
 \omega_{\psi,V} \left(\left(\begin{array}{cc} a & 0 \\ 0 & a^{-1}\end{array}\right),\epsilon \right)\phi(x) & = 
 \epsilon^m \chi_{\psi,V}(a) |a|^{m/2}\phi(ax) \\
 \omega_{\psi,V} \left(\left(\begin{array}{cc} 1 & b \\ 0 & 1\end{array}\right),1 \right)\phi(x) & = \psi(Q(x)b)\phi(x), \\
 \omega_{\psi,V} \left(\left(\begin{array}{cc} 0 & 1 \\ -1 & 0\end{array}\right),1 \right)\phi(x) & = \gamma(\psi,V) \int_F \phi(y)\psi((x,y))dy.
\end{align*}

Here, $\gamma(\psi,V)$ is the Weil index, which is an $8$-th root of unity. When $m=1$ and $Q(x) = x^2$, we will simply write $\omega_{\psi}, \chi_{\psi}$, and $\gamma(\psi)$ for $\omega_{\psi,V}, \chi_{\psi,V}$, and $\gamma(\psi,V)$, respectively. In this case, one has 
\[
 \chi_{\psi}(a) = (a,-1)_F \frac{\gamma(\psi^a)}{\gamma(\psi)},
\]
where $(\,,\,)_F$ denotes the Hilbert symbol. This satisfies $\chi_{\psi}(ab) = (a,b)_F\chi_{\psi}(a)\chi_{\psi}(b)$, $\chi_{\psi}(a^2) = 1$, and $\chi_{\psi^a} = \chi_{\psi} \cdot \chi_a$, where $\chi_a$ stands for the quadratic character $(a, \cdot)_F$.

For a general $V$, if $Q(x) = a_1x_1^2 + \cdots + a_mx_m^2$ with respect to some orthogonal basis, then $\gamma(\psi,V) = \prod_{i} \gamma_{\psi^{a_i}}$, and $\chi_{\psi,V} = \prod_i \chi_{\psi^{a_i}}$. This does not depend on the chosen basis. 

When $m$ is even, the above simplifies considerably. Indeed, if $m$ is even, it is clear from the above description that the Weil representation descends to a representation of $\SL_2(F) \times \mathrm{O}(V)$ on $\mathcal S(V)$. Further, the Weil index $\gamma(\psi,V)$ is a $4$-th root of unity in this case, and $\chi_{\psi,V}$ becomes the quadratic character associated with the quadratic space $(V,Q)$. This means that 
\[
 \chi_{\psi,V}(a) = (a, (-1)^{m/2}\det(V))_F, \quad a \in F^{\times}.
\]

It will be useful in some settings to extend the Weil representation $\omega_{\psi,V}$ described above. If $m$ is even, one defines 
\[
 R = \mathrm G(\SL_2 \times \mathrm O(V)) = \{(g,h) \in \GL_2 \times \mathrm{GO}(V): \det(g) = \nu(h)\},
\]
and then $\omega_{\psi,V}$ extends to a representation of $R(F)$ on $\mathcal S(V)$ by setting
\[
 \omega_{\psi,V}(g,h)\phi = \omega_{\psi,V}\left(g\left(\begin{array}{cc} 1 & 0 \\ 0 & \det(g)^{-1}\end{array}\right),1\right)L(h)\phi \quad \text{for } (g,h) \in R(F) \text{ and } \phi \in \mathcal S(V), 
\]
where $L(h)\phi(x) = |\nu(h)|_F^{-m/4} \phi(h^{-1}x)$ for $x \in V$.

\subsection{Theta functions and theta lifts}

Now let $F$ be a number field (for our purposes, we will just consider $F = \Q$), and consider a quadratic space $V$ over $F$ of dimension $m$. Fix a non-trivial additive character $\psi$ of $\A_F/F$ and let $\omega = \omega_{\psi,V}$ be the Weil representation of $\widetilde{\SL}_2(\A_F) \times \mathrm O(V)(\A_F)$ on $\mathcal S(V(\A_F))$ with respect to $\psi$. Given $(g,h) \in \widetilde{\SL}_2(\A_F) \times \mathrm O(V)(\A_F)$ and $\phi \in \mathcal S(V(\A_F))$, let 
\[
 \theta(g,h;\phi) := \sum_{x \in V(F)} \omega(g,h) \phi(x).
\]
Then $(g,h) \mapsto \theta(g,h;\phi)$ defines an automorphic form on $\widetilde{\SL}_2(\A_F) \times \mathrm O(V)(\A_F)$, called a {\em theta function}. When $m$ is even, this may be regarded as an automorphic form on $\SL_2(\A_F) \times \mathrm O(V)(\A_F)$.

Let $f$ be a cusp form on $\SL_2(\A_F)$ if $m$ is even, and a genuine cusp form on $\widetilde{\SL}_2(\A_F)$ if $m$ is odd. If $\phi \in \mathcal S(V(\A_F))$, put 
\[
 \theta(h;f,\phi) = \int_{\SL_2(F)\backslash\SL_2(\A_F)} \theta(g,h;\phi)f(g) dg, \quad h \in \mathrm O(V)(\A_F).
\]
Then $\theta(f,\phi): h \mapsto \theta(h;f,\phi)$ defines an automorphic form on $\mathrm O(V)(\A_F)$. If $m$ is even and $\pi$ is an irreducible cuspidal automorphic representation of $\SL_2(\A_F)$, or if $m$ is odd and $\pi$ is an irreducible genuine cuspidal automorphic representation of $\widetilde{\SL}_2(\A_F)$, put 
\[
 \Theta_{\widetilde{\SL}_2 \times \mathrm O(V)}(\pi) := \{\theta(f,\phi): f \in \pi, \phi \in \mathcal S(V(\A_F))\}.
\]
Then $\Theta_{\widetilde{\SL}_2 \times \mathrm O(V)}(\pi)$ is an automorphic representation of $\mathrm O(V)(\A_F)$, called the {\em theta lift of $\pi$}. Going in the opposite direction, one defines similarly the theta lift $\theta(f',\phi)$ of a cusp form $f'$ on $\mathrm O(V)(\A_F)$ and the theta lift $\Theta_{\mathrm O(V) \times \widetilde{\SL}_2}(\pi')$ of an irreducible cuspidal automorphic representation $\pi'$ of $\mathrm O(V)(\A)$.

Suppose that $m$ is even. As we did for the Weil representation, theta lifts can also be extended as follows. If $(g,h) \in R(\A_F)$ and $\phi \in \mathcal S(V(\A_F))$, one defines $\theta(g,h;\phi)$ via the same expression as above (using the extended Weil representation). Then, if $f$ is a cusp form on $\GL_2(\A_F)$ and $h \in \mathrm{GO}(V)(\A_F)$, choose $g' \in \GL_2(\A_F)$ with $\det(g') = \nu(h)$ and set 
\[
 \theta(h;f,\phi) = \int_{\SL_2(F)\backslash \SL_2(\A_F)} \theta(gg',h;\phi) f(gg') dg.
\]
The integral does not depend on the choice of the auxiliary element $g'$, and $\theta(f,\phi): h \mapsto \theta(h;f,\phi)$ defines now an automorphic form on $\mathrm{GO}(V)(\A_F)$. If $\pi$ is an irreducible cuspidal automorphic representation of $\GL_2(\A_F)$, then its theta lift $\Theta_{\GL_2 \times \mathrm{GO}(V)}(\pi)$ is formally defined exactly as before (and the same applies for $\Theta_{\mathrm{GO}(V) \times \GL_2}(\pi')$ if $\pi'$ is an irreducible cuspidal automorphic representation of $\mathrm{GO}(V)$).

\section{Three theta correspondences}

\subsection{The pair $(\GL_2, \mathrm{GO}_{2,2})$}\label{theta:GL2GSO4}

Let $V_4$ be the (split) four-dimensional quadratic space as above, and write from now on $\mathrm{GSO}_{2,2} \subset \mathrm{GO}_{2,2}$ for the groups $\mathrm{GSO}(V_4) \subset \mathrm{GO}(V_4)$, and likewise $\mathrm{O}_{2,2}$ for $\mathrm{O}(V_4)$. The theta correspondence for the pair $(\GL_2, \mathrm{GO}_{2,2})$ is sometimes referred to as the Jacquet--Langlands--Shimizu correspondence \cite{Shimizu} (cf. also \cite[Section 5]{IchinoIkeda}, \cite[Section 6]{Ichino-pullbacks}). We will be interested in the restriction of automorphic forms on $\mathrm{GO}_{2,2}(\A)$ to $\mathrm{GSO}_{2,2}(\A)$ (particularly in those arising as theta lifts from automorphic forms on $\GL_2(\A)$). 

By virtue of the exact sequence in \eqref{exseq:GSO4}, automorphic forms on $\mathrm{GSO}_{2,2}(\A)$ might be seen through the homomorphism $\rho$ as representations $\tau_1 \boxtimes \tau_2$ of $\GL_2 \times \GL_2$ with $\tau_1$ and $\tau_2$ having {\em the same} central character. The involution $x \mapsto x^*$ induces an element of order two $c \in \mathrm{GO}_{2,2}$, and $\mathrm{GO}_{2,2} = \mathrm{GSO}_{2,2} \rtimes \langle c \rangle$. 

Notice that for the theta correspondence to hold between $\GL_2$ and $\mathrm{GO}_{2,2}$ it does not suffice to consider the Weil representation of $\SL_2(\A) \times \mathrm{O}_{2,2}$; one needs to consider the Weil representation extended to the group $R(\A) = \{(g,h) \in \GL_2(\A) \times \mathrm{GO}_{2,2}(\A): \det(g) = \nu(h)\}$ as explained above. If $\tau$ is an irreducible cuspidal unitary representation of $\GL_2(\A)$, then one has $\Theta_{\GL_2 \times \mathrm{GSO}_{2,2}}(\tau) = \tau \boxtimes \tau$, where in line with the above comment $\Theta_{\GL_2 \times \mathrm{GSO}_{2,2}}(\tau)$ might be read as the restriction to $\mathrm{GSO}_{2,2}(\A)$ of the theta lift $\Theta_{\GL_2 \times \mathrm{GO}_{2,2}}(\tau)$. Conversely, suppose that $\Upsilon_0$ is an irreducible cuspidal unitary $\mathrm{GSO}_{2,2}(\A)$-representation. Then there is a unique extension of $\Upsilon_0$ to an automorphic representation $\Upsilon$ of $\mathrm{GO}_{2,2}(\A)$ on which there is a non-zero $\mathrm{O}(V_4')(\A)$-
invariant distribution, where $V_4' = \{x \in V_4: \mathrm{tr}(x) = 0\}$. If $\Theta_{\mathrm{GO}_{2,2}\times \GL_2}(\Upsilon) \neq 0$, then $\Upsilon_0$ is of the form $\tau \otimes \tau$ for some irreducible cuspidal unitary representation $\tau$ of $\GL_2(\A)$ and $\Theta_{\mathrm{GO}_{2,2}\times \GL_2}(\Upsilon) = \tau$.

Let us consider a normalized newform $g \in S_{k+1}(N,\chi)$ as in the introduction, hence $k \geq 1$ is an odd integer, $N\geq 1$ is an odd square-free integer, and $\chi$ is an even Dirichlet character modulo $N$. Write $M$ for the conductor of $\chi$ (in particular, $M$ is also odd and square-free). Let $\tau$ be the irreducible cuspidal automorphic $\GL_2(\A)$-representation associated with $g$, with central character $\underline{\chi}$. Then $\tau \boxtimes \tau$ can be regarded as a representation of $\mathrm{GSO}_{2,2}(\A)$ and it extends to a unique automorphic representation $\Upsilon$ of $\mathrm{GO}_{2,2}(\A)$ on which there is a non-zero $\mathrm{O}(V_4')(\A)$-invariant distribution. Then one has 
\[
 \Theta(\tau) = \Upsilon, \quad \Theta(\Upsilon) = \tau.
\]

Let $\mathbf g \in \tau$ be the adelization of $g$. Then the cusp form $\mathbf g \otimes \mathbf g \in \tau \boxtimes \tau$ extends to a cusp form $\mathbf G \in \Upsilon$ on $\mathrm{GO}_{2,2}(\A)$ such that $\mathbf G(hh') = \mathbf G(h)$ for all $h \in \mathrm{GO}_{2,2}(\A)$ and $h' \in \mu_2(\A)$, where $\mu_2$ is the subgroup of $\mathrm{O}_{2,2}$ generated by the involution $\ast$ on $V_4$. 

Define $\mathbf g^{\sharp} = \otimes_v \mathbf g_v^{\sharp}$ by setting $\mathbf g_v^{\sharp} = \mathbf g_v$ for all places $v\neq 2$, and $\mathbf g_2^{\sharp} = \tau_2(t(2^{-1})_2)\mathbf g_v$, where 
\[
 t(2^{-1})_2 = \left( \begin{array}{cc} 2^{-1} & 0 \\ 0 & 2 \end{array}\right) \in \SL_2(\Q_2).
\]
Further, consider the $M$-th level raising operator $\mathbf V_M$ acting on $\tau$ by $\varphi \mapsto \tau(\varpi_M)\varphi$, where $\varpi_M\in \GL_2(\A)$ is $1$ at every place $v\nmid M$, and equals $\varpi_p = \left(\begin{smallmatrix} p^{-1} & 0 \\ 0 & 1\end{smallmatrix}\right) \in \GL_2(\Q_p)$ at each prime $p\mid M$. Then define $\breve{\mathbf g} := \mathbf V_M\mathbf g^{\sharp} = (\mathbf V_M\mathbf g)^{\sharp}$. The cusp form $\breve{\mathbf g}$ thus obtained is one of the factors in our test vector.

Accordingly, we also modify slightly the cusp form $\mathbf G$ in the following manner. For each prime $p \mid M$, consider the element $h_p = (1,\varpi_p) \in \GL_2(\Q_p) \times \GL_2(\Q_p)$, which we identify with its image $\rho(h_p) \in \mathrm{GSO}_{2,2}(\Q_p) \subseteq \mathrm{GO}_{2,2}(\Q_p)$. Let $\mathbf Y_p$ denote the operator acting on $\Upsilon$ given by $\Upsilon(h_p)$, and $\mathbf Y_M := \prod_{p\mid M} \mathbf Y_p$. Similarly as above, if $h_M$ denotes the element in $\mathrm{GO}_{2,2}(\A)$ whose entries are trivial at every place $v$ away from $M$, and equals $h_p$ at each prime $p\mid M$, then we consider the cusp form $\mathbf Y_M \mathbf G = \Upsilon(h_M)\mathbf G \in \Upsilon$. With this definition, observe that 
\[
 \mathbf Y_M \mathbf G_{|\GL_2\times \GL_2} = \mathbf g \otimes \mathbf V_M \mathbf g \in \tau \boxtimes \tau.
\]

With the above definitions, the main purpose of this paragraph is to prove the following explicit identity between the cusp forms $\breve{\mathbf g}$ and $\mathbf Y_M\mathbf G$. This is made upon the choice of a Bruhat--Schwartz function $\phi_{\breve{\mathbf g}} \in \mathcal S(V_4(\A))$ to be defined below.

\begin{proposition}\label{ThetaIdentity1}
 With the above notation, one has
 \[
  \theta(\mathbf Y_M\mathbf G, \phi_{\breve{\mathbf g}}) = 2^{k-1} M^{-1} [\SL_2(\Z):\Gamma_0(N)]^{-1}\zeta_{\Q}(2)^{-2}\langle g,g \rangle\breve{\mathbf g}.
 \]
\end{proposition}

\vspace{0.3cm}

Following the approach of Ichino--Ikeda \cite[Section 5]{IchinoIkeda}, it is useful to consider a different model of the Weil representation. If $\varphi \in \mathcal S(V_4(\A))$, one defines its partial Fourier transform $\check{\varphi} \in \mathcal S(V_4(\A))$ by 
\begin{equation}\label{partialFT}
 \check{\varphi}\left(\left(\begin{array}{cc} x_1 & x_2 \\ x_3 & x_4 \end{array}\right)\right) = 
  \int_{\A^2} \varphi\left(\left(\begin{array}{cc} x_1 & y_2 \\ x_3 & y_4 \end{array}\right)\right)\psi(x_2y_4-x_4y_2)dy_2dy_4,
\end{equation}
where $dy_2$, $dy_4$ are the self-dual measure on $\A$ with respect to our fixed non-trivial additive character $\psi$ of $\A/\Q$. Then, one defines a representation $\check{\omega}$ of $R(\A)=\{(g,h) \in \GL_2(\A) \times \mathrm{GO}_{2,2}(\A): \det(g) = \nu(h)\}$ on $\mathcal S(V_4(\A))$ by setting
\[
 \check{\omega}(g,h)\check{\varphi} = (\omega(g,h)\varphi)\check{}.
\]
Observe that $\check{\omega}(g,1)\check{\varphi}(x) = \check{\varphi}(xg)$ for $g \in \SL_2(\A)$.

We start defining a Bruhat--Schwartz function $\phi_{\mathbf g} = \otimes_v \phi_{\mathbf g,v} \in \mathcal S(V_4(\A))$ associated with $\mathbf g$ as follows:
\begin{itemize}
 \item[i)] At primes $q \nmid N$, $\phi_{\mathbf g,q}$ equals the characteristic function on $\M_2(\Z_q)$.
 \item[ii)] At primes $p\mid N$, $\phi_{\mathbf g,p}$ is determined by requiring that $\check{\phi}_{\mathbf g, p}$ is given by  
 \[
  \check{\phi}_{\mathbf g, p}\left(\left(\begin{smallmatrix} x_1 & x_2 \\ x_3 & x_4 \end{smallmatrix}\right) \right) = 
  \mathbf 1_{\Z_p}(x_1)\mathbf 1_{\Z_p}(x_2)\mathbf 1_{p\Z_p}(x_3)\mathbf 1_{\Z_p^{\times}}(x_4)\underline{\chi}_p^{-1}(x_4).
 \]
 \item[iii)] At the archimedean place, 
 \[
  \phi_{\mathbf g, \infty}\left(\left(\begin{smallmatrix} x_1 & x_2 \\ x_3 & x_4 \end{smallmatrix}\right) \right) = (x_1 + \sqrt{-1}x_2 + \sqrt{-1}x_3 - x_4)^{k+1}\mathrm{exp}(-\pi \mathrm{tr}(x^t x)).
 \]
\end{itemize}
Notice that the local components at primes dividing $N$ are defined through their partial Fourier transforms. For later use, we provide an explicit recipe for $\phi_{\mathbf g,p}$ at such primes:

\begin{lemma}\label{lemma:phi-gp}
 Let $p$ be a prime dividing $N$, and write $\varepsilon(1/2,\underline{\chi}_p)$ for the root number of the character $\underline{\chi}_p: \Q_p^{\times} \to \C^{\times}$. Then, one has 
 \[
  \phi_{\mathbf g, p}\left(\left(\begin{smallmatrix} x_1 & x_2 \\ x_3 & x_4 \end{smallmatrix}\right) \right) = 
  \begin{cases}
   p^{-1/2}\varepsilon(1/2,\underline{\chi}_p)\mathbf 1_{\Z_p}(x_1)\mathbf 1_{\Z_p}(x_4)\mathbf 1_{p\Z_p}(x_3)\mathbf 1_{p^{-1}\Z_p^{\times}}(x_2)\underline{\chi}_p(x_2) & \text{if } p \mid M, \\
   \mathbf 1_{\Z_p}(x_1)\mathbf 1_{\Z_p}(x_4)\mathbf 1_{p\Z_p}(x_3)\left(\mathbf 1_{\Z_p}(x_2) - p^{-1}\mathbf 1_{p^{-1}\Z_p}(x_2)\right) & \text{if } p \mid N/M.
  \end{cases}
 \]
\end{lemma}

The next statement adapts \cite[Proposition 5.2]{IchinoIkeda} to our slightly different model for $\mathrm{GSO}_{2,2}$.

\begin{proposition}
With the above notation,
 \[
 \theta(\mathbf g, \phi_{\mathbf g})_{|\GL_2\times \GL_2} = 2^{k+1} \zeta_{\Q}(2)^{-1}[\SL_2(\Z):\Gamma_0(N)]^{-1} \mathbf g \otimes \mathbf g.
\]
\end{proposition}
\begin{proof}
We know that $\theta(\mathbf g,\phi_{\mathbf g}) \in \tau \boxtimes \tau$. A routine calculation shows the following assertions.
 \begin{itemize}
  \item[i)] If $p \nmid N$, and $(g,h) \in R(\Z_p)$, then $\omega(g,h)\phi_{\mathbf g,p} = \phi_{\mathbf g,p}$.
  \item[ii)] If $p\mid N$, and $h_1, h_2 \in K_0(N;\Z_p)$, then 
  \[
   \omega\left(\left(\begin{smallmatrix} \det(h_1h_2) & 0 \\ 0 & 1\end{smallmatrix}\right), (h_1,h_2)\right) \phi_{\mathbf g,p} = \underline{\chi}_p(h_1h_2) \phi_{\mathbf g,p}.
  \]
  \item[iii)] If $k_{\theta}$, $k_{\theta_1}$, $k_{\theta_2} \in \mathrm{SO}(2)$, then 
  \[
   \omega(k_{\theta}, (k_{\theta_1}, k_{\theta_2}))\phi_{\mathbf g,\infty} = \mathrm{exp}(\sqrt{-1}(k+1)(-\theta+\theta_1+\theta_2))\phi_{\mathbf g,\infty}.
  \]
 \end{itemize}
 It follows from these properties that there is a constant $C$ satisfying 
 \[
  \theta(\mathbf g,\phi_{\mathbf g})_{|\GL_2 \times \GL_2} = C \mathbf g \otimes \mathbf g,
 \]
 and one finds $C=2^{k+1} \mathrm{vol}(\hat{\Gamma}_0(N))$ by comparing the Fourier coefficients $W_{\mathbf g,1}(1)$ and $W_{\theta(\mathbf g,\phi_{\mathbf g}),1,1}(1)$ (cf. loc. cit. for details), where $\hat{\Gamma}_0(N) = \mathrm{SO}(2)\SL_2(\hat{\Z}) \cap \mathrm{SO}(2)\Gamma_0(N;\hat{\Z})$. The statement then follows by using that $\mathrm{vol}(\hat{\Gamma}_0(N)) = \zeta_{\Q}(2)^{-1}[\SL_2(\Z):\Gamma_0(N)]^{-1}$.
\end{proof}

\begin{corollary}\label{corollaryGg}
 With the above notation, 
 \begin{equation}\label{eqn:corollaryGg}
  \theta(\mathbf G, \phi_{\mathbf g}) = 2^{k+1} [\SL_2(\Z):\Gamma_0(N)]^{-1}\zeta_{\Q}(2)^{-2}\langle g,g \rangle \mathbf g.
 \end{equation}
\end{corollary}
\begin{proof}
The invariance properties of $\phi_{\mathbf g}$ imply that there is a constant $C$ such that $\theta(\mathbf G,\phi_{\mathbf g}) = C \mathbf g$. Thus we need to determine the precise value of $C$. On the one hand we have $\langle \theta(\mathbf G,\phi_{\mathbf g}), \mathbf g\rangle = C \langle \mathbf g,\mathbf g \rangle$. And on the other hand, by the seesaw principle together with the previous proposition we also have 
 \begin{align*}
  \langle \theta(\mathbf G,\phi_{\mathbf g}), \mathbf g\rangle & = \langle \mathbf G, \theta(\mathbf g,\phi_{\mathbf g})\rangle = 2^{k+1} \zeta_{\Q}(2)^{-1}[\SL_2(\Z):\Gamma_0(N)]^{-1} \langle \mathbf G,\mathbf G\rangle = \\
  & = 2^{k+1} \zeta_{\Q}(2)^{-1}[\SL_2(\Z):\Gamma_0(N)]^{-1} \langle\mathbf g,\mathbf g \rangle^2.
 \end{align*}
 Hence, using that $\langle\mathbf g,\mathbf g \rangle = \zeta_{\Q}(2)^{-1}\langle g,g\rangle$ we conclude that $C = 2^{k+1} [\SL_2(\Z):\Gamma_0(N)]^{-1}\zeta_{\Q}(2)^{-2}\langle g,g \rangle$.
 \end{proof}

Proposition \ref{ThetaIdentity1} will follow straightforward from the identity in Corollary \ref{corollaryGg}. We need to suitably modify the Bruhat--Schwartz function $\phi_{\mathbf g}$ in order to translate \eqref{eqn:corollaryGg} into an explicit analogous relation between $\breve{\mathbf g}$ and $\mathbf Y_M\mathbf G$.
To do so, define $\phi_{\mathbf g^{\sharp}} = \otimes_v \phi_{\mathbf g^{\sharp},v}$ by setting $\phi_{\mathbf g^{\sharp},v} = \phi_{\mathbf g,v}$ at all places $v \neq 2$, and $\phi_{\mathbf g^{\sharp},2} = 2^{-2}\omega_2(t(2^{-1})_2,1)\phi_{\mathbf g,2}$. From the definition of $\phi_{\mathbf g,2}$, one can easily check that $\phi_{\mathbf g^{\sharp},2}(x) = \mathbf 1_{V_4(2\Z_2)}(x)$ for $x \in V_4(\Q_2)$. With this slight modification at $p=2$, Corollary \eqref{corollaryGg} gets easily rephrased:

\begin{corollary}
 With the above notation, 
 \[
  \theta(\mathbf G, \phi_{\mathbf g^{\sharp}}) = 2^{k-1} [\SL_2(\Z):\Gamma_0(N)]^{-1}\zeta_{\Q}(2)^{-2}\langle g,g \rangle \mathbf g^{\sharp}.
 \]
\end{corollary}
\begin{proof}
 This follows from the very definitions. Indeed, recall that for $x \in \GL_2(\A)$ one has 
\[
 \theta(\mathbf G, \phi_{\mathbf g})(x) = \int_{[\mathrm O_{2,2}]} \Theta(x,y'y;\phi_{\mathbf g}) \mathbf G(y'y) dy = \int_{[\mathrm O_{2,2}]} \left(\sum_{v\in V_4(\Q)}\omega(x,y'y)\phi_{\mathbf g}(v)\right) \mathbf G(y'y) dy,
\]
where $y' \in \mathrm{GO}_{2,2}(\A)$ is any element with $\nu(y') = \det(x)$. From the last expression, observe that if we replace $\phi_{\mathbf g}$ by $\omega(g,1)\phi_{\mathbf g}$ with $g \in \SL_2(\A)$, then 
\begin{align*}
  \theta(\mathbf G, \omega(g,1)\phi_{\mathbf g})(x) & = \int_{[\mathrm O_{2,2}]} \left(\sum_{v\in V_4(\Q)}\omega(x,y'y)\omega(g,1)\phi_{\mathbf g}(v)\right) \mathbf G(y'y) dy = \\
  & = \int_{[\mathrm O_{2,2}]} \left(\sum_{v\in V_4(\Q)}\omega(xg,y'y)\phi_{\mathbf g}(v)\right) \mathbf G(y'y) dy = \theta(\mathbf G, \phi_{\mathbf g})(xg) = \tau(g)\theta(\mathbf G, \phi_{\mathbf g})(x).
\end{align*}
Applying this for $g = t(2^{-1})_2 \in \SL_2(\Q_2) \hookrightarrow \SL_2(\A) \subseteq \GL_2(\A)$, we deduce that 
\[
 \theta(\mathbf G, \phi_{\mathbf g^{\sharp}}) = \theta(\mathbf G, 2^{-2}\omega(t(2^{-1})_2,1)\phi_{\mathbf g}) = 2^{-2}\tau(t(2^{-1})_2)\theta(\mathbf G, \phi_{\mathbf g}),
\]
and the statement follows from the previous corollary together with the definition of $\mathbf g^{\sharp}$.
\end{proof}

Finally, we define $\phi_{\breve{\mathbf g}} = \otimes_v \phi_{\breve{\mathbf g},v}$ by keeping $\phi_{\breve{\mathbf g},v} = \phi_{\mathbf g^{\sharp},v}$ for all places $v \nmid M$, and setting
\[
\phi_{\breve{\mathbf g},p} = p^{-1}\omega_p(\varpi_p, h_p)\phi_{\mathbf g^{\sharp},p} = p^{-1}\omega_p(\varpi_p, h_p)\phi_{\mathbf g,p}
\]
at each prime $p\mid M$. In other words, if $\varpi_M$ and $h_M$ are as before, we see that $\phi_{\breve{\mathbf g}} = M^{-1}\omega(\varpi_M, h_M)\phi_{\mathbf g^{\sharp}}$.

\begin{proof}[Proof of Proposition \ref{ThetaIdentity1}]
 As above, if $x \in \GL_2(\A)$ notice that 
 \[
  \theta(\mathbf G, \phi_{\mathbf g^{\sharp}})(x) = \int_{\mathrm O_{2,2}(\Q)\backslash \mathrm O_{2,2}(\A)} \left(\sum_{v\in V_4(\Q)}\omega(x,y'y)\phi_{\mathbf g^{\sharp}}(v)\right) \mathbf G(y'y) dy,
 \]
 where $y' \in \mathrm{GO}_{2,2}(\A)$ is such that $\det(x) = \nu(y')$. In particular, observe that $\det(x\varpi_M) = \nu(y'h_M)$, hence
 \begin{align*}
  \theta(\mathbf Y_M \mathbf G, \phi_{\breve{\mathbf g}})(x) & = M^{-1}\int_{[\mathrm O_{2,2}]} \left(\sum_{v\in V_4(\Q)}\omega(x,y'y)\omega(\varpi_M, h_M)\phi_{\mathbf g^{\sharp}}(v)\right) \mathbf G(y'y h_M) dy = \\
  & = M^{-1}\int_{[\mathrm O_{2,2}]} \left(\sum_{v\in V_4(\Q)}\omega(x\varpi_M,y'h_My)\phi_{\mathbf g^{\sharp}}(v)\right) \mathbf G(y' h_My) dy.
 \end{align*}
 From this we see that $\theta(\mathbf Y_M \mathbf G, \phi_{\breve{\mathbf g}}) = M^{-1}\tau(\varpi_M)\theta(\mathbf G, \phi_{\mathbf g^{\sharp}})$, and hence the statement follows from the previous corollary, together with the fact that $\breve{\mathbf g} = \tau(\varpi_M)\mathbf g^{\sharp}$.
\end{proof}

For later use, we will need an explicit description of the Bruhat--Schwartz function $\phi_{\breve{\mathbf g}}$. At places $v \nmid 2M$, observe that $\phi_{\breve{\mathbf g},v} = \phi_{\mathbf g,v}$, thus we have:
\begin{itemize}
 \item[i)] If $p\nmid 2M$ is a finite prime, then $\phi_{\breve{\mathbf g},p}(x) = \mathbf 1_{V_4(\Z_p)}(x)$ for all $x \in V_4(\Q_p)$.
 \item[ii)] If $p \mid N/M$ is an (odd) prime, then 
 \[
  \phi_{\breve{\mathbf g},p}\left(\left(\begin{smallmatrix} x_1 & x_2 \\ x_3 & x_4 \end{smallmatrix}\right) \right) = \mathbf 1_{\Z_p}(x_1)\mathbf 1_{\Z_p}(x_4)\mathbf 1_{p\Z_p}(x_3)\left(\mathbf 1_{\Z_p}(x_2) - p^{-1}\mathbf 1_{p^{-1}\Z_p}(x_2)\right).
 \]
 \item[iii)] At $v=\infty$, 
 \[
  \phi_{\mathbf g, \infty}\left(\left(\begin{smallmatrix} x_1 & x_2 \\ x_3 & x_4 \end{smallmatrix}\right) \right) = (x_1 + \sqrt{-1}x_2 + \sqrt{-1}x_3 - x_4)^{k+1}\mathrm{exp}(-\pi \mathrm{tr}(x^t x)).
 \]
\end{itemize}

At primes $p\mid 2M$, we describe $\phi_{\breve{\mathbf g},p}$ in the following lemma.

\begin{lemma}
 With the above notation, the following assertions hold.
 \begin{itemize}
  \item[i)] At $p = 2$, $\phi_{\breve{\mathbf g},2}(x) = \mathbf 1_{V_4(2\Z_2)}(x)$ for all $x \in V_4(\Q_2)$.
  \item[ii)] At (odd) primes $p\mid M$, for $x = \left(\begin{smallmatrix} x_1 & x_2 \\ x_3 & x_4\end{smallmatrix}\right) \in V_4(\Q_p)$ we have
  \[
   \phi_{\breve{\mathbf g},p}(x) = p^{-1/2}\varepsilon(1/2,\lambda_p)\mathbf 1_{p\Z_p}(x_1)\mathbf 1_{\Z_p}(x_4)\mathbf 1_{p^2\Z_p}(x_3)\mathbf 1_{p^{-1}\Z_p^{\times}}(x_2)\underline{\chi}_p(x_2).
  \]
 \end{itemize}
\end{lemma}
\begin{proof}
 One just have to compute $\phi_{\breve{\mathbf g},2}$ and $\phi_{\breve{\mathbf g},p}$ ($p\mid M$) using the definitions of $\phi_{\mathbf g,2}$ and $\phi_{\mathbf g,p}$ together with the properties of the Weil representation, since $\phi_{\breve{\mathbf g},2} = 2^{-2}\omega_2(t(2^{-1})_2,1)\phi_{\mathbf g,2}$ and $\phi_{\breve{\mathbf g},p} = p^{-1}\omega_p(\varpi_p,h_p)\phi_{\mathbf g,p}$ for primes $p \mid M$.
 We omit the details and leave them to the reader.
\end{proof}

\subsection{The pair $(\PGL_2, \widetilde{\SL}_2)$}\label{theta:PGL2SL2}

Shimura's correspondence between half-integral weight modular forms and integral weight modular forms was investigated by Waldspurger as a theta correspondence between automorphic representations of $\PGL_2(\A)$ and automorphic representations of $\widetilde{\SL}_2(\A)$. Here, $\PGL_2$ is identified as the special orthogonal group $\SO(V)$ of a three-dimensional quadratic space $V$ as in Section \ref{spaces:lowrank}.

The theta correspondence for the pair $(\PGL_2, \widetilde{\SL}_2)$ depends on the choice of an additive character $\psi$ of $\Q \backslash \A$. To emphasize this dependence, we will write 
\[
 \Theta_{\PGL_2 \times \widetilde{\SL}_2}(\pi, \psi) \quad \text{(resp. } \Theta_{\widetilde{\SL}_2 \times \PGL_2}(\tilde{\pi}, \psi)\text{)}
\]
for the automorphic representation of $\widetilde{\SL}_2(\A)$ (resp. $\PGL_2(\A)$) obtained as the theta lift of the automorphic representation $\pi$ (resp. $\tilde{\pi}$) of $\PGL_2(\A)$ (resp. $\widetilde{\SL}_2(\A)$). On the local setting, we write 
\[
 \Theta_{\PGL_2 \times \widetilde{\SL}_2}(\pi_v, \psi_v) \quad \text{and} \quad \Theta_{\widetilde{\SL}_2 \times \PGL_2}(\tilde{\pi}_v, \psi_v)
\]
for the local theta lifts of $\pi_v$ and $\tilde{\pi}_v$, respectively. We will omit the subscripts $\PGL_2 \times \widetilde{\SL}_2$ or $\widetilde{\SL}_2 \times \PGL_2$ if the direction of the theta lift is clear.

In the following, for a fixed irreducible cuspidal automorphic representation $\pi$ of $\PGL_2(\A)$, and $D$ varying over the set of fundamental discriminants, the representations $\Theta(\pi \otimes \chi_D, \psi^D)$ (and their local counterparts $\Theta(\pi_v \otimes \chi_D, \psi_v^D)$) will play a crucial role. Waldspurger's description of the theta correspondence for $(\PGL_2, \widetilde{\SL}_2)$ tells us that the set $\{\Theta(\pi \otimes \chi_D, \psi^D): D \in \Q^{\times} \text{ fund. discr.}\}$ is {\em finite}. 

In order to describe the local theory, fix a place $v$ of $\Q$ and let $\mathcal P_{0,v}$ denote the set of special or supercuspidal representations (or discrete series representations if $v = \infty$) of $\PGL_2(\Q_v)$. For $D \in \Q_v^{\times}$, define the symbol 
\[
 \left( \frac{D}{\pi_v}\right) = \chi_D(-1)\epsilon(\pi_v,1/2)\epsilon(\pi_v \otimes \chi_D, 1/2),
\]
and consider the associated partition $\Q_v^{\times} = \Q_v^+(\pi_v) \sqcup \Q_v^-(\pi_v)$, where 
\[
 \Q_v^{\pm}(\pi_v) := \left\lbrace D \in \Q_v^{\times} : \left( \frac{D}{\pi_v}\right) = \pm 1 \right \rbrace.
\]
The next statement summarizes Waldspurger's local theory.

\begin{theorem}[Waldspurger]\label{thm:Waldspurger-local}
With the above notation, the following assertions hold.
 \begin{itemize}
  \item[i)] If $\pi_v \not\in \mathcal P_{0,v}$, then $\Q_v^+(\pi_v) = \Q_v^{\times}$ and $\tilde{\pi}_v := \Theta(\pi_v, \psi_v) = \Theta(\pi_v \otimes \chi_D, \psi_v^D)$ for all $D \in \Q_v^{\times}$.
  \item[ii)] If $\pi_v \in \mathcal P_{0,v}$, then there are two representations $\tilde{\pi}_v^+$ and $\tilde{\pi}_v^-$ of $\widetilde{\SL}_2(\Q_v)$ such that 
  \[
   \Theta(\pi_v \otimes \chi_D, \psi_v^D) = 
   \begin{cases}
    \tilde{\pi}_v^+ & \text{if } D \in \Q_v^+(\pi_v), \\
     \tilde{\pi}_v^- & \text{if } D \in \Q_v^-(\pi_v).
   \end{cases}
  \]
  \item[iii)] The equality $\Theta(\pi_v \otimes \chi_D, \psi_v^D) = \Theta(\pi_v, \psi_v)$ holds if and only if $\Theta(\pi_v, \psi_v)$ admits a non-trivial $\psi_v^D$-Whittaker model.
 \end{itemize}
\end{theorem}

\begin{rem}
 When $\pi_v \in \mathcal P_{0,v}$, one has $\tilde{\pi}_v^+ = \Theta(\pi_v,\psi_v)$ and $\tilde{\pi}_v^- = \Theta(\mathrm{JL}(\pi_v),\psi_v))$, where $\mathrm{JL}(\pi_v)$ is the Jacquet--Langlands lift of $\pi_v$ to an admissible representation of $PB_v^{\times}$, with $B_v$ the unique division quaternion algebra over $\Q_v$. We warn the reader that the labelling $+/-$ in ii) depends on the choice of the additive character $\psi_v$.
\end{rem}

Altogether, the {\em local Waldspurger packet} $\mathrm{Wald}_{\psi_v}(\pi_v)$ is defined to be the singleton $\{\tilde{\pi}_v\}$ if $\pi_v \not\in \mathcal P_{0,v}$, and the set $\{\tilde{\pi}_v^+, \tilde{\pi}_v^-\}$ if $\pi_v \in \mathcal P_{0,v}$. Although the labelling $+/-$ in the set $\mathrm{Wald}_{\psi_v}(\pi_v)$ depends on $\psi_v$, the packet $\mathrm{Wald}_{\psi_v}(\pi_v)$ itself does not.

To state the global side of this theory, write $\tilde{\mathcal A}_{00}$ for the subspace of cuspidal automorphic forms on $\widetilde{\SL}_2(\A)$ which are orthogonal to the theta series generated by quadratic forms of one variable. Let $\mathcal A_{0,i}$ be the subspace of cuspidal automorphic forms on $\PGL_2(\A)$ such that for any subrepresentation $\pi$ of $\mathcal A_{0,i}$ there exists some $D \in \Q^{\times}$ with $L(\pi\otimes \chi_D, 1/2) \neq 0$.

Let $\pi$ be an irreducible cuspidal automorphic representation of $\PGL_2(\A)$, and let $\Sigma(\pi)$ denote the set of rational places $v$ such that $\pi_v \in \mathcal P_{0,v}$. For each $D\in \Q^{\times}$, let $\epsilon(D,\pi) \in \{\pm 1\}^{|\Sigma(\pi)|}$ be the tuple determined by setting $\epsilon(D,\pi)_v = (\frac{D}{\pi_v})$ for each $v \in \Sigma(\pi)$. Observe that one has by construction
\[
 \epsilon(\pi \otimes \chi_D, 1/2) = \epsilon(\pi,1/2) \prod_{v\in \Sigma(\pi)} \left(\frac{D}{\pi_v}\right).
\]
For an arbitrary tuple $\epsilon = (\epsilon_v)_{v\in \Sigma(\pi)} \in \{\pm 1\}^{|\Sigma(\pi)|}$, define the set
$\Q^{\epsilon}(\pi) = \{D \in \Q^{\times}: \epsilon(D,\pi) = \epsilon\}$. In particular, $\pi$ determines a partition \[
\Q^{\times} = \bigsqcup_{\epsilon} \Q^{\epsilon}(\pi).
\]

Having settled this notation, we summarize Waldspurger's global theory as follows. Below, two irreducible subrepresentations of $\tilde{\mathcal A}_{00}$ are called {\em near equivalent}, denoted $\tilde{\pi}_1 \sim \tilde{\pi}_2$, if it holds that $\tilde{\pi}_{1,v} \simeq \tilde{\pi}_{2,v}$ for almost all places. 

\begin{theorem}[Waldspurger]\label{thm:Waldspurger-global}
With the above notation, the following assertions hold. 
 \begin{itemize}
  \item[i)] The global theta correspondence between $\PGL_2$ and $\widetilde{\SL}_2$ is compatible with the local correspondence. That is, if $\Theta(\tilde{\pi},\psi) \neq 0$, then $\Theta(\tilde{\pi},\psi) \simeq \otimes_v \Theta(\tilde{\pi}_v,\psi_v)$. And analogously, if $\Theta(\pi,\psi) \neq 0$, then $\Theta(\pi,\psi) \simeq \otimes_v \Theta(\pi_v,\psi_v)$.
  \item[ii)] $\Theta(\pi,\psi) \neq 0$ if and only if $L(\pi,1/2) \neq 0$. And $\Theta(\tilde{\pi},\psi)\neq 0$ if and only if $\tilde{\pi}$ has a non-trivial $\psi$-Whittaker model.
  \item[iii)] If $\tilde{\pi}$ is an irreducible subrepresentation of $\tilde{\mathcal A}_{00}$, then there is a unique irreducible automorphic representation $\pi = Sh_{\psi}(\tilde{\pi})$ of $\PGL_2(\A)$ associated to $\tilde{\pi}$ such that $\Theta(\tilde{\pi},\psi^D) \neq 0 \Rightarrow \Theta(\tilde{\pi},\psi^D) \otimes \chi_D = \pi$. The association $\tilde{\pi} \to Sh_{\psi}(\tilde{\pi})$ defines a bijection between $\tilde{\mathcal A}_{00}/\sim$ and $\mathcal A_{0,i}$.
  \item[iv)] If $\pi = Sh_{\psi}(\tilde{\pi})$, then the near equivalence class of $\tilde{\pi}$ consists of all the non-zero lifts of the form $\Theta(\pi\otimes \chi_D,\psi_D)$.
  \item[v)] Let $\epsilon \in \{\pm 1\}^{|\Sigma(\pi)|}$. If $\prod_{v\in \Sigma(\pi)} \epsilon_v \neq \epsilon(\pi,1/2)$, then $\Theta(\pi\otimes \chi_D, \psi^D) = 0$ for all $D \in \Q^{\epsilon}(\pi)$. If  $\prod_{v\in \Sigma(\pi)} \epsilon_v = \epsilon(\pi,1/2)$, then there is a unique $\tilde{\pi}^{\epsilon}$ such that for every $D \in \Q^{\epsilon}(\pi)$ it holds 
  \[
   \Theta(\pi\otimes \chi_D, \psi^D) = \begin{cases}
                                        \tilde{\pi}^{\epsilon} & \text{if } L(\pi\otimes \chi_D,1/2) \neq 0,\\
                                        0  & \text{if } L(\pi\otimes \chi_D,1/2) = 0.
                                       \end{cases}
  \]
 \end{itemize}
\end{theorem}

In the previous theorem, if $\epsilon = (\epsilon_v)_v \in \{\pm 1\}^{|\Sigma(\pi)|}$, then $\tilde{\pi}^{\epsilon}$ denotes the irreducible cuspidal automorphic representation of $\widetilde{\SL}_2(\A)$ whose local components equal $\tilde{\pi}_v$ at all rational places $v \not\in \Sigma(\pi)$, and whose local components at places $v \in \Sigma(\pi)$ equal $\tilde{\pi}_v^{\epsilon_v}$. Together with the local theory, the above result motivates the definition of the {\em (global) Waldspurger packet} $\mathrm{Wald}_{\psi}(\pi)$ associated with $\pi$ as the finite set 
\[
 \mathrm{Wald}_{\psi}(\pi) = \left\lbrace \tilde{\pi}^{\epsilon}: \prod_{v \in \Sigma(\pi)} \epsilon_v = \epsilon(\pi,1/2)\right\rbrace.
\]
Notice that $\mathrm{Wald}_{\psi}(\pi) = \mathrm{Wald}_{\psi^D}(\pi \otimes \chi_D)$ for all $D \in \Q^{\times}$, similarly as locally at each place $v$ one has $\mathrm{Wald}_{\psi_v}(\pi_v) = \mathrm{Wald}_{\psi_v^a}(\pi_v \otimes \chi_a)$ for all $a \in \Q_v^{\times}$.

\subsubsection{On the result of Baruch--Mao}\label{sec:BMexplained}

Having recalled Waldspurger's theory, we explain briefly how Baruch--Mao's result stated in Theorem \ref{thm:BM}, leading to a generalization of Kohnen's formula, fits in this theory. So let $f \in S_{2k}^{new}(N)$ and $\chi$ be as in Theorem \ref{thm:BM}, and let $\pi$ be the irreducible cuspidal automorphic representation of $\PGL_2(\A)$ associated with $f$. Choose once and for all a fundamental discriminant $D \in \mathfrak D(N,M)$, where the set $\mathfrak D(N,M)$ is defined in \eqref{D(N,M)}, and assume further that $L(f,D,k) \neq 0$. 

Let $\psi$ be the standard additive character of $\A/\Q$, and write $\overline{\psi} = \psi^{-1}$ and $\overline{\psi}^D = \psi^{-D}$ for the $(-1)$-th and $(-D)$-th twists of $\psi$, respectively. Then, consider the theta lift $\tilde{\pi} := \Theta(\pi \otimes \chi_D, \overline{\psi}^D)$ of $\pi \otimes \chi_D$ with respect to the additive character $\overline{\psi}^{D}$. Because of the assumption $L(f,D,k) \neq 0$, we have $\Theta(\pi \otimes \chi_D, \overline{\psi}^{D}) \neq 0$, and so $\tilde{\pi} = \otimes_v \tilde{\pi}_v$ with $\tilde{\pi}_v \simeq \Theta(\pi_v \otimes \chi_D, \overline{\psi}_v^{D})$.

Let $\epsilon := \epsilon(D,\pi) \in \{\pm 1\}^{|\Sigma(\pi)|}$ be defined as above, so that $\mathfrak D(N,M)$ is the set of fundamental discriminants in $\Q^{\epsilon}(\pi)$. It is proved in \cite[Section 10]{BaruchMao} that $\tilde{\pi} = \tilde{\pi}^{\epsilon} \in \mathrm{Wald}_{\overline{\psi}}(\pi)$. In other words, the automorphic representation $\tilde{\pi} = \Theta(\pi \otimes \chi_D, \overline{\psi}^{D})$ of $\widetilde{\SL}_2(\A)$ corresponds to the element in the Waldspurger packet $\mathrm{Wald}_{\overline{\psi}}(\pi)$ whose labelling coincides with $\epsilon$ (recall that the labelling of $\tilde{\pi}$ depends on the choice of $\overline{\psi}$). Moreover, one has $\epsilon_{\infty} = -1$, and for each prime $p\mid N$ 
\[
\epsilon_p = \left(\frac{D}{\pi_p}\right) = \left(\frac{D}{p}\right) =
\begin{cases}
 w_p & \text{if } p \mid N/M,\\
 -w_p & \text{if } p \mid M.
\end{cases}
\]

If $V_{\tilde{\pi}}$ denotes the representation space for $\tilde{\pi}$, then Baruch and Mao show that $V_{\tilde{\pi}} \cap \tilde{\mathcal A}_{k+1/2}^+(4N,\underline{\chi}_0)$ is one-dimensional. The one-dimensional subspace of $S_{k+1/2}^{+,new}(4NM;\chi)$ denoted $S_{k+1/2}^{+,new}(4NM,\chi;f\otimes\chi)$ in Theorem \ref{thm:BM} is then the preimage of $V_{\tilde{\pi}} \cap \tilde{\mathcal A}_{k+1/2}^+(4N,\underline{\chi}_0)$ under the adelization map \eqref{adelization-Waldspurger}, and $h$ can be taken to be the de-adelization of any new vector in $V_{\tilde{\pi}} \cap \tilde{\mathcal A}_{k+1/2}^+(4N,\underline{\chi}_0)$. 

Summing up, the set $\mathfrak D(N,M)$ in Theorem \ref{thm:BM} singles out a precise element $\tilde{\pi}$ in the Waldspurger packet $\mathrm{Wald}_{\overline{\psi}}(\pi) = \mathrm{Wald}_{\overline{\psi}^{D}}(\pi \otimes \chi_D)$, where the adelizations of the classical half-integral modular forms in $S_{k+1/2}^{+,new}(4NM,\chi;f\otimes\chi)$ belong to. Together with the local assumptions on $\chi$, this allows Baruch and Mao to have a clean description of the local types for $\tilde{\pi}$ at primes $p$ dividing $N$. For later purposes, we briefly describe these local types, according to whether $p$ divides $M$ or not.

First suppose that $p$ is a prime dividing $N/M$. The local representation $\pi_p$ is a quadratic twist of the Steinberg representation, say $\pi_p = \mathrm{St}_p \cdot \chi_{u}$, for some $u \in \Z_p^{\times}$. If $w_p=1$ (resp. $w_p=-1$), then $u$ is a non-quadratic residue (resp. quadratic residue) modulo $p$. We have 
\[
 \tilde{\pi}_p \simeq \Theta(\pi_p \otimes \chi_D, \overline{\psi}_p^{D}) = \Theta(\mathrm{St}_p\cdot \chi_{u D}, \overline{\psi}_p^{D}),
\]
and notice that $\delta := u D \in \Z_p^{\times}$ is a non-square in $\Z_p^{\times}$, since $(\frac{D}{p}) = w_p$. In this case, it follows that $\tilde{\pi}_p$ is a special representation of $\widetilde{\SL}_2(\Q_p)$, denoted $\tilde{\sigma}^{\delta}(\psi_p^{-D})$ in \cite[(10.5.3)]{BaruchMao}. This representation is realized as the space of functions $\tilde{\varphi}: \widetilde{\SL}_2(\Q_p) \to \C$ such that 
\[ 
\tilde{\varphi} \left(\left[\left(\begin{array}{cc} a & \ast \\ 0 & a^{-1}\end{array}\right), \epsilon \right] g \right) = \epsilon \chi_{\overline{\psi}_p^D}(a)\chi_{\delta}(a) |a|^{3/2}_p \tilde{\varphi}(g)
\]
for all $g\in \widetilde{\SL}_2(\Q_p)$ and $a \in \Q_p^{\times}$, and satisfying also the vanishing condition 
\[ \int_{\Q_p} \tilde{\varphi} (\tilde{w}\tilde{n}(x))\overline{\psi}_p^D(-\delta \Delta^2 x) dx = 0  \quad \text{for all } \Delta \in \Q_p.
\]
Here, notice that $\chi_{\delta}$ is the unique non-trivial quadratic character of $\Q_p^{\times}$. A newvector $\tilde{\varphi}_p \in \tilde{\pi}_p$ can be chosen as in \cite[Lemma 8.3]{BaruchMao}, see also Lemma \ref{sl2testvector_special} below.

Secondly, suppose that $p$ is a prime dividing $M$. We have again $\pi_p = \mathrm{St}_p\cdot \chi_{u}$, for some $u \in \Z_p^{\times}$, and therefore $\tilde{\pi}_p \simeq \Theta(\mathrm{St}_p\cdot \chi_{u D}, \overline{\psi}_p^{D})$. But now, $\delta := uD \in \Z_p^{\times}$ is a square in $\Z_p^{\times}$ because $(\frac{D}{p}) = -w_p$. Then $\tilde{\pi}_p$ is a supercuspidal representation of $\widetilde{\SL}_2(\Q_p)$. More precisely, it is the {\em odd} Weil representation $r_{\overline{\psi}_p^{D}}^-$ associated with the character $\overline{\psi}_p^{D}$ (cf. \cite[(10.5.4)]{BaruchMao}). In this case, a choice of newvector $\tilde{\varphi}_p \in \tilde{\pi}_p$ is described in \cite[Proposition 8.5]{BaruchMao}, see also Lemma \ref{sl2testvector_weil} below.

\subsection{The pair $(\widetilde{\SL}_2,\PGSp_2)$}\label{theta:SL2PGSP2}

Now we focus on the theta correspondence for the pair $(\widetilde{\SL}_2,\PGSp_2)$, which explains the classical Saito--Kurokawa lift introduced in Section \ref{MF:Siegel}. We identify $\PGSp_2$ with the special orthogonal group $\SO(3,2) = \SO(V_5)$, where $V_5$ is the five-dimensional quadratic space of determinant $1$ and Witt index $2$ as in Section \ref{spaces:lowrank}. As we did for the pair $(\PGL_2,\widetilde{\SL}_2)$, let us fix a non-trivial additive character $\psi$ of $\A/\Q$. 

Global theta lifts can be defined in the same fashion as we have already explained in the previous two instances, so that for an irreducible cuspidal representation $\Pi$ of $\PGSp_2(\A)$ and an irreducible component $\tilde{\pi} \subset \mathcal A_{00}(\widetilde{\SL}_2)$, one can define their lifts $\Theta_{\PGSp_2\times\widetilde{\SL}_2}(\Pi;\psi)$ and $\Theta_{\widetilde{\SL}_2\times\PGSp_2}(\tilde{\pi};\psi)$, respectively. The following assertions concerning this theta correspondence can be found in \cite{PS83}.

\begin{itemize}
 \item[a)] If $\Theta_{\PGSp_2\times\widetilde{\SL}_2}(\Pi;\psi)$ is not zero, then it is irreducible cuspidal.
 \item[b)] If $\Theta_{\widetilde{\SL}_2\times\PGL_2}(\tilde{\pi};\psi) = 0$, then $\Theta_{\widetilde{\SL}_2\times\PGSp_2}(\tilde{\pi};\psi)$ is irreducible cuspidal.
 \item[c)] If $\Theta_{\widetilde{\SL}_2\times\PGL_2}(\tilde{\pi};\psi)$ is not zero, then $\Theta_{\widetilde{\SL}_2\times\PGSp_2}(\tilde{\pi};\psi)$ is irreducible noncuspidal and occurs in the discrete spectrum of $\PGSp_2$.
 \item[d)] $\Theta_{\PGSp_2\times\widetilde{\SL}_2}(\Pi;\psi) = \tilde{\pi}$ if and only if $\Theta_{\widetilde{\SL}_2\times\PGSp_2}(\tilde{\pi};\psi)=\Pi$.
\end{itemize}

Similarly as for Waldspurger packets, one can introduce the notion of local and global {\em Saito--Kurokawa packets}. Indeed, let $v$ be a place of $\Q$, and $\pi_v$ be an infinite-dimensional irreducible admissible representation of $\PGL_2(\Q_v)$. If $\epsilon_v \in \{\pm 1\}$ and $\tilde{\pi}_v^{\epsilon_v}\in \mathrm{Wald}_{\psi_v}(\pi_v)$, write $\Pi_v^{\epsilon_v} := \Theta_{\widetilde{\SL}_2\times\PGSp_2}(\tilde{\pi}_v^{\epsilon_v};\psi_v)$. Then the local Saito--Kurokawa packet of $\pi_v$ is defined to be 
\[
 \SK(\pi_v) := \{\Pi_v^{\epsilon_v}: \tilde{\pi}_v^{\epsilon_v} \in \mathrm{Wald}_{\psi_v}(\pi_v)\}.
\]
Now if $\pi$ is an irreducible cuspidal automorphic representation of $\PGL_2(\A)$, the associated global Saito--Kurokawa packet is just 
\[
 \SK(\pi) := \{\Theta(\tilde{\pi};\psi): \tilde{\pi} \in \mathrm{Wald}_{\psi}(\pi)\}.
\]
Given a tuple $\epsilon = (\epsilon_v)_v$ such that $\epsilon_v\in \{\pm 1\}$ for every place $v$, and $\epsilon_v = +1$ for all $v$ such that $\pi_v$ is not square-integrable, set $\Pi^{\epsilon} = \otimes_v \Pi_v^{\epsilon_v}$. Then we have 
\[
 \SK(\pi) := \{\Pi^{\epsilon}: \prod_v \epsilon_v = \varepsilon(1/2,\pi)\}.
\]

Although the Saito--Kurokawa packet $\SK(\pi)$ associated with $\pi$ is defined through $\mathrm{Wald}_{\psi}(\pi)$, it turns out that $\SK(\pi)$ does not depend on the choice of the additive character $\psi$. Also, it is well-known that $\SK(\pi)$ consists only of cuspidal members when $L(1/2,\pi) = 0$.

In the rest of this section, we focus on an explicit relation between a newform $\mathbf h \in \tilde{\pi}$, where $\tilde{\pi}$ is as in Section \ref{sec:BMexplained}, and a theta lift of $\mathbf h$. From now on, we fix $\psi$ to be the standard additive character of $\A/\Q$.

If $\mathbf F$ is a cuspidal automorphic form on $\mathrm{SO}(V_5)(\A) \simeq \PGSp_2(\A)$ and $B \in \mathrm{Sym}_2(\Q)$, we will regard $\mathbf F$ as an automorphic form on $\GSp_2(\A)$ trivial on the center. Then its $B$-th Fourier coefficient $\mathcal W_{\mathbf F,B}: \GSp_2(\A) \to \C$ is defined as in \eqref{automFC:GSp2}. In particular, for an automorphic form $\mathbf h$ on $\widetilde{\SL}_2(\A)$ and a Bruhat--Schwartz function $\varphi \in \mathcal S(V_5(\A))$, the $B$-th Fourier coefficient of the theta lift $\theta(\mathbf h,\varphi)$ is the function 
\[
 g \mapsto \mathcal W_{\mathbf \theta(\mathbf h,\varphi),B}(g) = \int_{\mathrm{Sym}_2(\Q) \backslash \mathrm{Sym}_2(\A)} \theta(\mathbf h,\varphi)(n(X)g) \overline{\psi(\mathrm{tr}(BX))}dX, \quad g \in \GSp_2(\A).
\]

As in the introduction, let $k, N\geq 1$ be odd integers, with $N$ square-free, and let $\chi$ be an even Dirichlet character modulo $N$, of conductor $M\mid N$. Let $f \in S_{2k}^{new}(N)$ be a normalized newform of weight $2k$ and level $N$, and let $\pi$ be the automorphic representation of $\PGL_2$ associated with $f$. We assume Hypotheses \eqref{Hyp1} and \eqref{Hyp2}, so that $\chi_{(p)}(-1) = -1$ and $w_p = -1$ for all primes $p\mid M$.

Let $\tilde{\pi} \in \mathrm{Wald}_{\overline{\psi}}(\pi)$ be the automorphic representation of $\widetilde{\SL}_2(\A)$ obtained by theta correspondence as explained in Section \ref{sec:BMexplained}. Let $h \in S_{k+1/2}^{+,new}(4NM,\chi; f \otimes\chi)$ be a Shimura lift of $f$ as in Theorem \ref{thm:BM}, and let $\mathbf h \in \tilde{\pi}$ be its adelization. Besides, let $\mathbf F_{\chi} \in \mathcal S_{k+1}(K_0^{(2)}(NM),\underline{\chi})$ be the adelization of the Saito--Kurokawa lift $F_{\chi} = \mathbf M(\mathbf Z(h)) \in S_{k+1}(\Gamma_0^{(2)}(NM),\chi)$ as defined in \eqref{adelizationFchi} (cf. also Section \ref{MF:Siegel}). Recall the operator $\mathfrak R_M$, and consider the automorphic form 
\[
 \mathfrak R_M\mathbf F_{\chi} \in \mathcal S_{k+1}(K(NM,M), \underline{\chi}),
\]
which by Lemma \ref{lemma:RpF} is the adelization of $\mathfrak R_M F_{\chi} \in S_{k+1}(\Gamma_0^{(2)}(NM) \cap \Gamma^{\text{param}}(M),\chi)$.

The main purpose for the rest of this section is to prove the following identity, where the Bruhat--Schwartz function $\phi_{\mathbf h} \in \mathcal S(V_5(\A))$ will be defined below.

\begin{proposition}\label{ThetaIdentity2}
 With the above notation, 
 \begin{equation}\label{ThetaIdentity2-eqn}
  \theta(\mathbf h,\phi_{\mathbf h}) \otimes \underline{\chi} = 
  2^{-2} \chi(2)^{-1} M^{-1}[\SL_2(\Z):\Gamma_0(N)]^{-1} \zeta_{\Q}(2)^{-1} \mathfrak R_M \mathbf F_{\chi}.
 \end{equation}
\end{proposition}

Here, observe that $\theta(\mathbf h,\phi_{\mathbf h})$ is an automorphic cusp form on $\PGSp_2(\A)$, and as in Section \ref{AF:GSP2}, the automorphic cusp form $\theta(\mathbf h,\phi_{\mathbf h}) \otimes \underline{\chi}$ on $\GSp_2(\A)$ is defined by 
\[
 (\theta(\mathbf h,\phi_{\mathbf h}) \otimes \underline{\chi})(g) = \theta(\mathbf h,\phi_{\mathbf h})(g) \underline{\chi}(\nu(g)),
\]
where $\nu: \GSp_2(\A) \to \A^{\times}$ is the similitude morphism. In particular, the above proposition says that the automorphic form $\theta(\mathbf h,\phi_{\mathbf h}) \otimes \underline{\chi}$ is {\em classical}, in the sense that it is obtained by adelization of $2^{-2} \chi(2)^{-1} M^{-1}[\SL_2(\Z):\Gamma_0(N)]^{-1} \zeta_{\Q}(2)^{-1} \mathfrak R_M F_{\chi}$.

The proof of Proposition \ref{ThetaIdentity2} will proceed by comparing the Fourier coefficients 
$W_{\mathbf \theta(\mathbf h,\varphi)\otimes \underline{\chi}, B}$ and $\mathcal W_{\mathfrak R_M \mathbf F_{\chi}, B}$ of the automorphic forms $\mathbf \theta(\mathbf h,\varphi)\otimes \underline{\chi}$ and $\mathfrak R_M \mathbf F_{\chi}$ appearing in \eqref{ThetaIdentity2-eqn}, for arbitrary symmetric matrices $B \in \mathrm{Sym}_2(\Q)$. Concerning $\mathfrak R_M\mathbf F_{\chi}$, we know from Lemma \ref{FourierCoeffRMFchi} that $\mathcal W_{\mathfrak R_M \mathbf F_{\chi}, B}$ is zero unless $B$ is positive definite, half-integral and $b_3 \in M\Z$ if 
\[
 B = \left(\begin{array}{cc} b_1 & b_2/2 \\ b_2/2 & b_3 \end{array} \right).
\]
And in that case, $\mathcal W_{\mathfrak R_M \mathbf F_{\chi}, B}$ is uniquely determined by the values $\mathcal W_{\mathfrak R_M \mathbf F_{\chi}, B}(g_{\infty})$ at the elements 
\begin{equation}\label{ginfty}
 g_{\infty} = n(X) m(A,1) = \left(\begin{array}{cc} \mathbf 1_2 & X \\ 0 & \mathbf 1_2\end{array}\right)\left(\begin{array}{cc} A & 0 \\ 0 & {}^tA^{-1} \end{array}\right) \in \GSp_2(\R)
\end{equation}
with $X\in \mathrm{Sym}_2(\R)$, and $A \in \GL_2^+(\R)$, for which we know (cf. Corollary \ref{FourierCoeffRMFchi} and equation \eqref{WFchiB:ginfty}) that 
\[
\mathcal W_{\mathfrak R_M \mathbf F_{\chi}, B}(g_{\infty}) = \mathcal W_{\mathbf F_{\chi}, B}(g_{\infty}) = 
A_{\chi}(B) \det(Y)^{(k+1)/2}e^{2\pi\sqrt{-1}\mathrm{Tr}(BZ)},
\]
where $Y = A {}^t A$, $Z = X + \sqrt{-1}Y \in \mathcal H_2$. Here, $A_{\chi}(B)$ is the classical $B$-th Fourier coefficient of $F_{\chi}$, which can be made precise as in \eqref{Fchi:Fourier}.

Regarding $\theta(\mathbf h, \phi_{\mathbf h})$, it will follow from Lemma \ref{lemma:phih-properties} below that $\theta(\mathbf h,\phi_{\mathbf h}) \otimes {\underline{\chi}}$ satisfies the same invariance properties with respect to $K(NM,M) \subseteq \GSp_2(\hat{\Z})$ as $\mathfrak R_M \mathbf F_{\chi}$ does (namely, $K(NM,M)$ acts through $\underline{\chi}$ on both $\theta(\mathbf h,\phi_{\mathbf h}) \otimes {\underline{\chi}}$ and $\mathfrak R_M \mathbf F_{\chi}$). Therefore, by comparing the Fourier coefficients $\mathcal W_{\theta(\mathbf h,\phi_{\mathbf h})\otimes \underline{\chi}, B}$ with $\mathcal W_{\mathbf F_{\chi}, B}$ at elements $g_{\infty} \in \GSp_2(\R)$ as above we will be able to deduce a relation between $\mathcal W_{\theta(\mathbf h,\phi_{\mathbf h})\otimes \underline{\chi}, B}$ and $\mathcal W_{\mathbf F_{\chi}, B}$ as functions on $\GSp_2(\A)$, leading to the identity claimed in \eqref{ThetaIdentity2-eqn}. Furthermore, observe that from the very definitions we have 
\[
 \mathcal W_{\theta(\mathbf h,\phi_{\mathbf h})\otimes \underline{\chi}, B}(g_{\infty}) =  \mathcal W_{\theta(\mathbf h,\phi_{\mathbf h}), B}(g_{\infty}).
\]
For this reason, we will focus on the computation of Fourier coefficients of the automorphic form $ \theta(\mathbf h,\phi_{\mathbf h})$ on $\PGSp_2(\A)$ obtained as a theta lift from $\mathbf h$.

In order to determine the Fourier coefficients $\mathcal W_{\theta(\mathbf h,\phi_{\mathbf h}),B}$, as in the previous section it is useful to consider another model for the Weil representation. Recall the $3$-dimensional quadratic subspace $V_3 \subset V_5$ on which the quadratic form is given by $Q_1$, i.e. $V_3 = \langle v_2, v_3, v_4\rangle$. We identify $V_3^{\perp}$ with $F^2$, in a compatible way with the fact that 
\[
 Q = \left( \begin{array}{ccc} 0 & 0 & -1 \\ 0 & Q_1 & 0 \\ -1 & 0 & 0\end{array}\right).
\]
We consider the partial Fourier transform 
\[
 \mathcal S(V_5(\A)) \longrightarrow \mathcal S(V_3(\A))\otimes \mathcal S(\A^2), \quad \phi \longmapsto \hat{\phi},
\]
defined by setting  
\begin{equation}\label{defn:hat}
 \hat{\phi}(x;y) = \int_{\A} \phi (z; x; y_1) \psi(-y_2z)dz,
\end{equation}
where $x \in V_3(\A)$ and $y = (y_1,y_2)\in \A^2$. As usual, here $dz$ is the self-dual measure with respect to the additive character $\psi$. The Weil representation $\omega$ of $\widetilde{\SL}_2(\A) \times \mathrm O(V_5)(\A)$ on $\mathcal S(V_5(\A))$ gives rise then to a representation $\hat{\omega}$ of $\widetilde{\SL}_2(\A) \times \mathrm O(V_5)(\A)$ on $\mathcal S(V_3(\A))\otimes \mathcal S(\A^2)$ by setting 
\[
 \hat{\omega}(g,h) \hat{\phi} = (\omega(g,h)\phi){\hat{}}.
\]
If $\hat{\phi} = \phi_1 \otimes \phi_2$ with $\phi_1 \in \mathcal S(V_3(\A))$ and $\phi_2 \in \mathcal S(\A^2)$, then one has 
\begin{equation}\label{phihat-split}
 \hat{\omega}((g,\epsilon),1)\hat{\phi}(x;y) = \omega((g,\epsilon),1)\phi_1(x) \cdot \phi_2(yg)
\end{equation}
for $(g,\epsilon) \in \widetilde{\SL}_2(\A)$. This change of polarization helps to get simpler computations, and the identity in \eqref{phihat-split}, which we will use later, is an instance of this. Most importantly, in terms of this new model one can express the Fourier coefficients of $\theta(\mathbf h,\varphi)$, for a given $\varphi \in \mathcal S(V_5(\A))$, in terms of the Fourier coefficients of $\mathbf h$. Recall that if $\xi \in \Q$, then the $\xi$-th Fourier coefficient of $\mathbf h$ is by definition the function 
\[
g \mapsto W_{\mathbf h,\xi}(g) = \int_{\Q \backslash \A} \mathbf h(u(x)g)\overline{\psi(\xi x)}dx, \quad g \in \widetilde{\SL}_2(\A).
\]
As quoted in \eqref{FC:classicalautomorphic}, one has $c(n) = e^{2\pi n} W_{\mathbf h,n}(1)$ for all integers $n\geq 1$. With this, the following is proved in \cite[Lemma 4.2]{Ichino-pullbacks}.

\begin{lemma}
 If $\varphi \in \mathcal S(V_5(\A))$, then for $B \neq 0$ one has
\begin{equation}\label{FCSO5-lifts}
 \mathcal W_{\theta(\mathbf h,\varphi),B}(h) = \int_{U(\A)\backslash\SL_2(\A)} \hat{\omega}(g,h)\hat{\varphi}(\beta;0,1)W_{\mathbf h,\xi}(g) dg,
\end{equation}
where $\xi = \det(B)$ and $\beta = (b_3, b_2/2, -b_1)$ if $B = \left(\begin{smallmatrix}b_1 & b_2/2 \\ b_2/2 & b_3\end{smallmatrix}\right)$.
\end{lemma}

The identity in \eqref{FCSO5-lifts} is a crucial ingredient in our computation of the $B$-th Fourier coefficients of $\theta(\mathbf h, \phi_{\mathbf h})$ towards the proof of Proposition \ref{ThetaIdentity2}. To proceed with this computation, we still need to address two tasks. First, we must describe an explicit choice of Bruhat--Schwartz function $\phi_{\mathbf h} \in \mathcal S(V_5(\A))$. And second, we must express the integral on the right hand side of \eqref{FCSO5-lifts} as a product of local integrals, one for each rational place $v$. After this is done, we will be able to proceed with the computation of the Fourier coefficients of $\theta(\mathbf h, \phi_{\mathbf h})$ by performing local computations prime by prime. 

Concerning the choice of $\phi_{\mathbf h} \in \mathcal S(V_5(\A))$, recall the Bruhat--Schwartz function $\phi_{\breve{\mathbf g}} \in \mathcal S(V_4(\A))$ considered in the previous section, which is involved in the explicit formula in Proposition \ref{ThetaIdentity1}. For our proof of the main theorem, it is crucial that $\phi_{\mathbf h}$ is chosen so that its restriction to $V_4$ coincides with $\phi_{\breve{\mathbf g}}$. Since we have an explicit description of $\phi_{\breve{\mathbf g}}$, we only need to define $\phi_{\mathbf h}$ on the orthogonal complement of $V_4$, which is the one-dimensional quadratic subspace $V_1 := V_4^{\perp}$ of $V_5$ spanned by $v_3$. Notice that we may identify the space $\mathcal S(V_1(\A))$ with $\mathcal S(\A)$, by identifying $V_1$ with the one-dimensional quadratic space over $\Q$ endowed with the quadratic form $q(x) = x^2$. In $\mathcal S(V_1(\A))$, we consider the Bruhat--Schwartz function $\pmb{\phi} = \otimes_v \pmb{\phi}_v$ determined by its local components as 
follows:
\[
  \pmb{\phi}_v(x) =
  \begin{cases}
   \mathbf 1_{\Z_q}(x) & \text{if } v = q \neq \infty,\\
   e^{-2\pi x^2} & \text{if } v = \infty.
  \end{cases}
 \]
Considering the basis $v_1, \dots, v_5$ of $V_5$ as we did above, so that $V_1 = V_4^{\perp}$ is generated by $v_3$, and taking into account the embedding $V_4 \subset V_5$ explained in Section \ref{spaces:lowrank}, then $\phi_{\mathbf h}$ is defined by setting 
\[
 \phi_{\mathbf h}(z) := \pmb{\phi}(x_3) \phi_{\breve{\mathbf g}}\left(\left(\begin{smallmatrix} x_2 & x_1 \\ x_5 & x_4 \end{smallmatrix}\right)\right).
\]
for $z = x_1v_1 + x_2v_2 + x_3v_3 + x_4v_4 + x_5v_5 \in V_5(\A)$. By construction, the local components of $\phi_{\mathbf h} = \otimes_v \phi_{\mathbf h,v}$ can be easily given in terms of the local components of $\pmb{\phi}$ and $\phi_{\breve{\mathbf g}}$. For the reader's convenience, we describe such local components: if $v$ is a rational place, $z = x_1v_1 + x_2v_2 + x_3v_3 + x_4v_4 + x_5v_5 \in V_5(\Q_v)$ and we put $X = \left(\begin{smallmatrix} x_2 & x_1 \\ x_5 & x_4 \end{smallmatrix}\right)$, then $\phi_{\mathbf h,v}(z)$ is given as detailed below.
\begin{itemize}
 \item[i)] If $v = q \nmid N$ is an {\em odd} prime, then $\phi_{\mathbf h,q}$ is the characteristic function of $V_5(\Z_q)$. Indeed, 
 \[
  \phi_{\mathbf h,q}(z) = \pmb{\phi}_q(x_3) \phi_{\breve{\mathbf g},q}(X) = \mathbf 1_{\Z_q}(x_1)\mathbf 1_{\Z_q}(x_2)\mathbf 1_{\Z_q}(x_3)\mathbf 1_{\Z_q}(x_4)\mathbf 1_{\Z_q}(x_5).
 \]
 \item[ii)] If $v = 2$, $\phi_{\mathbf h,2}$ is the characteristic function of $\Z_2 v_3 + V_5(2\Z_2)$. Indeed, 
 \[
  \phi_{\mathbf h,2}(z) = \pmb{\phi}_2(x_3) \phi_{\breve{\mathbf g},2}(X) = \mathbf 1_{2\Z_2}(x_1)\mathbf 1_{2\Z_2}(x_2)\mathbf 1_{\Z_2}(x_3)\mathbf 1_{2\Z_2}(x_4)\mathbf 1_{2\Z_2}(x_5).
 \]
 \item[iii)] If $v = p \mid M$ is prime, then 
 \[
  \phi_{\mathbf h,p}(z) = \pmb{\phi}_p(x_3) \phi_{\breve{\mathbf g},p}(X) = 
  p^{-1/2}\varepsilon(1/2,\underline{\chi}_p) \mathbf 1_{p^{-1}\Z_p^{\times}}(x_1) \mathbf 1_{p\Z_p}(x_2)
  \mathbf 1_{\Z_p}(x_3)
  \mathbf 1_{\Z_p}(x_4) \mathbf 1_{p^2\Z_p}(x_5)\underline{\chi}_p(x_1).
 \]
 \item[iv)] If $v= p \mid N/M$ is prime, then 
 \[
  \phi_{\mathbf h,p}(z) = \pmb{\phi}_p(x_3) \phi_{\breve{\mathbf g},p}(X) = 
  \left(\mathbf 1_{\Z_p}(x_1) - p^{-1}\mathbf 1_{p^{-1}\Z_p}(x_1)\right) \mathbf 1_{\Z_p}(x_2)
  \mathbf 1_{\Z_p}(x_3) \mathbf 1_{\Z_p}(x_4) \mathbf 1_{p\Z_p}(x_5).
 \]
 \item[v)] For the archimedean prime $v = \infty$, 
 \[
  \phi_{\mathbf h, \infty}(z) = (x_2 + \sqrt{-1}x_1 + \sqrt{-1}x_5 - x_4)^{k+1}\exp(-\pi(x_1^2 + x_2^2 + 2x_3^2 + x_4^2 + x_5^2)).
 \]
\end{itemize}

In particular, observe that $\phi_{\mathbf h} = \otimes_v \phi_{\mathbf h,v}$ coincides with the Bruhat--Schwartz function $\varphi^{(5)} = \otimes_v \varphi_v^{(5)}$ defined in \cite[Section 7]{Ichino-pullbacks} locally at every place $v \nmid N$. Therefore, we can use Ichino's computations in loc. cit. at all such places.

We need to understand the action of the Weil representation of $\widetilde{\SL}_2(\A) \times \mathrm O(V_5)(\A)$ on $\phi_{\mathbf h} \in \mathcal S(V_5(\A))$. For later purposes, the properties we are interested in are collected in the following lemma. 

\begin{lemma}\label{lemma:phih-properties}
 Let $\phi_{\mathbf h}$ be defined as above, and $v$ be a rational place. Then the following assertions hold:
 \begin{itemize}
  \item[i)] If $v = q$ is an odd prime not dividing $N$, then $\hat{\phi}_{\mathbf h,q} = \phi_{\mathbf h,q}^1 \otimes \phi_{\mathbf h,q}^2$, where $\phi_{\mathbf h,q}^1 \in \mathcal S(V_3(\Q_p))$ and $\phi_{\mathbf h,q}^2 \in \mathcal S(\Q_p^2)$ are the characteristic functions of $V_3(\Z_q)$ and $\Z_q^2$, respectively. Besides, $\omega_q((k,s_q(k)),k')\phi_{\mathbf h,q} = \phi_{\mathbf h,q}$ for all $k \in \SL_2(\Z_q)$ and all $k' \in \mathrm{GSp}_2(\Z_q)$.
  \item[ii)] If $v = 2$, then $\omega_2(k,k')\phi_{\mathbf h,2} = \tilde{\epsilon}_2(k) \phi_{\mathbf h,2}$ for all $k \in \Gamma_0(4;\Z_2)$ and all $k' \in \mathrm{GSp}_2(\Z_2)$. 
  \item[iii)] If $v = \infty$, then 
  \[
   \omega_{\infty}(\tilde{k}_{\theta},k')\phi_{\mathbf h,\infty} = e^{-\sqrt{-1}(k+1/2)\theta}\det(\mathbf k)^{k+1}\phi_{\mathbf h,\infty}
  \]
  for all $\tilde{k}_{\theta} \in \widetilde{\mathrm{SO}(2)}$ and $k' = \left(\begin{smallmatrix} \alpha & \beta \\ -\beta & \alpha \end{smallmatrix}\right)\in \mathrm{Sp}_2(\R)$, with $\mathbf k = \alpha + \sqrt{-1}\beta \in \mathrm U(2)$. 
  \item[iv)] If $p \mid N/M$, then $\omega_p((k,s_p(k)),k')\phi_{\mathbf h,p} = \phi_{\mathbf h,p}$ for all $k \in \Gamma_0(NM;\Z_p) = \Gamma_0(p;\Z_p) \subseteq \SL_2(\Z_p)$ and all $k' \in K_0^{(2)}(NM;\Z_p) = K_0^{(2)}(p;\Z_p) \subseteq \mathrm{GSp}_2(\Z_p)$. And if $p \mid M$, then $\omega_p((k,s_p(k)),k')\phi_{\mathbf h,p} = \phi_{\mathbf h,p}$ for all $k \in \Gamma_0(NM;\Z_p) = \Gamma_0(p^2;\Z_p) \subseteq \SL_2(\Z_p)$ and all $k' \in K(NM, p; \Z_p) = K(p^2,p; \Z_p) \subseteq \mathrm{GSp}_2(\Z_p)$. Explicit expressions for $\hat{\phi}_{\mathbf h,p}$ in these cases are given in Lemma \ref{lemma:phih-badprimes} below.
 \end{itemize}
\end{lemma}
\begin{proof}
 Part i) is as in \cite[Section 7.3]{Ichino-pullbacks}; part ii) is worked out in \cite[Section 7.4]{Ichino-pullbacks}, where one also finds an explicit expression for $\hat{\omega}(r,1)\hat{\phi}_{\mathbf h,2}$, where $r$ varies over a set of representatives for $\SL_2(\Z_2)/K_0(4;\Z_2)$ (which consists only of $3$ elements); part iii) is covered in \cite[Section 7.5]{Ichino-pullbacks}. As for part iv), one can check it by routine (and tedious) computation using the explicit description of $\phi_{\mathbf h,p}$ together with the rules for the Weil representation and the explicit model of $\SO(V_5)$ that we are using (cf. Section \ref{spaces:lowrank}). We omit this computation and leave it for the reader.
\end{proof}

At primes $p\mid N$, we will also need the partial Fourier transforms $\hat{\phi}_{\mathbf h,p}$ of each local component $\phi_{\mathbf h,p}$, which we collect in the next lemma.

\begin{lemma}\label{lemma:phih-badprimes}
 With notation as above, let $p$ be a prime dividing $N$, and let $x=(x_1,x_2,x_3) \in V_3(\Q_p)$, $y=(y_1,y_2)\in \Q_p^2$. Then one has 
 \[
  \hat{\phi}_{\mathbf h,p}(x;y) = 
  \begin{cases}
  \mathbf 1_{\Z_p}(x_1)\mathbf 1_{\Z_p}(x_2)\mathbf 1_{\Z_p}(x_3)\mathbf 1_{p\Z_p}(y_1)\mathbf 1_{\Z_p^{\times}}(y_2) & \text{if } p \mid N/M, \\
  \mathbf 1_{p\Z_p}(x_1)\mathbf 1_{\Z_p}(x_2)\mathbf 1_{\Z_p}(x_3)\mathbf 1_{p^2\Z_p}(y_1) \mathbf 1_{\Z_p^{\times}}(y_2) \underline{\chi}_p^{-1}(y_2) & \text{if } p \mid M.
  \end{cases}
 \]
\end{lemma}

\begin{proof}
 We consider first the case $p\mid M$. By applying the definition of the partial Fourier transform and the recipe in iii) above we have, for $x = (x_1,x_2,x_3) \in V_3(\Q_p)$ and $(y_1,y_2)\in \Q_p^2$,
\begin{align*}
\hat{\phi}_{\mathbf h,p}(x;y) & = p^{-1/2}\varepsilon(1/2,\underline{\chi}_p)\mathbf 1_{p\Z_p}(x_1)\mathbf 1_{\Z_p}(x_2)\mathbf 1_{\Z_p}(x_3)\mathbf 1_{p^2\Z_p}(y_1) \int_{p^{-1}\Z_p^{\times}} \underline{\chi}_p(z)\psi_p(-y_2z) dz = \\
& = \mathbf 1_{p\Z_p}(x_1)\mathbf 1_{\Z_p}(x_2)\mathbf 1_{\Z_p}(x_3)\mathbf 1_{p^2\Z_p}(y_1) \mathbf 1_{\Z_p^{\times}}(y_2) \underline{\chi}_p^{-1}(y_2).
\end{align*}

Similarly, for primes $p\mid N/M$, we find for $x = (x_1,x_2,x_3) \in V_3(\Q_p)$ and $(y_1,y_2)\in \Q_p^2$ that 
\[
 \hat{\phi}_{\mathbf h,p}(x;y) = \mathbf 1_{\Z_p}(x_1)\mathbf 1_{\Z_p}(x_2)\mathbf 1_{\Z_p}(x_3)\mathbf 1_{p\Z_p}(y_1) \int_{\Q_p} \left(\mathbf 1_{\Z_p}(z)- p^{-1}\mathbf 1_{p^{-1}\Z_p}(z)\right)\psi_p(-y_2z) dz.
\]
The last integral is easily seen to equal $\mathbf 1_{\Z_p^{\times}}(y_2)$, hence the result follows.
\end{proof}

Having described our choice for $\phi_{\mathbf h}$, together with its main properties, we now focus on the right hand side of \eqref{FCSO5-lifts}. We want to decompose the Fourier coefficients $W_{\mathbf h,\xi}$ as a product of local Whittaker functions $W_{\mathbf h_v,\xi}$, so that the integral on the right hand side of \eqref{FCSO5-lifts} decomposes as a product of local integrals that will be eventually computed place by place. By multiplicity one, any decomposition of $W_{\mathbf h,\xi}$ as a product of local Whittaker functions will differ from a fixed one by a non-zero scalar factor. 
Our choice will follow closely the discussion in \cite[Section 8]{BaruchMao}, with slight renormalizations so that our decomposition will reflect the identity proved in \eqref{cxi-Psip}. Let $\xi \in \Q^+$, and write $\xi = \mathfrak d_{\xi}\mathfrak f_{\xi}^2$ with $\mathfrak d_{\xi} \in \mathbb N$ and $\mathfrak f_{\xi} \in \Q^+$, so that $-\mathfrak d_{\xi}$ is the discriminant of $\Q(\sqrt{-\xi})/\Q$. Write $e_p := \mathrm{ord}_p(\mathfrak f_{\xi})$, and recall the functions $\Psi_p(\xi;X) \in \C[X,X^{-1}]$ defined in Section \ref{AF:SL2}. 
For each rational place $v$, we define the local Whittaker function $W_{v,\xi} = W_{\mathbf h_v,\xi}$ attached to $\mathbf h_v$ and $\xi$ as follows (we do not enter here in a description of the local types of $\tilde{\pi}_p$ at primes $p \mid N$, which has been given in Section \ref{sec:BMexplained} and will be recalled and used again in Sections \ref{sec:periodsN/M} and \ref{sec:periodsM}, where we will need them to perform the computation of local $\SL_2$-periods).
\begin{itemize}
 \item[i)] If $v = p$ is a finite prime with $p\nmid N$, then we define $W_{p, \xi}$ as in \cite[Section 7.2, Appendix A.3]{Ichino-pullbacks}. In particular, for all primes $p\nmid N$ we have $W_{p,\xi}(1) = \Psi_p(\xi;\alpha_p)$.
 \item[ii)] If $v = p$ is a finite prime with $p \mid N/M$, then $\tilde{\pi}_p$ is the {\em special representation} $\tilde{\sigma}^{\delta}(\overline{\psi}_p^D)$ as explained in Section \ref{sec:BMexplained}, where $\delta \in \Z_p^{\times}$ is any non-square unit. The $p$-th component $\mathbf h_p \in \tilde{\pi}_p$ of $\mathbf h$ lies in the one-dimensional subspace of vectors fixed by $\widetilde{\Gamma}_0(p) \subseteq \widetilde{\SL}_2(\Z_p)$, and hence it is a multiple of the newvector $\tilde{\varphi}_p$ chosen as in \cite[Lemma 8.3]{BaruchMao} (see Lemma \ref{sl2testvector_special} below). We consider the local Whittaker function 
 \[
  W_{\tilde{\varphi}_p,\xi}: g \, \mapsto \, \int_{\Q_p} \tilde{\varphi}_p(s^{-1}u(x)g) \overline{\psi_p(\xi x)} dx
 \]
 associated with $\tilde{\varphi}_p$, with respect to $\psi_p^{\xi}$. We may assume that $p \nmid \mathfrak d_{\xi}$, as otherwise $W_{\mathbf h,\xi}(1) = 0$. Then, it follows from the computations in \cite[Section 8]{BaruchMao} that 
 \[
  W_{\tilde{\varphi}_p,\xi}(1) = 
  \begin{cases}
  2p^{-e_p} & \text{if } \chi_{-\xi}(p)+w_p \neq 0, \\
  0 & \text{if } \chi_{-\xi}(p)+w_p = 0.
  \end{cases}
 \]
 From the definition of the function $\Psi_p(\xi;X)$ in Section \ref{AF:SL2}, we see that 
 \[
  \Psi_p(\xi;\alpha_p) = p^{e_p(3/2-k)}a(p^{e_p})W_{\tilde{\varphi}_p,\xi}(1).
 \]
 We define $W_{p,\xi} := p^{e_p(3/2-k)}a(p^{e_p})W_{\tilde{\varphi}_p,\xi}$, so that we have $W_{p,\xi}(1) = \Psi_p(\xi;\alpha_p)$.
 
 \item[iii)] If $v=p$ is a finite prime with $p\mid M$, then $\tilde{\pi}_p$ is the {\em supercuspidal} odd Weil representation $r_{\overline{\psi}_p^D}^-$ explained in Section \ref{sec:BMexplained}. The subspace of vectors in $\tilde{\pi}_p$ on which $\widetilde{\Gamma}_0(p^2) \subseteq \widetilde{\SL}_2(\Z_p)$ acts through $\underline{\chi}_{0,p} (= \underline{\chi}_p)$ is one-dimensional, and hence the $p$-th component $\mathbf h_p \in \tilde{\pi}_p$ of $\mathbf h$ is a multiple of the newvector $\tilde{\varphi}_p$ chosen as in \cite[Proposition 8.5]{BaruchMao} (see Lemma \ref{sl2testvector_weil} below). The representation $\tilde{\pi}_p = r_{\overline{\psi}_p^D}^-$ is distinguished, in the sense that it only has non-trivial $\psi_p^{\xi}$-Whittaker functionals for $\xi$ in the same square class as $-D$. Equivalently, this holds if and only if $-\xi$ is in the same square class as $D$. If this is the case, we choose the local Whittaker function $W_{\tilde{\varphi}_p,\xi}$ associated with $\tilde{\varphi}_
p$ to satisfy $W_{\tilde{\varphi}_p,\xi}(1) = \mathbf 1_{\Z_p^{\times}}(\mathfrak f_{\xi})\underline{\chi}_{p}(\mathfrak f_{\xi})^{-1}$ (this normalization differs slightly from the one chosen in \cite[Section 8]{BaruchMao}). From the definition of the function $\Psi_p(\xi;X)$ in Section \ref{AF:SL2}, we now have 
 \[
  \Psi_p(\xi;\alpha_p) =  2p^{e_p(1/2-k)}a(p^{e_p})\underline{\chi}_{p}(\mathfrak f_{\xi})W_{\tilde{\varphi}_p,\xi}(1).
 \]
 We define $W_{p,\xi} :=  2p^{e_p(1/2-k)}a(p^{e_p}) W_{\tilde{\varphi}_p,\xi}$, so that $\Psi_p(\xi;\alpha_p) = \underline{\chi}_{p}(\mathfrak f_{\xi})W_{p,\xi}(1)$.
 
 \item[iv)] At the archimedean place $v=\infty$, again as in \cite[Section 7.2]{Ichino-pullbacks} we define $W_{\infty,\xi}$ by setting 
 \[
  W_{\infty,\xi}(u(x)t(a) k_{\theta}) = e^{2\pi\sqrt{-1}\xi x} a^{k+1/2}e^{-2\pi\xi a^2}e^{\sqrt{-1}(k+1/2)\theta}
 \]
 for $x\in \R$, $a\in \R_+^{\times}$, $\theta \in \R/4\pi\Z$ and $k_{\theta} = \left(\begin{smallmatrix} \cos \theta & \sin \theta \\ -\sin \theta & \cos \theta\end{smallmatrix}\right) \in \mathrm{SO}(2)$. In particular, $W_{\infty,\xi}(1) = e^{-2\pi\xi}$.
\end{itemize}

With these choices, we see from equation \eqref{FC:classicalautomorphic} and Lemma \ref{lemma:c(xi)} that 
\begin{equation}\label{Wh=Whlocals}
 W_{\mathbf h,\xi}(1)  = e^{-2\pi\xi}2^{-\nu(N)}c(\mathfrak d_{\xi})\chi(\mathfrak f_{\xi})\mathfrak f_{\xi}^{k-1/2} \prod_p \Psi_p(\xi;\alpha_p) = 
 2^{-\nu(N)}c(\mathfrak d_{\xi})\mathfrak f_{\xi}^{k-1/2} \prod_v W_{v,\xi}(1),
\end{equation}
and so the next lemma follows immediately by combining \eqref{Wh=Whlocals} with \eqref{FCSO5-lifts}.

\begin{lemma}
 With the above notations, for $B \neq 0$ we have 
 \begin{equation}\label{FCthetah:global-local}
 \mathcal W_{\theta(\mathbf h,\phi_{\mathbf h}),B} = 
 \begin{cases}
  2^{-\nu(N)}c(\mathfrak d_{\xi})\mathfrak f_{\xi}^{k-1/2} \zeta_{\Q}(2)^{-1}\prod_v \mathcal W_{B,v} & \text{if } \xi > 0, \\
  0 & \text{if } \xi \leq 0,
 \end{cases}
\end{equation}
where 
\[
 \mathcal W_{B,v}(h) = \int_{U(\Q_v)\backslash\SL_2(\Q_v)} \hat{\omega}(g,h)\hat{\phi}_{\mathbf h}(\beta;0,1) W_{v,\xi}(g)dg \times 
 \begin{cases}
  \mathrm{vol}(\SL_2(\Z_p))^{-1} & \text{if } v = p, \\
  \mathrm{vol}(\mathrm{SO}(2))^{-1} & \text{if } v = \infty.
 \end{cases}
\]
As before, here  $\xi = \det(B)$ and $\beta = (b_3, b_2/2, -b_1)$ if $B = \left(\begin{smallmatrix}b_1 & b_2/2 \\ b_2/2 & b_3\end{smallmatrix}\right)$.
\end{lemma}

By using this lemma, we can now determine the Fourier coefficients $\mathcal W_{\theta(\mathbf h,\phi_{\mathbf h}),B}$ by computing the local functions $\mathcal W_{B,v}$. Because of the invariance properties spelled out in Lemma \ref{lemma:phih-properties}, we see that $\theta(\mathbf h,\phi_{\mathbf h})$ is right invariant for $K(NM,M)$. In particular, as commented above the Fourier coefficients $\mathcal W_{\theta(\mathbf h,\phi_{\mathbf h}),B}$ are determined by the values $\mathcal W_{\theta(\mathbf h,\phi_{\mathbf h}),B}(g_{\infty})$ with $g_{\infty} = n(X)m(A,1)$ as in \eqref{ginfty}. Hence we only need to determine the values $\mathcal W_{B,p}(1)$ at finite primes $p$, together with $\mathcal W_{B,\infty}(n(X)m(A,1))$. We discuss case by case such computations.

\subsubsection{Computation at primes $p \nmid N$}

At primes $p \nmid N$, we can compute $\mathcal W_{B,p}(1)$ literally as in \cite[Sections 7.3, 7.4]{Ichino-pullbacks}. We summarize the outcome of such computation. At each prime $p \nmid N$, we continue to denote by $\{\alpha_p, \alpha_p^{-1}\}$ the Satake parameter of $f$ at $p$, and consider the function $\Psi_p(\xi; X) \in \C[X + X^{-1}]$ as above. Recall that $\phi_{\mathbf h,p}$ is the characteristic function of $V_5(\Z_p)$. Further, from Lemma \ref{lemma:phih-properties} we have 
\[
 \omega_p((k,s_p(k)),k')\phi_{\mathbf h,p} = \phi_{\mathbf h,p}
\]
for all $k \in \SL_2(\Z_p)$ and $k'\in \GSp_2(\Z_p)$, and $\hat{\phi}_{\mathbf h,p} = \phi_{p,1}\otimes \phi_{p,2}$, where $\phi_{p,1}\in \mathcal S(V_3(\Q_p))$ and $\phi_{p,2}\in \mathcal S(\Q_p^2)$ are the characteristic functions of $V_3(\Z_p)$ and $\Z_p^2$, respectively. In this case, one finds out that for $\xi\neq 0$
\begin{equation}\label{WBgoodp}
 \mathcal W_{B,p}(1) = \begin{cases}
               \sum_{n=0}^{\mathrm{min}(\mathrm{ord}_p(b_i))} p^{n/2}\Psi_p(p^{-2n}\xi;\alpha_p) & \text{if } b_1, b_2, b_3 \in \Z_p,\\
               0 & \text{otherwise}.
              \end{cases}
\end{equation}

At the prime $p=2$, the $2$-component $\phi_{\mathbf h,2}$ of the Bruhat--Schwartz function $\phi_{\mathbf h}$ is the characteristic function of $\Z_2 v_3 + V_5(2\Z_2)$, which satisfies $\omega_2(k,k')\phi_{\mathbf h,2} = \tilde{\epsilon}_2(k)\phi_{\mathbf h,2}$ for all $k\in \Gamma_0(4;\Z_2)$, $k' \in \GSp_2(\Z_2)$ (cf. Lemma \ref{lemma:phih-properties}). In this case one finds, for $\xi\neq 0$,
\begin{equation}\label{WB2}
 \mathcal W_{B,2}(1) = \begin{cases}
               2^{-7/2}\sum_{n=0}^{\mathrm{min}(\mathrm{ord}_2(b_i))} 2^{n/2}\Psi_2(2^{-2n+2}\xi;\alpha_2) & \text{if } b_1, b_2, b_3 \in \Z_2,\\
               0 & \text{otherwise.}\end{cases}
\end{equation}

\subsubsection{Computation at primes $p \mid N/M$}

Let $p$ be a prime dividing $N/M$. Now $\hat{\phi}_{\mathbf h,p}$ is not $\SL_2(\Z_p)$-invariant, but only $\Gamma_0(p)$-invariant. Let $R_p$ be a set of representatives for $\SL_2(\Z_p)/\Gamma_0(p)$. Then, using that $\SL_2(\Q_p) = U(\Q_p)T(\Q_p)\SL_2(\Z_p)$ we have 
\begin{align*}
 \mathcal W_{B,p}(h) & = \mathrm{vol}(\SL_2(\Z_p))^{-1}\int_{\Q_p^{\times}} \int_{\SL_2(\Z_p)} 
 \hat{\omega}_p(t(a)k,h)\hat{\phi}_{\mathbf h,p}(\beta;0,1)W_{p,\xi}(t(a)k) |a|_p^{-2} dk d^{\times}a = \\
 & = \mathfrak c_p  \int_{\Q_p^{\times}} \sum_{r\in R_p} 
 \hat{\omega}_p(t(a)r,h)\hat{\phi}_{\mathbf h,p}(\beta;0,1)W_{p,\xi}(t(a)r) |a|_p^{-2} d^{\times}a,
\end{align*}
where $\mathfrak c_p := [\SL_2(\Z_p):\Gamma_0(p)]^{-1} = \frac{\mathrm{vol}(\Gamma_0(p))}{\mathrm{vol}(\SL_2(\Z_p))}$. We will compute $\mathcal W_{B,p}(1)$. If $a \in \Q_p^{\times}$ and $r \in R_p$, recall from \eqref{phihat-split} that 
\[
 \hat{\omega}_p(t(a)r,1)\hat{\phi}_{\mathbf h,p}(\beta;0,1) = \omega_p((t(a)r,1))\phi_{1,p}(\beta) \phi_{2,p}((0,1)t(a)r),
\]
where $\phi_{1,p}(x) = \mathbf 1_{\Z_p}(x_1)\mathbf 1_{\Z_p}(x_2) \mathbf 1_{\Z_p}(x_3)$ and $\phi_{2,p}(y) = \mathbf 1_{p\Z_p}(y_1)\mathbf 1_{\Z_p^{\times}}(y_2)$. We take $R_p$ to be the set consisting of the elements 
\[
\left(\begin{array}{cc} 1 & 0 \\ b & 1 \end{array}\right), \text{ with } b \in \Z_p/p\Z_p, \, \text{and } \left(\begin{array}{cc} 0 & -1 \\ 1 & 0 \end{array}\right).
\]
Therefore, the elements $t(a)r$ with $a\in \Q_p^{\times}$, $r \in R_p$, are precisely the elements of the form 
\[
 \left(\begin{array}{cc} a & 0 \\ a^{-1}b & a^{-1} \end{array}\right), \, 
 \left(\begin{array}{cc} 0 & -a \\ a^{-1} & 0 \end{array}\right), \quad a \in \Q_p^{\times}, \, b \in \Z_p/p\Z_p.
\]
From the very definition of $\phi_{2,p}$, it is immediate to see that $\phi_{2,p}((0,1)t(a)r) = 0$ unless $r = 1$ and $a \in \Z_p^{\times}$.
Therefore, we deduce that 
\begin{align*}
 \mathcal W_{B,p}(1) & = \mathfrak c_p \int_{\Z_p^{\times}} \hat{\omega}_p(t(a),1)\hat{\phi}_{\mathbf h,p}(\beta;0,1)W_{p,\xi}(t(a)) d^{\times}a = \mathfrak c_p  \phi_{1,p}(\beta) \int_{\Z_p^{\times}} W_{p,\xi}(t(a)) d^{\times}a = \\
 & = \mathfrak c_p  \mathbf 1_{\Z_p}(b_3)\mathbf 1_{\Z_p}(b_2)\mathbf 1_{\Z_p}(b_1) \int_{\Z_p^{\times}} W_{p,\xi}(t(a)) d^{\times}a  =\mathfrak c_p  \mathbf 1_{\Z_p}(b_3)\mathbf 1_{\Z_p}(b_2)\mathbf 1_{\Z_p}(b_1)\Psi_p(\xi;\alpha_p),
\end{align*}
where we have used that for $a \in \Z_p^{\times}$ one has (here, recall that $\psi = \psi_p^{-D}$, where $\psi_p$ is the standard additive character of $\Q_p$, and $\delta \in \Z_p^{\times}$ is any non-square unit)
\[
 W_{p,\xi}(t(a))= \chi_{\psi}(a^{-1})\chi_{\delta}(a^{-1})|a|_p^{1/2}W_{p,a^2\xi}(1) = W_{p,a^2\xi}(1) = \Psi_p(a^2\xi;\alpha_p) = \Psi_p(\xi;\alpha_p).
\]
We conclude that 
\begin{equation}\label{WBpN}
  \mathcal W_{B,p}(1) = 
  \begin{cases}
   [\SL_2(\Z_p):\Gamma_0(p)]^{-1}\Psi_p(\xi;\alpha_p) & \text{if } b_1, b_2, b_3 \in \Z_p, \\
   0 & \text{otherwise.}
  \end{cases}
\end{equation}

\subsubsection{Computation at primes $p \mid M$}

We proceed similarly as in the previous case. But now if $p$ is a prime dividing $M$, then $\hat{\phi}_{\mathbf h,p}$ is only $\Gamma_0(p^2)$-invariant. If we denote by $R_{p^2}$ a set of representatives for $\SL_2(\Z_p)/\Gamma_0(p^2)$, then we have  
\begin{align*}
 \mathcal W_{B,p}(h) & = \mathrm{vol}(\SL_2(\Z_p))^{-1}\int_{\Q_p^{\times}} \int_{\SL_2(\Z_p)} 
 \hat{\omega}_p(t(a)k,h)\hat{\phi}_{\mathbf h,p}(\beta;0,1)W_{p,\xi}(t(a)k) |a|_p^{-2} dk d^{\times}a = \\
 & = \mathfrak c_{p^2} \int_{\Q_p^{\times}} \sum_{r\in R_{p^2}} 
 \hat{\omega}_p(t(a)r,h)\hat{\phi}_{\mathbf h,p}(\beta;0,1)W_{p,\xi}(t(a)r) |a|_p^{-2} d^{\times}a,
\end{align*}
where now $\mathfrak c_{p^2} := [\SL_2(\Z_p):\Gamma_0(p^2)]^{-1} = \frac{\mathrm{vol}(\Gamma_0(p^2))}{\mathrm{vol}(\SL_2(\Z_p))}$. To compute $\mathcal W_{B,p}(1)$, as before we notice that for $a \in \Q_p^{\times}$ and $r \in R_{p^2}$ we have 
\[
 \hat{\omega}_p(t(a)r,1)\hat{\phi}_{\mathbf h,p}(\beta;0,1) = \omega_p((t(a)r,1))\phi_{1,p}(\beta) \phi_{2,p}((0,1)t(a)r),
\]
where now $\phi_{1,p}(x) = \mathbf 1_{p\Z_p}(x_1)\mathbf 1_{\Z_p}(x_2)\mathbf 1_{\Z_p}(x_3)$ and $\phi_{2,p}(y) = \mathbf 1_{p^2\Z_p}(y_1)\mathbf 1_{\Z_p^{\times}}(y_2)\underline{\chi}_{0,p}^{-1}(y_2)$. We might take as a set of representatives for $\SL_2(\Z_p)/\Gamma_0(p^2)$ the set $R_{p^2}$ consisting of the elements 
\[
\left(\begin{array}{cc} 1 & 0 \\ b & 1 \end{array}\right), \text{ with } b \in \Z_p/p^2\Z_p, \, \text{and } \left(\begin{array}{cc} 0 & -1 \\ 1 & 0 \end{array}\right),
\]
hence the elements $t(a)r$ with $a\in \Q_p^{\times}$, $r \in R_p$, are precisely the elements of the form 
\[
 \left(\begin{array}{cc} a & 0 \\ a^{-1}b & a^{-1} \end{array}\right), \, 
 \left(\begin{array}{cc} 0 & -a \\ a^{-1} & 0 \end{array}\right), \quad a \in \Q_p^{\times}, \, b \in \Z_p/p^2\Z_p.
\]
As in the previous case, one easily checks that $\phi_{2,p}((0,1)t(a)r) = 0$ unless $r=1$ and $a \in \Z_p^{\times}$,
hence it follows that 
\begin{align*}
 \mathcal W_{B,p}(1) & = \mathfrak c_{p^2}
 \int_{\Z_p^{\times}} \hat{\omega}_p(t(a),1)\hat{\phi}_{\mathbf h,p}(\beta;0,1)W_{p,\xi}(t(a)) d^{\times}a = \mathfrak c_{p^2} \phi_{1,p}(\beta) \int_{\Z_p} \underline{\chi}_{0,p}(a) W_{p,\xi}(t(a)) d^{\times}a = \\
 & = \mathfrak c_{p^2} \mathbf 1_{p\Z_p}(b_3)\mathbf 1_{\Z_p}(b_2)\mathbf 1_{\Z_p}(b_1) \int_{\Z_p} \underline{\chi}_{0,p}(a) W_{p,\xi}(t(a)) d^{\times}a = \\
 & = \mathfrak c_{p^2} \mathbf 1_{p\Z_p}(b_3)\mathbf 1_{\Z_p}(b_2)\mathbf 1_{\Z_p}(b_1)\underline{\chi}_{0,p}(\mathfrak f_{\xi})^{-1}\Psi_p(\xi;\alpha_p),
\end{align*}
where we have used that for $a \in \Z_p^{\times}$ one has $W_{p,\xi}(t(a)) = \underline{\chi}_{p}(a \mathfrak f_{\xi})^{-1}\Psi_p(\xi;\alpha_p)$ for $a \in \Z_p^{\times}$. We thus obtain  
\begin{equation}\label{WBpM}
  \mathcal W_{B,p}(1) = 
  \begin{cases}
   [\SL_2(\Z_p):\Gamma_0(p^2)]^{-1}\chi_{(p)}(\mathfrak f_{\xi})\Psi_p(\xi;\alpha_p) & \text{if } b_1, b_2 \in \Z_p, b_3 \in p\Z_p, \\
   0 & \text{otherwise.}
  \end{cases}
\end{equation}

\subsubsection{Computation at the archimedean place}

At the archimedean place, the determination of $\mathcal W_{B,\infty}$ is carried out in \cite[Section 7.5]{Ichino-pullbacks}. Recall that the $\infty$-component $\phi_{\mathbf h,\infty} \in \mathcal S(V_5(\R))$ of $\phi_{\mathbf h}$ is given by 
\[
 \phi_{\mathbf h,\infty}(x) = (x_2 + \sqrt{-1}x_1 + \sqrt{-1}x_5 - x_4)^{k+1}e^{-\pi(x_1^2+x_2^2+2x_3^2+x_4^2+x_5^2)},
\]
and it satisfies (cf. Lemma \ref{lemma:phih-properties})
\[
 \omega_{\infty}(\tilde k_{\theta},k')\phi_{\mathbf h,\infty} = e^{-\sqrt{-1}(k+1/2)\theta}\det(\mathbf k)^{k+1}\phi_{\mathbf h,\infty}
\]
for $\tilde k_{\theta} \in \widetilde{\SO}(2)$ and $k' = \left(\begin{smallmatrix} \alpha & \beta \\ -\beta & \alpha  \end{smallmatrix}\right) \in \Sp_2(\R)$ with $\mathbf k = \alpha + \sqrt{-1}\beta \in \mathrm U(2)$. Then it is proved in \cite[Lemma 7.6]{Ichino-pullbacks} that for $\xi > 0$, $A \in \GL_2^+(\R)$ and $X \in \Sym_2(\R)$, one has 
\begin{equation}\label{WBinfty}
 \mathcal W_{B,\infty}(n(X)\mathrm{diag}(A,{}^tA^{-1})) = \begin{cases}
                             2^{k+1}\det(Y)^{(k+1)/2}e^{2\pi\sqrt{-1}\Tr(BZ)} & \text{if } B > 0,\\
                             0 & \text{if } B < 0,
                            \end{cases}
\end{equation}
where $Y = A {}^tA$, $Z = X + \sqrt{-1}Y$.

\subsubsection{Proof of Proposition \ref{ThetaIdentity2}}

We are finally in position to compute the Fourier coefficients for $\theta(\mathbf h,\phi_{\mathbf h})$. Let us fix $B \in \mathrm{Sym}_2(\Q)$, and 
\[
 g_{\infty} = n(X)m(A,1) = n(X)\mathrm{diag}(A,{}^tA^{-1}) \in \mathrm{Sp}_2(\R)
\]
with $X\in \Sym_2(\R)$ and $A \in \GL_2^+(\R)$. By using \eqref{WBgoodp}-\eqref{WBinfty} above, we may assume that $B>0$, and that $b_1, b_2, b_3 \in \Z$ with $b_3 \in M\Z$; otherwise we have $\mathcal W_{\theta(\mathbf h,\phi_{\mathbf h})}(g_{\infty}) = 0$. In that case, by virtue of \eqref{FCthetah:global-local} we have 
\[
 \mathcal W_{\theta(\mathbf h,\phi_{\mathbf h}),B}(g_{\infty}) = 
 2^{-\nu(N)} \zeta_{\Q}(2)^{-1}c(\mathfrak d_{\xi})\mathfrak f_{\xi}^{k-1/2} \mathcal W_{B,\infty}(g_{\infty}) \prod_p \mathcal W_{B,p}(1),
\] 
where $\mathcal W_{B,\infty}(g_{\infty}) = 2^{k+1}\det(Y)^{(k+1)/2}e^{2\pi\sqrt{-1}\Tr(BZ)}$ as in \eqref{WBinfty}. For ease of notation, let us write $\mathcal W_{\infty} := 2^{-k-1}\mathcal W_{B,\infty}(g_{\infty}) = \det(Y)^{(k+1)/2}e^{2\pi\sqrt{-1}\Tr(BZ)}$, and to abbreviate put 
\[
 \mu_{N,M} := [\SL_2(\Z):\Gamma_0(NM)]^{-1} = M^{-1}[\SL_2(\Z):\Gamma_0(N)]^{-1}.
\]
Then, using \eqref{WBgoodp}-\eqref{WBpM}, $\mathcal W_{\theta(\mathbf h,\phi_{\mathbf h}),B}(g_{\infty})$ equals  
\[
 2^{k-\nu(N)-5/2} \zeta_{\Q}(2)^{-1} \mu_{N,M}c(\mathfrak d_{\xi})\chi(\mathfrak f_{\xi})\mathfrak f_{\xi}^{k-1/2} 
 \mathcal W_{\infty} \times \prod_{p\nmid N} \sum_{n=0}^{\mathrm{min}(\mathrm{ord}_p(b_i))} p^{n/2}\Psi_p\left(\frac{4\xi}{p^{2n}};\alpha_p\right) \prod_{p\mid N} \Psi_p(\xi;\alpha_p).
\]
Next, observe that $\mathfrak d_{4\xi} = \mathfrak d_{\xi}$, hence $c(\mathfrak d_{\xi}) = c(\mathfrak d_{4\xi})$, and that $\mathfrak f_{\xi}^{k-1/2} = 2^{-k+1/2}\mathfrak f_{4\xi}^{k-1/2}$. Therefore, the previous expression can be rewritten as  
\begin{align*}
2^{-\nu(N)-2}\zeta_{\Q}(2)^{-1}\mu_{N,M} c(\mathfrak d_{4\xi})\chi(\mathfrak f_{\xi})\mathfrak f_{4\xi}^{k-1/2} \mathcal W_{\infty} 
 \times \sum_{\substack{d \mid (b_1,b_2,b_3),\\ (d,N)=1}} d^{1/2} \prod_{p\nmid N}\Psi_p\left(\frac{4\xi}{d^2};\alpha_p\right) \prod_{p\mid N} \Psi_p\left(\xi;\alpha_p\right).
\end{align*}
Now, for every integer $d$ with $(d,N)=1$, we have $c(\mathfrak d_{4\xi}) = c(\mathfrak d_{4\xi/d^2})$, $\mathfrak f_{4\xi}^{k-1/2} = d^{k-1/2}\mathfrak f_{4\xi/d^2}$, and 
$\Psi_p\left(\xi;\alpha_p\right) = \Psi_p\left(\frac{4\xi}{d^2};\alpha_p\right)$ for every prime $p\mid N$. Hence, 
\begin{align*}
 \mathcal W_{\theta(\mathbf h,\phi_{\mathbf h}),B}(g_{\infty}) & = \frac{\mu_{N,M}\mathcal W_{\infty}}{2^{2}\zeta_{\Q}(2)}  \sum_{\substack{d \mid (b_1,b_2,b_3),\\ (d,N)=1}} d^k 2^{-\nu(N)}c(\mathfrak d_{4\xi/d^2})\chi(\mathfrak f_{\xi})\mathfrak f_{4\xi/d^2}^{k-1/2}\prod_p\Psi_p\left(\frac{4\xi}{d^2};\alpha_p\right) = \\
 & =\frac{\mu_{N,M}\mathcal W_{\infty}}{2^{2}\zeta_{\Q}(2)} \sum_{\substack{d \mid (b_1,b_2,b_3),\\ (d,N)=1}} \chi(d/2) d^k 2^{-\nu(N)}c(\mathfrak f_{4\xi/d^2})\chi(\mathfrak f_{4\xi/d^2})\mathfrak f_{4\xi/d^2}^{k-1/2}\prod_p\Psi_p\left(\frac{4\xi}{d^2};\alpha_p\right) = \\
 & = \frac{\mu_{N,M}\mathcal W_{\infty}}{2^2\chi(2)\zeta_{\Q}(2)} \sum_{\substack{d \mid (b_1,b_2,b_3),\\ (d,N)=1}} \chi(d) d^k c(4\xi/d^2) = \frac{\mu_{N,M}}{2^2\chi(2)\zeta_{\Q}(2)}\det(Y)^{(k+1)/2}e^{2\pi\sqrt{-1}\Tr(BZ)} A_{\chi}(B).
\end{align*}

Thus, we conclude that when $B > 0$ and $b_1,b_2,b_3 \in \Z$ with $b_3 \in M\Z$,
\[
 \mathcal W_{\theta(\mathbf h,\phi_{\mathbf h})\otimes \underline{\chi},B}(g_{\infty}) = \mathcal W_{\theta(\mathbf h,\phi_{\mathbf h}),B}(g_{\infty}) = C  \mathcal W_{\mathbf F_{\chi},B}(g_{\infty}) = C \mathcal W_{\mathfrak R_M\mathbf F_{\chi},B}(g_{\infty}),
\]
where $C = 2^{-2}\chi(2)^{-1} \mu_{N,M} \zeta_{\Q}(2)^{-1}$. Since both $\mathcal W_{\theta(\mathbf h,\phi_{\mathbf h})\otimes\underline{\chi}, B}$ and $\mathcal W_{\mathfrak R_M\mathbf F_{\chi},B}$ vanish when either $B$ is not positive definite or $(b_1,b_2,b_3) \not\in \Z\times \Z \times M\Z$, it follows that $\theta(\mathbf h,\phi_{\mathbf h}) \otimes \underline{\chi} = C \mathfrak R_M\mathbf F_{\chi}$ as desired.

\section{The main result}\label{sec:proof}

In this section we finally state precisely and prove the main result of this paper, relegating the technical local computations to the subsequent sections. Since our approach relies crucially on a decomposition formula of Qiu \cite{Qiu} for an automorphic $\SL_2$-period, we first explain how such period is related to the central $L$-value that we want to compute.

\subsection{Qiu's decomposition formula for the $\SL_2$-period}

Let $\pi$ (resp. $\tau$) be an irreducible cuspidal automorphic representation of $\PGL_2(\A)$ (resp. $\GL_2(\A)$). Fix a non-trivial additive character $\psi$ of $\A/\Q$, and let $\tilde{\pi} \in \mathrm{Wald}_{\psi}(\pi)$ be an irreducible cuspidal automorphic representation of $\widetilde{\SL}_2(\A)$, belonging to the Waldspurger packet of $\pi$ with respect to $\psi$ as explained in Section \ref{theta:PGL2SL2}. Let also $\omega = \omega_{\psi}$ be the Weil representation of $\widetilde{\SL}_2(\A)$ acting on the space $\mathcal S(\A)$ of Bruhat--Schwartz functions (for the one dimensional quadratic space endowed with bilinear form $(x,y) = 2xy$) with respect to $\psi$. Associated with $\tilde{\pi}$, $\tau$ and $\omega$, there is a (global) $\SL_2$-period functional
\[
  \mathcal Q: \tilde{\pi} \otimes \tilde{\pi} \otimes \tau \otimes \tau \otimes \omega \otimes \omega \,\, \longrightarrow \,\, \C
\]
defined by associating to each choice of decomposable vectors $\mathbf h_1, \mathbf h_2 \in \tilde{\pi}$, $\mathbf g_1, \mathbf g_2 \in \tau$, $\pmb{\phi}_1, \pmb{\phi}_2 \in \omega$, the product of integrals 
\[
 \mathcal Q(\mathbf h_1, \mathbf h_2,\mathbf g_1, \mathbf g_2,\pmb{\phi}_1, \pmb{\phi}_2) := 
 \left(\int_{[\SL_2]}\overline{\mathbf h_1(g)}\mathbf g_1(g)\Theta_{\pmb{\phi}_1}(g)dg\right) \cdot 
 \overline{\left(\int_{[\SL_2]}\overline{\mathbf h_2(g)}\mathbf g_2(g)\Theta_{\pmb{\phi}_2}(g)dg\right)}.
\]
It is proved in \cite[Theorem 4.5]{Qiu} that this global period, if it is non-vanishing, decomposes as a product of local $\SL_2$-periods up to certain $L$-values. Namely, one has 
\begin{equation}\label{SL2period-decomposition}
 \mathcal Q(\mathbf h_1, \mathbf h_2,\mathbf g_1, \mathbf g_2,\pmb{\phi}_1, \pmb{\phi}_2) = 
 \frac{1}{4}\frac{L(1/2, \pi \times \mathrm{ad}\tau)}{L(1,\pi,\mathrm{ad})L(1,\tau,\mathrm{ad})} 
 \prod_v \mathcal I_v(\mathbf h_{1,v},\mathbf h_{2,v}, \mathbf g_{1,v}, \mathbf g_{2,v},\pmb{\phi}_{1,v}, \pmb{\phi}_{2,v}),
\end{equation}
where for each rational place $v$, the local period $\mathcal I_v(\mathbf h_{1,v},\mathbf h_{2,v}, \mathbf g_{1,v}, \mathbf g_{2,v},\pmb{\phi}_{1,v}, \pmb{\phi}_{2,v})$ is defined by integrating a product of matrix coefficients, and equals
\[
\frac{L(1,\pi_v,\mathrm{ad})L(1,\tau_v,\mathrm{ad})}{L(1/2,\pi_v\times \mathrm{ad}\tau_v)}
\int_{\SL_2(\Q_v)} \overline{\langle\tilde{\pi}(g_v)\mathbf h_{1,v}, \mathbf h_{2,v}\rangle} \langle\tau(g_v)\mathbf g_{1,v},\mathbf g_{2,v} \rangle \langle \omega_v(g_v)\pmb{\phi}_{1,v},\pmb{\phi}_{2,v} \rangle dg_v.
\]

When $\pi$ (resp. $\tau$) is the automorphic representation of $\PGL_2(\A)$ (resp. $\GL_2(\A)$) associated with the newform $f$ (resp. $g$) as in the introduction, notice that $L(1/2, \pi \times \mathrm{ad}\tau)$ coincides indeed with the special value $\Lambda(f\otimes\mathrm{Ad}(g),k) = \Lambda(f'\otimes\mathrm{Sym}^2(g),2k)$ that we are concerned with.

\begin{rem}\label{rmk:localperiodsQiu}
 Qiu's definition of the local periods $\mathcal I_v$ includes a factor $\zeta_v(2)^{-1}$, and then the decomposition formula in \eqref{SL2period-decomposition} has accordingly a factor $\zeta_{\Q}(2)$ on the right hand side. We have chosen to redefine the local periods by the above expression due to our different choice of local measures $dg_v$.
\end{rem}

The proof of the central value formula in Theorem \ref{mainthm} will rest crucially on this decomposition result, thus it is essential to characterize the conditions under which $\mathcal Q$ does not vanish. In this sense, we shall note the following (see \cite[Proposition 4.1]{Qiu}, and \cite[Theorem 7.1]{GanGurevich}):

\begin{proposition}\label{prop:nonvanishingQ}
 The functional $\mathcal Q$ is non-vanishing on $\tilde{\pi} \otimes \tilde{\pi} \otimes \tau \otimes \tau \otimes \omega \otimes \omega$ if and only if the following conditions hold:
 \begin{itemize}
  \item[i)] $L(1/2,\pi\times \mathrm{ad} \tau) \neq 0$;
  \item[ii)] $\tilde{\pi} = \tilde{\pi}^{\epsilon}$ with $\epsilon_v = \epsilon(1/2,\pi_v\otimes \tau_v \otimes \tau_v^{\vee})$;
  \item[iii)] $\epsilon(1/2,\pi_v\otimes \tau_v \otimes \tau_v^{\vee}) = 1$ when $\pi_v$ is not square-integrable.
 \end{itemize}
\end{proposition}

In condition ii), $\tilde{\pi}^{\epsilon}$ refers to the automorphic representation in $\mathrm{Wald}_{\psi}(\pi)$ labelled by the tuple $\epsilon = (\epsilon_v)_v \in \{\pm 1\}^{|\Sigma(\pi)|}$ as in Theorem \ref{thm:Waldspurger-global}. In particular, notice that fixed the automorphic representation $\pi$ there is {\em only one} automorphic representation for $\widetilde{\SL}_2(\A)$ in the (finite) set $\mathrm{Wald}_{\psi_v}(\pi_v)$ which makes the period $\mathcal Q$ non-vanishing. 
When $\pi_v$ is not square-integrable, recall from Theorem \ref{thm:Waldspurger-local} that the local Waldspurger packet $\mathrm{Wald}_{\psi_v}(\pi_v)$ consists of a single element, labelled $\tilde{\pi}_v$. Therefore, condition iii) is meant to ensure that condition ii) is not failing by obvious reasons. Also, condition i) implies that $\varepsilon(1/2,\pi\times \mathrm{ad} \tau)=1$, and hence  $\prod_v \epsilon_v = \epsilon(1/2,\pi_v\otimes \tau_v \otimes \tau_v^{\vee}) = \epsilon(1/2,\pi)$, ensuring that $\tilde{\pi}^{\epsilon}$ is a global automorphic representation.

\subsection{The central value formula}

Now we put ourselves in the setting of interest in this paper. Let $k\geq 1$ be an odd integer, $N\geq 1$ be an odd square-free integer, and let $f \in S_{2k}^{new}(\Gamma_0(N))$ be a normalized newform of weight $2k$, level $N$, and trivial nebentype. Let also $\chi$ be a Dirichlet character modulo $N$, and $g \in S_{k+1}^{new}(\Gamma_0(N),\chi)$ be a normalized newform of weight $k+1$, level $N$, and nebentype character $\chi$. Write $M \geq 1$ for the conductor of $\chi$, which divides $N$, and assume Hypotheses \eqref{Hyp1} and \eqref{Hyp2}. That is to say, $\chi_{(p)}(-1) = -1$ for all $p\mid M$ and $\varepsilon_p(f) = -1$ for all $p\mid M$. Here, the Dirichlet character $\chi_{(p)}: (\Z/p\Z)^{\times} \to \C^{\times}$ is the $p$-th component of $\chi$.

Let $\pi$ and $\tau$ be the cuspidal automorphic representations of $\PGL_2(\A)$ and $\GL_2(\A)$, respectively, associated with $f$ and $g$. Fix a fundamental discriminant $D \in \mathfrak D(N,M)$ as in Theorem \ref{thm:BM}, and assume further that $L(f,D,k) \neq 0$. Let $\psi$ be the standard additive character of $\A/\Q$, and let 
\[
\tilde{\pi} = \Theta_{\PGL_2\times\widetilde{\SL}_2}(\pi \otimes \chi_D, \psi^{-D}) \in \mathrm{Wald}_{\overline{\psi}}(\pi)                                                                                                                                                                                                                                                                                                                                                                                                                                                                                                                                                                                                                                                                                            \]
be the automorphic representation of $\widetilde{\SL}_2(\A)$ as explained at the end of Section \ref{theta:PGL2SL2}, where $\overline{\psi} = \psi^{-1}$ is the $(-1)$-th twist (equivalently, the inverse, or the complex conjugate) of $\psi$. Recall that, as an element of $\mathrm{Wald}_{\overline{\psi}}(\pi)$, $\tilde{\pi}$ corresponds to the automorphic representation labelled by the tuple $\epsilon = (\epsilon_v)_v \in \{\pm 1\}^{|\Sigma(\pi)|}$ with $\epsilon_{\infty} = -1$ and 
\[
 \epsilon_p = w_p \text{ for all primes } p \mid N/M, \quad \epsilon_p = -w_p \text{ for all primes } p \mid M.
\]
Under the hypothesis \eqref{Hyp2}, we thus have $\epsilon_p = 1$ for all primes $p\mid M$. Recall also that the adelization of the half-integral weight cuspidal forms $h \in S_{k+1/2}^{+,new}(4NM,\chi;f\otimes \chi)$ as in Theorem \ref{thm:BM} belong to this particular element $\tilde{\pi}$ in the Waldspurger packet $\mathrm{Wald}_{\overline{\psi}}(\pi)$. 

In this setting and under our assumptions, the criterion for the non-vanishing of the $\SL_2$-period
\[
 \mathcal Q: \tilde{\pi} \otimes \tilde{\pi} \otimes \tau \otimes \tau \otimes \omega_{\overline{\psi}} \otimes \omega_{\overline{\psi}} \, \longrightarrow \, \C
\]
is reduced to the following statement:

\begin{proposition}\label{prop:nonvanishing}
 With the above choices, the functional $\mathcal Q$ is non-vanishing if and only if
 \[
  \Lambda(f \otimes \mathrm{Ad}(g), k) \neq 0.
 \]
\end{proposition}
\begin{proof}
 In the current setting, condition iii) in Proposition \ref{prop:nonvanishingQ} obviously holds, and we claim that condition ii) holds if and only if hypothesis \eqref{Hyp2} is satisfied. Indeed, we only need to take care of condition ii) in Proposition \ref{prop:nonvanishingQ} at places $v \mid N\infty$. First of all, at the archimedean place $v = \infty$ we have $\epsilon_{\infty} = -1$, and our choice of weights implies that $\varepsilon(1/2, \pi_{\infty} \otimes \tau_{\infty} \otimes \tau_{\infty}^{\vee}) = -1$ as well. Secondly, suppose that $p$ is a prime dividing $N/M$. Then both $\pi_p$ and $\tau_p$ are quadratic twists of the Steinberg representation, and \cite[Proposition 8.6]{Prasad} shows that $\varepsilon(1/2,\pi_p \otimes \tau_p \otimes\tau_p^{\vee}) = \varepsilon(1/2,\pi_p) = w_p$, which agrees with $\epsilon_p$. And finally, suppose that $p$ is a prime factor of $M$. In this case, $\pi_p$ is again a quadratic twist of the Steinberg representation, but $\tau_p$ is now a (ramified) principal series representation. Then, \cite[Proposition 8.4]{Prasad} implies that $\varepsilon(1/2,  \pi_p\otimes\tau_p\otimes\tau_p^{\vee}) = 1$. Under hypothesis \eqref{Hyp2}, this indeed agrees with $\epsilon_p$ as pointed out above, and therefore the statement follows from Proposition \ref{prop:nonvanishingQ}.
\end{proof}

In light of this proposition, the period functional $\mathcal Q$ is therefore non-vanishing if we assume that $\Lambda(f \otimes \mathrm{Ad}(g), k) \neq 0$. When this holds, the strategy for proving our central value formula is now clear. Indeed, let $\mathbf h \in \tilde{\pi}$, $\mathbf g\in \tau$, and $\pmb{\phi} \in \omega_{\overline{\psi}}$ be decomposable vectors, and write $\mathcal Q(\mathbf h, \mathbf g, \pmb{\phi}) := \mathcal Q(\mathbf h,  \mathbf h, \mathbf g, \mathbf g, \pmb{\phi}, \pmb{\phi})$, and $\mathcal I_v(\mathbf h, \mathbf g, \pmb{\phi}) := \mathcal I_v(\mathbf h_v,  \mathbf h_v, \mathbf g_v, \mathbf g_v, \pmb{\phi}_v, \pmb{\phi}_v)$, so that 
\[
 \mathcal I_v(\mathbf h, \mathbf g, \pmb{\phi}) = \frac{L(1,\pi_v,\mathrm{ad})L(1,\tau_v,\mathrm{ad})}{L(1/2,\pi_v\times \mathrm{ad}\tau_v)}
\int_{\SL_2(\Q_v)} \overline{\langle\tilde{\pi}(g_v)\mathbf h_v, \mathbf h_v\rangle} \langle\tau(g_v)\mathbf g_v,\mathbf g_v \rangle \langle \omega_{\overline{\psi}_v}(g_v)\pmb{\phi}_v,\pmb{\phi}_v \rangle dg_v.
\]
Regularizing these local periods by setting 
\[
 \mathcal I_v^{\sharp}  (\mathbf h, \mathbf g, \pmb{\phi}) := 
\frac{\mathcal I_v(\mathbf h, \mathbf g, \pmb{\phi})}{\langle \mathbf h_v, \mathbf h_v \rangle\langle \mathbf g_v, \mathbf g_v \rangle \langle \pmb{\phi}_v, \pmb{\phi}_v \rangle} = 
\frac{\mathcal I_v(\mathbf h, \mathbf g, \pmb{\phi})}{||\mathbf h_v||^2 ||\mathbf g_v||^2||\pmb{\phi}_v||^2},
\]
we will write 
\begin{equation}\label{Iv-alphav}
  \mathcal I_v^{\sharp}  (\mathbf h, \mathbf g, \pmb{\phi}) = \frac{L(1,\pi_v,\mathrm{ad})L(1,\tau_v,\mathrm{ad})}{L(1/2,\pi_v\times \mathrm{ad}\tau_v)} \alpha_v^{\sharp} (\mathbf h, \mathbf g, \pmb{\phi}),
\end{equation}
where 
\begin{equation}\label{alphav}
 \alpha_v^{\sharp} (\mathbf h, \mathbf g, \pmb{\phi}) := \int_{\SL_2(\Q_v)} \frac{\overline{\langle\tilde{\pi}(g_v)\mathbf h_v, \mathbf h_v\rangle}}{||\mathbf h_v||^2} \frac{\langle\tau(g_v)\mathbf g_v,\mathbf g_v \rangle}{||\mathbf g_v||^2} \frac{\langle \omega_{\overline{\psi}_v}(g_v)\pmb{\phi}_v,\pmb{\phi}_v \rangle}{||\pmb{\phi}_v||^2} dg_v.
\end{equation}
Then, if one can choose $\mathbf h$, $\mathbf g$ and $\pmb{\phi}$ so that $\mathcal Q(\mathbf h, \mathbf g, \pmb{\phi})$ does not vanish, then one might rewrite \eqref{SL2period-decomposition} as 
\begin{equation}\label{centralvalueformula-Q}
 \Lambda(f \otimes \mathrm{Ad}(g), k) = \frac{4 L(1,\pi,\mathrm{ad})L(1,\tau,\mathrm{ad})}{\langle \mathbf h, \mathbf h \rangle\langle \mathbf g, \mathbf g\rangle\langle \pmb{\phi}, \pmb{\phi} \rangle} \left( \prod_v \mathcal I_v^{\sharp}(\mathbf h, \mathbf g, \pmb{\phi})^{-1}\right) \mathcal Q(\mathbf h, \mathbf g, \pmb{\phi}).
\end{equation}
By virtue of this identity, our proof of Theorem \ref{mainthm} consists essentially in choosing an appropriate {\em test vector} $\mathbf h \otimes \mathbf g \otimes \pmb{\phi} \in \tilde{\pi} \otimes \tau \otimes \omega_{\overline{\psi}}$ for which the right hand side of \eqref{centralvalueformula-Q} does not vanish. Computing the regularized local periods $\mathcal I_v^{\sharp}(\mathbf h, \mathbf g, \pmb{\phi})$ and translating the global period $\mathcal Q(\mathbf h, \mathbf g, \pmb{\phi})$ into classical terms, leads eventually to the desired explicit central value formula. 

\begin{theorem}\label{MainResult}
Let $k, N\geq 1$ be odd integers. Let $f \in S_{2k}^{new}(\Gamma_0(N))$ and $g \in S_{k+1}^{new}(\Gamma_0(N),\chi)$ be normalized newforms, and assume \eqref{SF}, \eqref{Hyp1}, and \eqref{Hyp2}. If $h \in S_{k+1/2}^{+,new}(4NM,\chi;f\otimes \chi)$ and $F_{\chi}$ denote a Shimura lift of $f$ as above and its Saito--Kurokawa lift, then 
\begin{equation}\label{mainformula}
 \Lambda(f \otimes \mathrm{Ad}(g), k) = 2^{k+1-\nu(M)} C(N,M,\chi) \frac{\langle f,f \rangle}{\langle h,h\rangle} 
 \frac{|\langle (\mathrm{id}\otimes U_M) F_{\chi |\mathcal H \times \mathcal H}, g \times g\rangle|^2}{\langle g,g\rangle ^2},
\end{equation}
where $\nu(M)$ denotes the number of prime divisors of $M$, and
\[
 C(N,M,\chi) = |\chi(2)|^{-2} M^{3-k}N^{-1} \prod_{p\mid N}(p+1)^2\prod_{p\mid M} (p+1).
\]
\end{theorem}
\begin{proof}
 Suppose first that $\Lambda(f \otimes \mathrm{Ad}(g), k)\neq 0$. By Proposition \ref{prop:nonvanishing}, the functional period $\mathcal Q$ does not vanish, and so we might use \eqref{centralvalueformula-Q} for a suitable choice of {\em test vector}. Keeping the notations as above, consider the pure tensor  
 \[
  \mathbf h \otimes \breve{\mathbf g} \otimes \pmb{\phi} \in \tilde{\pi} \otimes \tau \otimes \omega_{\overline{\psi}},
 \]
 where $\mathbf h \in \tilde{\mathcal A}_{k+1/2}(4NM,\underline{\chi}_0)$ is the adelization of $h$, $\breve{\mathbf g} = V_M \mathbf g^{\sharp} \in \tau$ is the automorphic cusp form obtained from the adelization $\mathbf g\in \tau$ of $g$ as in Section \ref{theta:GL2GSO4}, and $\pmb{\phi} \in \omega_{\overline{\psi}}$ is the one-dimensional Bruhat--Schwartz function determined by requiring that $\pmb{\phi}_q  = \mathbf 1_{\Z_q}$ at all finite primes $p$, and $\pmb{\phi}_{\infty}(x) = e^{-2\pi x^2}$ for all $x \in \R$. 
 
 With this choice, we have 
 \[
  \langle \mathbf h, \mathbf h\rangle = 2^{-1}\zeta_{\Q}(2)^{-1}\langle h, h \rangle, \quad 
  \langle \breve{\mathbf g}, \breve{\mathbf g}\rangle =  \langle \mathbf g, \mathbf g\rangle = \zeta_{\Q}(2)^{-1}\langle g, g \rangle, \quad 
  \langle \pmb{\phi}, \pmb{\phi} \rangle = ||\pmb{\phi}_{\infty}||^2 = 2^{-1}.
 \]
 And by \cite[Theorem 5.15]{Hida-MFGalCoh}, \cite[\S 3.2.1]{Watson}, it is known that
 \[
  L(1,\pi,\mathrm{ad}) = 2^{2k}N^{-1}[\SL_2(\Z):\Gamma_0(N)]\langle f,f\rangle, \quad L(1,\tau,\mathrm{ad}) = 2^{k+1}N^{-1}[\SL_2(\Z):\Gamma_0(N)]\langle g,g\rangle.
 \] 
 Therefore, the first factor on the right hand side of \eqref{centralvalueformula-Q} for our choice of test vector reads 
 \[
  \frac{4 L(1,\pi,\mathrm{ad})L(1,\tau,\mathrm{ad})}{\langle \mathbf h, \mathbf h \rangle\langle \mathbf g, \mathbf g\rangle\langle \pmb{\phi}, \pmb{\phi} \rangle} = 
  \frac{2^{3k+5} \zeta_{\Q}(2)^2 [\SL_2(\Z):\Gamma_0(N)]^2 \langle f,f\rangle}{N^2 \langle h, h \rangle}.
 \]
 
 Because of our choice of $\mathbf h$, $\breve{\mathbf g}$, and $\pmb{\phi}$, it follows from \cite[Lemma 4.4]{Qiu} that $\mathcal I_q^{\sharp}(\mathbf h, \breve{\mathbf g}, \pmb{\phi}) = 1$ for all finite primes $q \nmid 2N$ (notice that our choice for the local measure $dg_q$ on $\SL_2(\Q_q)$ is different to Qiu's choice, but we have modified accordingly the definition of the local periods $\mathcal I_v$ after \eqref{SL2period-decomposition}, cf. Remark \ref{rmk:localperiodsQiu}). In the next sections we will compute the regularized local periods at the remaining places: from Propositions \ref{prop:IN/M}, \ref{prop:IM}, \ref{prop:I2}, and \ref{prop:Ireal}, we have:
 \[
  \mathcal I_v^{\sharp}(\mathbf h, \breve{\mathbf g}, \pmb{\phi})^{-1} = 
  \begin{cases}
   p & \text{if } v = p, \, p\mid N/M, \\
   \frac{p(p+1)}{2}& \text{if } v = p, \, p\mid M, \\
   1 & \text{if } v = 2 \text{ or } v = \infty.\\
  \end{cases}
 \]
 Therefore, in \eqref{centralvalueformula-Q} we have 
 \begin{align*}
  \mathcal I & := \prod_v \mathcal I_v^{\sharp}(\mathbf h,\breve{\mathbf g},\pmb{\phi})^{-1} = 
  \prod_{p \mid 2N} \mathcal I_v^{\sharp}(\mathbf h,\breve{\mathbf g},\pmb{\phi})^{-1} = 
  N \prod_{p\mid M} \frac{(p+1)}{2} = 2^{-\nu(M)} N \prod_{p\mid M} (p+1).
 \end{align*}

 Finally, it remains to deal with the global $\SL_2$-period $\mathcal Q(\mathbf h,\breve{\mathbf g},\pmb{\phi})$. 
 Recall from Sections \ref{theta:GL2GSO4} and \ref{theta:SL2PGSP2} that we have associated Bruhat--Schwartz functions $\phi_{\breve{\mathbf g}} \in \mathcal S(V_4(\A))$ and $\phi_{\mathbf h} = \pmb{\phi} \otimes \phi_{\breve{\mathbf g}} \in \mathcal S(V_5(\A))$ to $\breve{\mathbf g}$ and $\mathbf h$. By virtue of Proposition \ref{ThetaIdentity1}, we have 
 \[
  \theta(\mathbf Y_M \mathbf G, \phi_{\breve{\mathbf g}}) = C_1 {\breve{\mathbf g}},
 \]
 where $C_1 = 2^{k-1} M^{-1} [\SL_2(\Z):\Gamma_0(N)]^{-1}\zeta_{\Q}(2)^{-2}\langle g,g \rangle$, and $\mathbf Y_M \mathbf G \in \Upsilon = \Theta(\tau)$ is an automorphic form for $\GSO_{2,2}$ whose restriction to $\GL_2 \times \GL_2$ coincides with $\mathbf g \otimes \mathbf V_M \mathbf g \in \tau \boxtimes \tau$. Write $\Pi$ for the automorphic representation of $\PGSp_2(\A)$ associated with $\theta(\mathbf h,\phi_{\mathbf h})$. It follows from the proof of \cite[Theorem 5.3]{Qiu} that 
 \[
  \mathcal Q(\mathbf h, \breve{\mathbf g}, \pmb{\phi}) = C_1^{-2} \mathcal P(\theta(\mathbf h,\phi_{\mathbf h}),\mathbf Y_M \mathbf G),
 \]
 where on the right hand side $\mathcal P: \Pi \otimes \Pi \otimes \Upsilon \otimes \Upsilon \to \C$ is an $\SO(V_4)$-period defined by associating to any choice of decomposable vectors $\mathbf F_1,\mathbf F_2 \in \Pi$, $\mathbf G_1, \mathbf G_2 \in \Upsilon$ the value 
 \[
  \mathcal P(\mathbf F_1,\mathbf F_2, \mathbf G_1, \mathbf G_2) := 
  \left(\int_{[\SO(V_4)]} \mathbf F_1(h)\overline{\mathbf G_1(h)} dh \right) \left(\int_{[\SO(V_4)]} \overline{\mathbf F_2(h)}\mathbf G_2(h) dh\right),
 \]
 and we have abbreviated $\mathcal P(\theta(\mathbf h,\phi_{\mathbf h}),\mathbf Y_M \mathbf G) = \mathcal P(\theta(\mathbf h,\phi_{\mathbf h}),\theta(\mathbf h,\phi_{\mathbf h}),\mathbf Y_M \mathbf G,\mathbf Y_M \mathbf G)$ (cf. \cite[Section 5]{Qiu}). Let $\Pi_{\chi} = \Pi \otimes \underline{\chi}$ be the automorphic representation of $\GSp_2(\A)$ associated with $\theta(\mathbf h,\phi_{\mathbf h})\otimes \underline{\chi}$. Since the similitude morphism is trivial on $\mathrm{SO}(V_4)$, if one defines a period functional 
 \[
  \mathcal P_{\chi}: \Pi_{\chi} \otimes \Pi_{\chi} \otimes \Upsilon \otimes \Upsilon \, \longrightarrow \, \C
 \]
 by the same recipe as for $\mathcal P$, then one has $\mathcal P(\mathbf F_1,\mathbf F_2, \mathbf G_1, \mathbf G_2) = \mathcal P_{\chi}(\mathbf F_1 \otimes \underline{\chi},\mathbf F_2\otimes \underline{\chi}, \mathbf G_1, \mathbf G_2)$ for all decomposable vectors $\mathbf F_1,\mathbf F_2 \in \Pi$, $\mathbf G_1, \mathbf G_2 \in \Upsilon$. In particular, 
 \[
  \mathcal Q(\mathbf h, \breve{\mathbf g}, \pmb{\phi}) = C_1^{-2} \mathcal P_{\chi}(\theta(\mathbf h,\phi_{\mathbf h})\otimes \underline{\chi},\mathbf Y_M \mathbf G),
 \]
 and thanks to Proposition \ref{ThetaIdentity2} we deduce that 
 \[
  \mathcal Q(\mathbf h, \breve{\mathbf g}, \pmb{\phi}) = C_1^{-2}|C_2|^2 \mathcal P_{\chi}(\mathfrak R_M\mathbf F_{\chi},\mathbf Y_M \mathbf G),
 \]
 where $C_2 = 2^{-2}\chi(2)^{-1} M^{-1} [\SL_2(\Z):\Gamma_0(N)]^{-1}\zeta_{\Q}(2)^{-1}$. Now, when restricted to $\SO(V_4)$, $\mathfrak R_M\mathbf F_{\chi}$ coincides with the adelization of $\mathfrak R_M F_{\chi|\mathcal H \times \mathcal H}$, which in turn equals $(\mathrm{id}\otimes V_MU_M)F_{\chi|\mathcal H \times \mathcal H}$. Besides, $\mathbf Y_M\mathbf G$ restricted to $\GL_2 \times \GL_2$ is the adelization of $M^{(k+1)/2-1} g \times V_Mg$. Therefore, we have 
 \[
  \mathcal P_{\chi}(\mathfrak R_M\mathbf F_{\chi},\mathbf Y_M \mathbf G) = 
  C_3^2M^{k-1}|\langle (\mathrm{id}\otimes V_MU_M)F_{\chi|\mathcal H \times \mathcal H}, (\mathrm{id}\otimes V_M) g\times g\rangle|^2,
 \]
 where $C_3 = 2^{-1}\zeta_{\Q}(2)^{-2}$ (cf. \cite[Section 9]{IchinoIkeda-periods}). Furthermore, one can easily check that 
 \[
  \langle (\mathrm{id}\otimes V_MU_M)F_{\chi|\mathcal H \times \mathcal H}, (\mathrm{id}\otimes V_M) g\times g\rangle = M^{2-k} \langle (\mathrm{id}\otimes U_M)F_{\chi|\mathcal H \times \mathcal H}, g\times g\rangle,
 \]
 where now the right hand side makes sense as a Petersson product with respect to $\Gamma_0(NM) \times \Gamma_0(NM)$. Therefore, 
 \[
  \mathcal P_{\chi}(\mathfrak R_M\mathbf F_{\chi},\mathbf Y_M \mathbf G) = 
  C_3^2M^{3-k}|\langle (\mathrm{id}\otimes U_M)F_{\chi|\mathcal H \times \mathcal H}, g\times g\rangle|^2.
 \]
 Altogether, we have 
 \[
  \Lambda(f \otimes \mathrm{Ad}(g), k) = \mathcal C \cdot \mathcal I \cdot\frac{\langle f,f\rangle}{\langle h, h \rangle} |\langle (\mathrm{id}\otimes U_M) F_{\chi |\mathcal H \times \mathcal H}, g \times g\rangle|^2,
 \]
 where we put $\mathcal C = 2^{3k+5} \zeta_{\Q}(2)^2 [\SL_2(\Z):\Gamma_0(N)]^2 M^{3-k} N^{-2} |C_2|^2C_3^2C_1^{-2}$. 
 Plugging the values of the constants, and of the product of local periods $\mathcal I$, one concludes that
 \[
  \Lambda(f \otimes \mathrm{Ad}(g),k) = 2^{k+1-\nu(M)} C(N,M,\chi) \frac{\langle f,f\rangle}{\langle h,h\rangle} \frac{|\langle (\mathrm{id}\otimes U_M) F_{\chi |\mathcal H \times \mathcal H}, g \times g\rangle|^2}{\langle g,g\rangle^2},
 \]
 where $C(N,M,\chi)$ is as in the statement.
 
 Finally, when $\Lambda(f \otimes \mathrm{Ad}(g), k) = 0$ the global functional $\mathcal Q$ vanishes by Proposition \ref{prop:nonvanishing}. In particular, $\mathcal Q(\mathbf h, \breve{\mathbf g},\pmb{\phi}) = 0$, and this implies that $\langle (\mathrm{id}\otimes U_M) F_{\chi |\mathcal H \times \mathcal H}, g \times g\rangle = 0$. Thus we see that the formula in the statement holds trivially in this case.
\end{proof}

An immediate application of Theorem \ref{MainResult} is the following algebraicity result, predicted by Deligne's conjecture, in which $c^+(f)$ denotes the period associated with $f$ by Shimura as in \cite{ShimuraPeriods}.

\begin{corollary}
 Let $f$ and $g$ be as in Theorem \ref{MainResult}. If $\sigma \in \mathrm{Aut}(\C)$, then 
 \[
  \left(\frac{\Lambda(f \otimes \mathrm{Ad}(g), k)}{\langle g,g\rangle^2 c^+(f)}\right)^{\sigma} = 
  \frac{\Lambda(f^{\sigma}\otimes \mathrm{Ad}(g^{\sigma}), k)}{\langle g^{\sigma},g^{\sigma}\rangle^2 c^+(f^{\sigma})}.
 \]
 In particular, if $\Q(f,g)$ denotes the number field generated by the Fourier coefficients of $f$ and $g$, then 
 \[
  \Lambda(f \otimes \mathrm{Ad}(g), k)^{\mathrm{alg}} := \frac{\Lambda(f \otimes \mathrm{Ad}(g),k)}{\langle g,g\rangle^2 c^+(f)} \in \Q(f,g).
 \]
\end{corollary}
\begin{proof}
 First of all, we may assume that the Fourier coefficients of $h$, and hence of $F_{\chi}$, belong to the number field $\Q(f,\chi)$ generated by the Fourier coefficients of $f$ together with the values of $\chi$. This is either a totally real field or a CM field.
 
 Choose a fundamental discriminant $D < 0$, $D \in \mathfrak D(N,M)$, with $L(f,D,k)\neq 0$. By Theorem \ref{thm:BM} (see also Remark \ref{rmk:BMcompleted}), if $c_h(|D|)$ denotes the $|D|$-th Fourier coefficient of $h$, then 
 \[
  \frac{\langle f, f\rangle}{\langle h,h\rangle} = 2^{k-1+\nu(N)} \left(\prod_{p\mid M}\frac{p}{p+1}\right)\frac{|D|^{k-1/2}\Lambda(f,D,k)}{|c_h(|D|)|^2}.
 \]
 Combined with the central value formula in Theorem \ref{MainResult}, 
 \[
  \Lambda(f \otimes \mathrm{Ad}(g), k)^{\mathrm{alg}} =  2^{2k+\nu(N/M)} \tilde{C}(N,M,\chi) \cdot \frac{|D|^{k-1/2}\Lambda(f,D,k)}{c^+(f)|c_h(|D|)|^2} \cdot \frac{|\langle (\mathrm{id}\otimes U_M)F_{\chi|\mathcal H \times \mathcal H}, g \times g \rangle|^2}{\langle g,g \rangle^4},
 \]
 where $\tilde C(N,M,\chi) = C(N,M,\chi) \prod_{p\mid M} \frac{p}{p+1}$. Now fix $\sigma \in \mathrm{Aut}(\C)$. One has $c_h(|D|)^{\sigma} = c_{h^{\sigma}}(|D|)$, and hence by the properties of the period $c^+(f)$ we have 
 \[
  \left(\frac{|D|^{k-1/2}\Lambda(f,D,k)}{c^+(f)c_h(|D|)^2}\right)^{\sigma} =  \frac{|D|^{k-1/2}\Lambda(f^{\sigma},D,k)}{c^+(f^{\sigma})c_{h^{\sigma}}(|D|)^2}.
 \]
 Besides, $F_{\chi}^{\sigma}$ is the Saito--Kurokawa lift of $h^{\sigma}$ and  
 \[
  \left(\frac{|\langle (\mathrm{id}\otimes U_M)F_{\chi|\mathcal H \times \mathcal H}, g \times g \rangle|^2}{\langle g,g \rangle^4}\right)^{\sigma} = 
  \frac{|\langle (\mathrm{id}\otimes U_M)F_{\chi|\mathcal H \times \mathcal H}^{\sigma}, g^{\sigma} \times g^{\sigma} \rangle|^2}{\langle g^{\sigma},g^{\sigma} \rangle^4},
 \]
 so the statement follows.
\end{proof}

\section{Computation of local periods at primes $p \mid N/M$}\label{sec:periodsN/M}

This section is devoted to compute the regularized local periods $\mathcal I_p^{\sharp}(\mathbf h, \breve{\mathbf g}, \pmb{\phi})$ at primes $p\mid N/M$. First of all, we will describe the local components $\breve{\mathbf g}_p$ and $\mathbf h_p$ (up to scalar multiple), according to the local types of the representations $\tau$ and $\tilde{\pi}$ at such primes. Then we will compute the matrix coefficients  $\langle \tau_p(g)\breve{\mathbf g}_p, \breve{\mathbf g}_p\rangle$ and $\langle \tilde{\pi}_p(g)\mathbf h_p, \mathbf h_p\rangle$, for $g \in \SL_2(\Q_p)$, which together with the Weil parings $\langle \omega_{\overline{\psi}_p}(g)\pmb{\phi}_p, \pmb{\phi}_p\rangle$ will lead to the determination of $\mathcal I_p^{\sharp}(\mathbf h, \breve{\mathbf g}, \pmb{\phi})$. Thus let us fix through all this section a prime factor $p$ of $N/M$, and let $\psi_p$ denote the $p$-th component of the standard additive character $\psi: \A/\Q \to \C^{\times}$.

For the $\GL_2$ case, $\tau_p$ a twist of the Steinberg representation $\mathrm{St}_p$ by some unramified quadratic character $\xi: \Q_p^{\times} \to \C^{\times}$. That is to say, it is the unique irreducible subrepresentation of the induced representation $\pi(\xi|\cdot|_p^{1/2},\xi|\cdot|_p^{-1/2})$. The representation $\pi(\xi|\cdot|_p^{1/2},\xi|\cdot|_p^{-1/2})$ is realized as the space of all locally constant functions $\varphi: \GL_2(\Q_p) \to \C$ satisfying the transformation property 
\begin{equation}\label{rule:St}
 \varphi\left(\left(\begin{smallmatrix}a & b \\ 0 & d\end{smallmatrix}\right)x\right) = \xi(ad) \left|\frac{a}{d}\right|_p\varphi(x) \quad \text{for all } a, d \in \Q_p^{\times}, b \in \Q_p, x \in \GL_2(\Q_p),
\end{equation}
and the subspace corresponding to $\tau_p$ is that of such functions which, in addition, satisfy a certain vanishing condition.

To describe $\breve{\mathbf g}_p \in \tau_p$, notice first of all that $\breve{\mathbf g}_p = \mathbf g_p$ because $p\nmid M$. Therefore, $\breve{\mathbf g}_p = \mathbf g_p$ belongs to the space $\tau_p^{K_0}$ of $K_0$-fixed vectors, where we abbreviate 
\[
 K_0 = K_0(p) := \left\lbrace \left(\begin{array}{cc} a & b \\ c & d\end{array}\right) \in \GL_2(\Z_p): c \equiv 0 \pmod p\right\rbrace.
\]
The space $\tau_p^{K_0}$ is well-known to be one-dimensional. When replacing $\mathbf g_p$ by a scalar multiple, the ratio $\langle \tau_p(g)\breve{\mathbf g}_p, \breve{\mathbf g}_p\rangle/||\mathbf g_p||^2$ remains invariant, so we can freely choose a new vector in $\tau_p^{K_0}$ and suppose that $\mathbf g_p$ coincides with such choice. Following \cite[Section 2.1]{SchmidtRemarksGL2}, we can choose $\mathbf g_p: \GL_2(\Q_p) \to \C$ in the induced model to be the local vector characterized by the property that 
\begin{equation}\label{newvector:St}
 \mathbf g_{p| \GL_2(\Z_p)} = \mathbf 1_{K_0} - \frac{1}{p}\mathbf 1_{K_0 w K_0},
\end{equation}
where $w = \left(\begin{smallmatrix} 0 & 1 \\ 1 & 0\end{smallmatrix}\right)$ and $\mathbf 1_X$ denotes the characteristic function of $X \subseteq \GL_2(\Z_p)$. Thanks to the Iwasawa decomposition for $\GL_2(\Q_p)$, this together with the rule \eqref{rule:St} determines uniquely $\mathbf g_p$. Notice that being $\xi: \Q_p^{\times} \to \C^{\times}$ {\em unramified} means that $\xi(a) = 1$ for all $a \in \Z_p^{\times}$, hence $\xi$ is completely determined by the value $\xi(p)$. Since $\xi$ is {\em quadratic}, we have $\xi(p) = \pm 1$, and it is well-known that $\xi(p) = -1$ (resp. $+1$) if and only if the local root number (or the Atkin--Lehner eigenvalue) of $\tau_p$ is $1$ (resp. $-1$). 

Now we move to the case of $\widetilde{\SL}_2$. Recall from Section \ref{sec:BMexplained} that $\tilde{\pi}_p$ is the special representation $\tilde{\sigma}^{\delta}(\overline{\psi}_p^D)$ of $\widetilde{\SL}_2(\Q_p)$, where $\delta \in \Z_p^{\times}$ is any non-square unit, and $D \in \Z_p^{\times}$ satisfies $(\frac{D}{p}) = w_p = \varepsilon(1/2,\pi_p)$. In order to lighten the notation, we will write from now on $\psi := \overline{\psi}_p^D$. As explained in Section \ref{sec:BMexplained}, the representation space of $\tilde{\pi}_p$ is then the space of locally constant functions $\tilde{\varphi}: \widetilde{\SL}_2(\Q_p) \to \C$ such that 
\begin{equation}\label{PS}
 \tilde{\varphi} \left(\left[\left(\begin{smallmatrix} a & \ast \\ 0 & a^{-1}\end{smallmatrix}\right), \epsilon \right] g \right) = \epsilon \chi_{\psi}(a)\mu(a)|a|_p\tilde{\varphi}(g) = \epsilon \chi_{\psi}(a)\chi_{\delta}(a)|a|^{3/2}_p\tilde{\varphi}(g)
\end{equation}
for all $g\in \widetilde{\SL}_2(\Q_p)$ and $a \in \Q_p^{\times}$, together with a certain vanishing condition that we will not need here. Notice that $\chi_{\delta}$ is the unique non-trivial quadratic character of $\Q_p^{\times}$ hence its restriction to $\Z_p^{\times}$ is trivial, and $\chi_{\delta}(p)= -1$.

In order to describe $\mathbf h_p$, recall first that $\SL_2(\Z_p)$ embeds into $\widetilde{\SL}_2(\Q_p)$ by $g \mapsto [g,s_p(g)]$, and let $\widetilde{\Gamma}_0$ denote the image of $\Gamma_0$ in $\widetilde{\SL}_2(\Z_p)$ under this embedding. Then let $\mathbf 1_{\widetilde{\SL}_2(\Z_p)}$ be the (genuine) function on $\widetilde{\SL}_2(\Q_p)$ which sends $[g,\epsilon]$ to $0$ if $g \not\in \SL_2(\Z_p)$ and to $\epsilon s_p(g)$ otherwise (thus it takes value $1$ if $[g,\epsilon]$ lies in the image of $\SL_2(\Z_p)$, and $-1$ if $g \in \SL_2(\Z_p)$ but $s_p(g) = -\epsilon$). Similarly, let $\mathbf 1_{\widetilde{\Gamma}_0}$ be the function on $\widetilde{\SL}_2(\Q_p)$ which sends $[g,\epsilon]$ to $0$ if $g \not\in \Gamma_0$ and to $\mathbf 1_{\widetilde{\SL}_2(\Z_p)}([g,\epsilon])$ otherwise. With these notations, the following is proved in \cite[Lemma 8.3]{BaruchMao}.

\begin{lemma}\label{sl2testvector_special}
The space of $\widetilde{\Gamma}_0$-fixed vectors in $\tilde{\pi}_p$ is one-dimensional, and a new vector generating such space is given by the function $\tilde{\varphi}_p: \widetilde{\SL}_2(\Q_p) \to \C$ whose restriction to $\widetilde{\SL}_2(\Z_p)$ equals $\mathbf 1_{\widetilde{\SL}_2(\Z_p)} - (p+1)\mathbf 1_{\widetilde{\Gamma}_0}$.
\end{lemma}

The condition in the statement determines completely $\tilde{\varphi}_p$, thanks to the Iwasawa decomposition of $\widetilde{\SL}_2(\Q_p)$ (which is lifted from that of $\SL_2(\Q_p)$). The $p$-th component $\mathbf h_p$ of the adelization $\mathbf h$ of the half-integral weight modular form $h$ is therefore a scalar multiple of $\tilde{\varphi}_p$. Since $\langle \tilde{\pi}_p(g)\mathbf h_p, \mathbf h_p\rangle/||\mathbf h_p||^2$ is invariant under replacing $\mathbf h_p$ by a scalar multiple of it, we may assume that $\mathbf h_p = \tilde{\varphi}_p$. 

\begin{lemma}\label{lemma:norm-varphip}
 For the above choice of $\mathbf h_p$, we have $||\mathbf h_p||^2 = p^{-1}(p^2-1)$.
\end{lemma}
\begin{proof}
 This is a straightforward computation. Indeed, when restricted to $\SL_2(\Z_p)$ and $\Gamma_0$ the functions $\mathbf 1_{\widetilde{\SL}_2(\Z_p)}$ and $\mathbf 1_{\widetilde{\Gamma}_0}$ become the characteristic functions of $\SL_2(\Z_p)$ and $\Gamma_0$, respectively. Therefore,
 \[  
 ||\mathbf h_p||^2 = \int_{\SL_2(\Z_p)} \mathbf h_p(h)\overline{\mathbf h_p(h)} dh = \int_{\SL_2(\Z_p) \setminus \Gamma_0} dh + p^2 \int_{\Gamma_0} dh = \vol(\SL_2(\Z_p)-\Gamma_0)) + p^2\vol(\Gamma_0).
 \]
 Since our measure is normalized so that $\SL_2(\Z_p)$ has volume $\zeta_p(2)^{-1} = (1-p^{-1})(1+p^{-1})$, and $\Gamma_0$ has index $p+1$ in $\SL_2(\Z_p)$, it follows easily that 
 $\vol(\Gamma_0) = p^{-1}(1-p^{-1})$ and $ ||\mathbf h_p||^2 = p^{-1}(p^2-1)$.
\end{proof}

Finally, for simplicity we will write $\omega_p = \omega_{\overline{\psi}_p}$ for the local Weil representation of $\widetilde{\SL}_2(\Q_p)$ acting on the space of Bruhat--Schwartz functions $\mathcal S(\Q_p)$, with respect to the character $\overline{\psi}_p = \psi_p^{-1}$. Here, $\Q_p$ is to be regarded as the one-dimensional quadratic space endowed with the bilinear form $(x,y) = 2xy$. In our choice of test vector, the $p$-th component of $\pmb{\phi}$ is $\pmb{\phi}_p = \mathbf 1_{\Z_p}$, the characteristic function of $\Z_p$. It is easily checked that $\pmb{\phi}_p$ is invariant under the action of $\SL_2(\Z_p)$. In particular, it is also invariant under $\Gamma_0$.

Having described the local components $\mathbf g_p\in \tau_p$, $\mathbf h_p \in \tilde{\pi}_p$, and $\pmb{\phi}_p \in \mathcal S(\Q_p)$, observe that the three of them are invariant under the action of $\Gamma_0$. It thus follows that for any $g \in \SL_2(\Q_p)$ the values 
\[
 \Phi_{\breve{\mathbf g}_p}(g) := \frac{\langle \tau_p(g)\breve{\mathbf g}_p, \breve{\mathbf g}_p\rangle}{||\breve{\mathbf g}_p||^2},  \quad \Phi_{\mathbf h_p}(g) := \frac{\langle \tilde{\pi}_p(g)\mathbf h_p, \mathbf h_p\rangle}{||\mathbf h_p||^2}, \quad \Phi_{\pmb{\phi}_p}(g) := \frac{\langle \omega_p(g)\pmb{\phi}_p, \pmb{\phi}_p\rangle}{||\pmb{\phi}_p||^2}
\]
depend only on the double coset $\Gamma_0 g \Gamma_0$. Thus we only need to compute these values for $g$ varying in a set of representatives for the double cosets for $\Gamma_0$ in $\SL_2(\Q_p)$. Define elements $\alpha, \beta \in \SL_2(\Q_p)$ by 
\[
 \alpha := \left( \begin{array}{cc} p & 0\\ 0& p^{-1}\end{array}\right), \quad 
 \beta := s\alpha = \left( \begin{array}{cc} 0 & p^{-1}\\ -p & 0\end{array}\right),
\]
where $s=\left(\begin{smallmatrix} 0 & 1\\ -1& 0\end{smallmatrix}\right)$. Then the Cartan decomposition for $\SL_2(\Q_p)$ relative to the maximal compact open subgroup  $\SL_2(\Z_p)$ gives 
\[
 G = \bigsqcup_{n\geq 0} \SL_2(\Z_p) \alpha_n \SL_2(\Z_p) = \bigsqcup_{n\geq 0} \SL_2(\Z_p) \alpha_{-n} \SL_2(\Z_p),
\]
where we put $\alpha_n := \alpha^n$ for any integer $n$. Combining this with the so-called Bruhat decomposition for $\SL_2$ over the residue field $\mathbb F_p$, one obtains also a decomposition for $\SL_2(\Q_p)$ in terms of $\Gamma_0$:
\begin{equation}\label{SL2-decomp-Gamma0}
 \SL_2(\Q_p) = \bigsqcup_{n\in \Z} \Gamma_0 \alpha_n \Gamma_0 \quad \sqcup \quad  \bigsqcup_{m\in \Z} \Gamma_0 \beta_m \Gamma_0.
\end{equation}
where $\beta_m := s\alpha_m = s\alpha^m$. By virtue of \eqref{SL2-decomp-Gamma0} and our above observation, it will be enough to compute the values $\Phi_{\mathbf g_p}(g)$, $\Phi_{\mathbf h_p}(g)$, and $\Phi_{\pmb{\phi}_p}(g)$ for  $g \in \{\alpha_n, \beta_m: n,m \in \Z\}$. We will need the volumes of the double cosets $\Gamma_0 \alpha_n \Gamma_0$ and $\Gamma_0 \beta_m \Gamma_0$, hence we collect them in the following lemma for ease of reference.

\begin{lemma}\label{lemma:volumesGamma0}
 Keep the same notation as above.
 \begin{itemize}
  \item[i)] $\mathrm{vol}(\Gamma_0\alpha_0\Gamma_0) = \mathrm{vol}(\Gamma_0) = p^{-1}(1-p^{-1})$, and for $n\neq 0$, one has 
  \[
   \mathrm{vol}(\Gamma_0\alpha_n\Gamma_0) = 
   \begin{cases}
     p^{2n-1}(1-p^{-1}) & \text{if } n > 0,\\
     p^{-2n-1}(1-p^{-1}) & \text{if } n < 0. 
   \end{cases}
  \]
 \item[ii)] $\mathrm{vol}(\Gamma_0\beta_0\Gamma_0) = \mathrm{vol}(\Gamma_0s\Gamma_0) = (1-p^{-1})$, and for $m\neq 0$, one has 
  \[
   \mathrm{vol}(\Gamma_0\beta_m\Gamma_0) =  
   \begin{cases}
    p^{2m-2}(1-p^{-1}) & \text{if } m > 0,\\
    p^{-2m}(1-p^{-1}) & \text{if } m < 0. 
   \end{cases}
  \]
 \end{itemize}
\end{lemma}
\begin{proof}
 For each $g \in \{\alpha_n,\beta_m: n,m\in \Z\}$, one writes the double coset $\Gamma_0 g \Gamma_0$ as a disjoint union of finitely many cosets $\Gamma_0 g_i$. Then the volume of $\Gamma_0 g \Gamma_0$ equals the volume of $\Gamma_0$ multiplied by the number of coset representatives $g_i$. We omit the details.
\end{proof}

\subsection{The $\GL_2$ case}

In the case of matrix coefficients for $\GL_2$, the values $\Phi_{\mathbf g_p}(g)$ for $g \in \SL_2(\Q_p)$ can be easily deduced from the results in \cite[\S 3]{Woodbury}. We explain briefly how to get such values.

\begin{proposition}\label{prop:Phialphas-GL-St}
 For an integer $n$, we have $\Phi_{\mathbf g_p}(\alpha_n) = p^{-2|n|}$.
\end{proposition}
\begin{proof}
 Suppose that $n\geq 0$, and notice first of all that, as elements of $\GL_2(\Q_p)$, we have
\[
 \alpha_n = p^{-n}\rho_{2n}, \quad \text{where } \rho_{2n} = \left(\begin{array}{cc} p^{2n} & 0 \\ 0 & 1\end{array}\right).
\]
In particular, it follows that $\Phi_{\mathbf g_p}(\alpha_n) = \xi(p^{-n})^2 \Phi_{\mathbf g_p}(\rho_{2n}) = \xi(p)^{-2n} \Phi_{\mathbf g_p}(\rho_{2n})$. According to \cite[Proposition 3.8]{Woodbury}, for $n\geq 1$ we have $\Phi_{\mathbf g_p}(\rho_{2n}) = \xi(p)^{2n} p^{-2n}$, hence combining the last two identities we find out that $\Phi_{\mathbf g_p}(\alpha_n) = p^{-2n}$. For $n=0$, $\alpha_0 = \mathrm{Id}_2$ and we trivially find $\Phi_{\mathbf g_p}(\mathrm{Id}_2) = 1$. Thus the statement holds as well. When $n<0$, we can proceed similarly. Indeed, we now have the identity 
\[
 \alpha_n = p^n w \rho_{-2n} w, \quad \text{where } w = \left(\begin{array}{cc} 0 & 1 \\ 1 & 0 \end{array}\right),
\]
so that $\Phi_{\mathbf g_p}(\alpha_n) = \xi(p)^{2n} \Phi_{\mathbf g_p}(w\rho_{-2n}w)$. Now $-2n > 0$, and again by \cite[Proposition 3.8]{Woodbury} $\Phi_{\mathbf g_p}(w \rho_{-2n}w) = \xi(p)^{-2n} p^{2n}$. Altogether, we conclude as desired that $\Phi_{\mathbf g_p}(\alpha_n) = p^{2n}$.
\end{proof}

The next lemma deals with the case $g = \beta_m = s\alpha_m$, for $m \in \Z$.

\begin{proposition}\label{prop:Phibetas-GL-St}
 For an integer $m$, we have $\Phi_{\mathbf g_p}(\beta_m) = -p^{-|2m-1|}$.
\end{proposition}
\begin{proof}
 First suppose that $m > 0$, and observe that
\[
 \beta_m = \left(\begin{array}{cc}0 & p^{-m} \\ -p^m & 0\end{array}\right) = 
 \left(\begin{array}{cc}0 & p^{-m} \\ p^m & 0\end{array}\right)\left(\begin{array}{cc}-1 & 0 \\ 0 & 1\end{array}\right) = p^{-m}\left(\begin{array}{cc}0 & 1 \\ p^{2m} & 0\end{array}\right)\left(\begin{array}{cc}-1 & 0 \\ 0 & 1\end{array}\right).
\]
Since $\tau_p$ is invariant by the rightmost element, we have $\Phi_{\mathbf g_p}(\beta_m) = \xi(p)^{-2m}\Phi_{\mathbf g_p}(w\rho_{2m})$. Again by \cite[Proposition 3.8]{Woodbury},  $\Phi_{\mathbf g_p}(w\rho_{2m}) = -\xi(p)^{2m}p^{1-2m}$, hence we conclude that $\Phi_{\mathbf g_p}(\beta_m) = -p^{-2m+1}$. When $m = 0$, observe that $\beta_0 = s$. Since 
\[
 s = w\left(\begin{array}{cc} -1 & 0 \\ 0 & 1\end{array}\right)
\]
and $\tau_p$ is invariant by the rightmost element, we see that $\Phi_{\mathbf g_p}(s) = \Phi_{\mathbf g_p}(w)$. By \cite[Proposition 3.8]{Woodbury}, $\Phi_{\mathbf g_p}(w) = -p^{-1}$, hence for $m=0$ we have $\Phi_{\mathbf g_p}(\beta_0) = -p^{-1}$, which fits into the statement. Finally, suppose that $m < 0$ and write 
\[
 \beta_m = p^m \left(\begin{array}{cc} 0 & p^{-2m} \\ -1 & 0\end{array}\right) = p^m\left(\begin{array}{cc} 0 & p^{-2m} \\ 1 & 0\end{array}\right)\left(\begin{array}{cc} -1 & 0 \\ 0 & 1\end{array}\right).
\]
From this, $\Phi_{\mathbf g_p}(\beta_m) = \xi(p)^{2m}\Phi_{\mathbf g_p}(\rho_{-2m}w)$. Since $-2m > 0$, now \cite[Proposition 3.8]{Woodbury} implies that $\Phi_{\mathbf g_p}(\rho_{-2m}w)= -\xi(p)^{-2m}p^{2m-1}$, and hence for $m < 0$ we conclude as claimed that $\Phi_{\mathbf g_p}(\beta_m) = -p^{2m-1}$.
\end{proof}

\subsection{The $\widetilde{\SL}_2$ case}

For the computation of matrix coefficients in the $\widetilde{\SL}_2$ case, it will be useful to introduce the following subsets of $\SL_2(\Z_p)$. For each integer $j\geq 0$, we define 
\[
 \mathcal L_j := \left\lbrace \left(\begin{smallmatrix} a & b \\c & d \end{smallmatrix}\right) \in \SL_2(\Z_p): \mathrm{ord}_p(c) = j\right\rbrace, \quad 
 \mathcal R_j := \left\lbrace \left(\begin{smallmatrix} a & b \\c & d \end{smallmatrix}\right) \in \SL_2(\Z_p): \mathrm{ord}_p(d) = j\right\rbrace. 
\]
Recall that our measure on $\SL_2(\A)$ is chosen so that $\mathrm{vol}(\SL_2(\Z_p)) = \zeta_p(2)^{-1} = 1-p^{-2}$. By expressing the sets $\mathcal L_j$ and $\mathcal R_j$ in terms of the subgroups $\Gamma_0(p^j)\subseteq \SL_2(\Z_p)$, for $j\geq 0$, one can easily prove: 

\begin{lemma}
 With the above notation, $\mathrm{vol}(\mathcal L_0) = \mathrm{vol}(\mathcal R_0) = 1-p^{-1}$, $\mathrm{vol}(\mathcal L_0 \cap \mathcal R_0) = (1-p^{-1})^2$, and 
 \[
  \mathrm{vol}(\mathcal L_j) = \mathrm{vol}(\mathcal R_j) = p^{-j}(1-p^{-1})^2 \text{ for all } j > 0.
 \]
\end{lemma}

For the computation of the (normalized) matrix coefficients $\Phi_{\mathbf h_p}(\alpha_n)$ and $\Phi_{\mathbf h_p}(\beta_m)$, recall that $\widetilde{\SL}_2(\Q_p) = \SL_2(\Q_p) \times \{\pm 1\}$ as sets. The group operation is given by
\[
 [g_1, \epsilon_1][g_2, \epsilon_2] = [g_1g_2, \epsilon(g_1,g_2)\epsilon_1\epsilon_2],
\]
where $\epsilon(g_1,g_2) = (x(g_1)x(g_1g_2), x(g_2)x(g_1g_2))_p$, with $x: \SL_2(\Q_p) \to \Q_p$ being defined as
\[
 g = \left(\begin{smallmatrix} a & b \\ c & d\end{smallmatrix}\right) \longmapsto x(g) = \begin{cases}
                                                                                          c & \text{if } c \neq 0,\\
                                                                                          d & \text{if } c = 0.
                                                                                         \end{cases}
\]

Recall that we regard $\SL_2(\Z_p)$ as a subgroup of $\widetilde{\SL}_2(\Q_p)$ via the splitting $k \, \longmapsto \, [k, s_p(k)]$, where 
\[
 s_p \left(\left( \begin{array}{cc} a & b \\ c & d \end{array}\right)\right) 
 = \begin{cases}
       (c,d)_p & \text{if } cd \neq 0 \text{ and } \mathrm{ord}_p(c) \text{ is odd}, \\
       1 & \text{otherwise.}
   \end{cases}
\]
And recall also that $\mathbf h_p \in \tilde{\pi}_p$ is the function $\widetilde{\SL}_2(\Q_p) \to \C$ described in Lemma \ref{sl2testvector_special}, which satisfies the transformation rule spelled out in \eqref{PS}.

\subsubsection{Computation of $\Phi_{\mathbf h_p}(\alpha_n)$}

Fix throughout this paragraph $n\in \Z$, and identify $\alpha_n$ with the element $[\alpha_n,1] \in \widetilde{\SL}_2(\Q_p)$. In the computation of $\Phi_{\mathbf h_p}(\alpha_n)$, we encounter products of the form $h\alpha_n$, with $h \in \SL_2(\Z_p)$ and $n \in \Z$. This is to be seen as the product 
\[
 [h,s_p(h)][\alpha_n,1] = [h\alpha_n, \epsilon] \in \widetilde{\SL}_2(\Q_p),
\]
where using the above recipe we have $\epsilon = s_p(h)\epsilon(h,\alpha_n)$. The sign $s_p(h) \in \{\pm 1\}$ is given as above. And by using the definition of $\epsilon(\cdot, \cdot)$, we have 
\[
 \epsilon(h,\alpha_n) = (x(h)x(h\alpha_n), x(\alpha_n)x(h\alpha_n))_p.
\]
Write 
\[
 h = \left(  \begin{array}{cc} a & b \\ c & d \end{array}\right), \quad 
 h\alpha_n = \left(  \begin{array}{cc} ap^n & bp^{-n} \\ cp^n & dp^{-n} \end{array}\right),
\]
so that we have 
\[
 s_p(h) = \begin{cases}
       (c,d)_p & \text{if } cd \neq 0, \mathrm{ord}_p(c) \text{ is odd}, \\
       1 & \text{otherwise,}
   \end{cases}
 \quad 
 x(h) = \begin{cases}
       c & \text{if } c \neq 0,\\
       d & \text{if } c = 0,
   \end{cases}
 \quad
 x(h\alpha_n) = \begin{cases}
       cp^n & \text{if } c \neq 0,\\
       dp^{-n} & \text{if } c = 0.
   \end{cases}
\]
Then one can easily check that 
\begin{equation}\label{epsilon-alphan}
  \epsilon = 
 \begin{cases}
  (d,p^n)_p & \text{if } c = 0,\\
  (c,dp^n)_p & \text{if } cd \neq 0 \text{ and } \mathrm{ord}_p(c) \text{ odd},\\
  (c,p^n)_p & \text{otherwise}.
 \end{cases}
\end{equation}

To compute $\Phi_{\mathbf h_p}(\alpha_n)$, we will need to evaluate $\mathbf h_p$ at the element $[h\alpha_n,\epsilon]$. And to do so, we need to write down an Iwasawa decomposition of this element, induced from an Iwasawa decomposition for $h\alpha_n$ in $\SL_2(\Q_p)$. The shape of such an Iwasawa decomposition will vary according to whether $h$ belongs to the subregions $\mathcal A_1(n)$ or $\mathcal A_2(n)$ of $\SL_2(\Z_p)$ that we now discuss.
\begin{itemize}
 \item[i)] Let $\mathcal A_1(n) \subseteq \SL_2(\Z_p)$ be the subset of $h = \left(\begin{smallmatrix} a & b \\ c & d\end{smallmatrix}\right) \in \SL_2(\Z_p)$ with $|cp^{n}|_p < |dp^{-n}|_p$. Equivalently, $|c|_p < p^{2n}|d|_p$. Observe that one always has $d \neq 0$ in $\mathcal A_1(n)$. If $h \in \mathcal A_1(n)$, we have 
 \[
  h\alpha_n = \left(  \begin{array}{cc} ap^n & bp^{-n} \\ cp^n & dp^{-n} \end{array}\right) = 
  \left(  \begin{array}{cc} d^{-1}p^n & \ast \\ 0 & dp^{-n} \end{array}\right)
  \left(  \begin{array}{cc} 1 & 0 \\ cd^{-1}p^{2n} & 1 \end{array}\right) =: g_1g_2,
 \]
 where notice that the rightmost element $g_2$ belongs to $\Gamma_0$. This lifts to an identity 
 \begin{equation}\label{Iwdecomposition:A1}
  [h\alpha_n,\epsilon] = 
  \left[\left(  \begin{array}{cc} d^{-1}p^n & \ast \\ 0 & dp^{-n} \end{array}\right),e\right]
  \left[\left(  \begin{array}{cc} 1 & 0 \\ cd^{-1}p^{2n} & 1 \end{array}\right),1\right],
 \end{equation}
 where $e = \epsilon(g_1,g_2)\epsilon$. Since $g_1g_2 = h\alpha_n$, we have 
 \[
 \epsilon(g_1,g_2) = (x(g_1)x(h\alpha_n),x(g_2)x(h\alpha_n))_p = 
 \begin{cases}
  (dc, dp^n)_p & \text{if } c \neq 0, \\
  1 & \text{if } c = 0.
 \end{cases}
 \]
 From our recipe for $\epsilon$ in \eqref{epsilon-alphan}, noticing that $d \neq 0$ in $\mathcal A_1(n)$ and using elementary properties of the Hilbert symbol, the above recipe for $e = e(h)$ when $h \in \mathcal A_1(n)$ gets simplified to 
 \begin{equation}\label{e-alphan-1}
  e = 
  \begin{cases}
   (cd,d)_p(d,p^n)_p & \text{if } c \neq 0 \text{ and } \mathrm{ord}_p(c) \text{ even,} \\
   (d,p^n)_p & \text{otherwise.}
  \end{cases}
 \end{equation}

 \item[ii)] Let $\mathcal A_2(n) \subseteq \SL_2(\Z_p)$ be the subset of elements $h = \left(\begin{smallmatrix} a & b \\ c & d\end{smallmatrix}\right) \in \SL_2(\Z_p)$ with $|cp^{n}| \geq |dp^{-n}|_p$, or equivalently, $|c|_p \geq p^{2n}|d|_p$. In $\mathcal A_2(n)$, observe that one always has $c \neq 0$. For $h \in \mathcal A_2(n)$, we have 
 \[
  h\alpha_n = \left(  \begin{array}{cc} ap^n & bp^{-n} \\ cp^n & dp^{-n} \end{array}\right) = 
  \left(  \begin{array}{cc} c^{-1}p^{-n} & \ast \\ 0 & cp^{n} \end{array}\right)
  \left(  \begin{array}{cc} 0 & -1 \\ 1 & dc^{-1}p^{-2n} \end{array}\right) =: g_1g_2,
 \]
 where now observe that $g_2 \in \SL_2(\Z_p) - \Gamma_0$. This identity lifts now to 
 \begin{equation}\label{Iwdecomposition:A2}
  [h\alpha_n,\epsilon] =  
  \left[ \left(  \begin{array}{cc} c^{-1}p^{-n} & \ast \\ 0 & cp^{n} \end{array}\right), e \right]
  \left[\left(  \begin{array}{cc} 0 & -1 \\ 1 & dc^{-1}p^{-2n} \end{array}\right), 1\right],
 \end{equation}
 again with $e = \epsilon(g_1,g_2)\epsilon$. In this case, since $c\neq 0$ for all $h \in \mathcal A_2(n)$, we have  
 \[
  \epsilon(g_1,g_2) = (x(g_1)x(h\alpha_n),x(g_2)x(h\alpha_n))_p = (c^2p^{2n},cp^n)_p = 1,
 \]
 and therefore, when $h \in \mathcal A_2(n)$ the value of $e=e(h)$ is given as follows:
 \begin{equation}\label{e-alphan-2}
  e = \epsilon = 
  \begin{cases}
  (c,d)_p(c,p^n)_p & \text{if } d \neq 0 \text{ and } \mathrm{ord}_p(c) \text{ odd},\\
  (c,p^n)_p & \text{if } d = 0 \text{ or } \mathrm{ord}_p(c) \text{ even}.
  \end{cases}
 \end{equation}
\end{itemize}

As we know $||\mathbf h_p||^2$ from Lemma \ref{lemma:norm-varphip}, to compute $\Phi_{\mathbf h_p}(\alpha_n)$ we need to compute $\langle \tilde{\pi}_p(\alpha_n)\mathbf h_p, \mathbf h_p\rangle$. And since $\mathcal A_1(n)$ and $\mathcal A_2(n)$ are disjoint by definition, and their union is the whole $\SL_2(\Z_p)$, we see that 
\[
 \langle \tilde{\pi}_p(\alpha_n)\mathbf h_p, \mathbf h_p\rangle = \int_{\mathcal A_1(n)} \mathbf h_p(h\alpha_n)\overline{\mathbf h_p(h)} dh + \int_{\mathcal A_2(n)} \mathbf h_p(h\alpha_n)\overline{\mathbf h_p(h)} dh =: A_1(n) + A_2(n).
\]
In the remaining of this section we will compute separately $A_1(n)$ and $A_2(n)$. Observe that $\mathbf h_p(h)$ only takes values $1$ and $-p$ for $h \in \SL_2(\Z_p)$, hence we have $\overline{\mathbf h_p(h)}= \mathbf h_p(h)$ in the above integrals.

\subsubsection*{Computation of $A_1(n)$} 

We proceed now with the computation of $A_1(n) = \int_{\mathcal A_1(n)} \mathbf h_p(h\alpha_n)\mathbf h_p(h) dh$. Writing $h = \left(\begin{smallmatrix} a & b \\ c & d\end{smallmatrix}\right)$ as usual, by virtue of \eqref{Iwdecomposition:A1} and \eqref{PS} we see that 
\[
 A_1(n) = - p^{1-3n/2} (-1)^n \chi_{\psi}(p^n) \int_{\mathcal A_1(n)}  (d,p^n)_p\chi_{\psi}(d) e(h) \chi_{\delta}(d)|d|^{-3/2}_p\mathbf h_p(h) dh,
\]
where $e(h)$ is given by the recipe in \eqref{e-alphan-1}. Denote by $\mathcal I_1(n)$ the last integral. Defining 
\[
 \mathcal A_1^+(n) := \{h \in \mathcal A_1(n): c \neq 0 \text{ and } \mathrm{ord}_p(c) \text{ even}\}, \quad \mathcal A_1^-(n) := \{h \in \mathcal A_1(n): c = 0 \text{ or } \mathrm{ord}_p(c) \text{ odd}\},
\]
we have, according to \eqref{e-alphan-1},
\[
 \mathcal I_1(n) = \int_{\mathcal A_1^+(n)} (cd,d)_p\chi_{\psi}(d)\chi_{\delta}(d)|d|_p^{-3/2} \mathbf h_p(h) dh + \int_{\mathcal A_1^-(n)} \chi_{\psi}(d)\chi_{\delta}(d)|d|_p^{-3/2} \mathbf h_p(h) dh.
\]
Observe that $d$ is always a unit in $\mathcal A_1^-(n)$, thus the above becomes 
\[
 \mathcal I_1(n) = \int_{\mathcal A_1^+(n)} (cd,d)_p\chi_{\psi}(d)\chi_{\delta}(d)|d|_p^{-3/2}\mathbf h_p(h) dh + \int_{\mathcal A_1^-(n)} \mathbf h_p(h) dh.
\]
From this expression we will now easily obtain the value of $A_1(n)$.

\begin{lemma}\label{lemma:A1}
With the above notation, 
\[
 A_1(n) = 
 \begin{cases}
  p^{-3n/2} (-1)^n \chi_{\psi}(p^n) (1-p^{-1})(1+p-p^n) & \text{if } n>0, \\
  p^{3n/2} (-1)^n \chi_{\psi}(p^n) (1-p^{-1}) p^{1-n} & \text{if } n \leq 0.
 \end{cases}
\]
\end{lemma}
\begin{proof}
 Suppose that $n > 0$. First of all, observe that $\mathcal A_1^-(n) = \mathcal A_1^-(0)$, which up to a set of measure zero is the disjoint union of the sets $\mathcal L_{2j+1}$ with $j\geq 0$. Therefore we have
 \begin{align*}
  \int_{\mathcal A_1^-(n)} \mathbf h_p(h) dh & = \sum_{j\geq 0} \int_{\mathcal L_{2j+1}} \mathbf h_p(h) dh = 
  -p \sum_{j\geq 0} \mathrm{vol}(\mathcal L_{2j+1}) = -p \sum_{j\geq 0} p^{-2j-1}(1-p^{-1})^2 = \\
  & = -(1-p^{-1})^2 \sum_{j\geq 0} p^{-2j} = \frac{-(1-p^{-1})^2}{1-p^{-2}} = \frac{-(1-p^{-1})}{1+p^{-1}} = \frac{1-p}{p+1}. 
 \end{align*}
 On the other hand, $\mathcal A_1^+(n)$ is the disjoint union of the sets $\mathcal L_{2j}$ with $j>0$, $\mathcal L_0 \cap \mathcal R_0$, and $\mathcal R_j$ with $1\leq j \leq 2n-1$. If $h \in \mathcal L_{2j}$ with $j > 0$, then $h \in \Gamma_0$, $\mathrm{ord}_p(c)$ is even, and $d$ is a unit. Thus
 \[
  \int_{\mathcal L_{2j}} (cd,d)_p\chi_{\psi}(d)\chi_{\delta}(d)|d|_p^{-3/2}\mathbf h_p(h) dh = -p \int_{\mathcal L_{2j}} 1dh = -p \mathrm{vol}(\mathcal L_{2j}) = -p^{1-2j}(1-p^{-1})^2,
 \]
 and one deduces that
 \[
  \sum_{j>0}\int_{\mathcal L_{2j}} (cd,d)_p\chi_{\psi}(d)\chi_{\delta}(d)|d|_p^{-3/2}\mathbf h_p(h) dh = 
  -p(1-p^{-1})^2 \sum_{j>0} p^{-2j} = -p^{-1}\frac{(1-p^{-1})^2}{(1-p^{-2})} = \frac{1-p}{p(p+1)}.
 \]
  When $h \in \mathcal L_0 \cap \mathcal R_0$, both $c$ and $d$ are units, and $\mathbf h_p(h)=1$, hence 
\[
  \int_{\mathcal L_0 \cap \mathcal R_0} (cd,d)_p\chi_{\psi}(d)\chi_{\delta}(d)|d|_p^{-3/2}\mathbf h_p(h) dh = \mathrm{vol}(\mathcal L_0 \cap \mathcal R_0) = (1-p^{-1})^2.
 \]
 Finally, if $h \in \mathcal R_j$ with $j > 0$, we have $c \in \Z_p^{\times}$ and $\mathrm{ord}_p(d) = j$. Using elementary properties of the quadratic symbol, it follows that 
 \[
  \int_{\mathcal R_j} (cd,d)_p\chi_{\psi}(d)\chi_{\delta}(d)|d|_p^{-3/2}\mathbf h_p(h) dh = 
  (-1)^j\chi_{\psi}(p^j)p^{3j/2}\int_{\mathcal R_j} (cd,p^j)_p dh.
 \]
 If $j = 2t$ is even, then $(cd,p^j)_p = 1$ and we deduce that  
 \[
  \int_{\mathcal R_{2t}} (cd,d)_p\chi_{\psi}(d)\chi_{\delta}(d)|d|_p^{-3/2}\mathbf h_p(h) dh = 
  p^{3t} \mathrm{vol}(\mathcal R_{2t}) = p^t (1-p^{-1})^2.
 \]
 In contrast, when $j$ is odd, by applying the automorphism of $\mathcal R_j$ given by conjugation by $\gamma_u = \left(\begin{smallmatrix} u & 0 \\ 0 & 1\end{smallmatrix}\right)$, with $u \in \Z_p^{\times}$ a non-quadratic residue, one sees that the above integral equals the same integral multiplied by $-1$. Therefore, it must vanish. As a consequence, 
 \begin{align*}
  & \sum_{j=1}^{2n-1} \int_{\mathcal R_j} (cd,d)_p\chi_{\psi}(d)\chi_{\delta}(d)|d|_p^{-3/2}\mathbf h_p(h) dh = 
  \sum_{t=1}^{n-1} \int_{\mathcal R_{2t}} (cd,d)_p\chi_{\psi}(d)\chi_{\delta}(d)|d|_p^{-3/2}\mathbf h_p(h) dh = \\
  & = (1-p^{-1})^2 \sum_{t=1}^{n-1}  p^t = (1-p^{-1})^2 \frac{p-p^n}{1-p} = 
  (1-p^{-1})^2 \frac{p^{n-1}-1}{1-p^{-1}} = (1-p^{-1})(p^{n-1}-1).
 \end{align*}
 
Summing all the contributions, we find for $n > 0$ 
\[
 \mathcal I_1(n) = -(1-p^{-1}) + (1-p^{-1})^2 + (1-p^{-1})(p^{n-1}-1) = p^{-1}(1-p^{-1})(p^n-p-1),
\]
and hence $A_1(n) = p^{-3n/2} (-1)^n \chi_{\psi}(p^n) (1-p^{-1})(1+p-p^n)$.

Next assume that $n\leq 0$. In this case, one easily checks that $\mathcal A_1^+(n)$ is the disjoint union of the sets $\mathcal L_{2t}$ for $t > -n$. Therefore, it follows that (notice $d$ is a unit)
\begin{align*}
 & \int_{\mathcal A_1^+(n)}(cd,d)_p\chi_{\psi}(d)\chi_{\delta}(d)|d|_p^{-3/2}\mathbf h_p(h) dh = 
 \sum_{t > -n} \int_{\mathcal L_{2t}}(cd,d)_p\chi_{\psi}(d)\chi_{\delta}(d)|d|_p^{-3/2}\mathbf h_p(h) dh = \\
 & = -p \sum_{t > -n} \mathrm{vol}(\mathcal L_{2t}) = -p(1-p^{-1})^2 \sum_{t > -n} p^{-2t} = 
 -p^{2n-1}\frac{(1-p^{-1})^2}{1-p^{-2}} = -p^{2n}\frac{p-1}{p(p+1)}.
\end{align*}
Besides, one can check that up to a set of measure zero the set $\mathcal A_1^-(n)$ equals the disjoint union of the sets $\mathcal L_{2j+1}$ with $j \geq -n$. Therefore, the integral over $\mathcal A_1^-(n)$ in the expression for $\mathcal I_1(n)$ equals 
\[
 \sum_{j\geq -n} \int_{\mathcal L_{2j+1}} \mathbf h_p(h) dh = -p \sum_{j\geq -n} \mathrm{vol}(\mathcal L_{2j+1}) = -p(1-p^{-1})^2 \sum_{j\geq -n} p^{-2j-1} = -p^{2n}\frac{(1-p^{-1})^2}{1-p^{-2}} = -p^{2n}\frac{p-1}{p+1}.
\]
Altogether, we find for $n\leq0$ that 
\[
 \mathcal I_1(n) = -p^{2n}\frac{p-1}{p(p+1)} -p^{2n}\frac{p-1}{p+1} = -p^{2n}\frac{p-1}{p+1}(1+p^{-1}) = -p^{2n-1}(p-1)
\]
and therefore $A_1(n) = p^{-3n/2} (-1)^n \chi_{\psi}(p^n)(p-1)p^{2n} = p^{3n/2} (-1)^n \chi_{\psi}(p^n) (1-p^{-1}) p^{1-n}$. 
\end{proof}

\subsubsection*{Computation of $A_2(n)$} 

We now deal with the computation of $A_2(n) = \int_{\mathcal A_2(n)} \mathbf h_p(h\alpha_n)\mathbf h_p(h) dh$. Writing $h = \left(\begin{smallmatrix} a & b \\ c & d\end{smallmatrix}\right)$ as usual, by virtue of \eqref{Iwdecomposition:A2} and \eqref{PS} we see that 
\[
 A_2(n) =  p^{3n/2} (-1)^n \chi_{\psi}(p^n) \int_{\mathcal A_2(n)}  (c,p^n)_p\chi_{\psi}(c) e(h) \chi_{\delta}(c)|c|^{-3/2}_p\mathbf h_p(h) dh,
\]
where $e(h)$ is given by the recipe in \eqref{e-alphan-2}. Write $\mathcal I_2(n)$ for the integral on the right hand side. Then define 
\[
 \mathcal A_2^+(n) := \{h \in \mathcal A_2(n): d = 0 \text{ or } \mathrm{ord}_p(c) \text{ even}\}, \quad \mathcal A_2^-(n) := \{h \in \mathcal A_2(n): d \neq 0 \text{ and } \mathrm{ord}_p(c) \text{ odd}\},
\]
so that $\mathcal A_2 = \mathcal A_2^+(n) \sqcup \mathcal A_2^-(n)$ and 
\[
 \mathcal I_2(n) = \int_{\mathcal A_2^+(n)}  \chi_{\psi}(c) \chi_{\delta}(c)|c|^{-3/2}_p\mathbf h_p(h) dh + 
 \int_{\mathcal A_2^-(n)}  (c,d)_p\chi_{\psi}(c) \chi_{\delta}(c)|c|^{-3/2}_p\mathbf h_p(h) dh.
\]

Concerning the second integral, applying conjugation by $\gamma_u = \left(\begin{smallmatrix} u & 0 \\ 0 & 1 \end{smallmatrix}\right)$ shows that such integral equals itself multiplied by $-1$, hence it vanishes. Therefore, we have 
\[
 \mathcal I_2(n)= \int_{\mathcal A_2^+(n)}  \chi_{\psi}(c) \chi_{\delta}(c)|c|^{-3/2}_p\mathbf h_p(h) dh.
\]

Using this expression, we can easily compute $\mathcal I_2(n)$, and hence $A_2(n)$.

\begin{lemma}\label{lemma:A2}
 With the above notation, 
 \[
  A_2(n) = 
  \begin{cases}
  p^{-3n/2}(-1)^n \chi_{\psi}(p^n)(1-p^{-1})p^{n} & \text{if } n > 0,\\
  p^{3n/2} (-1)^n \chi_{\psi}(p^n)(1-p^{-1})(1+p-p^{1-n}) & \text{if } n \leq 0.
  \end{cases}
 \]
\end{lemma}
\begin{proof}
 The proof goes along the same lines of the previous lemma, so we omit the computations. We just point out that when $n > 0$ (resp. $n\leq 0$) the set $\mathcal A_2^+(n)$ coincides, up to a zero measure set, with the disjoint union of the sets $\mathcal R_j$, with $j \geq 2n$ (resp. the sets $\mathcal L_{2t}$, with $0\leq t \leq -n$).
 %
%
\end{proof}


\begin{proposition}\label{prop:Phialphas-SL-St}
 With the above notation, $\Phi_{\mathbf h_p}(\alpha_n) = p^{-3|n|/2} (-1)^n \chi_{\psi}(p^n)$.
\end{proposition}
\begin{proof}
Recalling that $\Phi_{\mathbf h_p}(\alpha_n) = \langle \tilde{\pi}_p(\alpha_n)\mathbf h_p, \mathbf h_p \rangle/||\mathbf h_p||^2$, and that $||\mathbf h_p||^2 = p^{-1}(p^2-1)$ from Lemma \ref{lemma:norm-varphip}, the statement follows by combining Lemmas \ref{lemma:A1} and \ref{lemma:A2} since $\langle \tilde{\pi}_p(\alpha_n)\mathbf h_p, \mathbf h_p \rangle = A_1(n) + A_2(n)$.
\end{proof}

\subsubsection{Computation of $\Phi_{\mathbf h_p}(\beta_m)$}

Now we proceed with the computation of $\Phi_{\mathbf h_p}(\beta_m)$. Thus fix from now on an integer $m$, and let $\beta_m = s\alpha_m = \left(\begin{smallmatrix} 0 & p^{-m} \\ -p^m & 0\end{smallmatrix}\right)\in \SL_2(\Q_p)$. As before, we identify $\alpha_m$ with the element $[\alpha_m,1] \in \widetilde{\SL}_2(\Q_p)$. We will need to evaluate $\mathbf h_p$ at products of the form 
\[
 [h s ,s_p(hs)][\alpha_m,1] = [h s \alpha_m,\epsilon] \in \widetilde{\SL}_2(\Q_p),
\]
where $\epsilon = s_p(hs) \epsilon(hs,\alpha_m)$. If we write $h = \left(\begin{smallmatrix} a & b \\ c & d \end{smallmatrix}\right)$, then  
\[
 h s = \left(  \begin{array}{cc} -b & a \\ -d & c \end{array}\right), \quad 
 h\beta_m = hs\alpha_m = \left( \begin{array}{cc} -bp^m & ap^{-m} \\ -dp^m & cp^{-m} \end{array}\right),
\]
so that  
\[
 s_p(h s) = \begin{cases}
       (c,-d)_p & \text{if } cd \neq 0, \mathrm{ord}_p(d) \text{ is odd}, \\
       1 & \text{otherwise,}
   \end{cases}
 \quad 
 x(h s) = \begin{cases}
       -d & \text{if } d \neq 0,\\
       c & \text{if } d = 0,
   \end{cases}
 \quad
 x(h\beta_m) = \begin{cases}
       -dp^m & \text{if } d \neq 0,\\
       cp^{-m} & \text{if } d = 0.
   \end{cases}
\]
Besides, $x(\alpha_m) = p^{-m}$, hence we can compute the sign  
\[
 \epsilon(h s,\alpha_m) = (x(h s)x(h\beta_m), x(\alpha_m)x(h\beta_m))_p
\]
as follows. If $d \neq 0$, we have $\epsilon(h s,\alpha_m) = (-d,p^m)_p$, whereas if $d = 0$, we have $\epsilon(h s,\alpha_m) = (c,p^m)_p$. Together with the recipe for $s_p(hs)$, we have 
\begin{equation}\label{epsilon-betam}
 \epsilon = s_p(hs) \epsilon(hs,\alpha_m) = 
 \begin{cases}
   (c,-d)_p(-d,p^m)_p & \text{if } cd \neq 0, \mathrm{ord}_p(d) \text{ is odd}, \\
   (-d,p^m)_p & \text{if } cd \neq 0, \mathrm{ord}_p(d) \text{ is even},\\
   (c,p^m)_p & \text{if } d = 0,\\
   (-d,p^m)_p & \text{if } c = 0.
 \end{cases}
\end{equation}

As in the case of the elements $\alpha_n$, now to compute $\Phi_{\mathbf h_p}(\beta_m)$ we need an Iwasawa decomposition for $[h\beta_m,\epsilon]$, induced from the Iwasawa decomposition in $\SL_2(\Q_p)$. Such a decomposition will depend on two subregions $\mathcal B_1(m)$ and $\mathcal B_2(m)$ of $\SL_2(\Z_p)$. The discussion is analogous to the one we did above for the sets $\mathcal A_1(n)$ and $\mathcal A_2(n)$, thus we omit some details. 
\begin{itemize}
 \item[i)] Let $\mathcal B_1(m) \subseteq \SL_2(\Z_p)$ be the set of $h = \left(\begin{smallmatrix} a & b \\ c & d\end{smallmatrix}\right) \in \SL_2(\Z_p)$ with $|dp^m|_p < |cp^{-m}|_p$, or equivalently $|d|_p < p^{2m}|c|_p$. Notice that one always has $c \neq 0$ in $\mathcal B_1(m)$. If $h \in \mathcal B_1(m)$, one has 
 \begin{equation}\label{Iwdecomposition:B1}
  [h\beta_m,\epsilon] = \left[\left(\begin{array}{cc} c^{-1}p^m & \ast \\ 0 & cp^{-m}\end{array}\right),e\right]\left[\left(\begin{array}{cc} 1 & 0 \\ \frac{-d}{c}p^{2m} & 1\end{array}\right),1\right],
 \end{equation}
 where one can check that the sign $e$ is given by the recipe
 \begin{equation}\label{e-betam-1}
  e = \begin{cases}
 (c,p^m)_p & \text{if } d = 0 \text{ or } \mathrm{ord}_p(d) \text{ is odd,}\\
 (c,d)_p(c,p^m)_p & \text{if } d \neq 0 \text{ and } \mathrm{ord}_p(d) \text{ is even.}
 \end{cases}
 \end{equation} 
 
 \item[ii)] Let $\mathcal B_2(m)\subseteq \SL_2(\Z_p)$ be the set of elements $h = \left(\begin{smallmatrix} a & b \\ c & d\end{smallmatrix}\right) \in \SL_2(\Z_p)$ such that $|dp^m|_p \geq |cp^{-m}|_p$, or equivalently, $|d|_p \geq p^{2m}|c|_p$. Now, one always has $d\neq 0$ in $\mathcal B_2(m)$. If $h \in \mathcal B_2(m)$, one has
 \begin{equation}\label{Iwdecomposition:B2}
  [h\beta_m,\epsilon] = \left[\left(\begin{array}{cc} -d^{-1}p^{-m} & \ast \\ 0 & -dp^m \end{array}\right),e \right]\left[\left(\begin{array}{cc} 0 & -1 \\ 1 & \frac{c}{d}p^{-2m}\end{array}\right),1\right],
 \end{equation}
 where the sign $e$ is now computed according to the recipe
 \begin{equation}\label{e-betam-2}
  e = \begin{cases}
  (c,-d)_p(-d,p^m)_p & \text{if } \mathrm{ord}_p(d) \text{ is odd}, \\
  (-d,p^m)_p & \text{if } \mathrm{ord}_p(d) \text{ is even}.
 \end{cases}
 \end{equation}
 \end{itemize}

Now, to proceed with the computation of $\Phi_{\mathbf h_p}(\beta_m)$, we may focus on the computation of the matrix coefficients $\langle \tilde{\pi}_p(\beta_m)\mathbf h_p, \mathbf h_p\rangle$, since we know $||\mathbf h_p||^2$ from Lemma \ref{lemma:norm-varphip}, and we clearly have 
\[
 \langle \tilde{\pi}_p(\beta_m)\mathbf h_p, \mathbf h_p\rangle = \int_{\mathcal B_1(m)} \mathbf h_p(h\beta_m)\overline{\mathbf h_p(h)} dh + \int_{\mathcal B_2(m)} \mathbf h_p(h\beta_m)\overline{\mathbf h_p(h)} dh =: B_1(m) + B_2(m).
\]
Thus from now on, we will focus on the computation of the integrals $B_1(m)$ and $B_2(m)$. Notice again that in the above expression one has $\overline{\mathbf h_p(h)}= \mathbf h_p(h)$.

\subsubsection{Computation of $B_1(m)$} 

We proceed now with the computation of $B_1(m) = \int_{\mathcal B_1(m)} \mathbf h_p(h\beta_m)\mathbf h_p(h) dh$. Writing $h = \left(\begin{smallmatrix} a & b \\ c & d\end{smallmatrix}\right)$ as usual, by virtue of \eqref{Iwdecomposition:B1} and \eqref{PS} we see that 
\[
 B_1(m) =  - p^{1-3m/2} (-1)^m \chi_{\psi}(p^m) \int_{\mathcal B_1(m)}  (c,p^m)_p\chi_{\psi}(c) e(h) \chi_{\delta}(c)|c|^{-3/2}_p\mathbf h_p(h) dh,
\]
where $e(h)$ is given by the recipe in \eqref{e-betam-1}. Let us denote by $\mathcal J_1(m)$ the last integral, and define 
\[
 \mathcal B_1^+(m) := \{h \in \mathcal B_1(m): d = 0 \text{ or }\mathrm{ord}_p(d) \text{ odd}\}, \quad 
 \mathcal B_1^-(m) := \{h \in \mathcal B_1(m): d\neq 0 \text{ and } \mathrm{ord}_p(d) \text{ even}\}.
\]
Then, according to \eqref{e-betam-1} we have 
\[
 \mathcal J_1(m) = \int_{\mathcal B_1^+(m)} \chi_{\psi}(c) \chi_{\delta}(c)|c|^{-3/2}_p\mathbf h_p(h) dh + 
 \int_{\mathcal B_1^-(m)}  (c,d)_p\chi_{\psi}(c) \chi_{\delta}(c)|c|^{-3/2}_p\mathbf h_p(h) dh.
\]
Observe that $\mathcal B_1^+(m) \subseteq \SL_2(\Z_p)-\Gamma_0$, hence $\mathbf h_p(h)=1$ and $c$ is a unit for all $h \in \mathcal B_1^+(m)$. Therefore, the first integral in the last expression equals $\mathrm{vol}(\mathcal B_1^+(m))$, and 
\[
 \mathcal J_1(m) = \mathrm{vol}(\mathcal B_1^+(m)) + 
 \int_{\mathcal B_1^-(m)}  (c,d)_p\chi_{\psi}(c) \chi_{\delta}(c)|c|^{-3/2}_p\mathbf h_p(h) dh.
\]

\begin{lemma}\label{lemma:B1}
 With the above notation, 
 \[
  B_1(m) =
  \begin{cases}
  p^{1-3m/2} (-1)^{m+1} \chi_{\psi}(p^m)(1-p^{-1})(1 + p - p^m) & \text{if } m > 0,\\
  p^{3m/2}(-1)^{m+1} \chi_{\psi}(p^m)(1-p^{-1})p^{-m} & \text{if } m \leq 0.
  \end{cases}
 \]
\end{lemma}
\begin{proof}

First, suppose that $m \leq 0$. Then, up to a set of measure zero, $\mathcal B_1^+(m)$ equals the union of the sets $\mathcal R_{2j+1}$ with $j \geq -m$. Thus 
 \[
  \mathrm{vol}(\mathcal B_1^+(m)) = \sum_{j\geq -m} \mathrm{vol}(\mathcal R_{2j+1}) = p^{-1}(1-p^{-1})^2\sum_{j\geq -m} p^{-2j} = p^{2m-1}(1-p^{-1})^2 \frac{1}{1-p^{-2}} = p^{2m-1}\frac{1-p^{-1}}{1+p^{-1}}.
 \]
 Besides, $\mathcal B_1^-(m)$ equals the union of the sets $\mathcal R_{2j}$ with $j > -m$. When $h\in \mathcal R_{2j}$ we have $\mathbf h_p(h) = 1$ and $c$ is a unit, and one has 
\[
  \int_{\mathcal B_1^-(m)}  (c,d)_p\chi_{\psi}(c) \chi_{\delta}(c)|c|^{-3/2}_p\mathbf h_p(h) dh = \sum_{j>-m} \mathrm{vol}(\mathcal R_{2j}) = (1-p^{-1})^2\sum_{j>-m} p^{-2j} 
 = p^{2m-2}\frac{1-p^{-1}}{1+p^{-1}}.
\]
 Altogether, in this case we see that $\mathcal J_1(m) = p^{2m-1}(1-p^{-1})$, and hence 
 \[
  B_1(m) = - p^{1-3m/2} (-1)^m \chi_{\psi}(p^m)\mathcal J_1(m) = p^{3m/2}(-1)^{m+1} \chi_{\psi}(p^m)(1-p^{-1})p^{-m}.
 \]

 Now suppose that $m > 0$. One easily checks that $\mathcal B_1^+(m)=\mathcal B_1^+(0)$, and that $\mathcal B_1^-(m)$ is the union of $\mathcal B_1^-(0)$ together with the sets $\mathcal L_0 \cap \mathcal R_0$ and the sets $\mathcal L_j$ with $1\leq j\leq 2m-1$. Integration over $\mathcal L_0 \cap \mathcal R_0$ gives $\mathrm{vol}(\mathcal L_0\cap \mathcal R_0)=(1-p^{-1})^2$. For {\em odd} integers $j$ with $1\leq j \leq 2m-1$, we see with an already used argument that
 \[
  \int_{\mathcal L_j}(c,d)_p\chi_{\psi}(c) \chi_{\delta}(c)|c|^{-3/2}_p\mathbf h_p(h) dh 
 \]
 vanishes. And for $j=2t$ even, $1\leq t \leq m-1$, we have 
 \[
   \int_{\mathcal L_{2t}}(c,d)_p\chi_{\psi}(c) \chi_{\delta}(c)|c|^{-3/2}_p\mathbf h_p(h) dh = -p^{1+3t} \mathrm{vol}(\mathcal L_{2t}) = -p^{1+t} (1-p^{-1})^2.
 \]
 Summing up all the contributions, for $m > 0$ one checks that $\mathcal J_1(m) = (1-p^{-1})(1 + p - p^m)$, and hence in this case $B_1(m) = p^{1-3m/2} (-1)^{m+1} \chi_{\psi}(p^m)(1-p^{-1})(1 + p - p^m)$.
\end{proof}

\subsubsection*{Computation of $B_2(m)$} 

We now deal with the computation of $B_2(m) = \int_{\mathcal B_2(m)} \mathbf h_p(h\beta_m)\mathbf h_p(h) dh$. Writing $h = \left(\begin{smallmatrix} a & b \\ c & d\end{smallmatrix}\right)$ as usual, by virtue of \eqref{Iwdecomposition:B2} and \eqref{PS} we see that 
\[
 B_2(m) = p^{3m/2}(-1)^m\chi_{\psi}(p^m) \int_{\mathcal B_2(m)} e(h)(-d,p^m) \chi_{\psi}(-d)  \chi_{\delta}(d)|d|_p^{-3/2}\mathbf h_p(h) dh,
\]
where $e(h)$ is given by the recipe in \eqref{e-betam-2}. Let us denote by $\mathcal J_2(m)$ the last integral, and define 
\[
 \mathcal B_2^-(m) := \{h \in \mathcal B_2(m): \mathrm{ord}_p(d) \text{ odd}\}, \quad 
 \mathcal B_2^+(m) := \{h \in \mathcal B_2(m): \mathrm{ord}_p(d) \text{ even}\}.
\]
Then, by \eqref{e-betam-2} we have 
\[
 \mathcal J_2(m) = \int_{\mathcal B_2^-(m)} (c,-d)_p \chi_{\psi}(d) \chi_{\delta}(d)|d|_p^{-3/2}\mathbf h_p(h) dh  + \int_{\mathcal B_2^+(m)} \chi_{\psi}(d) \chi_{\delta}(d)|d|_p^{-3/2}\mathbf h_p(h)  dh.
\]

\begin{lemma}\label{lemma:B2}
 With the above notation, 
 \[
  B_2(m) = 
  \begin{cases}
  p^{1-3m/2}(-1)^{m+1}\chi_{\psi}(p^m)(1-p^{-1})p^{m} & \text{if } m>0,\\
  p^{3m/2}(-1)^{m+1}\chi_{\psi}(p^m)(1-p^{-1})(1+p^{-1}-p^{-m}) & \text{if } m\leq 0.
  \end{cases}
 \]
\end{lemma}
\begin{proof}
The proof goes along the same lines of the previous lemma, and hence we omit the computations. We just stress that when $m > 0$ the set $\mathcal B_2^-(m)$ is empty, and $\mathcal B_2^+(m)$ is the disjoint union of the sets $\mathcal L_j$ with $j \geq 2m$. And for $m\leq 0$, the set $\mathcal B_2^-(m)$ equals the disjoint union of the sets $\mathcal R_{2j+1}$ with $0\leq j\leq -m-1$, whereas $\mathcal B_2^+(m)$ is the disjoint union of the sets $\mathcal R_{2j}$ with $0\leq j \leq -m$.
\end{proof}


\begin{proposition}\label{prop:Phibetas-SL-St}
 With the above notation, $\Phi_{\mathbf h_p}(\beta_m) = p^{-|3m/2-1|} (-1)^{m+1} \chi_{\psi}(p^m)$.
\end{proposition}
\begin{proof}
Recalling that $\Phi_{\mathbf h_p}(\beta_m) = \langle \tilde{\pi}_p(\beta_m)\mathbf h_p, \mathbf h_p \rangle/||\mathbf h_p||^2$, and that $||\mathbf h_p||^2= p^{-1}(p^2-1)$ from Lemma \ref{lemma:norm-varphip}, the statement follows by combining Lemmas \ref{lemma:B1} and \ref{lemma:B2} since $\langle \tilde{\pi}_p(\beta_m)\mathbf h_p, \mathbf h_p \rangle=B_1(m)+B_2(m)$.
%
\end{proof}

\subsection{Weil pairings}

Finally, we also need to compute the Weil pairings 
\[
 \Phi_{\pmb{\phi}_p}(\alpha_n) = \langle \omega_{\overline{\psi}_p}(\alpha_n)\pmb{\phi}_p, \pmb{\phi}_p \rangle 
 \quad \text{and} \quad 
 \Phi_{\pmb{\phi}_p}(\beta_m) = \langle \omega_{\overline{\psi}_p}(\beta_m)\pmb{\phi}_p, \pmb{\phi}_p \rangle
\]
for $n, m \in \Z$. Notice that $||\pmb{\phi}_p||^2 = \int_{\Q_p} \mathbf 1_{\Z_p}(x) \overline{\mathbf 1_{\Z_p}(x)} dx = \mathrm{vol}(\Z_p) = 1$. The next statement is actually valid for all primes $p$, and we will use it later not only for primes $p\mid N/M$, but also for primes $p\mid M$.

\begin{proposition}\label{weilpairings}
 Let $p$ be a prime, and $\pmb{\phi}_p = \mathbf 1_{\Z_p} \in \mathcal S(\Q_p)$ be the characteristic function of $\Z_p$. With the above notation, if $n, m \in \Z$ we have 
 \[
  \frac{\Phi_{\pmb{\phi}_p}(\alpha_n)}{||\pmb{\phi}_p||^2} = \chi_{\overline{\psi}_p}(p^n)p^{-|n|/2}, \quad 
  \frac{\Phi_{\pmb{\phi}_p}(\beta_m)}{||\pmb{\phi}_p||^2} = \chi_{\overline{\psi}_p}(p^m)p^{-|m|/2}.
 \]
\end{proposition}
\begin{proof}
 This follows immediately from the definitions. 
\end{proof}

\subsection{Computation of $\mathcal I_p^{\sharp}(\mathbf h, \breve{\mathbf g}, \pmb{\phi})$}

Continue to fix a prime $p|N/M$. Recall from \eqref{Iv-alphav}, \eqref{alphav} that 
\[
  \mathcal I_p^{\sharp}  (\mathbf h, \breve{\mathbf g}, \pmb{\phi}) = \frac{L(1,\pi_v,\mathrm{ad})L(1,\tau_v,\mathrm{ad})}{L(1/2,\pi_v\times \mathrm{ad}\tau_v)} \alpha_p^{\sharp} (\mathbf h, \breve{\mathbf g}, \pmb{\phi}),
\]
where $\alpha_p^{\sharp} (\mathbf h, \breve{\mathbf g}, \pmb{\phi}) = \int_{\SL_2(\Q_p)} \Omega_p(g) dg$ and $\Omega_p(g) := \overline{\Phi_{\mathbf h_p}(g)}\Phi_{\breve{\mathbf g}_p}(g)\Phi_{\pmb{\phi}_p}(g)$ for $g \in \SL_2(\Q_p)$. By using the double coset decomposition of $\SL_2(\Q_p)$ as in \eqref{SL2-decomp-Gamma0}, we have 
\[
 \alpha_p^{\sharp} (\mathbf h, \breve{\mathbf g}, \pmb{\phi}) = 
 \sum_{n\in \Z} \Omega_p(\alpha_n)\mathrm{vol}(\Gamma_0\alpha_n\Gamma_0) + 
 \sum_{m\in \Z} \Omega_p(\beta_m)\mathrm{vol}(\Gamma_0\beta_m\Gamma_0).
\]

\begin{proposition}
 Let $p$ be a prime dividing $N/M$. Then 
 \[
  \alpha_p^{\sharp} (\mathbf h, \breve{\mathbf g}, \pmb{\phi}) = \frac{p-w_p}{p^2+w_p} \zeta_p(2)^{-1}.
 \]
\end{proposition}
\begin{proof}
 For $n=0$ we have $\Omega_p(\alpha_0)\mathrm{vol}(\Gamma_0\alpha_0\Gamma_0) = \mathrm{vol}(\Gamma_0) = p^{-1}(1-p^{-1})$. Let $n\neq 0$ be an integer. From Proposition \ref{weilpairings}, we have $\Phi_{\pmb{\phi}_p}(\alpha_n) = \chi_{\overline{\psi}_p}(p^n)p^{-|n|/2}$, whereas from Propositions \ref{prop:Phialphas-GL-St} and \ref{prop:Phialphas-SL-St} we have
\[
 \Phi_{\breve{\mathbf g}_p}(\alpha_n) = p^{-2|n|}, 
 \qquad
 \Phi_{\mathbf h_p}(\alpha_n) = p^{-3|n|/2}(-1)^n \chi_{\psi}(p^n).
\]
Since $\psi = \overline{\psi}_p^D$, one has $\chi_{\psi} = \chi_{\overline{\psi}_p}\cdot \chi_D$, and therefore $\chi_{\psi}(p^n) = (\frac{D}{p})^n \chi_{\overline{\psi}_p}(p^n) = w_p^n\chi_{\overline{\psi}_p}(p^n)$. From this, using the volumes from Lemma \ref{lemma:volumesGamma0}, we deduce that
\[
 \Omega_p(\alpha_n)\mathrm{vol}(\Gamma_0\alpha_n\Gamma_0) = (-w_p)^n p^{-2|n|-1}(1-p^{-1}),
\]
and therefore 
\[
  \sum_{n\in \Z} \Omega_p(\alpha_n)\mathrm{vol}(\Gamma_0\alpha_n\Gamma_0) = 
  p^{-1}(1-p^{-1})\left(1 + \sum_{n > 0} (-w_p p^{-2})^n + \sum_{n < 0} (-w_p p^{-2})^{-n}\right).
\]
The two geometric sums on the right hand side are the same, and equal $\frac{-w_p}{p^2+w_p}$. Hence, 
\begin{equation}\label{sumalphas}
 \sum_{n\in \Z} \Omega_p(\alpha_n)\mathrm{vol}(\Gamma_0\alpha_n\Gamma_0) = 
 p^{-1}(1-p^{-1})\left(1 - \frac{2w_p}{p^2+w_p} \right) = p^{-1}(1-p^{-1}) \frac{p^2-w_p}{p^2+w_p}.
\end{equation}

Besides, for $m=0$ we have $\Omega_p(\beta_0)\mathrm{vol}(\Gamma_0\beta_0\Gamma_0) = (-p^{-1})(-p^{-1})(1-p^{-1}) = p^{-2}(1-p^{-1})$. And if $m \neq 0$, we have again $\Phi_{\pmb{\phi}_p}(\beta_m) = \chi_{\overline{\psi}_p}(p^m)p^{-|m|/2}$, and Propositions \ref{prop:Phibetas-GL-St} and \ref{prop:Phibetas-SL-St} tell us that
\[
 \Phi_{\breve{\mathbf g}_p}(\beta_m) = -p^{-|2m-1|}, 
 \qquad
 \Phi_{\mathbf h_p}(\beta_m) =  p^{-|3m/2-1|} (-1)^{m+1} \chi_{\psi}(p^m).
\]
Using again that $\chi_{\psi}(p^m) = w_p^m \chi_{\overline{\psi}_p}(p^m)$, and the volumes from Lemma \ref{lemma:volumesGamma0}, it follows that when $m\neq 0$
\[
 \Omega_p(\beta_m)\mathrm{vol}(\Gamma_0\beta_m\Gamma_0) = 
 \begin{cases}
    p^{-2m}(-w_p)^m (1-p^{-1}) & \text{if } m > 0,\\
    p^{2m-2} (-w_p)^m (1-p^{-1}) & \text{if } m < 0.
 \end{cases}
\]
Summing up all the terms, one easily checks that 
\begin{equation}\label{sumbetas}
 \sum_{m\in \Z} \Omega_p(\beta_m)\mathrm{vol}(\Gamma_0\beta_m\Gamma_0) = (1-p^{-1})\frac{1-w_p}{p^2+w_p}.
\end{equation}

Finally, combining \eqref{sumalphas} and \eqref{sumbetas} we conclude that 
\[
 \alpha_p^{\sharp}(\mathbf h, \breve{\mathbf g}, \pmb{\phi}) = 
 \frac{1-p^{-1}}{p^2+w_p}(p^{-1}(p^2-w_p) + 1 - w_p) = \frac{p-1}{p^2}\cdot \frac{p^2+p(1-w_p)-w_p}{p^2+w_p} = \frac{p-w_p}{p^2+w_p} \zeta_p(2)^{-1}.
\]
\end{proof}

\begin{proposition}\label{prop:IN/M}
 Let $p$ be a prime dividing $N/M$. Then $\mathcal I_p^{\sharp} (\mathbf h, \breve{\mathbf g}, \pmb{\phi}) = p^{-1}$.
\end{proposition}
\begin{proof}
 Recall from the definition of $\mathcal I_p^{\sharp}(\mathbf h, \breve{\mathbf g}, \pmb{\phi})$ that
\[
\mathcal I_p^{\sharp}(\mathbf h, \breve{\mathbf g}, \pmb{\phi}) = \frac{L(1,\pi_p,\mathrm{ad})L(1,\tau_p,\mathrm{ad})}{L(1/2,\pi_p\times \mathrm{ad}\tau_p)}\alpha_p^{\sharp}(\mathbf h, \breve{\mathbf g}, \pmb{\phi}).
\]
Since $p\mid N/M$, both $\pi_p$ and $\tau_p$ are (unramified) special representations, say $\pi_p = \xi_1\mathrm{St}_p$, $\tau_p = \xi_2\mathrm{St}_p$. And recall that $\varepsilon(1/2,\pi_p) = - \xi_1(p)$ and similarly for $\tau_p$. Then (cf. \cite{JLbook} and \cite[Section 3]{Kudla-LocalLanglands}) we have 
\[
 L(1/2,\pi_p\otimes \tau_p\otimes \tau_p) = \frac{1}{(1+w_pp^{-2})(1+w_pp^{-1})^2} = \frac{p^4}{(p^2+w_p)(p+w_p)^2},
\]
where $w_p = \varepsilon(1/2,\pi_p)$. On the other hand, it is well-known that $L(1/2,\pi_p) = \frac{p}{p+w_p}$, so that (since $p\nmid M$)
\[
 L(1/2,\pi_p\times \mathrm{ad}\tau_p) = \frac{L(1/2,\pi_p\otimes \tau_p\otimes \tau_p)}{L(1/2,\pi_p)} = 
 \frac{p^3}{(p^2+w_p)(p+w_p)}.
\]
Also, from \cite[Section 10]{Hida-Galreps} (cf. also \cite{GelbartJacquet}), we have $L(1,\pi_p, \mathrm{ad}) = L(1,\tau_p,\mathrm{ad}) = \zeta_p(2)$. It thus follows from the previous proposition that
\[
\mathcal I_p^{\sharp}(\mathbf h, \breve{\mathbf g}, \pmb{\phi}) = 
\frac{\zeta_p(2)^{2}(p^2+w_p)(p+w_p)}{p^3} \frac{p-w_p}{\zeta_p(2)(p^2+w_p)} =  \zeta_p(2)\frac{(p+w_p)(p-w_p)}{p^3} = \frac{p^2(p^2-1)}{(p^2-1)p^3} = \frac{1}{p}.
\]
\end{proof}

\section{Computation of local integrals at primes $p \mid M$}\label{sec:periodsM}

In this section we focus on the regularized local periods $\mathcal I_p^{\sharp}(\mathbf h, \breve{\mathbf g}, \pmb{\phi})$ at primes $p \mid M$. First we will recall the local types of the representations $\tau$ and $\tilde{\pi}$ at such primes, and describe explicitly the local components $\breve{\mathbf g}_p$ and $\mathbf h_p$. After this, we will be concerned with the matrix coefficients $\langle \tau_p(g)\breve{\mathbf g}_p, \breve{\mathbf g}_p\rangle$ and $\langle \tilde{\pi}_p(g)\mathbf h_p, \mathbf h_p\rangle$, together with the Weil pairings $\langle \omega_{\overline{\psi}_p}(g)\pmb{\phi}_p, \pmb{\phi}_p\rangle$, for $g\in \SL_2(\Q_p)$, towards obtaining an explicit expression for $\mathcal I_p^{\sharp}(\mathbf h, \breve{\mathbf g}, \pmb{\phi})$. As in the previous section, for $g \in \SL_2(\Q_p)$ we write
\[
 \Phi_{\breve{\mathbf g}_p}(g) := \frac{\langle \tau_p(g)\breve{\mathbf g}_p, \breve{\mathbf g}_p\rangle}{||\breve{\mathbf g}_p||^2},  \quad \Phi_{\mathbf h_p}(g) := \frac{\langle \tilde{\pi}_p(g)\mathbf h_p, \mathbf h_p\rangle}{||\mathbf h_p||^2}, \quad \Phi_{\pmb{\phi}_p}(g) := \frac{\langle \omega_{\overline{\psi}_p}(g)\pmb{\phi}_p, \pmb{\phi}_p\rangle}{||\pmb{\phi}_p||^2}.
\]

When $p$ divides $M$, $\breve{\mathbf g}_p$ and $\mathbf h_p$ are not $\Gamma_{0}$-invariant, although our choice guarantees that for any element $g \in \SL_2(\Q_p)$, the product $\Omega_p(g) := \overline{\Phi_{\mathbf h_p}(g)}\Phi_{\mathbf g_p}(g)\Phi_{\pmb{\phi}_p}(g)$ depends only on the double coset $\Gamma_{00} g \Gamma_{00}$. Because of this, we will need to refine our decomposition of $\SL_2(\Q_p)$ in \eqref{SL2-decomp-Gamma0} into a decomposition in terms of double cosets for $\Gamma_{00}$. Then, we will need to compute the products $\Omega_p(g)$ for $g$ varying in a set of representatives for the double cosets for $\Gamma_{00}$ in $\SL_2(\Q_p)$. As we will see, many of the involved matrix coefficients vanish, so we will not need to compute all of them in order to obtain the quantities $\Omega_p(g)$. 

Fix a prime $p\mid M$ through all this section. As before, write $\psi_p$ for the $p$-th component of the standard additive character $\psi: \A/\Q \to \C^{\times}$.

\subsection{Local types and explicit description of test vectors}

As in last section, we start considering the $\GL_2$ case. Now $\tau_p$ is a (ramified) principal series. More precisely, $\tau_p = \pi(\xi_1,\xi_2)$ is the principal series representation induced by two characters $\xi_1,\xi_2: \Q_p^{\times} \to \C^{\times}$. In the induced model, this space is realized as the space of those functions $\varphi: \GL_2(\Q_p) \to \C$ such that 
\begin{equation}\label{rule:PS}
 \varphi\left(\left(\begin{smallmatrix}a & b \\ 0 & d\end{smallmatrix}\right)x\right) = \xi_1(a)\xi_2(d) \left|\frac{a}{d}\right|_p^{1/2} \varphi(x) \quad \text{for all } a, d \in \Q_p^{\times}, b \in \Q_p, x \in \GL_2(\Q_p).
\end{equation}
Because of our assumption that $N$ is square-free, we may assume that $\xi_1$ is unramified (hence $\xi_1(a)=1$ for all $a \in \Z_p^{\times}$) and $\xi_2$ is ramified of ($p$-power) conductor $1$ (meaning that $\xi_2(1+p\Z_p)=1$). Define 
\[
 K_0^1 = \left\lbrace \left(\begin{array}{cc} a & b \\ c & d\end{array}\right) \in \GL_2(\Z_p): c \equiv 0, d \equiv 1 \pmod p \right\rbrace,
\]
and notice that $\GL_2(\Z_p) = B(\Z_p)K_0^1 \sqcup B(\Z_p)wK_0^1$, where $B(\Z_p) = B(\Q_p)\cap \GL_2(\Z_p)$ is the subgroup of matrices of $\GL_2(\Z_p)$ which belong to the Borel subgroup of upper triangular matrices in $\GL_2(\Q_p)$, and $w = \left(\begin{smallmatrix} 0 & 1 \\ 1 & 0\end{smallmatrix}\right)$. In the induced model for $\tau_p = \pi(\xi_1,\xi_2)$, the subspace $\tau_p^{K_0^1} \subset \tau_p$ of vectors fixed by $K_0^1$ turns out to be one-dimensional, and a non-trivial $K_0^1$-invariant vector is described in \cite[Proposition 2.1.2]{SchmidtRemarksGL2}. Namely, the vector $\varphi_p: \GL_2(\Q_p) \to \C$ characterized by requiring that
\begin{equation}\label{newvector:PS}
 \varphi_p(h) = \begin{cases}
         \xi_1(p)^{-1}\xi_1(a)\xi_2(d)|ad^{-1}|^{1/2} & \text{if } h \in \left(\begin{smallmatrix} a & \ast \\ 0 & d\end{smallmatrix}\right)K_0^1, \quad a, d \in \Q_p^{\times}, \\
         0 & if h \not \in B(\Q_p)K_0^1.
        \end{cases}
\end{equation}
Since we are only interested in the normalized value $\Phi_{\breve{\mathbf g}_p}(g)$, we may assume that $\mathbf g_p$ coincides with the local vector $\varphi_p$ given by the above recipe. In contrast to the case of the previous section, however, when $p$ divides $M$ the new vector $\mathbf g_p$ does not ensure the non-vanishing of the local periods $\mathcal I_p$. To remedy this, we replace $\mathbf g_p$ by $\breve{\mathbf g}_p= \mathbf V_p \mathbf g_p \in \tau_p$, where $\mathbf V_p$ is the $p$-th level raising operator acting on $\tau_p$ by $\varphi \mapsto \tau_p(\varpi_p)\varphi$, with $\varpi_p = \left(\begin{smallmatrix} p^{-1} & 0 \\ 0 & 1 \end{smallmatrix}\right) \in \GL_2(\Q_p)$. The vector $\breve{\mathbf g}_p$ is no longer $K_0^1$-invariant, but defining  
\[
 K_{00}^1 = \left\lbrace \left(\begin{array}{cc} a & b \\ c & d\end{array}\right) \in \GL_2(\Z_p): c \equiv 0, d \equiv 1 \pmod{p^2} \right\rbrace,
\]
we have the following:

\begin{lemma}\label{lemma:brevegp}
 With the above notation, $\breve{\mathbf g}_p$ is $K_{00}^1$-invariant. Moreover, we have $||\breve{\mathbf g}_p||^2 = (p+1)^{-1}$.
\end{lemma}
\begin{proof}
 It is straightforward to check the first assertion, hence we focus in the computation of $||\breve{\mathbf g}_p||^2$. By definition, we have $||\breve{\mathbf g}_p||^2 = \langle\breve{\mathbf g}_p, \breve{\mathbf g}_p\rangle$. Using the decomposition $\GL_2(\Z_p) = K_0 \sqcup K_0wK_0$, we have  
\[
\langle \breve{\mathbf g}_p, \breve{\mathbf g}_p\rangle  = \int_{\GL_2(\Z_p)} \mathbf g_p(h\varpi_p) \overline{\mathbf g_p(h\varpi_p)} dh = \int_{K_0} \mathbf g_p(h \varpi_p)\overline{\mathbf g_p(h\varpi_p)} dh + \int_{K_0wK_0} \mathbf g_p(h \varpi_p) \overline{\mathbf g_p(h\varpi_p)} dh.
\]
We deal separately each of the two integrals. If $h = \left(\begin{smallmatrix} x & y \\ z & t \end{smallmatrix}\right) \in K_0$ is an arbitrary element in $K_0$, using Iwasawa decomposition in $K_0$, we can write
\[
 h\varpi_p = \left(\begin{array}{cc} xp^{-1} & y \\ zp^{-1} & t\end{array} \right) = 
 \left(\begin{array}{cc} p^{-1}t^{-1}\det(h) & y \\ 0 & t\end{array} \right)
 \left(\begin{array}{cc} 1 & 0 \\ t^{-1}p^{-1}z & 1\end{array} \right),
\]
where the rightmost element, call it $h_0$, belongs to $\GL_2(\Z_p)$. Moreover, if $h \in K_{00}$ then $h_0$ clearly belongs to $K_0^1$; and one can check that if $h \not \in K_{00}$ then $h_0 \not \in B(\Q_p)K_0^1$, where recall that $B(\Q_p)$ stands for the upper triangular Borel subgroup of $\GL_2(\Q_p)$. Thus, it follows from \eqref{newvector:PS} (recall that $\xi_1$ is unramified, $\xi_2$ is ramified of conductor $1$, and $\xi_1\xi_2 = \underline{\chi}_p$) that 
\[
 \mathbf g_p(h\varpi_p) = 
 \begin{cases}
      p^{1/2} \xi_1(p)^{-2}\underline{\chi}_p(h) & \text{if } h \in K_{00},\\
      0 & \text{if } h \not\in K_{00}.
 \end{cases}
\]

Now suppose that $h = \left(\begin{smallmatrix} x & y \\ z & t \end{smallmatrix}\right) \in K_0wK_0$ is an arbitrary element in $K_0wK_0$. Then, again using Iwasawa decomposition and noticing that $z \in \Z_p^{\times}$, we have
\[
 h\varpi_p = \left(\begin{array}{cc} xp^{-1} & y \\ zp^{-1} & t\end{array} \right) = 
 \left(\begin{array}{cc} -z^{-1}\det(h) & xp^{-1} \\ 0 & zp^{-1}\end{array} \right)
 \left(\begin{array}{cc} 0 & 1 \\ 1 & z^{-1}tp\end{array} \right).
\]
The rightmost element belongs to $\GL_2(\Z_p)$, but not to $K_0^1$. And further, it can be easily shown that it does not belong to $B(\Q_p)K_0^1$ either, hence by applying \eqref{newvector:PS} we deduce that $\mathbf g_p(h\varpi_p) = 0$  for all $h \in K_0wK_0$. Therefore, we conclude that $\langle \breve{\mathbf g}_p, \breve{\mathbf g}_p\rangle = p \mathrm{vol}(K_{00}) = (p+1)^{-1}$.
\end{proof}

\begin{rem}\label{rem:gpcoeffK00}
 The same arguments of the proof of the previous lemma show that, for an arbitrary $g \in \SL_2(\Q_p)$, we have 
 \[
  \langle \tau_p(g)\breve{\mathbf g}_p, \breve{\mathbf g}_p \rangle = 
  \int_{K_{00}} \mathbf g_p(h g \varpi_p) \overline{\mathbf g(h\varpi_p)} dh = 
  p^{1/2}\xi_1(p)^2 \int_{K_{00}} \mathbf g_p(h g \varpi_p)\underline{\chi}_p^{-1}(h) dh.
 \]
\end{rem}

Now we turn our attention to the representation $\tilde{\pi}_p$ of $\widetilde{\SL}_2(\Q_p)$. In order to lighten the notation, we will write as in the previous section $\psi = \overline{\psi}_p^D$, where $\overline{\psi}_p = \psi_p^{-1}$. Then, as explained in Section \ref{sec:BMexplained}, $\tilde{\pi}_p$ is the odd Weil representation $r_{\psi}^-$ (which is supercuspidal). The space of $r^{-}_{\psi}$ is the subspace of {\em odd} functions in $\mathcal S(\Q_p)$ (where $\Q_p$ is regarded as a quadratic space endowed with the bilinear form $(x,y) = 2xy$). Recall that the action of $\widetilde{\SL}_2(\Q_p)$ is determined by the following properties: if $\varphi \in \mathcal S(\Q_p)$ is odd, $a \in \Q_p^{\times}$, $x \in \Q_p^{\times}$, and we write $s = \left(\begin{smallmatrix} 0 & 1 \\ -1 & 0 \end{smallmatrix}\right)$, then
\begin{align*}\label{newvector:sl2weil}
 r^{-}_{\psi}\left[\left(\begin{array}{cc}a \\ & a^{-1}\end{array}\right),1\right] \varphi(x) & = |a|_p^{1/2}\chi_{\psi}(a) \varphi(ax),\\
 r^{-}_{\psi}\left[\left(\begin{array}{cc}1 & b \\ & 1\end{array}\right),1\right] \varphi(x) & = \psi(b x^2) \varphi(x), \\ 
 r^{-}_{\psi}\left[s,1\right] \varphi(x) & = \gamma(\psi)\int_{\Q_p} \varphi(y)\psi(2xy) dy.
\end{align*}
For our choice of $\psi$, we have $\gamma(\psi) = 1$, hence the third identity above simplifies to 
\[
 r^{-}_{\psi}\left[s,1\right] \varphi(x) = \int_{\Q_p} \varphi(y)\psi(2xy) dy.
\]

Let $\widetilde{\Gamma}_{00}$ denote the image of $\Gamma_{00}$ in $\widetilde{\SL}_2(\Z_p)$. The following is proved in \cite[Lemma 8.5]{BaruchMao}:

\begin{lemma}\label{sl2testvector_weil}
The space of vectors $\varphi_p$ in $r^{-}_{\psi}$ satisfying $r^{-}_{\psi} \varphi_p = \underline{\chi}_p (k) \varphi_p$ for all $k \in \widetilde{\Gamma}_{00}$ is one-dimensional, and it is generated by the function $\mathbf 1_{\Z_p^{\times}} \cdot \underline{\chi}_p^{-1}$.
\end{lemma}

The $p$-th component $\mathbf h_p$ of the vector $\mathbf h$ is thus a scalar multiple of the function given in the lemma. Since the matrix coefficients $\Phi_{\mathbf h_p}(g)$ are normalized so that they are invariant under replacing $\mathbf h_p$ by a scalar multiple, we will assume in the following that $\mathbf h_p = \mathbf 1_{\Z_p^{\times}} \cdot \underline{\chi}_p^{-1}$. Notice that 
\begin{equation}\label{norm:hp-M}
 ||\mathbf h_p||^2 = \int_{\Q_p} \mathbf h_p(x)\overline{\mathbf h_p(x)} dx = \int_{\Z_p^{\times}} dx = \text{vol}(\Z_p^{\times}) = 1-p^{-1}.
\end{equation}

Having described our choices for the $p$-th components $\breve{\mathbf g}_p$ and $\mathbf h_p$, we should note that by construction the function of $\SL_2(\Q_p)$ defined by $g \, \mapsto \, \Phi_{\breve{\mathbf g}_p}(g) \overline{\Phi_{\mathbf h_p}(g)}$ is $\Gamma_{00}$-biinvariant (and it is not $\Gamma_0$-biinvariant). Indeed, this follows immediately from the invariance properties of $\breve{\mathbf g}_p$ and $\mathbf h_p$, together with the properties of the matrix coefficients $\langle \tau_p(g) \breve{\mathbf g}_p, \breve{\mathbf g}_p \rangle$ and $\langle r_{\psi}^-(g) \mathbf h_p, \mathbf h_p \rangle$.

Finally, as in the previous section we write $\omega_p = \omega_{\overline{\psi}_p}$ for the Weil representation of $\widetilde{\SL}_2(\Q_p)$ acting on the space of Bruhat--Schwartz functions $\mathcal S(\Q_p)$, with respect to the character $\overline{\psi}_p = \psi_p^{-1}$. As before, by our choice of test vector we have $\pmb{\phi}_p = \mathbf 1_{\Z_p}$, and recall that Proposition \ref{weilpairings} continues to hold when $p|M$.

\subsection{Computation of $\alpha_p^{\sharp}(\mathbf h, \breve{\mathbf g}, \pmb{\phi})$ for $p \mid M$}

Recall from last section that 
\begin{equation}\label{SL2-decomp-K0}
 \SL_2(\Q_p) = \bigsqcup_{n\in \Z} \Gamma_0 \alpha_n \Gamma_0 \quad \sqcup \quad \bigsqcup_{n\in \Z} \Gamma_0\beta_n \Gamma_0.
\end{equation}

By the comment we have just made above, now one cannot compute $\alpha_p^{\sharp}(\mathbf h, \breve{\mathbf g}, \pmb{\phi})$ by only computing the matrix coefficients for $\breve{\mathbf g}_p$, $\mathbf h_p$, and $\pmb{\phi}_p$ at the elements $\alpha_n$ and $\beta_m$. However, starting from this decomposition we can refine it to obtain a decomposition in terms of double cosets for $\Gamma_{00}$. First of all, one might observe that $\Gamma_{00}$ is {\em not normal} in $\Gamma_0$. However, one has 
\[
 \Gamma_0 = \bigsqcup_{\gamma \in \Z_p/p\Z_p} \Gamma_{00} \nu_{\gamma} = \bigsqcup_{\gamma \in \Z_p/p\Z_p} \nu_{\gamma} \Gamma_{00}, \quad \nu_{\gamma} = \left(\begin{array}{cc}  1 & 0 \\ \gamma p & 1\end{array}\right) \in \SL_2(\Z_p),
\]
so that for each $n$ and $m$ we can write 
\begin{equation}\label{K0cosets}
 \Gamma_0\alpha_n\Gamma_0 = \bigcup_{\gamma,\delta \in \Z_p/p\Z_p} \Gamma_{00} \nu_{\gamma} \alpha_n \nu_{\delta} \Gamma_{00}, \quad 
 \Gamma_0\beta_m\Gamma_0 = \bigcup_{\gamma,\delta \in \Z_p/p\Z_p} \Gamma_{00} \nu_{\gamma} \beta_m \nu_{\delta} \Gamma_{00}.
\end{equation}
However, these unions are {\em not} disjoint. Nevertheless, it is not so difficult (although labourious) to reduce these expressions to disjoint unions, so that eventually one obtains a set of representatives for the double cosets for $\Gamma_{00}$ in $\SL_2(\Q_p)$. We describe this set in the following lemma, whose proof is skipped.

\begin{lemma}\label{SL2-decomposition-00}
Fix a non-quadratic residue $u \in \Z_p^{\times}$. Then
\begin{equation}
 \SL_2(\Q_p) = \bigsqcup_{r \in \mathcal R} \Gamma_{00} r \Gamma_{00},
\end{equation}
where the set $\mathcal R$ is the union of the following sets:
\begin{itemize}
 \item[I)] $\{1, \nu_1, \nu_u\}$;
 \item[II)] $\{\alpha_n, \alpha_n\nu_1, \alpha_n\nu_u: n > 0 \}$;
 \item[III)] $\{\alpha_n, \nu_1\alpha_n, \nu_u\alpha_n: n < 0 \}$;
 \item[IV)] $\{\beta_m: m \in \Z\}$;
 \item[V)] $\{\beta_m\nu_1, \beta_m\nu_u: m >0 \} \sqcup \{\nu_1\beta_m\nu_{\delta}, \nu_u\beta_m\nu_{\delta}: m>0, \delta \in \Z/p\Z\}$.
\end{itemize}
\end{lemma}

Now, we can finally proceed with the main goal of this section, namely the computation of $\alpha_p^{\sharp}(\mathbf h, \breve{\mathbf g}, \pmb{\phi})$, and therefore of the regularized local period $\mathcal I_p^{\sharp}(\mathbf h, \breve{\mathbf g}, \pmb{\phi})$. To ease the notation, for $g \in \SL_2(\Q_p)$ we will write $\Omega_p(g) := \overline{\Phi_{\mathbf h_p}(g)}\Phi_{\breve{\mathbf g}_p}(g)\Phi_{\pmb{\phi}_p}(g)$, so that 
\[
 \alpha_p^{\sharp}(\mathbf h, \breve{\mathbf g}, \pmb{\phi}) = \int_{\SL_2(\Q_p)} \Omega_p(g) dg.
\]
By our choice of $\mathbf h_p$, $\breve{\mathbf g}_p$, and $\pmb{\phi}_p$, we see that $\Omega_p(g)$ depends only on the double coset $\Gamma_{00} g \Gamma_{00}$. By using the decomposition explained in Lemma \ref{SL2-decomposition-00}, we see that 
\begin{equation}\label{alphap:pM}
 \alpha_p^{\sharp}(\mathbf h, \breve{\mathbf g}, \pmb{\phi}) = \sum_{r\in \mathcal R} \Omega_p(r) \mathrm{vol}(\Gamma_{00}r\Gamma_{00}).
\end{equation}
Therefore, we will proceed by computing $\Omega_p(r) \mathrm{vol}(\Gamma_{00}r\Gamma_{00})$ for each $r \in \mathcal R$. We will deal with the cases I) - V) listed in Lemma \ref{SL2-decomposition-00} one by one. First we will concentrate in computing $\Omega_p(r)$.

\subsubsection{Case I)}

We start computing $\Omega_{p}(\nu_{\gamma})$ for $\gamma \in \Z_p$. First of all we have the following vanishing statement for $\Phi_{\breve{\mathbf g}_p}$.

\begin{lemma}\label{lemma:case1}
 If $\gamma \in \Z_p^{\times}$, then $\Phi_{\breve{\mathbf g}_p}(\nu_{\gamma}) = 0$.
\end{lemma}
\begin{proof}
 Let $\gamma \in \Z_p^{\times}$. By Remark \ref{rem:gpcoeffK00}, we have  
 \[
  \langle \tau_p(\nu_{\gamma})\breve{\mathbf g}_p, \breve{\mathbf g}_p\rangle  = \int_{K_{00}} \mathbf g_p(h \nu_{\gamma} \varpi_p)\overline{\mathbf g_p(h\varpi_p)} dh.
 \]
 But for $h \in K_{00}$, an Iwasawa decomposition for $h\nu_{\gamma}\varpi_p$ reads
\[
 h\nu_{\gamma}\varpi_p = \left(\begin{array}{cc} p^{-1}t^{-1}\det(h) & y \\ 0 & t\end{array} \right)
 \left(\begin{array}{cc} 1 & 0 \\ t^{-1}p^{-1}z + \gamma & 1\end{array} \right).
\]
Under the assumption that $\gamma \in \Z_p^{\times}$, and taking into account that $h \in K_{00}$, we see from this identity that $h\nu_{\gamma}\varpi_p$ does not belong to $B(\Q_p)K_0^1$, thus it follows from \eqref{newvector:PS} that $\mathbf g_p(h \nu_{\gamma} \varpi_p) = 0$ for all $h \in K_{00}$, and hence the statement follows. 
\end{proof}

With this we can easily deduce $\Omega_p(r)$ for elements $r$ as in Case I).

\begin{proposition}\label{Psip:case1}
We have $\Omega_p(1) = 1$. And for $\gamma \in \Z_p^{\times}$, $\Omega_p(\nu_{\gamma}) = 0$.
\end{proposition}
\begin{proof}
Recall that $\Omega_p(g) := \overline{\Phi_{\mathbf h_p}(g)}\Phi_{\breve{\mathbf g}_p}(g)\Phi_{\pmb{\phi}_p}(g)$. By our normalization of $\Phi_{\mathbf h_p}$, $\Phi_{\breve{\mathbf g}_p}$, and $\Phi_{\pmb{\phi}_p}$, it is clear that $\Omega_p(1) = 1$. And for $g = \nu_{\gamma}$ with $\gamma \in \Z_p^{\times}$, the previous lemma implies that $\Omega_p(\nu_{\gamma}) = 0$.
\end{proof}

\subsubsection{Case II)}

Now we focus on elements of the form $\alpha_n\nu_c$, with $n > 0$ and $c \in \Z_p$. When $c \not\in \Z_p^{\times}$, $\nu_c \in K_{00}$ and therefore $\Omega_p(\alpha_n\nu_c) = \Omega_p(\alpha_n)$.  

\begin{lemma}\label{lemma:case2}
The following assertions hold:
\begin{itemize}
\item[i)] $\Phi_{\mathbf h_p}(\alpha_n) = 0$ for all $n>0$.
\item[ii)] If $n > 0$ and $c \in \Z_p^{\times}$, then $\Phi_{\mathbf h_p}(\alpha_n\nu_c)= 0$.
\end{itemize}
\end{lemma}
\begin{proof}
From one of the rules for the odd Weil representation $r_{\psi}^-$, we have 
\[
 (r^{-}_{\psi}(\alpha_n)\mathbf h_p)(x) = |p|_p^{n/2}\chi_{\psi}(p^n)\mathbf h_p(p^nx) = |p|_p^{n/2}\chi_{\psi}(p^n)\mathbf 1_{p^{-n}\Z_p^{\times}}(x)\underline{\chi}_p^{-1}(p^nx).
\]
Therefore, if $n \neq 0$ we find that
 \[
 \langle r^{-}_{\psi}(\alpha_n)\mathbf h_p, \mathbf h_p\rangle = |p|_p^{n/2}\chi_{\psi}(p^n) \int_{\Q_p} \mathbf 1_{p^{-n}\Z_p^{\times}}(x)\underline{\chi}_p^{-1}(p^nx)\mathbf 1_{\Z_p^{\times}}(x)\underline{\chi}_p(x)dx = 0,
 \]
and i) follows. Notice that the argument does not require $n>0$, but only $n \neq 0$. To show ii), notice that
 \[
  \nu_c = (-1)s\left(\begin{array}{cc} 1 & -cp\\ 0 & 1\end{array}\right)s.
 \]
By applying repeatedly the rules for the Weil representation to the elements on the right hand side of this identity, one arrives to
 \[
  r^{-}_{\psi}(\nu_c)\mathbf h_p(x) =  \int_{\Q_p} \psi(-2xy-cpy^2) \mathfrak G(2y,\underline{\chi}_p^{-1}) dy = \mathbf 1_{\Z_p}(x) \int_{\Q_p} \psi(-2xy-cpy^2) \mathfrak G(2y,\underline{\chi}_p^{-1}) dy,
 \] 
and applying $r^{-}_{\psi}(\alpha_n)$ to this expression we get 
\[
 r^{-}_{\psi}(\alpha_n\nu_c)\mathbf h_p(x) = \chi_{\psi}(p^n)p^{-n/2} \mathbf 1_{\Z_p}(xp^n) \int_{\Q_p} \psi(-2xp^ny-cpy^2)\mathfrak G(2y,\underline{\chi}_p^{-1}) dy.
\]
Completing the squares and with some elementary computation, we find more explicitly

\[
 r^{-}_{\psi}(\alpha_n\nu_c) \mathbf h_p(x) =  
 p^{(1-n)/2}\chi_{\psi}(p^n)\varepsilon(1/2,\underline{\chi}_p)\underline{\chi}_p(2)
 \mathbf 1_{p^{-n}\Z_p}(x) \psi\left(\frac{p^{2n}x^2}{cp}\right) \int_{\Z_p^{\times}} \psi\left(\frac{-c}{p}\left(y+\frac{p^nx}{c}\right)^2\right)\underline{\chi}_p(y) dy. 
\] 
Using that $\psi(\frac{p^{2n}x^2}{cp}) = 1$ and $\int_{\Z_p^{\times}} \psi(\frac{-c}{p}(y+\frac{p^nx}{c})^2)\underline{\chi}_p(y) dy = \int_{\Z_p^{\times}} \psi(\frac{-cy^2}{p})\underline{\chi}_p(y) dy$ for $n > 0$ and $x \in \Z_p^{\times}$, it follows that 
\[
\langle r^{-}_{\psi}(\alpha_n\nu_c)\mathbf h_p,\mathbf h_p\rangle = p^{(1-n)/2}\chi_{\psi}(p^n)\varepsilon(1/2,\underline{\chi}_p)\underline{\chi}_p(2) \int_{\Z_p^{\times}} \psi\left(\frac{-cy^2}{p}\right)\underline{\chi}_p(y) dy \int_{\Z_p^{\times}} \underline{\chi}_p(x) dx.
\]
But the last integral vanishes by orthogonality of Dirichlet characters, hence $\langle r^{-}_{\psi}(\alpha_n\nu_c)\mathbf h_p,\mathbf h_p\rangle = 0$, which gives $\Phi_{h_p}(\alpha_n\nu_c) = 0$ as we wanted to prove.
\end{proof}

As a consequence, we immediately have the next vanishing statement.

\begin{proposition}\label{Psip:case2}
For all $n>0$ and $c \in \Z_p$, $\Omega_p(\alpha_n\nu_c) = 0$.
\end{proposition}

\subsubsection{Case III)}

Now we consider the case of elements $\nu_c\alpha_n$ with $n < 0$ and $c\in \Z_p$.

\begin{lemma}\label{lemma:case3}
The following assertions hold:
\begin{itemize}
 \item[i)] $\Phi_{\mathbf h_p}(\alpha_n) = 0$ for all $n<0$.
 \item[ii)] If $n<0$ and $c \in \Z_p^{\times}$, then $\Phi_{\mathbf h_p}(\nu_c\alpha_n)= 0$.
\end{itemize}
\end{lemma}
\begin{proof}
The first assertion follows as in Lemma \ref{lemma:case2}, where we only used $n\neq 0$. To prove ii), observe that  
\[
 r^{-}_{\psi}(\alpha_n) \mathbf h_p(x) = \chi_{\psi}(p^n)p^{-n/2}\mathbf 1_{\Z_p^{\times}}(p^nx) \underline{\chi}_p^{-1}(p^nx) =  \chi_{\psi}(p^n)p^{-n/2}\mathbf 1_{p^{-n}\Z_p^{\times}}(x)\underline{\chi}_p^{-1}(x).
\]
Secondly, using again the decomposition $\nu_c = (-1)s\left(\begin{smallmatrix} 1 & -cp \\ 0 & 1\end{smallmatrix}\right) s$ as in the previous case, one finds
\[
 r^{-}_{\psi}(\nu_c\alpha_n)\mathbf h_p(x)= \chi_{\psi}(p^n)p^{-n/2} \int_{\Q_p} \psi(-2xp^nz-cp^{2n+1}z^2)\mathfrak G(2z,\underline{\chi}_p^{-1})dz.
\]
From this, we have  
\[
 \langle r^{-}_{\psi}(\nu_c\alpha_n)\mathbf h_p,\mathbf h_p\rangle  = \chi_{\psi}(p^n)p^{-n/2} \int_{\Z_p^{\times}} \left(
 \int_{\Q_p} \psi(-2xp^nz-cp^{2n+1}z^2)\mathfrak G(2z,\underline{\chi}_p^{-1})dz \right) \underline{\chi}_p(x) dx.
\]
We have $\mathfrak G(2z,\underline{\chi}_p^{-1}) = p^{-1/2}\varepsilon(1/2,\underline{\chi}_p)\underline{\chi}_p(2)\mathbf 1_{p^{-1}\Z_p^{\times}}(z)\underline{\chi}_p(z)$, thus the inner integral in the above expression for $\langle r^{-}_{\psi}(\nu_c\alpha_n)\mathbf h_p,\mathbf h_p\rangle$ equals 
 \[ 
 p^{1/2}\varepsilon(1/2,\underline{\chi}_p)\underline{\chi}_p(2) \int_{\Z_p^{\times}} \psi(-2xp^{n-1}z-cp^{2n-1}z^2)\underline{\chi}_p(z) dz.
 \]
By completing squares and plugging this in the expression for $\langle r^{-}_{\psi}(\nu_c\alpha_n)\mathbf h_p,\mathbf h_p\rangle $, one eventually obtains 
\[
 \langle r^{-}_{\psi}(\nu_c\alpha_n)\mathbf h_p,\mathbf h_p\rangle = \chi_{\psi}(p^n)p^{(1-n)/2}\varepsilon(1/2,\underline{\chi}_p)\underline{\chi}_p(2) \int_{\Z_p^{\times}} \psi(-cp^{2n-1}z^2) \underline{\chi}_p(z) \mathfrak G(-2zp^{n-1},\underline{\chi}_p) dz. 
\]
But notice that $-2zp^{n-1} \not\in p^{-1}\Z_p^{\times}$ for $z\in \Z_p^{\times}$ because $n<0$. Since $\underline{\chi}_p$ has conductor $1$, it follows that $\mathfrak G(-2zp^{n-1},\underline{\chi}_p) = 0$, and therefore $\langle r^{-}_{\psi}(\nu_c\alpha_n)\mathbf h_p,\mathbf h_p\rangle = 0$ as well, which implies $\Phi_{\mathbf h_p}(\nu_c\alpha_n) = 0$.
\end{proof}

\begin{proposition}\label{Psip:case3}
For all $n < 0$ and $c \in \Z_p$, $\Omega_p(\nu_c\alpha_n) = 0$.
\end{proposition}

\subsubsection{Case IV)}

Now we deal with Case IV) in Lemma \ref{SL2-decomposition-00}, consisting only of elements $\beta_m$ with $m \in \Z$.

\begin{lemma}\label{lemma:case4}
The following assertions hold:
\begin{itemize}
\item[i)] $\Phi_{\mathbf h_p}(\beta_m) = 0$ for all $m\neq 1$.
\item[ii)] $\Phi_{\breve{\mathbf g}_p}(\beta_1)=0$.
\end{itemize}
\end{lemma}

\begin{proof}
i) Since $\beta_m = s\alpha_m$, we compute $r^{-}_{\psi}(\beta_m)\mathbf h_p(x)$ by applying first $r^{-}_{\psi}(\alpha_m)$ and then $r^{-}_{\psi}(s)$. We have 
\[
 r^{-}_{\psi}(\alpha_m)\mathbf h_p(x) = \chi_{\psi}(p^m)p^{-m/2}\mathbf 1_{\Z_p^{\times}}(p^mx) \underline{\chi}_p^{-1}(p^mx) = \chi_{\psi}(p^m)p^{-m/2}\mathbf 1_{p^{-m}\Z_p^{\times}}(x) \underline{\chi}_p^{-1}(x),
\]
and therefore applying $r^{-}_{\psi}(s)$ gives
\[
 r^{-}_{\psi}(\beta_m)\mathbf h_p(x) = \chi_{\psi}(p^m)p^{-m/2} \int_{p^{-m}\Z_p^{\times}} \psi(2xy)\underline{\chi}_p^{-1}(y) dy = \chi_{\psi}(p^m)p^{m/2} \mathfrak G(2xp^{-m},\underline{\chi}_p^{-1}).
\]
 
From this, we have 
\[
\langle r^{-}_{\psi}(\beta_m)\mathbf h_p,\mathbf h_p\rangle  = \chi_{\psi}(p^m)p^{m/2} \int_{\Z_p^{\times}} \mathfrak G(2xp^{-m},\underline{\chi}_p^{-1})\underline{\chi}_p(x) dx.
\]
But $\mathfrak G(2xp^{-m},\underline{\chi}_p^{-1}) = p^{-1/2}\varepsilon(1/2,\underline{\chi}_p) \mathbf 1_{p^{m-1}\Z_p^{\times}}(x)\underline{\chi}_p(2x)$. If $m\neq 1$, we have $p^{m-1}\Z_p^{\times}\cap \Z_p^{\times} = \emptyset$, which implies that $\langle r^{-}_{\psi}(\beta_m)\mathbf h_p,\mathbf h_p\rangle = 0$ and therefore  $\Phi_{\mathbf h_p}(\beta_m) = 0$.
  
ii) If $h = \left(\begin{smallmatrix} x & y \\ z & t \end{smallmatrix}\right) \in K_{00}$ is an arbitrary element in $K_{00}$, an Iwasawa decomposition for $h\beta_1\varpi_p$ reads
\[
 h\beta_1\varpi_p = \left(\begin{array}{cc} -y & xp^{-1} \\ -t & zp^{-1} \end{array}\right) = 
 \left(\begin{array}{cc} t^{-1}\det(h) p^{-1} & -y \\ 0 & -t \end{array}\right)\left(\begin{array}{cc} 0 & 1 \\ 1 & -t^{-1}p^{-1}z \end{array}\right),
\]
and it is easy to check from this expression that $h\beta_1\varpi_p \not\in B(\Q_p)K_0^1$. Thus we have $\mathbf g_p(h\beta_1\varpi_p) = 0$ for all $h \in K_{00}$ by \eqref{newvector:PS}, and Remark \ref{rem:gpcoeffK00} implies that $\langle \tau_p(\beta_1)\breve{\mathbf g}_p, \breve{\mathbf g}_p\rangle = 0$. Therefore $\Phi_{\breve{\mathbf g}_p}(\beta_1)=0$ as well.
\end{proof}

Directly from the lemma, and the definition of $\Omega_p$, we deduce:

\begin{proposition}\label{Psip:case4} 
$\Omega_p(\beta_m) = 0$ for all $m \in \Z$.
\end{proposition}

\subsubsection{Case V)}

Finally, we consider the computation of matrix coefficients for elements in Case V) from Lemma \ref{SL2-decomposition-00}. We start considering the elements of the form $\beta_m \nu_{\delta}$ with $m>0$ and $\delta \in \Z_p^{\times}$. First of all, we note the following vanishing statement for $m>1$:

\begin{lemma}\label{case5:1}
If $\delta \in \Z_p^{\times}$, then $\Phi_{\mathbf h_p}(\beta_m\nu_{\delta}) = 0$ for all $m > 1$. 
\end{lemma}
\begin{proof}
As in previous lemmas, we have 
 \[
  r^{-}_{\psi}(\alpha_m\nu_{\delta})\mathbf h_p(x) = \chi_{\psi}(p^m)p^{-m/2} \mathbf 1_{p^{-m}\Z_p}(x) 
  \int_{\Q_p} \psi(-2xp^my - \delta p y^2)\mathfrak G(2y,\underline{\chi}_p^{-1}) dy.
 \]
 By applying $r^{-}_{\psi}(s)$, we deduce that 
 \begin{align*}
 r^{-}_{\psi}(\beta_m\nu_{\delta})\mathbf h_p(x) & =  \chi_{\psi}(p^m)p^{-m/2} \int_{p^{-m}\Z_p} \psi(2xz) 
  \left(\int_{\Q_p} \psi(-2zp^my - \delta p y^2)\mathfrak G(2y,\underline{\chi}_p^{-1}) dy \right) dz = \\
  & = \chi_{\psi}(p^m)p^{m/2} \int_{\Q_p} \psi(-\delta p y^2)\mathfrak G(2y,\underline{\chi}_p^{-1})\left( \int_{\Z_p} \psi(2(xp^{-m}-y)z) dz \right) dy.
 \end{align*}
 Using that $\mathfrak G(2y,\underline{\chi}_p^{-1}) = p^{-1/2}\varepsilon(1/2,\underline{\chi}_p)\underline{\chi}_p(2)\mathbf 1_{p^{-1}\Z_p^{\times}}(y)\underline{\chi}_p(y)$, one finds 
\[
  r^{-}_{\psi}(\beta_m\nu_{\delta})\mathbf h_p(x) = \chi_{\psi}(p^m)p^{(m+1)/2}\varepsilon(1/2,\underline{\chi}_p)\underline{\chi}_p(2) \int_{\Z_p^{\times}} \psi\left(\frac{-\delta y^2}{p}\right) \left( \int_{\Z_p} \psi\left(\frac{2(xp^{1-m}-y)z}{p}\right) dz \right) \underline{\chi}_p(y) dy.
 \]
 The inner integral vanishes unless $x \in p^{m-1}\Z_p^{\times}$ and $xp^{1-m} \equiv y \pmod p$, so one easily gets 
 \[
  r^{-}_{\psi}(\beta_m\nu_{\delta})\mathbf h_p(x) = \chi_{\psi}(p^m)p^{(m-1)/2}\varepsilon(1/2,\underline{\chi}_p)\underline{\chi}_p(2) \mathbf 1_{p^{m-1}\Z_p^{\times}}(x) \psi\left(\frac{-\delta x^2}{p^{2m-1}}\right)\underline{\chi}_p(x).
 \]
 From this, it follows easily that $\langle r^{-}_{\psi}(\beta_m\nu_{\delta})\mathbf h_p,\mathbf h_p\rangle =0$ for all $m>1$, since $\Z_p^{\times} \cap p^{m-1}\Z_p^{\times} = \emptyset$.
\end{proof}

It follows from this lemma that $\Omega_p(\beta_m\nu_{\delta}) = 0$ for all $\delta \in \Z_p^{\times}$ and $m>1$. So we are left with the case $m=1$. However, looking at $\Phi_{\breve{\mathbf g}_p}(\beta_1\nu_{\delta})$ we find:

\begin{lemma}\label{case5:2}
For all $\delta \in \Z_p^{\times}$, one has $\Phi_{\breve{\mathbf g}_p}(\beta_1\nu_{\delta}) = 0$.
\end{lemma}
\begin{proof}
Let $h = \left(\begin{smallmatrix} x & y \\ z & t \end{smallmatrix}\right) \in K_{00}$ be an arbitrary element. An Iwasawa decomposition for $h\beta_1\nu_{\delta}\varpi_p$ reads 
\[
 h\beta_1\nu_{\delta}\varpi_p = \left(\begin{array}{cc} t^{-1}\det(h) p^{-1} & \ast \\ 0 & -t \end{array}\right)\left(\begin{array}{cc} \delta & 1 \\ 1 - \delta t^{-1}p^{-1}z & -t^{-1}p^{-1}z \end{array}\right).
\]
Using this decomposition we see that $h\beta_1\nu_{\delta}\varpi_p \not\in B(\Q_p)K_0^1$. Thus it follows from \eqref{newvector:PS} that $\mathbf g_p(h\beta_1\nu_{\delta}\varpi_p) = 0$ for all $h \in K_{00}$, and hence by Remark \ref{rem:gpcoeffK00} $\langle \tau_p(\beta_1\nu_{\delta})\breve{\mathbf g}_p, \breve{\mathbf g}_p \rangle = 0$, which implies that $\Phi_{\breve{\mathbf g}_p}(\beta_1\nu_{\delta}) = 0$.
\end{proof}

Therefore, together with the above discussion we see that $\Omega_p(\beta_m\nu_{\delta}) = 0$ for all $m>0$ and all $\delta \in \Z_p^{\times}$. Let us next consider the elements of the form $\nu_{\gamma}\beta_m\nu_{\delta}$ when both $\gamma, \delta \in \Z_p^{\times}$. In the next lemma we will deal with matrix coefficients of the form $\langle \tau_p(\nu_{\gamma}\beta_m\nu_{\delta})\breve{\mathbf g}_p, \breve{\mathbf g}_p \rangle$, with $m>0$. To compute these, we need to find an Iwasawa decomposition for elements $h \nu_{\gamma} \beta_m \nu_{\delta} \varpi_p$ with $h \in K_{00}$ (cf. Remark \ref{rem:gpcoeffK00}). Using that $h \nu_{\gamma} \beta_m \nu_{\delta} \varpi_p = h \nu_{\gamma} \beta_m \varpi_p \left(\begin{smallmatrix}1 & 0 \\ \delta & 1\end{smallmatrix}\right)$, one can first determine an Iwasawa decomposition for 
\[
 h \nu_{\gamma} \beta_m \varpi_p = \left(\begin{array}{cc} -yp^{m-1} & xp^{-m} + \gamma yp^{1-m} \\ -tp^{m-1} & zp^{-m} + \gamma tp^{1-m}\end{array}\right),
\]
and then multiply it on the right by $\left(\begin{smallmatrix}1 & 0 \\ \delta & 1\end{smallmatrix}\right)$. This is how we proceed in the proof of the next lemma.

\begin{lemma}\label{case5:3}
 Let $m > 0$ be an integer, and let $\gamma, \delta \in \Z_p^{\times}$. Then $\Phi_{\breve{\mathbf g}_p}(\nu_{\gamma}\beta_m\nu_{\delta}) = 0$ unless $m = 1$ and $\gamma\delta \equiv 1 \pmod p$, and in that case one has $\Phi_{\breve{\mathbf g}_p}(\nu_{\gamma}\beta_1\nu_{\delta})= \underline{\chi}_p(\gamma)$.
\end{lemma}
\begin{proof}
Let us first consider the case $m=1$. In this case, if $h \in K_{00}$ then one has  
\[
 h\nu_{\gamma}\beta_1\nu_{\delta}\varpi_p = \left(\begin{array}{cc} t^{-1}p^{-1}\det(h) & -y \\ 0 & -t \end{array}\right)\left(\begin{array}{cc} \delta & 1 \\ 1 -\delta \eta(h) & -\eta(h) \end{array}\right),
\]
where $\eta(h) = t^{-1}p^{-1}z + \gamma \in \Z_p$. Since $h \in K_{00}$, observe that $\eta(h) \equiv \gamma \pmod p$, and therefore $1-\delta\eta(h) \equiv 1-\gamma\delta \pmod p$. By looking at the right hand side, it is easy to check that $h\nu_{\gamma}\beta_1\nu_{\delta}\varpi_p$ belongs to $B(\Q_p)K_0^1$ if and only if $\gamma\delta \equiv 1 \pmod p$. We deduce from \eqref{newvector:PS} that 
\[
 \mathbf g_p(h\nu_{\gamma}\beta_1\nu_{\delta}\varpi_p) = 
 \begin{cases}
 \xi_1(p)^{-2}\xi_2(\gamma t)|p^{-1}|_p^{1/2} = p^{1/2}\xi_1(p)^{-2}\underline{\chi}_p(\gamma)\underline{\chi}_p(h) & \text{if } \gamma\delta \equiv 1 \pmod p,\\
 0 & \text{otherwise}.
 \end{cases}
\]
In particular, it immediately follows from this that $\langle \tau_p(\nu_{\gamma}\beta_1\nu_{\delta})\breve{\mathbf g}_p, \breve{\mathbf g}_p \rangle = 0$ whenever $\gamma\delta \not\equiv 1 \pmod p$, and so $\Phi_{\breve{\mathbf g}_p}(\nu_{\gamma}\beta_1\nu_{\delta}) = 0$ as well in this case. When $\gamma \delta \equiv 1 \pmod p$, the above tells us that  
\[
 \langle \tau_p(\nu_{\gamma}\beta_1\nu_{\delta})\breve{\mathbf g}_p, \breve{\mathbf g}_p \rangle = p^{1/2}\xi_1(p)^2 \int_{K_{00}} \mathbf g_p(h \nu_{\gamma}\beta_1\nu_{\delta} \varpi_p) \underline{\chi}_p^{-1}(h) dh  = p \underline{\chi}_p(\gamma)\mathrm{vol}(K_{00}) = \frac{\underline{\chi}_p(\gamma)}{p+1}.
\]
Dividing by $||\breve{\mathbf g}_p||^2 = (p+1)^{-1}$ (cf. Lemma \ref{lemma:brevegp}), we get $\Phi_{\breve{\mathbf g}_p}(\nu_{\gamma}\beta_1\nu_{\delta}) = \underline{\chi}_p(\gamma)$ when $\gamma\delta \equiv 1 \pmod p$.

Now suppose that $m > 1$. Then an Iwasawa decomposition for $h \nu_{\gamma} \beta_m \nu_{\delta}\varpi_p$ reads 
\[
 h \nu_{\gamma} \beta_m \nu_{\delta}\varpi_p = \left(\begin{array}{cc} p^{m-2}t^{-1}\det(h)\eta(h)^{-1} & xp^{-m}+\gamma yp^{1-m} \\ 0 & p^{1-m}t\eta(h) \end{array}\right)\left(\begin{array}{cc} 1 & 0 \\ \delta -\eta(h)^{-1}p^{2m-2} & 1\end{array}\right),
\]
where $\eta(h) = t^{-1}p^{-1}z + \gamma$ as before. Since $h \in K_{00}$, we have $\eta(h)\in \Z_p^{\times}$, and since we are now assuming $m>1$, we see that $\delta -\eta(h)^{-1}p^{2m-2} \in \Z_p^{\times}$ because $\delta \in \Z_p^{\times}$. Therefore, the rightmost element in the above identity does not belong to $K_0^1$, and it is not difficult to see that it does not belong to $B(\Q_p)K_0^1$ either. Hence, by applying \eqref{rule:PS} we deduce that $\mathbf g_p(h \nu_{\gamma} \beta_m \nu_{\delta}\varpi_p) = 0$ for all $h \in K_{00}$. It follows that $\langle \tau_p(\nu_{\gamma}\beta_m\nu_{\delta})\breve{\mathbf g}_p,\breve{\mathbf g}_p\rangle = 0$, and hence $\Phi_{\breve{\mathbf g}_p}(\nu_{\gamma}\beta_m\nu_{\delta}) = 0$ as well. 
\end{proof}

By virtue of the last lemma, for any integer $m>0$ we have $\Omega_p(\nu_{\gamma}\beta_m\nu_{\delta}) = 0$ unless $m=1$ and $\gamma\delta \equiv 1 \pmod p$. Therefore, we only need to compute $\Phi_{\mathbf h_p}(\nu_{\gamma}\beta_m\nu_{\delta})$ for $m=1$ and $\gamma,\delta \in \Z_p^{\times}$ with $\gamma\delta \equiv 1 \pmod p$. Further, we see from Case V) in Lemma \ref{SL2-decomposition-00} that the latter condition is only satisfied for the representatives $\nu_1\beta_1\nu_1$ and $\nu_u\beta_1\nu_{u^{-1}}$. That is to say, there will be only two elements $r$ arising in Case V) for which $\Omega_p(r)$ might be non-zero. The next lemma addresses the computation of $\Phi_{\mathbf h_p}(\nu_{\gamma}\beta_1\nu_{\gamma^{-1}})$ for $\gamma \in \Z_p^{\times}$; for $\gamma = 1, u$ it thus provides the values that we are looking for. In the proof we will need to deal with certain quadratic Gauss sums. 
For each $a \in \Q_p^{\times}$, we write $\mathcal G(a) := \int_{\Z_p} \psi(az^2) dz$. Here, recall that $\psi = \overline{\psi}_p^D = \psi_p^{-D}$, and that $D \in \Z_p^{\times}$ is a square thanks to hypothesis \eqref{Hyp2}, together with the fact that $(\frac{D}{p}) = -w_p$. In particular, if $a \in \Z_p^{\times}$, then we have $\mathcal G(a/p) = p^{-1} G(a,p)$, where $G(a,p)$ denotes the usual quadratic Gauss sum defined by 
\[
 G(a,p) = \sum_{j=0}^{p-1} \zeta^{ax^2}, \quad \text{with } \zeta = e^{2\pi \sqrt{-1}/p}.
\]
Recall that one has $G(a,p) = \left(\frac{a}{p}\right)G(1,p)$, and $G(1,p)^2 = \left(\frac{-1}{p}\right)p$. 

\begin{lemma}\label{case5:4}
Let $\gamma \in \Z_p^{\times}$, and put $\nu = \nu_{\gamma}$ and $\nu' = \nu_{\gamma^{-1}}$. Then 
\[
 \Phi_{\mathbf h_p}(\nu \beta_1 \nu')= \frac{\chi_{\psi}(p)\underline{\chi}_p(\gamma) p^{-1/2}}{p-1}(p - G(-\gamma,p)).
\]
\end{lemma}
\begin{proof}
From previous lemmas, we know that 
\[
 r^{-}_{\psi}(\alpha_1\nu')\mathbf h_p(x) = \chi_{\psi}(p) p^{-1/2} \mathbf 1_{p^{-1}\Z_p}(x) \int_{\Q_p}\psi(-2xyp - \gamma^{-1} y^2 p) \mathfrak G(2y,\underline{\chi}_p^{-1}) dy.
\]
Observing that 
\[
 \nu\beta_1\nu' = (-1)s\left(\begin{array}{cc} 1 & -\gamma p\\ 0 & 1\end{array}\right)s s \alpha_1 \nu' = s\left(\begin{array}{cc} 1 & -\gamma p\\ 0 & 1\end{array}\right) \alpha_1 \nu'
\]
we compute $(r_{\psi}^-(\nu\beta_1\nu')\mathbf h_p)(x)$ by applying $r_{\psi}^-\left(s\left(\begin{smallmatrix} 1 & -\gamma p\\ 0 & 1\end{smallmatrix}\right)\right)$ to the above expression. This gives
\[
 r^{-}_{\psi}(\nu \beta_1 \nu') \mathbf h_p(x) = \chi_{\psi}(p) p^{1/2} \int_{\Z_p} \psi\left(\frac{z(2x-\gamma z)}{p}\right)\left( \int_{\Q_p}\psi(-y(2z+\gamma^{-1}yp)) \mathfrak G(2y,\underline{\chi}_p^{-1}) dy \right) dz.
\]
Using that $\mathfrak G(2y,\underline{\chi}_p^{-1}) = p^{-1/2}\varepsilon(1/2,\underline{\chi}_p)\underline{\chi}_p(2)\mathbf 1_{p^{-1}\Z_p^{\times}}(y)\underline{\chi}_p(y)$, the inner integral becomes 
\[
 p^{1/2}\varepsilon(1/2,\underline{\chi}_p)\underline{\chi}_p(2) \int_{\Z_p^{\times}} \psi\left(\frac{-y(2z+\gamma^{-1}y)}{p}\right)\underline{\chi}_p(y) dy.
\]
Setting $C = p \chi_{\psi}(p) \varepsilon(1/2,\underline{\chi}_p) \underline{\chi}_p(2)$, one can then rewrite $r^{-}_{\psi}(\nu \beta_1 \nu')\mathbf h_p(x)$ as
\[
 r^{-}_{\psi}(\nu \beta_1 \nu')\mathbf h_p(x) = C \int_{\Z_p^{\times}} \psi\left(\frac{-\gamma^{-1} y^2}{p}\right)\underline{\chi}_p(y) \left(\int_{\Z_p} \psi\left(\frac{2xz-2zy-\gamma z^2}{p}\right) dz\right) dy.
\]
Completing squares, the inner integral is essentially a (translated) quadratic Gauss sum, and it is not difficult to see that the above expression can be simplified to  
\[
 r^{-}_{\psi}(\nu \beta_1 \nu')\mathbf h_p(x) = C \mathcal G\left(\frac{-\gamma}{p}\right) \mathbf 1_{\Z_p}(x) \psi\left(\frac{x^2}{\gamma p}\right) \mathfrak G\left(\frac{-2x}{\gamma p},\underline{\chi}_p\right).
\]
Replacing $\mathfrak G\left(\frac{-2x}{\gamma p},\underline{\chi}_p\right)$ and $C$ by their value, we eventually find that 
\[
 r^{-}_{\psi}(\nu \beta_1 \nu')\mathbf h_p(x) = \chi_{\psi}(p) p^{1/2} \underline{\chi}_p(\gamma)\mathcal G\left(\frac{-\gamma}{p}\right)\mathbf 1_{\Z_p^{\times}}(x) \psi\left(\frac{x^2}{\gamma p}\right) \underline{\chi}_p^{-1}(x),
\]
and using this expression, we deduce that
\begin{align*}
 \langle r^{-}_{\psi}(\nu \beta_1 \nu')\mathbf h_p,\mathbf h_p\rangle & = 
 \chi_{\psi}(p)p^{1/2}\underline{\chi}_p(\gamma)\mathcal G\left(\frac{-\gamma}{p}\right) \int_{\Z_p^{\times}} \psi\left(\frac{x^2}{\gamma p}\right) dx = \chi_{\psi}(p)p^{1/2}\underline{\chi}_p(\gamma)\mathcal G\left(\frac{-\gamma}{p}\right)\left(\mathcal G\left(\frac{1}{\gamma p}\right) - \frac{1}{p}\right).
\end{align*}
Dividing by $||\mathbf h_p||^2 = 1-p^{-1} = p^{-1}(p-1)$, we obtain 
\[
 \Phi_{\mathbf h_p}(\nu \beta_1\nu') = \frac{\chi_{\psi}(p)\underline{\chi}_p(\gamma) p^{3/2}}{p-1} \cdot \mathcal G\left(\frac{-\gamma}{p}\right)\left(\mathcal G\left(\frac{\gamma^{-1}}{p}\right) - \frac{1}{p}\right).
\]
Finally, using that $\mathcal G\left(\frac{-\gamma}{p}\right) = \frac{1}{p}\left(\frac{-\gamma}{p}\right)G(1,p)$, and $\mathcal G\left(\frac{\gamma^{-1}}{p}\right) = \frac{1}{p}\left(\frac{\gamma}{p}\right)G(1,p)$, we conclude that 
\[
 \Phi_{\mathbf h_p}(\nu \beta_1\nu') = \frac{\chi_{\psi}(p)\underline{\chi}_p(\gamma) p^{-1/2}}{p-1}(p - G(-\gamma,p)).
\]
\end{proof}

Combining the previous four lemmas, we summarize our discussion for Case V): 

\begin{proposition}\label{Psip:case5}
Let $m>0$ be an integer, and let $\gamma, \delta \in \Z_p^{\times}$. Then $\Omega_p(\beta_m\nu_{\delta}) = 0$, and 
\[
\Omega_p(\nu_{\gamma}\beta_m\nu_{\delta}) = 
\begin{cases}
  \frac{p - G(\gamma,p)}{p(p-1)}  & \text{if } m = 1, \, \gamma\delta \equiv 1 \pmod p,\\ 
   0 & \text{otherwise}.
  \end{cases}
 \]
\end{proposition}
\begin{proof}
 The first assertion follows immediately from Lemmas \ref{case5:1} and \ref{case5:2}. As for the second one, combining Lemmas \ref{case5:3} and \ref{case5:4} yields that $\Omega_p(\nu_{\gamma}\beta_m\nu_{\delta})=0$ unless $m=1$ and $\gamma\delta \equiv 1 \pmod p$. And in that case, the same lemmas together with Proposition \ref{weilpairings} tell us that 
 \[
   \Phi_{\mathbf h_p}(\nu_{\gamma} \beta_1\nu_{\delta}) = \frac{\chi_{\psi}(p)\underline{\chi}_p(\gamma) p^{-1/2}}{p-1}(p - G(-\gamma,p)), \quad \Phi_{\pmb{\phi}_p}(\nu_{\gamma} \beta_1\nu_{\delta}) = \chi_{\overline{\psi}_p}(p)p^{-1/2}, \quad \Phi_{\breve{\mathbf g}_p}(\nu_{\gamma}\beta_1\nu_{\delta}) = \underline{\chi}_p(\gamma).
 \]
 Notice that $\chi_{\psi}(p) = \chi_{\overline{\psi}_p}(p)\chi_D(p) = \chi_{\overline{\psi}_p}(p)$ because $D \in \Z_p^{\times}$ is a square. Therefore, 
 \[
  \Omega_p(\nu_{\gamma} \beta_1\nu_{\delta}) = \overline{\Phi_{\mathbf h_p}(\nu_{\gamma} \beta_1\nu_{\delta})}\Phi_{\breve{\mathbf g}_p}(\nu_{\gamma}\beta_1\nu_{\delta})\Phi_{\pmb{\phi}_p}(\nu_{\gamma} \beta_1\nu_{\delta}) = \frac{p - \overline{G(-\gamma,p)}}{p(p-1)} = \frac{p - G(\gamma,p)}{p(p-1)}.
 \]
\end{proof}

\subsection{Computation of $\mathcal I_p^{\sharp}(\mathbf h, \breve{\mathbf g}, \pmb{\phi})$}

Finally, we conclude with the computation of the regularized local periods $\mathcal I_p^{\sharp}(\mathbf h, \breve{\mathbf g}, \pmb{\phi})$ at primes $p\mid M$. Recall from \eqref{alphap:pM} that 
\[
  \alpha_p^{\sharp}(\mathbf h, \breve{\mathbf g}, \pmb{\phi}) = \sum_{r\in \mathcal R} \Omega_p(r) \mathrm{vol}(\Gamma_{00}r\Gamma_{00}),
\]
where $\mathcal R$ is the set described in Lemma \ref{SL2-decomposition-00}. However, by Propositions \ref{Psip:case1}, \ref{Psip:case2}, \ref{Psip:case3}, \ref{Psip:case4}, \ref{Psip:case5},  
\begin{equation}\label{alphap:pM:easy}
 \alpha_p^{\sharp}(\mathbf h, \breve{\mathbf g}, \pmb{\phi}) = \Omega_p(1)\mathrm{vol}(\Gamma_{00}) + \Omega_p(r_1)\mathrm{vol}(\Gamma_{00}r_1\Gamma_{00}) + \Omega_p(r_u)\mathrm{vol}(\Gamma_{00}r_u\Gamma_{00}),
\end{equation}
where $u \in \Z_p^{\times}$ is a fixed non-quadratic residue, and $r_{\gamma} = \nu_{\gamma}\beta_1\nu_{\gamma^{-1}}$ for $\gamma = 1,u$. It is not hard to see that $\mathrm{vol}(\Gamma_{00}) = p^{-3}(p-1)$, and that 
\[
 \mathrm{vol}(\Gamma_{00}r_1\Gamma_{00}) = \mathrm{vol}(\Gamma_{00}r_u\Gamma_{00}) = \frac{p^{-3}(p-1)^2}{2}.
\]

\begin{proposition}
Let $p$ be a prime dividing $M$. Then  $\alpha_p^{\sharp}(\mathbf h, \breve{\mathbf g}, \pmb{\phi}) = 2p^{-3}(p-1)$.
\end{proposition}
\begin{proof}
 We just need to compute the three terms on the right hand side of \eqref{alphap:pM:easy}. Clearly, we have $\Omega_p(1)\mathrm{vol}(\Gamma_{00}) = \mathrm{vol}(\Gamma_{00}) = p^{-3}(p-1)$. As for the other two terms, it follows from Proposition \ref{Psip:case5} that 
\[
 \Omega_p(r_1)\mathrm{vol}(\Gamma_{00}r_1\Gamma_{00}) = \frac{p - G(1,p)}{p(p-1)}\mathrm{vol}(\Gamma_{00}r_1\Gamma_{00}), \quad
 \Omega_p(r_u)\mathrm{vol}(\Gamma_{00}r_u\Gamma_{00}) = \frac{p - G(u,p)}{p(p-1)}\mathrm{vol}(\Gamma_{00}r_u\Gamma_{00}).
\]
Therefore, the sum of these two equals
\[
 \frac{p^{-3}(p-1)}{2p} \left(2p - G(1,p) - G(u,p) \right) = p^{-3}(p-1),
\]
since $G(u,p) = -G(1,p)$. It follows that $\alpha_p^{\sharp}(\mathbf h, \breve{\mathbf g}, \pmb{\phi}) = 2p^{-3}(p-1)$.
\end{proof}

\begin{proposition}\label{prop:IM}
 Let $p$ be a prime dividing $M$. Then  
 \[
  \mathcal I_p^{\sharp}(\mathbf h, \breve{\mathbf g}, \pmb{\phi}) = \frac{2}{p(p+1)}.
 \]
\end{proposition}
\begin{proof}
From the definition of the regularized local period, we have 
\[
\mathcal I_p^{\sharp}(\mathbf h, \breve{\mathbf g}, \pmb{\phi}) = \frac{L(1,\pi_p,\mathrm{ad})L(1,\tau_p,\mathrm{ad})}{\zeta_{\Q_p}(2)L(1/2,\pi_p\times \mathrm{ad}\tau_p)}\alpha_p^{\sharp}(\mathbf h, \breve{\mathbf g}, \pmb{\phi}).
\]
As in the case where $p$ divides $N/M$, $\pi_p$ is an unramified Steinberg representation, so that we have again 
\[
  L(1,\pi_p, \mathrm{ad}) = \frac{p^2}{p^2-1}, \quad L(1/2,\pi_p) = \frac{1}{(1+w_p p^{-1})} = \frac{p}{p+w_p} = \frac{p}{p - 1},
\]
where now we are using that $w_p = -1$ by hypothesis \eqref{Hyp2}. In contrast, as already explained, $\tau_p = \pi(\xi_1,\xi_2)$ is a (ramified) principal series representation induced by a pair of characters $\xi_1, \xi_2: \Q_p^{\times} \to \C^{\times}$ (with $\xi_1$ unramified and $\xi_2$ of conductor $1$). In this case, from \cite[Section 10]{Hida-Galreps} (or \cite{GelbartJacquet}) we have 
\[
   L(1,\tau_p, \mathrm{ad}) = \frac{1}{1-p^{-1}} = \frac{p}{p-1}.
\]
Besides, one can now check using \cite{JLbook} and \cite[Section 3]{Kudla-LocalLanglands} that
\[
 L(1/2, \pi_p \otimes \tau_p \otimes \tau_p^{\vee}) = L(1/2,\pi_p)^2 = \frac{p^2}{(p-1)^2}.
\]
Therefore, $L(1/2,\pi_p\times \mathrm{ad}\tau_p) = \frac{p}{(p-1)}$ and the statement follows. 
\end{proof}

\section{Computation of local integrals at $p=2$ and $v=\infty$}

Let $f \in S_{2k}(N)$, $g \in S_{k+1}(N,\chi)$, and $h \in S_{k+1/2}^+(4N,\chi)$ be as usual, and let $\mathbf f$, $\mathbf g$, and $\mathbf h$ be their adelizations. Write $\pi$, $\tau$ and $\tilde{\pi}$ for the corresponding automorphic representations as in Section \ref{sec:proof}. Continue to consider our test vector $\mathbf h \otimes \breve{\mathbf g} \otimes \pmb{\phi} \in \tilde{\pi} \otimes \tau \otimes \omega$, where $\omega = \omega_{\overline{\psi}}$. Recal that $\psi$ denotes the standard additive character of $\A/\Q$, and $\overline{\psi} = \psi^{-1}$.

\subsection{Computation at $p=2$}\label{sec:p=2}

For simplicity, in the following we write $\psi$ for the local additive character $\overline{\psi}_2 = \psi_2^{-1}$ of $\Q_2$, so that $\omega_{2} = \omega_{\psi}$. Write also $t(2) = \left(\begin{smallmatrix} 2 & 0 \\ 0 & 2^{-1} \end{smallmatrix}\right) \in \SL_2(\Q_2)$, which we identify with $[t(2),1] \in \widetilde{\SL}_2(\Q_2)$, and recall that $\breve{\mathbf g}_2 = \mathbf g_2^{\sharp} = \tau_2(t(2)^{-1})\mathbf g_2$. Then define 
\[
 \mathbf h_2^{(2)} := \tilde{\pi}_2(t(2)) \mathbf h_2, \quad \pmb{\phi}_2^{(2)} := 2^{1/2}\omega_2(t(2)) \pmb{\phi}_2.
\]

\begin{lemma}
 With the above notation, we have $\pmb{\phi}_2^{(2)} = \mathbf 1_{\frac{1}{2}\Z_2}$. Furthermore, the following identities hold: 
 \[
  \langle \mathbf h_2^{(2)}, \mathbf h_2^{(2)}\rangle = \langle \mathbf h_2,\mathbf h_2 \rangle, 
  \quad \langle \breve{\mathbf g}_2, \breve{\mathbf g}_2 \rangle = \langle \mathbf g_2, \mathbf g_2 \rangle,
  \quad \langle \pmb{\phi}_2^{(2)}, \pmb{\phi}_2^{(2)}\rangle = 2 \langle \pmb{\phi}_2, \pmb{\phi}_2\rangle.
 \]
\end{lemma}
\begin{proof}
 Recall that $\pmb{\phi}_2 = \mathbf 1_{\Z_2}$. Then by applying the rules of the Weil representation we have 
\[
 \omega_2(t(2))\pmb{\phi}_2 = \omega_2(t(2))\mathbf 1_{\Z_2} = \chi_{\psi}(2)|2|_2^{1/2} \mathbf 1_{\frac{1}{2}\Z_2}.
\]
One can check that $\chi_{\psi}(2) = 1$ for our choice of $\psi$, and hence $\omega_2(t(2))\pmb{\phi}_2 = 2^{-1/2} \mathbf 1_{\frac{1}{2}\Z_2}$. It follows that $\pmb{\phi}_2^{(2)} = \mathbf 1_{\frac{1}{2}\Z_2}$ as stated. Furthermore, by definition we have 
\[
 \langle \pmb{\phi}_2^{(2)}, \pmb{\phi}_2^{(2)} \rangle = \int_{\Q_p} \mathbf 1_{\frac{1}{2}\Z_2}(x)\overline{\mathbf 1_{\frac{1}{2}\Z_2}(x)} dx = \mathrm{vol}(2^{-1}\Z_2) = 2 = 2\mathrm{vol}(\Z_2) = 2\langle \pmb{\phi}_2, \pmb{\phi}_2\rangle.
\]
As for the other identities, we show only the one concerning $\mathbf g$ (the other one can be dealt with similarly). First of all, $\mathbf g_2$ is a $\GL_2(\Z_2)$-fixed vector in the unramified principal series representacion $\tau_2 = \pi(\xi,\xi^{-1})$, where $\xi: \Q_2^{\times} \to \C^{\times}$ is an unramified character. The space of such vectors is known to be one-dimensional, and a non-trivial choice is given by the function $\varphi_2: \GL_2(\Z_2) \to \C$ defined by (cf. \cite[Section 2]{SchmidtRemarksGL2})
\begin{equation}\label{local-g2}
 \varphi_2(\gamma) = 
 \begin{cases}
  \xi(ad^{-1}) |ad^{-1}|_2^{1/2} & \text{if } \gamma \in \left(\begin{smallmatrix} a & \ast \\ 0 & d\end{smallmatrix}\right)\GL_2(\Z_2), \quad a, d \in \Q_2^{\times},\\
  0 & \text{otherwise.} 
 \end{cases}
\end{equation}
Since we are only interested in the ratio $\langle \breve{\mathbf g}_2, \breve{\mathbf g}_2 \rangle / \langle \mathbf g_2, \mathbf g_2 \rangle$, we may assume that $\mathbf g_2 = \varphi_2$. Then, one has $\langle \mathbf g_2, \mathbf g_2 \rangle = \int_{\GL_2(\Z_2)}\mathbf g_2(x) \overline{\mathbf g_2(x)} dx = \mathrm{vol}(\GL_2(\Z_2)) = 1$. To compute 
\[
 \langle \breve{\mathbf g}_2, \breve{\mathbf g}_2 \rangle = \langle \mathbf g_2^{\sharp}, \mathbf g_2^{\sharp} \rangle = \int_{\GL_2(\Z_2)}\mathbf g_2(xt(2^{-1})) \overline{\mathbf g_2(xt(2^{-1}))} dx,
\]
we divide $\GL_2(\Z_2)$ into three subregions, namely 
\[
 L_0 := \left\lbrace x = \left(\begin{smallmatrix} a & b \\ c & d\end{smallmatrix}\right)\in \GL_2(\Z_2): c \in \Z_2^{\times} \right\rbrace, \quad 
 L_1 := \left\lbrace x = \left(\begin{smallmatrix} a & b \\ c & d\end{smallmatrix}\right)\in \GL_2(\Z_2): c \in 2\Z_2^{\times} \right\rbrace,
\]
and $L_2 := \GL_2(\Z_2) - L_0 - L_1 = K_0(4)$. Working separately each of these regions, finding an Iwasawa decomposition for $xt(2^{-1})$ one checks that 
\[
 \mathbf g_2(xt(2^{-1})) \overline{\mathbf g_2(xt(2^{-1}))} = 
 \begin{cases}
  2^{-2} & \text{if } x \in L_0, \\
  1 & \text{if } x \in L_1, \\
  2^{2} & \text{if } x \in L_2.
 \end{cases}
\]
Therefore, $\langle \breve{\mathbf g}_2, \breve{\mathbf g}_2 \rangle = 2^{-4}\mathrm{vol}(L_0) + \mathrm{vol}(L_1) + 2^4\mathrm{vol}(L_2)$. Since $L_0 = \GL_2(\Z_2) - K_0(2)$ and $L_2 = K_0(4)$, and $K_0(2)$ (resp. $K_0(4)$) has index $3$ (resp. $6$) in $\GL_2(\Z_2)$, one can easily compute  $\langle \breve{\mathbf g}_2, \breve{\mathbf g}_2 \rangle = 1 = \langle \mathbf g_2, \mathbf g_2 \rangle$.
\end{proof}

The next proposition computes the local regularized period $\mathcal I_2^{\sharp}(\mathbf h,\breve{\mathbf g},\pmb{\phi})$, by relating it to the period $\mathcal I_2^{\sharp}(\mathbf h^{(2)}, \mathbf g, \pmb{\phi}^{(2)})$ and invoking the computation done in \cite[Section 6]{Xue-FourierJacobiSK}.

\begin{proposition}\label{prop:I2}
 $\mathcal I_2^{\sharp}(\mathbf h,\breve{\mathbf g},\pmb{\phi}) = 1$.
\end{proposition}
\begin{proof}
 Recall that by definition we have 
 \[
  \mathcal I_2^{\sharp}(\mathbf h,\breve{\mathbf g},\pmb{\phi}) = 
  \frac{\mathcal I_2(\mathbf h,\breve{\mathbf g},\pmb{\phi})}{\langle \mathbf h_2, \mathbf h_2\rangle\langle \breve{\mathbf g}_2, \breve{\mathbf g}_2\rangle\langle \pmb{\phi}_2, \pmb{\phi}_2\rangle},
 \]
 where 
 \[
 \mathcal I_2(\mathbf h, \breve{\mathbf g}, \pmb{\phi}) = \frac{L(1,\pi_2,\mathrm{ad})L(1,\tau_2,\mathrm{ad})}{L(1/2,\pi_2\times \mathrm{ad}\tau_2)} 
\int_{\SL_2(\Q_2)} \overline{\langle\tilde{\pi}_2(g)\mathbf h_2, \mathbf h_2\rangle} \langle\tau_2(g)\breve{\mathbf g}_2,\breve{\mathbf g}_2 \rangle \langle \omega_2(g)\pmb{\phi}_2,\pmb{\phi}_2 \rangle dg.
\]
Denote by $\alpha_2(\mathbf h_2,\breve{\mathbf g}_2,\pmb{\phi}_2)$ the integral on the right hand side. By replacing $g$ with $t(2)^{-1}gt(2)$, one checks
\[
 \alpha_2(\mathbf h_2,\breve{\mathbf g}_2,\pmb{\phi}_2) = \alpha_2(\tilde{\pi}_2(t(2))\mathbf h_2, \tau_2(t(2))\breve{\mathbf g}_2,\omega_2(t(2))\pmb{\phi}_2) = \alpha_2(\tilde{\pi}_2(t(2))\mathbf h_2, \mathbf g_2,\omega_2(t(2))\pmb{\phi}_2),
\]
using that $\breve{\mathbf g}_2 = \mathbf g_2^{\sharp} = \tau_2(t(2)^{-1})\mathbf g_2$. Since $\tilde{\pi}_2(t(2))\mathbf h_2 = \mathbf h_2^{(2)}$ and $\omega_2(t(2))\pmb{\phi}_2 = 2^{-1/2}\pmb{\phi}_2^{(2)}$ by definition, the above shows that $\mathcal I_2(\mathbf h, \breve{\mathbf g}, \pmb{\phi}) = 2^{-1} \mathcal I_2(\mathbf h^{(2)}, \mathbf g, \pmb{\phi}^{(2)})$. Hence, by the previous lemma, 
\[
 \mathcal I_2^{\sharp}(\mathbf h,\breve{\mathbf g},\pmb{\phi}) = 
  \frac{\mathcal I_2(\mathbf h,\breve{\mathbf g},\pmb{\phi})}{\langle \mathbf h_2, \mathbf h_2\rangle\langle \breve{\mathbf g}_2, \breve{\mathbf g}_2\rangle\langle \pmb{\phi}_2, \pmb{\phi}_2\rangle} = 
  \frac{2^{-1}\mathcal I_2(\mathbf h^{(2)}, \mathbf g, \pmb{\phi}^{(2)})}{\langle \mathbf h_2^{(2)}, \mathbf h_2^{(2)}\rangle \langle \mathbf g_2, \mathbf g_2 \rangle 2^{-1}\langle \pmb{\phi}_2^{(2)}, \pmb{\phi}_2^{(2)}\rangle} = \mathcal I_2^{\sharp}(\mathbf h^{(2)},\mathbf g,\pmb{\phi}^{(2)}).
\]
Finally, it follows from Xue's computation in \cite[Section 6]{Xue-FourierJacobiSK} that the right hand side equals $1$.
\end{proof}

\subsection{Computation at the archimedean place}\label{sec:infty}

Finally, we deal with the computation of the regularized local period $\mathcal I_{\infty}^{\sharp}(\mathbf h,\breve{\mathbf g},\pmb{\phi})$ at the real place $v = \infty$. The approach we follow here has already been considered in \cite{HangXue-generalizedIchino} (where the case that $g$ has weight $\ell+1$ with $\ell \geq k$ odd is also covered). For simplicity, in what follows we will write $\psi = \psi_{\infty}^{-1}$ for the twist of the standard additive character $\psi_{\infty}$ on $\R$ by $-1$, so that $\psi(x) = e^{-2\pi\sqrt{-1}x}$ for all $x \in \R$. Then $\omega_{\infty} = \omega_{\psi}$. By Iwasawa decomposition, every element $g \in \SL_2(\R)$ can be written as 
\[
 g = \left(\begin{array}{cc} y & 0 \\ 0 & y^{-1}\end{array}\right)\left(\begin{array}{cc} 1 & x \\ 0 & 1\end{array}\right) k
\]
for some $y \in \R_{>0}$, $x \in \R$ and $k \in \mathrm{SO}(2)$. We consider the Haar measure $dg = y^{-2}dxdydk$, where $dx$ and $dy$ are the Lebesgue measure on $\R$, and $dk$ is the Haar measure on $\mathrm{SO}(2)$ with $\mathrm{vol}(\mathrm{SO}(2))=\pi$.

Observe that $\tau_{\infty}$ (resp. $\pi_{\infty}$) is a discrete series representation of $\PGL_2(\R)$ of weight $k+1$ (resp. $2k$). The archimedean component $\breve{\mathbf g}_{\infty} =\mathbf g_{\infty}$ of $\breve{\mathbf g}$ is a lowest weight vector in $\tau_{\infty}$. Similarly, $\tilde{\pi}_{\infty}$ is a discrete series representation of $\widetilde{\SL}_2(\R)$ of lowest $K$-type $k+1/2$, and $\mathbf h_{\infty}$ is a lowest weight vector in $\tilde{\pi}_{\infty}$.

Let $J$ be the Jacobi group, which arises as the semidirect product of $\SL_2$ with the so-called Heisenberg group $H$, and it can be realized as a subgroup of $\mathrm{Sp}_4$ (see \cite[Section 1.1]{BerndtSchmidt}). In explicit terms, elements in $J$ can be written as products  
\[
\left(\begin{array}{cc} a & b \\ c & d\end{array}\right)(\lambda, \mu, \xi) = 
\left(\begin{array}{cccc} a & & b \\ & 1 \\ c & & d \\ & & & 1\end{array}\right)
\left(\begin{array}{cccc} 1 & & & \mu \\ \lambda & 1 & \mu & \xi \\ & & 1 & -\lambda \\ & & & 1\end{array}\right), 
\quad \left(\begin{array}{cccc} a & & b \\ & 1 \\ c & & d \\ & & & 1\end{array}\right) \in \SL_2, (\lambda,\mu,\xi) \in H.
\]

By virtue of \cite[Theorem 7.3.3]{BerndtSchmidt}, $\tilde{\pi}_{\infty} \otimes \omega_{\infty}$ is isomorphic to a discrete series representation $\rho_{\infty}$ of $J(\R)$ of lowest $K$-type $k+1$. In particular, the vector $\mathbf h_{\infty} \otimes \pmb{\phi}_{\infty} \in \tilde{\pi}_{\infty} \otimes \omega_{\infty}$ is then identified under the previous isomorphism with a lowest weight vector $\mathbf J_{\infty} \in \rho_{\infty}$. Being the isomorphism $\tilde{\pi}_{\infty} \otimes \omega_{\infty} \simeq \rho_{\infty}$ an isometry (see loc. cit.), we have 
\begin{align*}
\alpha_{\infty}^{\sharp}(\mathbf h,\breve{\mathbf g},\pmb{\phi}) & = 
\int_{\SL_2(\R)} 
\frac{\langle \tau(g)\breve{\mathbf g}_{\infty},\breve{\mathbf g}_{\infty}\rangle}{||\breve{\mathbf g}_{\infty}||^2} 
\frac{\overline{\langle \tilde{\pi}(g)\mathbf h_{\infty},\mathbf h_{\infty} \rangle}}{||\mathbf h_{\infty}||^2}
\frac{\langle \omega_{\infty}(g)\pmb{\phi}_{\infty}, \pmb{\phi}_{\infty}\rangle}{||\pmb{\phi}_{\infty}||^2} dg = \\
& = \int_{\SL_2(\R)} 
\frac{\langle \tau(g)\breve{\mathbf g}_{\infty},\breve{\mathbf g}_{\infty}\rangle}{||\breve{\mathbf g}_{\infty}||^2} 
\frac{\overline{\langle \rho(g)\mathbf J_{\infty},\mathbf J_{\infty} \rangle}}{||\mathbf J_{\infty}||^2} dg =: \alpha_{\infty}^{\sharp}(\breve{\mathbf g}_{\infty},\mathbf J_{\infty}).
\end{align*}

To compute $\alpha_{\infty}^{\sharp}(\mathbf h,\breve{\mathbf g},\pmb{\phi})$, we will consider the explicit model $D(k+1,N)$ of the discrete series representation $\rho_{\infty}$ that can be found in \cite[Chapter 3]{BerndtSchmidt}. As vector spaces, one has  
\[
 D(k+1,N) = \bigoplus_{\substack{m,\ell \geq 0},\\ \ell \text{ even}} \C \cdot v_{m,\ell},
\]
and $\mathrm{SO}_2(\R)$ acts on $v_{m,\ell}$ through the character $u \mapsto u^{k+1+m+\ell}$. The element $v_{0,0}$ is a lowest weight vector, and $\mathrm{SO}_2(\R)$ acts on the line spanned by $v_{0,0}$ through the character $u \mapsto u^{k+1}$. Let $\mathfrak r$ be the Lie algebra of $R(\R)$, and denote by $\mathfrak r_{\C}$ its complexification. Then there are certain operators $X_+$, $X_-$, $Y_+$, $Y_-$ acting on $\mathfrak r_{\C}$ (see loc. cit.). One has that $\mathfrak{sl}_2$ is a Lie subalgebra of $\mathfrak r$, and $X_{\pm} \in \mathfrak{sl}_{2,\C}$. Also, one has $\mathrm d\rho_{\infty} X_- \mathbf J_{\infty} = \mathrm d\rho_{\infty}Y_- \mathbf J_{\infty} = 0$. The action of these operators is given by the following recipe:
\begin{align*}
 \mathrm d\rho_{\infty}Y_+ v_{m,\ell} & = v_{m+1,\ell}, \quad \mathrm d\rho_{\infty}X_+ v_{m,\ell} = -\frac{1}{2\pi N} v_{m+2,\ell}, \\
 \mathrm d\rho_{\infty}Y_- v_{m,\ell} & = -2\pi N m v_{m-1,\ell}, \quad \mathrm d\rho_{\infty}X_- v_{m,\ell} = \pi N m(m-1) v_{m-2,\ell} - \frac{\ell}{4}(2k+\ell-1)v_{m,\ell-2}.
\end{align*}

The space $D(k+1,N)$ is endowed with an inner product $\langle \, , \, \rangle$, and the vectors $v_{m,\ell}$ form an orthogonal basis with respect to this inner product. We put $||v||^2 = \langle v, v \rangle$. For each pair of integers $m,\ell \geq 0$ with $\ell$ even, one can compute $||v_{m,\ell}||^2$ in a recursive manner from $||v_0||^2$, where we abbreviate $v_0:=v_{0,0}$. We shall normalize the inner product $\langle \, , \, \rangle$ so that $||v_0||^2 = 1$.

\begin{lemma}
 With the above notation, $\alpha_{\infty}^{\sharp}(\mathbf h,\breve{\mathbf g},\pmb{\phi}) = 2\pi^2 k^{-1}$.
\end{lemma}
\begin{proof}
With respect to the above model, $\tau_{\infty}$ might be realized as a subrepresentation of $\rho_{\infty}|_{\SL_2(\R)}$, spanned by $v_0$. Then we can assume the inner product for $\tau_{\infty}$ to be given by the restriction of the inner product for $\rho_{\infty}$. As $\alpha_{\infty}^{\sharp}(\breve{\mathbf g}_{\infty},\mathbf J_{\infty})$ is normalized so that it is invariant under replacing $\breve{\mathbf g}_{\infty}$ and $\mathbf J_{\infty}$ by scalar multiples of them, we can therefore assume that $\breve{\mathbf g}_{\infty} = \mathbf J_{\infty} = v_0$. Then we have 
\[
 \alpha_{\infty}^{\sharp}(\mathbf h,\breve{\mathbf g},\pmb{\phi}) = \alpha_{\infty}^{\sharp}(\breve{\mathbf g}_{\infty},\mathbf J_{\infty}) = \frac{1}{||v_0||^4} \int_{\SL_2(\R)} |\langle \tau_{\infty}(g)v_0, v_0\rangle|^2 dg = \int_{\SL_2(\R)} |\langle \tau_{\infty}(g)v_0, v_0\rangle|^2 dg.
\]

Write $A^+ := \left\lbrace \left(\begin{smallmatrix} e^t & \\ & e^{-t}\end{smallmatrix}\right): t \geq 0\right\rbrace$,
and consider the map 
\begin{align*}
(\SO_2(\R) \times A^+ \times \SO_2(\R))/\{\pm 1\} \, & \longrightarrow \, \SL_2(\R)/\{\pm 1\}, \\
\left(k, \left(\begin{array}{cc} e^t & \\ & e^{-t}\end{array}\right), k' \right) & \, \longmapsto \, 
k \left(\begin{array}{cc} e^t & \\ & e^{-t}\end{array}\right) k',
\end{align*}
where on the left hand side $-1$ is identified with the element $(-1,1,-1)$. This map is bijective outside the boundary of $A^+$, by virtue of Cartan decomposition. Using a similar argument to the one in \cite[Section 12]{IchinoIkeda-periods}, one deduces that $dg = 2 \mathrm{sinh}(2t)dtdkdk'$, where $dk$ and $dk'$ are the Haar measure as above for which $\SO_2(\R)$ has volume $\pi$, and $dt$ is the Lebesgue measure. It is well-known (cf. \cite{Knapp}) that 
\[
 \left\langle \tau_{\infty}\left(\left(\begin{array}{cc} e^t & \\ & e^{-t}\end{array}\right)\right)v_0, v_0 \right \rangle = \mathrm{cosh}(t)^{-(k+1)},
\]
and hence it follows that  
\begin{align*}
 \alpha_{\infty}^{\sharp}(\mathbf h, \breve{\mathbf g}, \pmb{\phi}) & = \int_{\SL_2(\R)} |\langle \tau_{\infty}(g)v_0,v_0\rangle|^2 dg = \mathrm{vol}(\SO_2(\R))^2 \int_0^{\infty} \mathrm{cosh}(t)^{-2(k+1)} 2 \mathrm{sinh}(2t)dt = \\
 & = 2\pi^2 \int_0^{\infty} \mathrm{cosh}(t)^{-2(k+1)}\mathrm{sinh}(2t)dt = 2\pi^2 k^{-1}.
\end{align*}
\end{proof}

\begin{proposition}\label{prop:Ireal}
 We have $\mathcal I_{\infty}^{\sharp}(\mathbf h,\breve{\mathbf g},\pmb{\phi}) = 1$.
\end{proposition}
\begin{proof}
We have seen in the previous lemma that  $\alpha_{\infty}^{\sharp}(\mathbf h,\breve{\mathbf g},\pmb{\phi}) = 2\pi^2 k^{-1}$. Besides, we have 
\[
 \frac{L(1,\pi_{\infty},\mathrm{Ad})L(1,\tau_{\infty}, \mathrm{Ad})}{L(1/2,\pi_{\infty}\times \mathrm{Ad}(g))} = 
 \frac{2(2\pi)^{-k-1}\Gamma(k+1)\pi^{-1}\Gamma(1)\cdot 2(2\pi)^{-2k}\Gamma(2k)\pi^{-1}\Gamma(1)}{2^2(2\pi)^{-2k-1}\Gamma(2k)\Gamma(k-k+1)\cdot 2(2\pi)^{-k}\Gamma(k)} = \frac{k}{2\pi^2},
\]
and thus it follows from the definition of $\mathcal I_{\infty}^{\sharp}(\mathbf h,\breve{\mathbf g},\pmb{\phi})$ that 
\[
\mathcal I_{\infty}^{\sharp}(\mathbf h,\breve{\mathbf g},\pmb{\phi}) = 
\frac{L(1,\pi_{\infty},\mathrm{Ad})L(1,\tau_{\infty}, \mathrm{Ad})}{L(1/2,\pi_{\infty}\times \mathrm{Ad}(g))}\alpha_{\infty}^{\sharp}(\mathbf h,\breve{\mathbf g},\pmb{\phi}) = 1.
\]
\end{proof}

\bibliographystyle{alpha}

\end{document}